\documentclass[a4paper,10pt]{amsart}

\usepackage{verbatim}
\usepackage{amssymb}
\usepackage{amsbsy}
\usepackage{amscd}
\usepackage{amsmath}
\usepackage{amsthm}
\usepackage[mathscr]{eucal}

\newtheorem{theorem}{Theorem}\numberwithin{theorem}{section}

\newtheorem{defn}[theorem]{Definition}

\newtheorem{ex}[theorem]{Example}

\newtheorem{prop}[theorem]{Proposition}

\newtheorem{lem}[theorem]{Lemma} 

\newtheorem{cor}[theorem]{Corollary}

\newtheorem{rem}[theorem]{Remark}

\def\Im{\operatorname{Im}}
\def\Hom{\operatorname{Hom}}

\def\L{\mathbb L}

\def\b{{\bf b}}
\def\x{{\bf x}}
\def\y{{\bf y}}
\def\pf{{\;+_{_F}\:}}

\newcommand{\N}{\mathbb{N}} 
\newcommand{\Q}{\mathbb{Q}}
\newcommand{\Z}{\mathbb{Z}}

\newcommand{\C}{\mathbb{C}}

\newcommand{\wh}{\widehat}

\newcommand{\hooklongrightarrow}{\lhook\joinrel\longrightarrow}
\newcommand{\twoheadlongrightarrow}{\relbar\joinrel\twoheadrightarrow}

\title%[Generalized (co)homology of  the loop spaces]
{
  Generalized (co)homology of  the loop spaces of classical groups   and 
  the universal factorial Schur $P$- and $Q$-functions   
}

\author{Masaki Nakagawa   and Hiroshi Naruse}

\thanks{The  first author is partially supported by the Grant-in-Aid for Scientific  Research 
          (C)  24540105, Japan Society for  the Promotion of Science}
\thanks{The second author is partially supported by the Grant-in-Aid for Scientific  Research 
          (C)  25400041, Japan Society for  the Promotion of Science}
\address{Department of  Education \endgraf
                       Okayama University \endgraf
                       Okayama  700-8530 \\ Japan 
}

\email{nakagawa@okayama-u.ac.jp}
\email{rdcv1654@okayama-u.ac.jp} 



\begin{document}

\begin{abstract}
 In this paper, we study the generalized (co)homology 
         Hopf algebras of 
        the loop spaces on the infinite classical groups, 
        generalizing  the work due to Kono-Kozima and Clarke. 
        We shall give a description of these Hopf algebras   in terms 
        of symmetric functions.   
        Based on   topological considerations  in the first half of this paper, 
        we then introduce a {\it universal} analogue of the factorial Schur $P$- 
        and $Q$-functions due to Ivanov and Ikeda-Naruse. 
        %Using the multiparameter analogues of the Cauchy identity, 
        %we define the universal factorial Schur $P$- and $Q$-functions. 
        We   investigate  various properties  of these functions such as the 
        {\it cancellation property}, which we call the {\it $\mathbb{L}$-supersymmetric property}, 
        the {\it factorization property}, and  the {\it vanishing property}. 
        We prove that the universal analogue of the Schur $P$-functions form 
        a formal basis for the ring of  symmetric functions with  the $\mathbb{L}$-supersymmetric property. 
        By using the  universal analogue of the Cauchy identity, we then 
        define the {\it dual} universal Schur $P$- and $Q$-functions.  
         We describe  the duality of these functions in terms of Hopf algebras.  
        %some properties of these dual Schur functions 
\end{abstract}

\maketitle
\tableofcontents

%%%%%%%%%%%%%%%%%%%%%%%%%%%%%%%%%%%%%%%%%%%%%%%%%%%%%%%%%%%%%%%%%%%%%%%%%%%%%%%%%%
\section{Introduction}     \label{sec:Introduction}  
Let $SU = SU(\infty)$, $Sp = Sp(\infty)$, and $SO = SO(\infty)$ be the 
infinite special unitary,  symplectic, and special orthogonal group 
respectively, and $\Omega SU$, $\Omega Sp$, and $\Omega_{0} SO$ 
denote its based loop space ($\Omega_{0}$ means the connected component 
of the identity).  These spaces have natural (homotopy-commutative) 
H-space structure 
given by the usual loop  multiplication.  
On the other hand, it is well-known that 
these loop spaces have no torsion  and no odd degree elements 
in (co)homology with integer coefficients (see e.g., 
the classical work due to 
Bott \cite{Bot54}, \cite{Bot56}%\footnote{
%Bott showed these facts by an application of the Morse theory. 
%On the other hand, by the result of Garland-Raghunathan \cite[Corollary 1.7]{Gar-Rag75} 
%and an unpublished work of Quillen (see also Mitchell \cite[Theorems 1.1 and 1.2]{Mit86}, 
%\cite[Theorems 1.1 and 1.4]{Mit87}, \cite[Theorem 4.2]{Mit88}), it is known that 
%the {\it affine Grassmannian} $\mathrm{Gr}_{G}$ for a simply-connected 
%simple complex algebraic group $G$ is homotopy equivalent to 
%the based loop space $\Omega K$ on a maximal compact subgroup $K$ of $G$. 
%The affine Grassmannian $\mathrm{Gr}_{G}$ has a natural {\it Schubert cell decomposition} 
%with {\it even}  dimensions
%(see e.g., Garland-Raghunathan \cite[Theorem 1.9]{Gar-Rag75}, 
%Mitchell  \cite[Theorem 1.3]{Mit86}, \cite[Theorem 1.2]{Mit87}, \cite[Corollary 3.2]{Mit88}). 
%From this, the above-mentioned property on the (co)homology of loop spaces 
%also follows. 
%}
). 
Therefore  the  H-space structure on $\Omega SU$ (resp. $\Omega Sp$, $\Omega_{0} SO$) 
 endows $H_{*}(\Omega SU)$ 
and $H^{*}(\Omega SU)$  (resp. $H_{*}(\Omega Sp)$ and $H^{*}(\Omega Sp$), 
$H_{*}(\Omega_{0} SO)$ and $H^{*}(\Omega_{0} SO)$) 
with the structure of dual Hopf algebras over 
the integers $\Z$. These Hopf algebras were intensively studied by Bott \cite{Bot58}.   
Although  Bott uses  the  technique of symmetric functions very much in 
that  paper (see e.g., 
\cite[\S 8]{Bot58}), he did not give explicitly the descriptions of these Hopf algebras 
in terms of symmetric functions.  
On the other hand,
by means of the celebrated Bott periodicity theorem (Bott \cite{Bot59}, 
\cite[Proposition 8.3]{Bot58}, Switzer \cite[11.60, 16.47]{Swi75}), 
there exists a homotopy equivalence of H-spaces $BU \simeq \Omega SU$, 
where $BU$ denotes the classifying space of the infinite unitary 
group $U = U(\infty)$.  The H-space structure of $BU$ is induced from 
the Whitney sum of complex vector bundles (see e.g., May \cite[p.201, Proposition]{May99}, 
Switzer \cite[p.213]{Swi75}). 
Therefore they have the isomorphic (co)homology. 
By the theory of characteristic classes of complex vector bundles, 
 it is well known that the integral cohomology ring 
of $BU$ is $H^{*}(BU) \cong \Z[c_{1}, c_{2}, \ldots]$, where 
$c_{i} \; (i = 1, 2, \ldots)$ denote the universal Chern classes. 
The homology ring of $BU$ is also known to be a polynomial ring of the form 
$H_{*}(BU) \cong \Z[\beta_{1}, \beta_{2}, \ldots]$, where $\beta_{i} \; (i = 1, 2, \ldots)$
is an element  of degree $2i$ induced from the natural map  
$BU(1) \simeq \C P^{\infty} \longrightarrow BU$ 
(see e.g., Switzer \cite[Corollary 16.11, Theorem 16.17]{Swi75}).  
Moreover, $H^{*}(BU)$ is a Hopf algebra which is {\it self-dual}: 
$H_{*}(BU)$ is isomorphic to $H^{*}(BU)$ as a Hopf algebra. 
In topology,  it is customary to think of Chern classes of complex vector bundles 
as elementary symmetric functions in  certain variables (sometimes called 
the {\it Chern roots}). From this, both $H^{*}(BU) \cong H^{*}(\Omega SU)$ 
and $H_{*}(BU) \cong H_{*}(\Omega SU)$ 
can be identified with the {\it ring of symmetric functions} denoted by 
 $\Lambda$. For cohomology, the universal Chern classes $c_{i} \; (i = 1, 2, \ldots)$
 correspond to the $i$-th elementary symmetric 
functions $e_{i}$, and for homology, the elements $\beta_{i} \; (i = 1, 2, \ldots)$ 
correspond to the $i$-th complete symmetric functions $h_{i}$ 
(for this point of view, see e.g., Baker-Richter \cite[\S 5]{Bak-Ric2008}, 
Lenart \cite[\S 4]{Len98}, Liulevicius \cite{Liu80}). 
%\footnote{
%It is further  known that under this isomorphism, the {\it Schubert classes} 
%of $BU \simeq \Omega SU$ 
%correspond to the {\it Schur functions} $s_{\lambda}$ ($\lambda$ are partitions),
% although it is hard to find 
%appropriate references of this `well-known fact' (see e.g., Lam \cite{Lam2011}). 
%}
% \footnote{
% For the terminologies on symmertic functions, we mainly follow Macdonald's 
% book \cite{Mac95}.  
%}
The starting point of our work  is  to give   
descriptions of other Hopf algebras $H^{*}(\Omega Sp)$, $H_{*}(\Omega Sp)$, 
$H^{*}(\Omega_{0} SO)$, and $H_{*}(\Omega_{0} SO)$ in terms of symmetric functions. 
By the Bott periodicity theorem again,   there exist 
 homotopy equivalences of H-spaces $Sp/U \simeq \Omega Sp$, 
$SO/U \simeq \Omega_{0} SO$ (see e.g., Switzer \cite[p.409]{Swi75}). 
Especially we have  the isomorphisms $H^{*}(\Omega Sp) \cong H^{*}(Sp/U)$  and $H^{*}(\Omega_{0} SO) 
\cong H^{*}(SO/U)$.  
On the other hand, by the work of Pragacz \cite[\S 6]{Pra91} and 
J\'{o}zefiak \cite{Joz91}, we know that 
the integral cohomology ring $H^{*}(Sp/U)$  (resp. $H^{*}(SO/U)$) 
of the infinite Lagrangian 
Grassmannian $Sp/U$ (resp.   infinite orthogonal Grassmannian $SO/U$) 
 is isomorphic to the ring of Schur $Q$-functions (resp. Schur $P$-functions) 
denoted by $\Gamma$ (resp. $\Gamma'$).   %\footnote{
%For the account of Schur $P$- and $Q$-functions, see e.g., Macdonald \cite[III, \S 8]{Mac95}. 
%}.    
%\footnote{
%As with  the case of $SU$,  it is also known that under these isomorphisms, 
%the {\it Schubert classes}  of $Sp/U \simeq \Omega Sp$ (resp. $SO/U \simeq \Omega_{0} SO$) 
%correspond to the {\it Schur $Q$-functions}   $Q_{\lambda}$  (resp. {\it Schur $P$-functions} $P_{\lambda}$), 
%where  $\lambda$ are strict partitions. 
%}. 
Therefore   we have the following isomorphism   abstractly: 
$H^{*}(\Omega Sp) \cong \Gamma$ and $H^{*}(\Omega_{0} SO) \cong \Gamma'$. 
Since $\Gamma$ and $\Gamma'$ are mutually dual Hopf algebras over $\Z$, 
we also have $H_{*}(\Omega Sp) \cong \Gamma'$ and $H_{*}(\Omega_{0} SO) \cong \Gamma$.

%%%%%%%%%%%%%%%%%%%%%%%%%%%%%%%%%%%%%%%%%%%%%%%%%%%%%%%%%%%%%%%%%%%%%%%%%%%%%%%%%%%%%%
Recently, the  study of the {\it affine Grassmannian}  $\mathrm{Gr}_{G}$ %:= G(\C((t)))/G(\C[[t]])$ 
for a simply-connected simple complex algebraic group $G$ in terms of 
{\it Schubert calculus} has been developed extensively (see e.g., 
Peterson \cite{Pet97}, 
Lam \cite{Lam2008}, 
Lam-Schilling-Shimozono \cite{Lam-Sch-Shi2010I}, \cite{Lam-Sch-Shi2010II}). 
On the one hand, 
%The Bruhat decomposition of $G(\C((t))$ induces a cell decomposition of 
$\mathrm{Gr}_{G}$ admits  a cell decomposition by {\it Schubert cells} 
from its presentation $\mathrm{Gr}_{G} = G(\C((t)))/G(\C[[t]])$%indexed by {\it Grassmannian elements}
%of the affine Weylg group $W_{\mathrm{aff}}$ of $G$ 
  (see e.g., Garland-Raghunathan \cite[Theorem 1.9]{Gar-Rag75}, 
Mitchell  \cite[Theorem 1.3]{Mit86}, \cite[Theorem 1.2]{Mit87}, \cite[Corollary 3.2]{Mit88}).  
From this, $H_{*}(\mathrm{Gr}_{G})$ and $H^{*}(\mathrm{Gr}_{G})$ have 
free $\Z$-module bases consisting of {\it Schubert classes}. 
On the other hand, by the result of Garland-Raghunathan \cite[Corollary 1.7]{Gar-Rag75} 
and an unpublished work of Quillen (see also Mitchell %\cite[Theorems 1.1 and 1.2]{Mit86}, 
\cite[Theorems 1.1 and 1.4]{Mit87}), \cite[Theorem 4.2]{Mit88}), 
it is known that 
%the  affine Grassmannian  
$\mathrm{Gr}_{G}$ %for a simply-connected simple complex algebraic group $G$ 
is homotopy equivalent to 
the based loop space $\Omega K$ on the  maximal compact subgroup $K$ of $G$. 
%The affine Grassmannian $\mathrm{Gr}_{G}$ has a natural {\it Schubert cell decomposition} 
%with {\it even}  dimensions
%(see e.g., Garland-Raghunathan \cite[Theorem 1.9]{Gar-Rag75}, 
%Mitchell  \cite[Theorem 1.3]{Mit86}, \cite[Theorem 1.2]{Mit87}, \cite[Corollary 3.2]{Mit88}). 
%From this, the above-mentioned property on the (co)homology of loop spaces 
%also follows. 
Therefore their work is closely 
related to our current work.  
In particular, Lam  studied the affine Grassmannian $\mathrm{Gr}_{SL(n, \C)} \simeq \Omega SU(n)$, 
and identified  $H_{*}(\mathrm{Gr}_{SL(n, \C)})$
and $H^{*}(\mathrm{Gr}_{SL(n, \C)})$   with a subring $\Lambda_{(n)}$ and a quotient 
$\Lambda^{(n)}$ of the ring of symmetric functions $\Lambda$. Moreover he identified 
  the Schubert classes of $H_{*}(\mathrm{Gr}_{SL(n, \C)})$
and $H^{*}(\mathrm{Gr}_{SL(n, \C)})$  as explicit symmetric functions 
(see \cite[Theorem 7.1]{Lam2008}).  
Also Lam-Schilling-Shimozono \cite{Lam-Sch-Shi2010I} considered the affine Grassmannian 
$\mathrm{Gr}_{Sp_{2n}(\C)} \simeq \Omega Sp (n)$ and identified 
$H_{*}(\mathrm{Gr}_{Sp_{2n}(\C)})$ and $H^{*}(\mathrm{Gr}_{Sp_{2n}(\C)})$ 
with certain dual Hopf algebras $\Gamma_{(n)}$ and $\Gamma^{(n)}$ 
of symmetric functions, defined in terms of  Schur $P$- and $Q$-functions.  
By taking limit $n \rightarrow \infty$,  we obtain immediately 
the above descriptions of $H_{*}(\Omega Sp)$ and $H^{*}(\Omega Sp)$. 
However, these deep results depend on Peterson's remarkable 
result (\cite{Pet97}):  In his lecture notes, 
he constructed an isomorphism between the torus equivariant homology 
$H_{*}^{T} (\mathrm{Gr}_{G})$ of the affine Grassmannian
and a certain subalgebra of the {\it affine nilHecke ring} $\mathbb{A}_{\mathrm{aff}}$. 
Thus it seems difficult to find a geometric or topological meaning of 
their result.

%%%%%%%%%%%%%%%%%%%%%%%%%%%%%%%%%%%%%%%%%%%%%%%%%%%%%%%%%%%%%%%%%%%%%%%%%%%%%%%%%%%%%%%%%
  We then turned our attention to topologists' work on 
the loop spaces on Lie groups (e.g., Clarke \cite{Cla74}, \cite{Cla81}, 
\cite{Cla81(2)}, Kono-Kozima \cite{Kon-Koz78}, Kozima \cite{Koz79}, \cite{Koz80},  \cite{Koz83}). 
Especially we focused on Kono-Kozima's work. In  \cite{Kon-Koz78}, they  considered the following 
homomorphisms in homology: 
\begin{equation}  \label{eqn:OmegaSU->OmegaSp->OmegaSU}   
 \Omega (c \circ q)_{*} = (\Omega c)_{*} \circ (\Omega q)_{*}: 
 H_{*}(\Omega SU)  \; \overset{(\Omega q)_{*}}{\longrightarrow} \; 
  H_{*}(\Omega Sp)  \; \overset{(\Omega c)_{*}}{\longrightarrow} \; 
  H_{*}(\Omega SU), 
\end{equation} 
where $q: SU \longrightarrow Sp$ and $c: Sp \longrightarrow SU$ are  induced from the 
{\it quaternionification} $SU(n)  \hooklongrightarrow Sp(n)$ and  the {\it complex restriction}
$Sp(n) \hooklongrightarrow SU(2n)$ (see (\ref{eqn:Omega(cq)_*}) in this paper).  
%\footnote{
%Let $f: X \longrightarrow Y$ be a continuous map between topological spaces 
%$X$ and $Y$ with base points. Then $\Omega f: \Omega X \longrightarrow \Omega Y$ denotes 
%the induced map on based loop spaces $\Omega X$ and $\Omega Y$.  
%}. 
It is easy to show that $(\Omega c)_{*}$ 
is a split monomorphism (Lemma \ref{lem:(Omega(c))_*}) 
and thus $H_{*}(\Omega Sp)$ can be regarded  as 
a subalgebra of $H_{*}(\Omega SU)$.  By a topological argument, they 
fixed elements $z_{i}  \in H_{2i}(\Omega Sp) \; (i = 1, 2, \ldots)$, and 
determined the Hopf algebra structure of $H_{*}(\Omega Sp)$ explicitly 
in terms of these elements (\cite[Theorem 2.18]{Kon-Koz78})\footnote{
In topology, it was known that $H_{*}(\Omega Sp)$ is a polynomial algebra 
generated by elements of degrees $4i - 2 \; (i = 1, 2, \ldots)$.  
The problem was to fix algebra generators and to give an explicit description 
of the coalgebra structure of $H_{*}(\Omega Sp)$. 
%Actually Bott gave an algorithm for computing the Hopf algebra $H_{*}(\Omega Sp(n))$ 
%by his ``generating variety approach'' (\cite[Theorem 3]{Bot58}). 
%In the case of $Sp(n)$, the generating variety is the Lagrangian Grassmannian 
%$Sp(n)/U(n)$. The problem in using this method to compute the Hopf algebra 
%structure of $H_{*}(\Omega Sp(n))$ is that $Sp(n)/U(n)$ is ``too big''. 
%Indeed it has $2^{n}$ additive generators  in $H_{*}(Sp(n)/U(n))$ 
(cf. Bott \cite[\S 11]{Bot58}).  
}.  In course of consideration, they showed the following formula (\cite[Theorem 2.9]{Kon-Koz78}):  
\begin{equation}   \label{eqn:Omega(cq)_*(beta(x))}  
   \Omega (c \circ q)_{*}(\beta (x)) = \beta (x)/\beta (-x), 
\end{equation} 
where we used the isomorphism $H_{*}(\Omega SU) \cong H_{*}(BU) \cong \Z[\beta_{1}, \beta_{2}, 
\ldots ]$ as mentioned before, and $\beta (x) := \sum_{i \geq 0}\beta_{i} x^{i} 
\in H_{*}(\Omega SU)[[x]]$,  a formal power series.
This formula is of particular importance in our present work: 
From (\ref{eqn:Omega(cq)_*(beta(x))}),  we see immediately 
that under the afore-mentioned 
isomorphism $H_{*}(BU) \cong H_{*}(\Omega SU) \cong \Lambda$, 
the monomorphism $(\Omega c)_{*}: H_{*}(\Omega Sp) \hooklongrightarrow H_{*}(\Omega SU)$ 
can be  identified with the natural inclusion $\Gamma' \hooklongrightarrow \Lambda$ 
(the similar consideration can be  found in Lam \cite[\S 2.3]{Lam2011}).  
In this way, we are able to give a sequence of homomorphisms 
(\ref{eqn:OmegaSU->OmegaSp->OmegaSU}) an interpretation in terms of symmetric functions
(for more details, see  \S \ref{subsec:IdentificationsSymmetricFunctions}). 
The advantage of our method is that it has immediate  application  to any 
generalized  homology theory $E_{*}(-)$ which is {\it complex oriented} 
in the sense of Adams   \cite[p.37]{Ada74}  (for the application to the $K$-homology theory, 
see Clarke \cite{Cla81}).

%%%%%%%%%%%%%%%%%%%%%%%%%%%%%%%%%%%%%%%%%%%%%%%%%%%%%%%%%%%%%%%%%%%%
Let $E^{*}(-)$  denote  a  complex oriented  generalized  (multiplicative) cohomology theory, 
and $E_{*}(-)$ the corresponding homology theory  
in the sense of Adams \cite[p.37]{Ada74}.
The coefficient rings for these two theories are 
given by $E_{*} := E_{*}(\mathrm{pt})$ ($E$-homology of a point) 
and $E^{*} := E^{*}(\mathrm{pt})$ ($E$-cohomology of a point).  
Our first aim is to  describe the $E_{*}$-(co)homology of 
$\Omega SU$, $\Omega Sp$, and $\Omega_{0} SO$  
in terms of symmetric functions.  Generalizing the approach 
due to Kono-Kozima \cite{Kon-Koz78} and Clarke \cite{Cla81}, 
we are able to describe the Hopf algebras $E_{*}(\Omega Sp)$ 
and $E^{*}(\Omega Sp)$ (resp. $E_{*}(\Omega_{0} SO)$ and 
$E^{*}(\Omega_{0} SO)$) as subalgebras of $E_{*}(\Omega SU)$ 
and $E^{*}(\Omega SU)$  (see Section \ref{sec:E-(co)homologyOmegaSpOmega_0SO}).  Motivated by 
the description of $E_{*}(\Omega_{0} SO)$ and $E_{*}(\Omega Sp)$ 
(resp. $E^{*}(\Omega Sp)$ and $E^{*}(\Omega_{0} SO)$),  in Section \ref{sec:RingsE-(co)homologySchurPQ-Functions}, 
we shall introduce certain subalgebras $\Gamma^{E}_{*}$, ${\Gamma'}^{E}_{*}$ 
of $\Lambda^{E}_{*}$  (resp.  $\Gamma_{E}^{*}$, ${\Gamma'}_{E}^{*}$ of 
$\Lambda_{E}^{*}$).  
By definition, using the identification 
$E_{*}(\Omega SU) \cong E_{*}(BU)$ with $\Lambda_{*}^{E} := E_{*} \otimes_{\Z} \Lambda$
(scalar extension) (resp. $E^{*}(\Omega SU) \cong E^{*}(BU)$ with 
$\Lambda^{*}_{E} := \Hom_{E_{*}} (\Lambda^{E}_{*}, E_{*})$ (graded dual)),
one sees immediately that $E_{*}(\Omega_{0} SO) \cong \Gamma^{E}_{*}$ and 
$E_{*}(\Omega Sp)  \cong {\Gamma'}^{E}_{*}$ (resp. $E^{*}(\Omega Sp) \cong \Gamma_{E}^{*}$ 
and $E^{*}(\Omega_{0} SO) \cong {\Gamma'}_{E}^{*}$). 
Thus   our first main result   identifies 
$E_{*}(\Omega Sp)$ and $E^{*}(\Omega Sp)$ 
(resp. $E_{*}(\Omega_{0} SO)$ and $E^{*}(\Omega_{0} SO)$) 
with certain 
dual Hopf algebras ${\Gamma'}^{E}_{*}$ and $\Gamma_{E}^{*}$ 
(resp. $\Gamma^{E}_{*}$ and ${\Gamma'}_{E}^{*}$)  
of symmetric functions, defined in terms of a generalization of 
the rings of  Schur  $P$- and $Q$-functions (Propositions \ref{prop:IdentificationE_*(Omega_0SO)E_*(OmegaSp)}, 
\ref{prop:IdentificationE^*(OmegaSp)E^*(Omega_0SO)}).

%%%%%%%%%%%%%%%%%%%%%%%%%%%%%%%%%%%%%%%%%%%%%%%%%%%%%%%%%%%%%%%%%%%%%%%%%%%%%%%%%%%%%%%%%%%%
Having constructed certain Hopf algebras  $\Gamma^{E}_{*}$, ${\Gamma'}^{E}_{*}$ of $\Lambda^{E}_{*}$ and $\Gamma_{E}^{*}$, ${\Gamma'}_{E}^{*}$ of $\Lambda_{E}^{*}$ 
of symmetric functions as above,  
our next task is to construct certain symmetric functions which may serve 
as  ``nice base''   for these algebras as free $E_{*}$ (or $E^{*}$)-modules. 
Since the complex cobordism theory $MU^{*}(-)$ is ``universal'' among 
complex oriented generalized cohomology theories  (Quillen \cite{Qui69}), 
it suffices to consider the case $E = MU$.  In this case, 
the coefficient ring $MU_{*} = MU^{-*}$ is known to be isomorphic to 
the {\it Lazard ring} $\L$  
(Lazard  \cite{Laz55}, Quillen \cite[Theorem 2]{Qui69}, Adams \cite[Part II, Theorem 8.2]{Ada74})\footnote{
$\L$ is known to be a polynomial algebra over the integers $\Z$  on generators of degrees 
$2, 4, 6, 8, \ldots$ (see e.g., Adams \cite[Part II, Theorem 7.1]{Ada74}, Ravenel \cite[Theorem A2.1.10]{Rav2004}).
One can define $\L$ as the quotient of a polynomial ring $P$ generated by 
formal symbols $a_{i, j} \; (i, j  \geq 1)$ of degree $2(i + j - 1)$
 by a certain ideal $I$ 
(see e.g., Adams \cite[Part II,  Theorem 5.1]{Ada74}, 
Ravenel \cite[Theorem A2.1.8]{Rav2004}). 
}.  
 We wish to construct certain symmetric functions which constitute 
 base   for ${\Gamma'}^{*}_{MU}$, $\Gamma^{*}_{MU}$,  and 
 their dual base for $\Gamma^{MU}_{*}$, ${\Gamma'}^{MU}_{*}$. 
 We also expect them to correspond to  {\it cohomology Schubert base} for  
$MU^{*}(\Omega_{0} SO) \cong MU^{*}(SO/U)$,  $MU^{*}(\Omega Sp) \cong MU^{*}(Sp/U)$, 
and {\it homology Schubert base} for 
$MU_{*}(\Omega_{0} SO) \cong MU_{*}(SO/U)$, $MU_{*}(\Omega Sp) \cong MU_{*}(Sp/U)$
respectively.  In Sections \ref{sec:UFSPQF} and \ref{sec:DUFSPQF}, 
we tackle this problem  along  the following lines: 
Firstly,  at the time of this writing,  it is not known   how to define or characterize 
geometrically 
the  Schubert classes of general flag varieties $G/P$  or affine Grassmannians
$\mathrm{Gr}_{G}$  in  {\it generalized}  (co)homology theories $E^{*}(-)$ and 
$E_{*}(-)$ (for the Schubert classes in generalized (co)homology theories, see e.g., 
Ganter-Ram \cite{Gan-Ram2012}). 
Thus we switch our attention from {geometry} to {algebra}, 
and  deal with the above problem purely algebraically. 
Secondly,  it seems plausible to define the Schubert classes 
 by way of the {\it torus equivariant cohomology theory}
 and  the technique  of the {\it localization theory} (see  Ganter-Ram \cite{Gan-Ram2012}, 
 Harada-Henriques-Holm \cite{HHH2005}, Ikeda-Naruse \cite{Ike-Nar2013}, 
Kostant-Kumar \cite{Kos-Kum90}).  In torus equivariant theory, 
various {\it factorial} analogues of Schur functions play an important role 
(for the {\it factorial Schur functions} and the complex Grassmannians, see 
Knutson-Tao \cite{Knu-Tao2003}, Molev-Sagan \cite{Mol-Sag99}; 
for the {\it factorial Schur $P$- and $Q$-functions} and 
the maximal isotropic Grassmannians, see  Ivanov \cite{Iva2004}, 
Ikeda \cite{Ike2007}, 
Ikeda-Naruse \cite{Ike-Nar2009};  
for the $K$-theoretic analogues of factorial Schur $P$- and $Q$-functions, 
see Ikeda-Naruse  \cite{Ike-Nar2013}). 
Thus we wish to  construct the required functions as a natural 
generalization of these factorial Schur functions.  
This suggests that our functions will contain multi-parameter 
$\b = (b_{1}, b_{2}, \ldots)$.  
%\vspace{1cm}  
The definition of our functions in {\it cohomology}, denoted $P^{\L}_{\lambda}(\x|\b)$ and 
$Q^{\L}_{\lambda} (\x|\b)$ with $\lambda$ a strict partition, 
$\x = (x_{1}, x_{2}, \ldots)$  an independent variables, $\b = (b_{1}, b_{2}, \ldots)$  
a {\it parameter},  
will be given in \S \ref{subsec:DefinitionP^L(x|b)Q^L(x|b)}, Definition \ref{df:DefinitionP^L(x_n|b)Q^L(x_n|b)}. 
We shall call them the {\it universal factorial Schur $P$- and $Q$-functions}. 
As the name suggests, if we  specialize  $a_{i, j} = 0$ for all $i, j \geq 1$,  
we will obtain the usual factorial Schur $P$- and $Q$-functions $P_{\lambda}(\x|\b)$ 
and $Q_{\lambda}(\x|\b)$ due to Ivanov \cite{Iva2004}\footnote{
In \cite[Definitions 2.10, 2.13]{Iva2004}, Ivanov introduced a multi-parameter generalization of the usual Schur $P$- and $Q$-functions denoted by $P_{\lambda; a}$ and $Q_{\lambda; a}$, where 
$\lambda$ is a strict partition and 
$a = (a_{k})_{k \geq 1}$ (with $a_{1} = 0$)
 is an arbitrary sequence of complex numbers. By definition, $Q_{\lambda; a} = 2^{\ell (\lambda)} 
 P_{\lambda;a}$, where $\ell (\lambda)$ denotes the length of $\lambda$.  
In this paper, we use the definition of these functions due to Ikeda-Mihalcea-Naruse 
\cite[\S 4.2]{IMN2011}.   They denote these functions by $P_{\lambda}(x|a)$ and $Q_{\lambda}(x|a)$, 
where $a = (a_{i})_{i \geq 1}$ is an infinite sequence of variables.   
By definition, $Q_{\lambda} (x|a) = 2^{\ell (\lambda)} P_{\lambda}(x|0, a)$.  
Note that
$P_{\lambda}(x|a)$ is the {\it even} limit of the corresponding polynomials 
$P^{(n)}_{\lambda}(x_{1}, \ldots, x_{n}|a)$ of finite variables because of the 
{\it mod $2$ stability} (see Ikeda-Naruse \cite[Proposition 8.2]{Ike-Nar2009}).  Also we shall  use  $[\x|\b]^{k} := \prod_{i=1}^{k} (x + b_{i})$ as 
a generalization of the  ordinary $k$-th power  in place of 
$(x|a)^{k} := \prod_{i=1}^{k} (x - a_{i})$.   
}, 
and if we specialize $a_{1, 1} = \beta$, ``Bott's element'' 
and $a_{i, j} = 0$ for all $(i, j) \neq (1, 1) $, 
we will obtain the $K$-theoretic analogue of the factorial Schur $P$- and $Q$-functions 
$GP_{\lambda}(\x|\b)$ and $GQ_{\lambda}(\x|\b)$ due to Ikeda-Naruse \cite{Ike-Nar2013}. 
Further we shall investigate various properties of our functions such as 
\begin{itemize} 
\item the {\it $\L$-supersymmetric property}  which is a genaralization of
      the ``$Q$-cancellation property'' due to Pragacz \cite[p.145]{Pra91}, 
      ``supersymmetricity'' due to Ivanov \cite[Definition 2.1]{Iva2004}, 
      and ``$K$-supersymmetric property ($K$-theoretic $Q$-cancellation property)''  due to Ikeda-Naruse \cite[Definition 1.1]{Ike-Nar2013};

\item the {\it factorization property}  (cf. Pragacz \cite[Proposition 2.2]{Pra91} and 
      Ikeda-Naruse \cite[Proposition 2.3]{Ike-Nar2013});

\item the {\it vanishing property} (cf. %Molev-Sagan  \cite[Theorem 2.1]{Mol-Sag99}, 
             Ivanov \cite[Theorem 5.3]{Iva2004},  Ikeda-Naruse \cite[Proposition 8.3]{Ike-Nar2009},  
Ikeda-Mihalcea-Naruse \cite[Proposition 4.2]{IMN2011}, Ikeda-Naruse \cite[Proposition 7.1]{Ike-Nar2013});

\item the {\it basis theorem} (cf. Ikeda-Mihalcea-Naruse \cite[Proposition 4.2]{IMN2011}, 
      Ikeda-Naruse \cite[Theorem 3.1, Propositions 3.2, 3.4, 3.5]{Ike-Nar2013}).  
\end{itemize}  
In particular, in the non-equivariant case, i.e., $\b  = 0$, 
our functions $\{ P^{\L}_{\lambda}(\x) \}$ (resp. $\{ Q^{\L}_{\lambda}(\x) \}$) 
turn out to constitute a formal  $\L$-basis for the ring ${\Gamma'}^{*}_{MU} \cong MU^{*}(\Omega_{0}  SO)$  (resp.  $\Gamma^{*}_{MU} \cong MU^{*}(\Omega Sp)$).  
Combining these  ``cohomology base'' $\{ P^{\L}_{\lambda}(\x) \}$ and 
$\{ Q^{\L}_{\lambda} (\x) \}$ with the  argument using an analogue of 
the  {\it Cauchy identity}, %\footnote{
%cf. Macdonald \cite[I,  (4.6)]{Mac95},  Stembridge \cite[\S 5]{Ste89}. 
%}
we shall define the dual base. More precisely, we  define functions 
$\wh{p}^{\L}_{\lambda}(\y)$ and $\wh{q}^{\L}_{\lambda}(\y)$ with 
$\lambda$ a strict partition, $\y= (y_{1},y_{2}, \ldots)$  one more set of independent 
variables, by the following identity (Definition \ref{df:Definitionwh{p}^Lwh{q}^L}): 
 \begin{equation*}     
\begin{array}{llll}  
\Delta(\x; \y) & =  \displaystyle{\prod_{i,j  \geq 1}}    \frac{1-\overline{x}_{i} y_{j}}{1- x_{i} y_{j}}
=  \sum_{\lambda: \; \text{strict} }
Q^{\L}_{\lambda} (\x) \, \widehat{p}^\L_\lambda(\y),   \medskip \\ 
\displaystyle
\Delta(\x; \y)  & =  \displaystyle{ \prod_{i,j \geq 1}}   \frac{1-\overline{x}_{i} y_{j}}{1-x_{i} y_{j}}
=\sum_{\lambda: \; \text{strict} }
P^{\L}_{\lambda} (\x) \,  \widehat{q}^\L_\lambda(\y).  \medskip 
\end{array} 
\end{equation*} 
We also show the {\it basis theorem} for these dual functions, 
and furthermore we will discuss the duality of these functions in terms 
of the theory of Hopf algebras.   
Thus  our second main result  is the 
{\it algebraic}   construction of the universal factorial Schur $P$- and $Q$-functions 
$P^{\L}_{\lambda}(\x|\b)$'s, $Q^{\L}_{\lambda}(\x|\b)$'s  and 
their duals $\wh{q}^{\L}_{\lambda}(\y)$'s, $\wh{p}^{\L}_{\lambda}  (\y)$'s (with $\b  =0$).   
At present, the geometric meaning of our functions
$P^{\L}_{\lambda} (\x|\b)$, $Q^{\L}_{\lambda}(\x|\b)$ and  
$\wh{p}^{\L}_{\lambda}(\y)$, $\wh{q}^{\L}_{\lambda}(\y)$ is not apparent.   
However,  for instance, the vanishing property of $Q^{\L}_{\lambda}(\x|\b)$'s  
strongly suggests that these functions will provide 
the ``Schubert basis'' for the torus equivariant complex cobordism 
$MU_{T}^{*}(LG(n))$ for a Lagrangian Grassmannian $LG(n) \cong Sp(n)/U(n)$, 
$T$ a maximal torus of $Sp(n)$. We shall discuss this problem elsewhere.

This paper is organized as follows: In Section \ref{sec:E-(co)homologyOmegaSpOmega_0SO}, 
we review the topologists' work concerning the loop spaces on $SU$, $Sp$, and $SO$. 
 Especially, we shall give  descriptions of $E$-(co)homology Hopf algebras 
of $\Omega Sp$ and $\Omega_{0} SO$, where $E^{*}(-)$ denotes 
a complex oriented generalized cohomology theory (see Theorems \ref{thm:E_*(OmegaSp)}, 
\ref{thm:E_*(Omega_0SO)}, and 
\ref{thm:E^*(OmegaSp)}).  In Section \ref{sec:RingsE-(co)homologySchurPQ-Functions}, 
 we introduce the  $E$-theoretic analogues of the rings of  Schur $P$- and $Q$-functions
(see Subsections \ref{subsec:E-homologySchurP,Q-functions}, \ref{subsec:E-cohomologySchurP,Q-functions}), 
%and   describe the $E$-(co)homology of $\Omega SU$, $\Omega Sp$, and $\Omega_{0} SO$ 
%in terms of symmetric functions. 
and  give an interpretation of  $E$-(co)homology Hopf alebras of $\Omega Sp$ and 
$\Omega_{0} SO$ in terms of 
symmetric functions (see Propositions \ref{prop:IdentificationE_*(Omega_0SO)E_*(OmegaSp)},
\ref{prop:IdentificationE^*(OmegaSp)E^*(Omega_0SO)}).  These are the first main result 
of this paper. 
In Section \ref{sec:UFSPQF}, we define the universal factorial Schur $P$- and $Q$-functions
$P^{\L}_{\lambda}(\x_{n}|\b)$ and $Q^{\L}_{\lambda} (\x_{n}|\b)$ 
(Definition \ref{df:DefinitionP^L(x_n|b)Q^L(x_n|b)}), and 
establishes the fundamental properties of these functions such as 
$\L$-supersymmetricity (Subsection \ref{subsec:L-supersymmetricity}), 
stability property (Subsection \ref{subsec:StabilityProperty}), 
factorization formula (Subsection  \ref{subsec:FactorizationFormula}), 
vanishing property (Subsection \ref{subsec:VanishingPropertyP^L(x|b)LQ^L(x|b)}), 
 and basis theorem (Subsections \ref{subsec:BasisTheoremP^L(x)Q^L(x)}, 
\ref{subsec:BasisTheoremP^L(x|b)Q^L(x|b)}).  
In this section, we also define the universal factorial Schur functions 
$s^{\L}_{\lambda}(\x_{n}|\b)$ and discuss their properties (Subsection \ref{subsec:UFSF}).  
In   Section \ref{sec:DUFSPQF}, we define the 
dual universal Schur $P$- and $Q$-functions 
$\wh{p}^{\L}_{\lambda}(\y)$ and $\wh{q}^{\L}_{\lambda}(\y)$, 
and establishes the basis theorem (Theorem \ref{thm:BasisTheoremwh{p}^L(y)wh{q}^L(y)}). 
Moreover, we describe the duality of these functions. 
Then our second main result is summarized in Theorem \ref{thm:MU_*(OmegaSp)MU_*(Omega_0SO)MU^*(OmegaSp)MU^*(Omega_0SO)}.  
In the Appendix, Section 6, we discuss  another version of the universal factorial 
Schur functions $s^{\L}_{\lambda} (\x_{n}||\b_{\Z})$, which are the universal 
analogues of Molev's double Schur functions \cite{Mol2009} (Subsection \ref{subsec:UFSF2}).  
We also collect the necessary data concerning  the Weyl groups, the root systems, 
etc. of classical types.

\vspace{0.3cm}  

\textbf{Acknowledgments.} \quad    
We would like to thank Takeshi Ikeda, Nobuaki Yagita  for helpful comments and valuable conversations.

%%%%%%%%%%%%%%%%%%Section2%%%%%%%%%%%%%%%%%%%%%%%%%%%%%%%%%%%%%%%%%%%%%%%%%%%%%%%%%%%%%%%%%%%%%%%%
\section{$E$-(co)homology of $\Omega Sp$ and $\Omega_{0} SO$}  \label{sec:E-(co)homologyOmegaSpOmega_0SO}
\subsection{Generalized (co)homology theory}  \label{subsec:Generalized(Co)homologyTheory}  
As in the introduction, $E^{*}(-)$  denotes  the complex oriented 
 generalized (multiplicative) cohomology theory, 
and $E_{*}(-)$ the corresponding homology theory  
in the sense of Adams \cite[Part II, p.37]{Ada74},  Switzer \cite[16.27]{Swi75}.   
  A generator 
$x^{E} \in \tilde{E}^{2}(\C P^{\infty})$, where $\tilde{E}^{*}(-)$ is the corresponding reduced cohomology theory, 
 is specified and it is 
called the {\it orientation class}.  
The coefficient  rings  for these two theories are given by 
 $E_{*} = E_{*}(\mathrm{pt})$ and $E^{*} = E_{-*}$.    
In what follows, the coefficient ring $E_{*}$  is assumed to be torsion free. %\footnote{
%In fact,  the condition that $E_{*}$ has no $2$-torsion will be sufficient for our study 
%(see the description of $E_{*}(\Omega Sp)$ in \S \ref{subsec:E-homology(OmegaSp)}). 
%}.  
Then it is known that the cohomology  ring of  the infinite projective space $\C P^{\infty}$ is $E^{*}(\C P^{\infty})
= E^{*}[[x^{E}]]$, a formal power series ring with   the  given   generator 
$x^{E}  \in \tilde{E}^{2}(\C P^{\infty})$ (see Adams \cite[Part II, Lemma 2.5]{Ada74}), %\footnote{
%In what follows, we often drop the superscript $E$ and write simply $x$ 
%instead of $x^{E}$ for simplicity.  
%},
 and   the homology  $E_{*}(\C P^{\infty})$ of  $\C P^{\infty}$  is a free $E_{*}$-module with a 
basis $\{ \beta_{i}^{E} \}_{i \geq 0}$ ($\beta_{0}^{E} = 1$) (Adams \cite[Part II,  Lemma 2.14]{Ada74}).  
With respect to the $E$-theory Kronecker product (pairing),  %\footnote{
%See e.g., Switzer \cite[p.283]{Swi75}. 
%}
we have
\begin{equation*} 
  \langle  (x^{E})^{i},  \beta_{j}^{E} \rangle = \delta_{ij}, 
\end{equation*}  
namely,  $\{ (x^{E})^{i} \}_{i \geq 0}$ and $\{ \beta_{j}^{E} \}_{j \geq 0}$
are dual bases over $E_{*}$.

Let 
\begin{equation*} 
     \mu_{E} (u, v)    =  u + v + \sum_{i, j \geq 1} a_{i, j}^{E} u^{i} v^{j}   \in  E^{*}[[u, v]]  \quad 
     (a_{i, j}^{E} \in  E^{2(1 - i - j)} = E_{2(i + j - 1)}) 
\end{equation*} 
be the (one dimensional commutative) formal group law  over the graded ring $E^{*}$ 
associated  with the cohomology theory $E^{*}(-)$.  %\footnote{
%For a   geometric interpretation of the formal group law $\mu_{E}(u, v)$, see 
%e.g., Adams \cite[pp.39--40]{Ada74}. 
%} \footnote{
%If we assign to $u$, $v$ the degree $2$, and to $a^{E}_{i, j}$ the degree $2(1 - i - j)$, 
%then  $\mu_{E}(u, v)$ is a homogeneous formal power series in $u$, $v$ of degree $2$.  
%}.  
The formal power series $\mu_{E}(u, v)$ satisfies the conditions
\begin{enumerate} 
\item [(i)] $\mu_{E}(u, 0) = u$, $\mu_{E}(0, v) = v$, 
\item [(ii)] $\mu_{E}(u, v) = \mu_{E}(v, u)$, 
\item [(iii)]  $\mu_{E}(u, \mu_{E}(v, w)) = \mu_{E}(\mu_{E}(u, v), w)$. 
\end{enumerate} 
It follows from (i), (ii) that 
\begin{equation*} 
  a_{i, 0} = \left \{ \begin{array}{llll} 
                     &\hspace{-0.3cm}   1 \quad & (i = 1), \medskip \\
                     &\hspace{-0.3cm}   0 \quad & (i \geq 2)\medskip 
                    \end{array} 
           \right.  
           \quad \text{and} \quad 
  a_{0, j} = \left \{ \begin{array}{llll} 
                     &\hspace{-0.3cm}   1 \quad & (j = 1), \medskip \\
                     &\hspace{-0.3cm}   0 \quad & (j \geq 2)\medskip 
                    \end{array} 
           \right.  
\end{equation*}  
and $a_{i, j} = a_{j, i} \; (i, j \geq 1)$.  Therefore $\mu_{E}(u, v)$ is in fact of the form 
\begin{equation*} 
\begin{array}{llll}  
  \mu_{E}(u, v) & = u + v + \displaystyle{\sum_{i, j \geq 1}}  a_{i, j}^{E}u^{i} v^{j}  \medskip \\
                & = u + v + a_{1, 1}uv + a_{1, 2}u^{2}v + a_{1, 2}uv^{2} + \cdots. \medskip 
                % +  a_{1, 3}u^{3}v + a_{2, 2}u^{2}v^{2} + a_{1, 3}uv^{3} + \cdots.  \medskip 
\end{array}  
\end{equation*}  
We shall use this formal group law to define the {\it formal sum}, 
{\it formal inverse}, and {\it formal subtraction}. Namely, 
for two  indeterminates $X$, $Y$, the formal sum $X +_{\mu} Y$ is defined 
as 
\begin{equation*} 
  X +_{\mu} Y := \mu_{E}(X, Y) = X + Y + \sum_{i, j \geq 1} a^{E}_{i, j} X^{i}Y^{j}  
\in E^{*}[[X, Y]]. 
\end{equation*} 
Denote by 
\begin{equation*}  
     [-1]_{E} (X)   = \iota_{E}(X)  = \overline{X} =   \sum_{j \geq 1} c_{j} X^{j}   \in E^{*}[[X]]
\end{equation*} 
the formal inverse series\footnote{
It might  be  convenient to use the notation $\overline{X}$ instead of $[-1]_{E}(X)$ in later sections 
(see  \S \ref{subsec:E-cohomology(OmegaSp)}, \S \ref{subsec:E-homologySchurP,Q-functions}).   
}.    Namely $[-1]_{E}(X)$ is the unique formal power series satisfying the condition  
$\mu_{E} (X,  [-1]_{E}(X))  \equiv 0$, or equivalently $X +_{\mu}   [-1]_{E}(X) = 0$. 
  It follows directly from the definition that we have
\begin{equation}   \label{eqn:[-1]_E(X)}  
\begin{array}{llll} 
  [-1]_{E}(X)  & = -X +  a_{1, 1}X^{2}  - a_{1, 1}^{2}X^{3}  + (a_{1, 1}^{3}
                       + a_{1, 1}a_{1,  2}  + 2a_{1, 3} - a_{2, 2})X^{4}   \medskip \\
               & + (-a_{1, 1}^{4}  - 3a_{1, 1}^{2} a_{1, 2} - 6a_{1, 1}a_{1, 3} + 3a_{1, 1}a_{2, 2})X^{5} + \cdots. \medskip 
\end{array}  
\end{equation}  
This formal inverse allows us to define the formal subtraction: 
\begin{equation*} 
   X -_{\mu} Y :=  X +_{\mu}  [-1]_{E}(Y) = X +_{\mu}  \overline{Y}. 
\end{equation*}  
Finally,  we define $[1]_{E}(X) := X$, and inductively,  
\begin{equation*}
   [n]_{E}(X)  :=  [n-1]_{E}(X) +_{\mu} X  =  \mu_{E}([n-1]_{E}(X), X) \quad (n \geq 2), 
\end{equation*}
and  $[-n]_{E}(X) := [n]_{E}([-1]_{E}(X)) = [-1]_{E}([n]_{E}(X)) \; (n \geq 1)$.   
We call $[n]_{E}(X)$ the {\it $n$-series} in the following.  
In later sections,  we often need  the 2-series $[2]_{E}(X) = X +_{\mu} X = \mu_{E}(X, X)$. 
If we put $[2]_{E}(X) = \displaystyle{\sum_{k \geq 1}}   \alpha^{E}_{k}X^{k}$, then 
\begin{equation*} 
  \alpha_{1}^{E} = 2, \; \alpha_{2}^{E} = a_{1, 1}, \; \alpha_{3}^{E} = 2a_{1, 2}, \; 
  \alpha_{4}^{E} = 2a_{1, 3} + a_{2, 2}, \; \alpha_{5}^{E} = 2a_{1, 4} + 2a_{2, 3}, \ldots.   
\end{equation*}

\begin{ex}   \label{ex:FGL}  
\quad 
\begin{enumerate} 
\item For the ordinary cohomology theory $($with integer coefficients$)$ $E = H$, 
the coefficient ring  is $H^{*} = H^{*}(\mathrm{pt}) = \Z$ $(H^{0} = \Z$, $H^{k} = 0 \; (k \neq 0))$.  
     We choose  the standard orientation, namely the class of a hyperplane $x^{H}  \in \tilde{H}^{2}(\C P^{\infty})$.  Then the associated formal group law is  the {\it additive} formal group law 
       $\mu_{H}(X, Y) = X + Y$. 
      The $2$-series  is $[2]_{H}(X) = 2X$,  and 
       the formal inverse is $[-1]_{H}(X) = -X$.

\item  For the $($topological$)$ $K$-theory  %\footnote{
    %  Here the $K$-theory means $\Z/2\Z$-graded (topological) $K$-theory: 
    % $K^{*}(-) = K^{0}(-) \bigoplus K^{1}(-)$. 
     %  Therefore the 
     % coefficient ring is $K^{*}  = K^{*}(\mathrm{pt}) =  K^{0}(\mathrm{pt}) = \Z$.  
      %If one uses the $\Z$-graded $K$-cohomology theory, the coefficient ring is 
     % given by the Laurent polynomial ring $K^{*} = \Z[\beta, \beta^{-1}]$, 
     % where $\beta \in K^{-2} \cong  \tilde{K}(S^{2}) \cong \tilde{K}(\C P^{1})$ 
      %is the usual Bott's element, i.e., $\beta = 1 - \eta^{*}$, where 
      %$\eta$ stands for the {\it tautological line bundle} over $\C P^{1}$
      %and $\eta^{*}$ its dual.  
      %Thus the associated formal group law is given by $\mu_{K}(X, Y) = X + Y - \beta XY$.  
%} 
   $E = K$,  the coefficient ring  is $K^{*} = K^{*}(\mathrm{pt}) = \Z[\beta, \beta^{-1}]$, 
with $\beta  := 1 -  \eta_{1}  \in K^{-2}(\mathrm{pt}) \cong  \tilde{K}(S^{2})$, 
where $\eta_{1}$ stands for the {\it tautological} $($or Hopf$)$ line bunlde over $\C P^{1} \cong S^{2}$. 
%$\deg \, (\beta) = -2$.   
We choose  the standard orientation 
 $x^{K} :=    \beta^{-1}(1 - \eta_{\infty}) \in \tilde{K}^{2}(\C P^{\infty})$, where $\eta_{\infty}$ 
stands for the tautological line bundle   over $\C P^{\infty}$ \footnote{
We adopt the  convention due to Bott \cite[Theorem 7.1]{Bot69}, 
Levin-Morel \cite[Example 1.1.5]{Lev-Mor2007}
so that the $K$-theory first Chern class of a line bundle $L$ (over a space $X$) 
is given by $c^{K}_{1}(L) = \beta^{-1}(1 - L^{*})$, where $L^{*}$ denotes the dual 
bundle of $L$.  In this convention,  the orientation class $x^{K}$ is equal to 
the $K$-theory first Chern class of  the dual bundle  $\eta_{\infty}^{*}$ 
(the {\it canonical line bundle} in the sense of algebraic geometry), 
namely $c^{K}_{1}(\eta_{\infty}^{*}) = \beta^{-1}(1  - \eta_{\infty})$.  
}. 
Then the associated formal group law is the {\it multiplicative} formal group law
 $\mu_{K}(X, Y) = X + Y -  \beta XY$. 
  The $2$-series is $[2]_{K}(X) = 2X -  \beta X^2$, and the formal inverse is 
  \begin{equation*} 
    [-1]_{K}(X)  =  -\dfrac{X}{1 -  \beta X}  =  -X - \beta X^2 - \beta^{2} X^3 - \beta^{3} X^4 - \cdots.  
                  %= \sum_{i \geq 0}(-1)^{i}X^{i}.  
  \end{equation*}  
     
\item For the complex cobordism theory $E  = MU$,   
        the coefficient ring $MU^{*} = MU^{*}(\mathrm{pt})$ is a polynomial algebra over $\Z$ 
      on generators of degrees $-2, -4, \ldots $ $($see e.g., Adams \cite[Part II, Theorem 8.1] {Ada74}$)$. 
       As in Adams \cite[Part II, Examples (2.4)]{Ada74}, Ravenel \cite[Example 4.1.3]{Rav2004},
     we take the orientation class $x^{MU}  
       \in \tilde{MU}^{2}(\C P^{\infty})$ to be the $($stable$)$ homotopy class of the map 
     $\C P^{\infty} \simeq BU(1) \overset{\sim}{\longrightarrow}  MU(1)$, 
     where $MU(1)$ denotes the Thom space of the universal line bundle over $BU(1)$. 
     Then the associated formal group law 
       \begin{equation*} 
          \mu_{MU}(X, Y) = X + Y + \sum_{i, j \geq 1} a^{MU}_{i, j} X^{i} Y^{j}, 
          \quad a^{MU}_{i, j}  \in MU^{2(1 -i - j)}
       \end{equation*} 
       is a {\it universal formal group law} first shown by 
       Quillen \cite[Theorem 2]{Qui69}. Namely, for any formal group law $\mu$ over a commutative ring 
       $R$ with unit, there exists a unique ring homomorphism $\theta: MU^{*} \longrightarrow R$ 
       such that $\mu (X, Y) =  (\theta_{*} \mu_{MU})(X, Y) := 
      X  + Y + \sum_{i, j \geq 1} \theta (a^{MU}_{i, j}) X^{i}Y^{j}$.  
       Quillen also showed that the coefficient ring $MU^{*}$ is 
       isomorphic to the {\it Lazard ring} $\L$ $($see also Section $\ref{sec:UFSPQF})$.  
\end{enumerate} 
\end{ex}

%%%%%%%%%%%%%%%%%%%%%%%%%%%%%%%%%%%%%%%%%%%%%%%%%%%%%%%%%%%%%%%%%%%%%%%%%%%%%%%%%%%%%%%%%%%%%%%%%%%%%%%%%%%%%%%%%%%
\subsection{$E$-(co)homology of the loop space on $SU$}   \label{subsec:E-(co)homology(OmegaSU)}  
Let $SU = SU(\infty)$ be the infinite special unitary group, and $\Omega SU$ its based loop space. 
By the celebrated Bott periodicity theorem (see   Bott  \cite[Proposition 8.3]{Bot58}, 
 \cite[p.314, Theorem II]{Bot59},     Switzer \cite[16.47]{Swi75}),  
there exists a homotopy equivalence 
$g_{\infty}: BU  \overset{\sim}{\longrightarrow} 
\Omega SU$ (%\footnote{
for the precise construction of the map $g_{\infty}$, see Bott \cite[Propositions 8.2, 8.3]{Bot58}, 
Switzer \cite[\S 16.47]{Swi75}, Kono-Kozima \cite[\S 1]{Kon-Koz78}). 
%} 
Therefore  they have isomorphic (co)homology for any %generalized 
 (co)homology  theory.  
Moreover,  both spaces $\Omega SU$ and $BU$ are equipped with H-space structures 
defined by the loop multiplication and the Whitney sum map respectively,  
and the above homotopy equivalence is actually an equivalence as H-spaces. 
%As we mentioned in the introduction, 
%the space $BU$ is equipped with an H-space structure induced from 
%the Whitney sum map,
% and the above homotopy equivalence is actually an `H-equivalence'.  
Since the integral homology of $BU$ has no torsion,   
 the (co)homology of these spaces are isomorphic as Hopf algebras for any 
(complex oriented) generalized cohomology theory $E^{*}(-)$.   
The following  facts are well known to topologists (see e.g., 
Adams \cite[Part II, Lemma 4.1, Lemma 4.3]{Ada74}, Switzer \cite[Theorems 16.31, 16.32]{Swi75}).

\begin{theorem}     \label{thm:E^*(OmegaSU)E_*(OmegaSU)}   
\quad 
\begin{enumerate}  
\item   
 The integral cohomology ring of $BU$ is given as follows$:$
\begin{equation*} 
    E^{*}(BU)  \cong    E^{*}[[c^{E}_{1}, c^{E}_{2}, \ldots, c^{E}_{n}, \ldots ]],   
\end{equation*}  
where $c^{E}_{n}  \in E^{2n}(BU) \; (n = 1, 2, \ldots)$ are the $E$-theory universal  Chern  classes. %\footnote{
%More appropriately, we should call them the ``$E$-theory universal Chern-Conner-Floyd classes''
%because Conner-Floyd first constructed the Chern classes of vector bundles in 
%any generalized cohomology theory (see Adams \cite[Part I,  4.  The Conner-Floyd Chern classes]{Ada74},  Conner-Floyd \cite[Theorem 7.6,  Corollary 8.3]{Con-Flo66}). However, in this paper, we simply call them Chern classes 
%for brevity.
%} 
%\footnote{
%As is well known to topologists, the Chern classes $c^{E}_{i}$'s have the 
%interpretation  as {\it elementary symmetric functions}  (see  \S \ref{subsec:IdentificationsSymmetricFunctions} %in this paper).
%}.  
 The coalgebra 
structure is given by 
\begin{equation*} 
      \phi (c^{E}_{n})  =  \sum_{i+j = n} c^{E}_{i} \otimes c^{E}_{j}  \quad (c^{E}_{0} := 1),  
\end{equation*}  
where the coproduct map\footnote{
In  the following, when we refer to the coproduct of a certain Hopf algebra, 
we shall always denote it by $\phi$ if there is no fear of confusion. 
} 
$\phi$ is induced from the multiplication $\mu : BU \times BU \longrightarrow BU$
that arises from the Whitney sum of complex vector bundles.

\item The integral homology ring of $BU$ is given as follows$:$
\begin{equation*} 
   E_{*}(BU)  \cong   E_{*}[\beta^{E}_{1}, \beta^{E}_{2}, \ldots, \beta^{E}_{n}, \ldots  ], 
\end{equation*}   
where $\beta^{E}_{n}  \in E_{2n}(BU)  \; (n = 1, 2, \ldots)$ are the images of the elements 
$\beta^{E}_{n} \in E_{2n}(\C P^{\infty})$
under the  induced homomorphism $E_{*}(\C P^{\infty})   \longrightarrow E_{*}(BU)$ 
from  the natural map   $\C P^{\infty} \simeq BU(1) \longrightarrow BU$.  %\footnote{
%Let $\langle - , - \rangle: E^{n}(BU)  \times E_{m}(BU)  \longrightarrow 
%E_{m - n} \; (n, m \in \Z)$  be the $E$-theory Kronecker  pairing.
% With respect to this pairing,  one can check that 
%$c^{E}_{n} \in E^{2n}(BU)$ is a unique element satisfying 
%\begin{equation*} 
%\begin{array}{llll} 
%&    \langle c^{E}_{n}, \beta^{E}_{i_{1}} \beta^{E}_{i_{2}} \cdots \beta^{E}_{i_{r}} \rangle 
%= \left \{ \begin{array}{llll} 
%             & 1  \quad &  \quad r = n  \; \text{and} \;  i_{1} = \cdots = i_{r} = 1,  \\
%             &  0 \quad &  \quad \text{otherwise}. 
%           \end{array} 
%  \right.   \medskip \\
%\Longleftrightarrow &  \langle c^{E}_{n},  (\beta^{E}_{1})^{n} \rangle  =  1  \quad 
%\text{and} \quad \langle c^{E}_{n}, M \rangle = 0, \medskip 
%\end{array}  
%\end{equation*}
%where $M$ is any monomial $(\beta^{E}_{1})^{i_{1}} (\beta^{E}_{2})^{i_{2}} \cdots (\beta^{E}_{r})^{i_{r}}$ 
%distinct from $(\beta^{E}_{1})^{n}$ (see Adams \cite[Part II,  Lemma 4.3]{Ada74}).  
% From this,  
%The elements $\beta^{E}_{i}$'s  have the interpretation as {\it complete $($homogeneous$)$ symmetric functions} 
%(see \S \ref{subsec:IdentificationsSymmetricFunctions}).
%}.    
The  coalgebra structure is given by 
\begin{equation*} 
     \phi  (\beta^{E}_{n})  = \sum_{i + j = n} \beta^{E}_{i} \otimes \beta^{E}_{j}  \quad (\beta^{E}_{0} := 1),    
\end{equation*}
where the coporoduct map $\phi$ is induced from the diagonal map $\Delta: BU \longrightarrow BU \times BU$.  
\end{enumerate}   
\end{theorem} 
\noindent
Notice that the Kronecker product
\begin{equation*} 
  \langle -, - \rangle: 
 E^{n}(BU)  \times E_{m}(BU) \longrightarrow   
E_{m-n}  \quad (n, m \in \Z) 
\end{equation*} 
induces the following isomorphism (see e.g., Switzer \cite[pp.289--290]{Swi75}) 
\begin{equation*} 
  E^{n} (BU) \overset{\sim}{\longrightarrow}   \Hom_{E_{*}}^{-n} (E_{*}(BU), E_{*})     
\end{equation*} 
for each $n \in  \Z$, and under this duality, 
 $E^{*}(BU)$ and $E_{*}(BU)$ are mutually dual Hopf algebras  over $E_{*}$.     
In what follows, we shall use the identification
\begin{equation*} 
\begin{array}{lll} 
  & E^{*}(\Omega SU) \overset{\sim}{\longrightarrow}  E^{*}(BU) \cong E^{*}[[c^{E}_{1}, c^{E}_{2}, \ldots, c^{E}_{n}, \ldots]], \medskip \\
  & E_{*}(\Omega SU)  \overset{\sim}{\longleftarrow}  E_{*}(BU)  \cong E_{*}[\beta^{E}_{1}, \beta^{E}_{2}, \ldots, \beta^{E}_{n}, \ldots]. \medskip 
\end{array}  
\end{equation*}

%%%%%%%%%%%%%%%%%%%%%%%%%%%%%%%%%%%%%%%%%%%%%%%%%%%%%%%%%%%%%%%%%%%%%%%%%%%%%%%%%%%%%%%%%%%%%%%%%%%%%%%%%%%%
\subsection{$E$-homology of the loop space   of $Sp$}   \label{subsec:E-homology(OmegaSp)}  
Let $Sp = Sp(\infty)$ be the infinite symplectic group  and $\Omega Sp$  its based loop space.  
Let  
\begin{equation*} 
  q: SU(n)  \hooklongrightarrow Sp(n),  \quad    
  c: Sp(n)  \hooklongrightarrow SU(2n)
\end{equation*} 
be the {\it quaternionification}  and  the {\it complexification} (or  {\it complex restriction}) 
respectively.   These maps   induce the following sequence of inclusions: 
\begin{equation*} 
  SU   \;  \overset{q}{\hooklongrightarrow} \;   Sp  \;  \overset{c}{\hooklongrightarrow} \;  SU.   
\end{equation*}  
Further they induce the following maps of based loop spaces: 
\begin{equation*} 
\Omega SU   \; \overset{\Omega q}{\longrightarrow}  \;   \Omega Sp  \;   \overset{\Omega c}{\longrightarrow}  \; \Omega SU.  
\end{equation*} 
Consider the induced homomorphisms in $E$-homology: 
\begin{equation}  \label{eqn:Omega(cq)_*}   
\begin{array}{llll} 
 & \Omega (c \circ q)_{*}:    E_{*}(\Omega SU)   \;  \overset{(\Omega q)_{*}}{\longrightarrow}   \;  E_{*}(\Omega Sp)  
                       \;  \overset{(\Omega c)_{*}}{\longrightarrow}  \;  E_{*}(\Omega SU).  \medskip                    
\end{array} 
\end{equation}   
Then one  can show the following fact: 
\begin{lem}  [Kono-Kozima \cite{Kon-Koz78},  Corollary 6.7]    \label{lem:(Omega(c))_*}   
  The homomorphism 
\begin{equation*} 
  (\Omega c)_{*}:  E_{*}(\Omega Sp)  \longrightarrow  E_{*}(\Omega SU)
\end{equation*} 
is a split monomorphism. %\footnote{
%As we shall show in \S \ref{subsec:E-cohomology(OmegaSp)}, Lemma \ref{lem:(Omega(c))^*(Omega(q))^*}, 
%the induced homomorphism in cohomology 
%\begin{equation*} 
%  (\Omega c)^{*}:  E^{*}(\Omega SU)  \longrightarrow E^{*}(\Omega Sp) 
%\end{equation*} 
%is surjective.  Dually $(\Omega c)_{*}: E_{*}(\Omega Sp) \longrightarrow E_{*}(\Omega SU)$
%is a split monomorphism. 
%}.  
\end{lem}   
\noindent
Therefore we can regard  the algebra $E_{*}(\Omega Sp)$ as a subalgebra of $E_{*}(\Omega SU)$.  
Following the idea of Kono-Kozima \cite{Kon-Koz78} and Clarke \cite{Cla81}, 
we shall describe this algebra explicitly.  We extend the algebra homomorphism (\ref{eqn:Omega(cq)_*}) 
to  the following algebra homomorphism of formal power series rings: 
\begin{equation*}
   \Omega (c \circ q)_{*}:  E_{*}(\Omega SU)[[T]]  \longrightarrow  E_{*}(\Omega SU)[[T]].  
\end{equation*}   
Let $\beta^{E}(T) := \sum_{i \geq 0} \beta_{i}^{E} T^{i}  \in  E_{*}(\Omega SU)[[T]]$ 
be the formal power series  with coefficients in $E_{*}(\Omega SU)$.   
By a topological argument, 
Kono-Kozima calculated the image of $\beta^{E}(T)$ under the homomorphism 
$\Omega (c \circ q)_{*}$:  
\begin{prop}  [Kono-Kozima \cite{Kon-Koz78},  Theorem 2.9, Proposition 6.10; Clarke \cite{Cla81}, p.18]   \label{prop:Omega(cq)_*(beta(t))}  
  \begin{equation}    \label{eqn:Omega(cq)_*(beta(t))}  
    \Omega (c \circ q)_{*}(\beta^{E}(T))  =  \dfrac{\beta^{E}(T)} {\beta^{E} ([-1]_{E}(T))}. 
  \end{equation}  
\end{prop}  
\noindent
This formula  will be  crucial for our study  later. 
\begin{rem}   \label{rem:h:BU->BU}
 Here we make one  important remark 
which   gives a more geometric interpretation 
about the map $\Omega (c \circ q)$.  Let $\gamma \longrightarrow BU$ be the universal 
$($virtual$)$ bundle over $BU$, and $\overline{\gamma}$ its conjugate bundle.  
Denote by $h: BU \longrightarrow BU$ the classifying map of $\gamma - \overline{\gamma}$. 
Then  Clarke \cite[Proof of Proposition 2]{Cla81} showed essentially that 
the following diagram is homotopy-commutative$:$ 
\begin{equation}   \label{eqn:h:BU->BU}  
\begin{CD} 
        BU        @>{\sim}>{g_{\infty}}>  \Omega SU \\
        @V{h}VV                               @VV{\Omega (c \circ q)}V  \\
        BU        @>{\sim}>{g_{\infty}}>    \Omega SU.   \\ 
\end{CD} 
\end{equation} 
From this, the expression $(\ref{eqn:Omega(cq)_*(beta(t))})$ follows easily
$($see  Clarke \cite[p.18]{Cla81}$)$.   The above construction $($in the case of $SO)$ of the map 
$h: BU \longrightarrow BU$  also appears in Baker \cite[p.711]{Bak86}, 
Ray \cite[Theorem 7.1]{Ray74}. 
\end{rem} 

\begin{ex}   \label{ex:Omega(cq)_*(beta_k)(k=1,2,3,4,5)}   
For example, we compute\footnote{
Here and in what follows, we often omit the superscript $E$ for simplicity.   
}$:$ 
\begin{equation*} 
\begin{array}{rlll} 
 \Omega (c \circ q)_{*} (\beta_{1})   &= &  2\beta_{1}, \medskip \\
 \Omega (c  \circ q)_{*} (\beta_{2})  &= &  2\beta_{1}^{2} - a_{1, 1}  \beta_{1}, \medskip \\
 \Omega (c \circ q)_{*} (\beta_{3}) 
       &=  &  2(\beta_{3} -  \beta_{2}\beta_{1} + \beta_{1}^{3})  + 2a_{1, 1}\beta_{2} 
          - 3a_{1, 1} \beta_{1}^{2}   + a_{1, 1}^{2} \beta_{1}.   \medskip 
\end{array}
\end{equation*}  
\end{ex} 

In Clarke \cite[pp.16--17]{Cla81},  he introduced the  elements 
$\eta_{i}^{E}  \in E_{2i}(\Omega Sp) \; (i = 1, 2, \ldots)$ 
    so as to satisfy the 
following relation
\begin{equation}  \label{eqn:eta^E}   
   (\Omega q)_{*}(\beta^{E}(T))  =  1 + [2]_{E}(T) \eta^{E}(T), 
\end{equation}   
where $\eta^{E}(T) :=  \sum_{j \geq 0} \eta_{j+1}^{E} T^{j}  \in E_{*}(\Omega Sp)[[T]]$.
More explicitly, we have
%\begin{equation*}
%  (\Omega q)_{*}(\beta (T)) 
%   =   1 + [2](T) \eta (T) 
%    =  1 + \displaystyle{\sum_{k \geq 1}}   \alpha_{k}T^{k} \cdot \sum_{j \geq 0} \eta_{j+1} T^{j} 
%   =   1 + \displaystyle{\sum_{l \geq 1}} \left ( \sum_{k=1}^{l}  \alpha_{k}  \eta_{l + 1 - k} \right ) T^{l},   
%\end{equation*} 
%and hence 
\begin{equation*} 
  (\Omega q)_{*}(\beta_{l})  =   \sum_{k = 1}^{l}  \alpha_{k} \eta_{l  + 1 - k} 
   = 2\eta_{l}  + \alpha_{2} \eta_{l-1} + \cdots + \alpha_{l} \eta_{1} \quad (l = 1, 2, \ldots). 
\end{equation*} 
Then in $E_{*}(\Omega SU)$, we have 
\begin{equation}   \label{eqn:Omegac_*(eta_l)}  
  \Omega (c \circ q)_{*}(\beta_{l})  =  2(\Omega c)_{*}(\eta_{l})  +  \alpha_{2} (\Omega c)_{*}(\eta_{l-1}) 
            + \cdots  +  \alpha_{l}  (\Omega c)_{*}(\eta_{1})  \quad (l = 1, 2, \ldots). 
\end{equation}  
Using Proposition \ref{prop:Omega(cq)_*(beta(t))} and (\ref{eqn:Omegac_*(eta_l)}),  
we can express $(\Omega c)_{*}(\eta_{l}) \; (l = 1, 2, \ldots)$  in terms of $\beta_{j} \; (j = 1, 2, \ldots)$.
\begin{ex}    \label{ex:Omegac(eta_k)(k=1,2,3,4,5)}   
Using Example $\ref{ex:Omega(cq)_*(beta_k)(k=1,2,3,4,5)}$ and $(\ref{eqn:Omegac_*(eta_l)})$, 
we compute 
\begin{equation*} 
\begin{array}{rlll} 
  (\Omega c)_{*}(\eta_{1}) &=  & \beta_{1}, \medskip \\
  (\Omega c)_{*}(\eta_{2}) &=  & \beta_{1}^{2} - a_{1, 1} \beta_{1}, \medskip \\
  (\Omega c)_{*}(\eta_{3}) 
     &=  & \beta_{3} - \beta_{2}\beta_{1} + \beta_{1}^{3}  + a_{1, 1}\beta_{2}  - 2a_{1, 1}\beta_{1}^{2} 
     + (a_{1, 1}^{2} - a_{1, 2}) \beta_{1}.   \medskip  
\end{array}  
\end{equation*}  
\end{ex}
\noindent
Then it can be shown that $E_{*}(\Omega Sp)$ is multiplicatively generated by $\eta^{E}_{1}, \eta^{E}_{2}, \ldots$
(Clarke  \cite[p.16, Corollary 4]{Cla81}).  Moreover, one sees that 
$E_{*}(\Omega Sp)$ is polynomially generated by $\eta^{E}_{1}, \eta^{E}_{3},  \eta^{E}_{5}, \ldots$
(Clarke \cite[p.17]{Cla81}). More precisely, one can express the {\it even} 
 $\eta^{E}_{2i} \; (i = 1, 2, \ldots$) 
in terms of the {\it odd} $\eta^{E}_{2i-1} \; (i = 1, 2, \ldots)$ in the following way: 
By Proposition \ref{prop:Omega(cq)_*(beta(t))}, we have 
\begin{equation*} 
   \dfrac{\beta^{E}(T)}{\beta^{E}([-1]_{E}(T))} 
   =    \Omega (c \circ q)_{*}(\beta^{E}(T))  =  (\Omega c)_{*} \circ (\Omega q)_{*} (\beta^{E}(T)) 
                                           =   (\Omega c)_{*} (1 + [2]_{E}(T) \eta^{E}(T)). 
\end{equation*}    
Therefore we obtain 
\begin{equation*} 
\begin{array}{rllll} 
  & (\Omega c)_{*}  (( 1 + [2]_{E}(T) \eta^{E}(T) )(1 + [2]_{E}( [-1]_{E}(T) ) \eta^{E}( [-1]_{E}(T))) )   \medskip \\
  & =   \dfrac{\beta^{E}(T)} {\beta^{E}([-1]_{E}(T))} \cdot  \dfrac{\beta^{E}([-1]_{E}(T))} {\beta^{E}(T)}   = 1.  \medskip 
\end{array} 
\end{equation*} 
Since $(\Omega c)_{*}$ is injective, the elements $\eta^{E}_{i}$'s satisty the  following relation: 
\begin{equation*}  \label{eqn:Relation(Sp)}   
      ( 1 + [2]_{E}(T) \eta^{E}(T) )(1 + [2]_{E}( [-1]_{E}(T) ) \eta^{E}( [-1]_{E}(T)))  = 1,     
\end{equation*}   
or equivalently, we have the following (see Clarke \cite[p.19]{Cla81}): 
\begin{equation} \label{eqn:relations(eta^E)}  
    [2]_{E}(T) \eta^{E}(T)  + [-2]_{E}(T) \eta^{E}([-1]_{E}(T)) + 
    [2]_{E}(T)[-2]_{E}(T) \eta^{E}(T) \eta^{E}([-1]_{E}(T)) = 0. 
\end{equation}  
Using the relation (\ref{eqn:relations(eta^E)}),  it can be shown directly that 
$\eta^{E}_{2i} \; (i = 1, 2, \ldots)$  can be eliminated (see Clarke \cite[p.19]{Cla81}, 
Examples \ref{ex:H_*(Omega(Sp))} (3), \ref{ex:K_*(Omega(Sp))} (3) below). 
The coproduct $\phi(\eta_{n}^{E})$ can be obtained immediately by the definition 
of $\eta^{E}_{n}$ (see Clarke  pp.17--18).     
Thus we obtain the following description of $E_{*}(\Omega Sp)$: 
%%%%%%%%%%%%%%%%%%%%%%%%%%%%%%%%%%%%%%%%%%%%%%%%%%%%%%%%%%%%%%%%%%%%%%%%%%%%%%%%%%%%%%%%%%%%%%%
\begin{theorem}  [Clarke  \cite{Cla81}]     \label{thm:E_*(OmegaSp)}  
The Hopf algebra structure of $E_{*}(\Omega Sp)$ is given as follows$:$ 
\begin{enumerate} 
\item As an algebra, 
\begin{equation*} 
\begin{array}{llll} 
  E_{*}(\Omega Sp)  &  =  \dfrac{E_{*}[\eta_{1}^{E}, \eta_{2}^{E}, \ldots, \eta_{i}^{E}, \ldots]}
                           {( ( 1 + [2]_{E}(T) \eta^{E}(T) )(1 + [2]_{E}( [-1]_{E}(T) ) \eta^{E}( [-1]_{E}(T)))     
                              = 1)}  \medskip \\
                   & = E_{*}[\eta_{1}^{E}, \eta_{3}^{E}, \ldots, \eta_{2i-1}^{E}, \ldots].     
\end{array}  
\end{equation*} 

\item  The coalgebra structure is given by 
          \begin{equation*} 
\begin{array}{rlll} 
 \phi (\eta_{1}^{E})  & = & \eta^{E}_{1} \otimes 1 + 1 \otimes \eta^{E}_{1}, \medskip \\
 \phi (\eta_{l}^{E})  & =  &  \eta^{E}_{l} \otimes 1  + 1 \otimes \eta^{E}_{l} + 
           \displaystyle{
                               \sum_{  
                                  {\tiny 
                                     \begin{array}{cccc} 
                                      &  i + j + k = l, \\
                                      &  i, j \geq 1 
                                     \end{array}  
                                   } 
                                 }    
                         }  \hspace{-0.5cm}              
                                \alpha^{E}_{k + 1} \eta^{E}_{i} \otimes \eta^{E}_{j}   \medskip \\
                      &=  &  \eta^{E}_{l} \otimes 1  + 1 \otimes \eta^{E}_{l} + 
        2 \hspace{-0.3cm} 
            \displaystyle{ 
                               \sum_{
                                 {\tiny 
                                     \begin{array}{cccc} 
                                      &  i + j = l, \\
                                      &  i, j \geq 1 
                                     \end{array}  
                                   } 
                                 }   
                         }   \hspace{-0.3cm}    
                                 \eta^{E}_{i} \otimes \eta^{E}_{j} 
                        +  \; \alpha^{E}_{2} \hspace{-0.3cm}  
                              \displaystyle{ 
                               \sum_{
                                 {\tiny 
                                     \begin{array}{cccc} 
                                      &  i + j = l-1, \\
                                      &  i, j \geq 1 
                                     \end{array}  
                                   } 
                                 }   
                         }   \hspace{-0.3cm} 
                                \eta^{E}_{i} \otimes \eta^{E}_{j}   \medskip \\
              & &     + \cdots 
                  + \alpha^{E}_{l-1} \eta^{E}_{1} \otimes \eta^{E}_{1} \quad (l \geq 2). \medskip              
\end{array} 
\end{equation*}

\item  The elements $\eta^{E}_{2n} \; (n = 1, 2, \ldots)$ are inductively determined by 
the recursive formula$:$ 
  \begin{equation*} 
      [2]_{E}(T) \eta^{E}(T)  + [-2]_{E}(T) \eta^{E}([-1]_{E}(T)) + 
    [2]_{E}(T)[-2]_{E}(T) \eta^{E}(T) \eta^{E}([-1]_{E}(T)) = 0. 
  \end{equation*} 
\end{enumerate}  
\end{theorem}   
%%%%%%%%%%%%%%%%%%%%%%%%%%%%%%%%%%%%%%%%%%%%%%%%%%%%%%%%%%%%%%%%%%%%%%%%%%%%%%%%%%%%%%%%%%%%%%%

\begin{ex} [Kono-Kozima \cite{Kon-Koz78}, Theorem 2.9]   \label{ex:H_*(Omega(Sp))}   
For the ordinary homology theory, we have 
\begin{equation*} 
   [2]_{H}(T) = 2T  \quad  \text{and}  \quad  [-1]_{H}(T) = -T. 
\end{equation*} 
From Theorem $\ref{thm:E_*(OmegaSp)}$,     the Hopf algebra structure of $H_{*}(\Omega Sp)$ 
is given as follows$:$  
\begin{enumerate} 
  \item As an algebra, 
         \begin{equation*} 
              H_{*}(\Omega Sp) \cong \Z[z_{1},  z_{3}, \ldots,  z_{2n-1}, \ldots],   
         \end{equation*} 
         where we set $z_{i} = \eta^{H}_{i} \; (i = 1, 2, \ldots)$.  
  \item  The coalgebra structure is given by 
         \begin{equation*} 
               \phi (z_{n})  = z_{n} \otimes 1 + 1 \otimes z _{n} 
                            + 2 \hspace{-0.3cm} \sum_{ 
                                       {\tiny  
                                            \begin{array}{cccc} 
                                              &   i + j = n,  \\
                                              &   i, j \geq 1
                                            \end{array}  
                                         }
                                     }  \hspace{-0.3cm}  z_{i} \otimes z_{j}.    
         \end{equation*}  
%where $z_{0} = \eta^{H}_{0}  = 0$.  
  \item  The elements $z_{2n} \; (n = 1, 2, \ldots)$ are  determined   by the recursive  formula$:$ 
         \begin{equation*} 
              z_{2n} + \hspace{-0.3cm} \sum_{ 
                              {\tiny  
                                  \begin{array}{cccc} 
                                          &   i + j = 2n,  \\
                                          &   i, j \geq 1
                                  \end{array}  
                              }
                             } \hspace{-0.3cm}   (-1)^{i} z_{i}z_{j} = 0.  
         \end{equation*} 
\end{enumerate}

\end{ex}

%%%%%%%%%%%%%%%%%%%%%%%%%%%%%%%%%%%%%%%%%%%%%%%%%%%%%%%%%%%%%%%%%%%%%%%%%%%%%%%%%%%%%
\begin{ex}  [Clarke  \cite{Cla81}, Theorem 1]  \label{ex:K_*(Omega(Sp))} 
For $\Z/2\Z$-graded $K$-theory\footnote{
We put formally $\beta = -1$ in Example \ref{ex:FGL}. 
},  we have 
\begin{equation*} 
 [2]_{K}(T)   =  2T + T^{2}  \quad \text{and}  \quad  [-1]_{K}(T) =  - \dfrac{T}{1 + T}.    
\end{equation*} 
From Theorem $\ref{thm:E_*(OmegaSp)}$, the Hopf algebra structure of $K_{0}(\Omega Sp)$ 
is given as follows$:$  
\begin{enumerate} 
  \item As an algebra, 
         \begin{equation*} 
              K_{0}(\Omega Sp) \cong \Z[\eta^{K}_{1}, \eta^{K}_{3}, \ldots, \eta^{K}_{2n-1}, \ldots].  
         \end{equation*} 
  \item  The coalgebra structure is given by 
         \begin{equation*} 
               \phi (\eta^{K}_{n})  = \eta^{K}_{n} \otimes 1 + 1 \otimes \eta^{K}_{n} 
                            + 2\hspace{-0.3cm} \sum_{
                                        {\tiny  
                                           \begin{array}{cccc} 
                                             &   i + j = n,  \\
                                             &   i, j \geq 1
                                           \end{array}  
                                        }
                                   } 
                                    \hspace{-0.3cm} \eta^{K}_{i} \otimes \eta^{K}_{j} 
                            +  \hspace{-0.3cm}  \sum_{
                                       {\tiny  
                                          \begin{array}{cccc} 
                                             &   i + j = n-1,  \\
                                             &   i, j \geq 1
                                           \end{array}  
                                        }   
                                    } \hspace{-0.3cm}  \eta^{K}_{i} \otimes \eta^{K}_{j}.    
         \end{equation*}  
%where $\eta_{0}^{K} := 0$.  
\item The elements $\eta^{K}_{2k} \; (k = 1, 2, \ldots)$ are determined by the recursive formula$:$ 
% (see Clarke \cite[p.11]{Cla81})%$:$ 
\begin{equation*} 
\begin{array}{llll}  
  &  (2\eta^{K}_{k}  + \eta^{K}_{k-1}) 
    + \displaystyle{\sum_{i=1}^{k-1}}    (-1)^{i}  (2\eta^{K}_{k-i} +  \eta^{K}_{k - i  -1}) 
  \sum_{j=1}^{i}  \left \{  2 \binom{i-1}{j-1}  + \binom{i-1}{j}  \right \}  \eta^{K}_{j}   \medskip \\
 &   +  (-1)^{k}   \displaystyle{\sum_{j=1}^{k}} 
       \left \{ 2 \binom{k-1}{j-1}  + \binom{k-1}{j}  \right \} 
     \eta^{K}_{j} = 0  \quad  (\eta^{K}_{0} := 0).   \medskip 
\end{array}
\end{equation*} 
\end{enumerate} 
\end{ex}

%%%%%%%%%%%%%%%%%%%%%%%%%%%%%%%%%%%%%%%%%%%%%%%%%%%%%%%%%%%%%%%%%%%%%%%%%%%%%%%%%%%%%%%%%%%%%%%%
\subsection{$E$-homology of the loop space of $SO$}   \label{subsec:E-homology(OmegaSO)}  
Quite analogously, we can describe the $E$-homology Hopf algebra of 
the loop space on an infinite special orthogonal group  $SO = SO(\infty)$. 
Let 
\begin{equation*} 
  r :  SU(n)  \hooklongrightarrow SO(2n),   \quad     
  c:  SO(n)  \hooklongrightarrow SU(n)
\end{equation*} 
be the {\it real restriction}  and the  {\it complexification} respectively. 
These maps induce the following   sequence of inclusions: 
\begin{equation*}
  SU  \;  \overset{r}{\hooklongrightarrow} \;    SO  \;  \overset{c}{\hooklongrightarrow} \;   SU.   
\end{equation*}  
Further they    induce the following maps on based loop spaces: 
\begin{equation*}    
\Omega SU   \overset{\Omega r}{\longrightarrow}  \;  \Omega_{0} SO   \; \overset{\Omega c}{\longrightarrow}  \; \Omega SU.  
\end{equation*} 
Consider the induced homomorphisms in $E$-homology: 
\begin{equation*} 
  \Omega (c \circ r)_{*}:    E_{*}(\Omega SU)   \;  \overset{(\Omega r)_{*}}{\longrightarrow}  \;  E_{*}(\Omega_{0} SO)  
                        \;  \overset{(\Omega c)_{*}}{\longrightarrow}  \;   E_{*}(\Omega SU).         
\end{equation*}   
We can show the following facts: 
\begin{lem}   \label{lem:(Omega(r))_*(Omega(c))_*}  
\quad 
\begin{enumerate} 
\item  The homomorphism  $(\Omega r)_{*}:  E_{*}(\Omega SU)  \longrightarrow  E_{*}(\Omega_{0} SO)$ is surjective.

\item  The homomorphism $(\Omega c)_{*}:  E_{*}(\Omega_{0} SO)  \longrightarrow E_{*}(\Omega SU)$ is injective.  %\footnote{
%Unlike the case of $Sp$,  $(\Omega c)_{*}$ is not a split monomorphism. 
%}.    
\end{enumerate} 
\end{lem}

\begin{proof} 
(1) According to Bott \cite[p.44]{Bot58}, as a generating variety for $SU(n)$ (resp. $SO(2n))$, 
    we can take the complex  projective space $\C P^{n-1} \cong U(n)/(U(1) \times U(n-1))$
    (resp.  the even dimensional complex quadric $Q_{2n-2} \cong SO(2n)/(SO(2) \times SO(2n-2))$). 
    Let $g'_{n}:  \C P^{n-1} \longrightarrow \Omega SU(n)$
    (resp. $g''_{n}: Q_{2n-2}  \longrightarrow \Omega_{0} SO(2n)$)  
   be the corresponding generating map. 
    Then we have the natural embedding   
            $i_{n}:  \C P^{n-1} \cong U(n)/(U(1) \times U(n-1))  \hooklongrightarrow  
                 Q_{2n-2}  \cong SO(2n)/(SO(2) \times SO(2n-2))$, and 
     we have the following commutative diagram: 
     \begin{equation*} 
     \begin{CD} 
           \C P^{n-1}  @>{g'_{n}}>>        \Omega SU(n)  \\
            @V{i_{n}}VV                  @VV{\Omega r}V \\
           Q_{2n-2}    @>{g''_{n}}>>  \Omega_{0}  SO(2n).       
      \end{CD} 
      \end{equation*} 
     Letting $n \rightarrow \infty$, this diagram induces  the following commutative diagram: 
     \begin{equation*} 
     \begin{CD} 
           \C P^{\infty}        @>{g'_{\infty}}>>        \Omega SU  \\
            @V{i_{\infty}}VV                  @VV{\Omega r}V \\
           Q_{\infty}           @>{g''_{\infty}}>>  \Omega_{0}  SO.       
      \end{CD} 
      \end{equation*}
     Note that the vertical map $i_{\infty}:  \C P^{\infty} \overset{\sim}{\longrightarrow}  Q_{\infty}$ 
     is  a homotopy equivalence.  Passing to homology, we have the following commutative diagram: 
     \begin{equation*} 
     \begin{CD} 
           H_{*}(\C P^{\infty})        @>{g'_{\infty *}}>>        H_{*}(\Omega SU)  \\
            @V{i_{\infty *}}V{\cong}V                  @VV{(\Omega r)_{*}}V \\
           H_{*}(Q_{\infty})           @>{g''_{\infty *}}>>       H_{*}(\Omega_{0}  SO).      
      \end{CD} 
      \end{equation*}
     By the result of Bott \cite[p.36, Theorem 1]{Bot58}  again, 
      the Pontrjagin ring $H_{*}(\Omega_{0} SO)$ is generated by 
     $\Im g''_{\infty *}$ as an algebra. From this along with  the isomorphism $i_{\infty *}$, 
     the surjectivity of $(\Omega r)_{*}$ follows. 
     For a complex oriented generalized homology theory $E_{*}(-)$, we argue as follows: 
     By the universality of the  complex bordism homology theory $MU_{*}(-)$, 
     it is sufficient to prove that the homomorphism $(\Omega r)_{*}: MU_{*}(\Omega SU) \longrightarrow MU_{*}(\Omega_{0} SO)$ is surjective. 
     Since both $H_{*}(\Omega SU)$ and $H_{*}(\Omega_{0} SO)$ are 
     concentrated in even degrees  and $MU_{*} = MU_{*}(\mathrm{pt})$ is also 
     evenly  graded, the homology  Atiyah-Hirzebruch spectral sequence 
    \begin{equation*}  
        E^{2} = H_{*}(X; MU_{*}) \cong H_{*}(X) \otimes   MU_{*}   
      \quad \Longrightarrow \quad MU_{*}(X) 
    \end{equation*} 
    collapses for $X = \Omega SU$ and $\Omega_{0} SO$.  
     As a consequence,  we have  an isomorphism 
        $MU_{*}(X) \cong H_{*}(X) \otimes MU_{*}$ for $X = \Omega SU$ and $\Omega_{0} SO$.  
     Therefore the surjection $(\Omega r)_{*}: H_{*}(\Omega SU)  \twoheadlongrightarrow H_{*}(\Omega_{0} SO)$ induces a surjection 
     $(\Omega r)_{*}: MU_{*}(\Omega SU) \cong H_{*}(\Omega SU) \otimes   MU_{*} 
     \longrightarrow MU_{*}(\Omega_{0} SO) \cong H_{*}(\Omega_{0} SO) \otimes   MU_{*}$
     since $MU_{*}$ is free as a $\Z$-module.  
     
  (2)  By the Bott periodicity, we have  homotopy equivalences 
       \begin{equation*} 
               SO/U \overset{\sim}{\longrightarrow} \Omega_{0} SO \quad \text{and}  \quad 
        g_{\infty}:  BU \overset{\sim}{\longrightarrow} \Omega SU, 
       \end{equation*}  
       and the following diagram is commutative (see e.g., Cartan \cite[p.11]{Car59I}): 
       \begin{equation}  \label{eqn:CD(SO/U)}  
        \begin{CD} 
              SO/U  @>{\chi}>>    BU   \\
               @V{\simeq}VV                   @V{g_{\infty}}V{\simeq}V  \\
              \Omega_{0} SO   @>>{\Omega c}>    \Omega SU,        
        \end{CD} 
       \end{equation}  
       where $\chi$ is induced from the inclusion 
$SO(2n)/U(n) \hooklongrightarrow U(2n)/(U(n) \times U(n))$ 
(note that $\chi$ is the fiber inclusion of the Borel fibration 
$SO/U   \overset{\chi}{\hooklongrightarrow}   BU \longrightarrow BSO$).   %\footnote{
%       Alternatively, the map $\chi: SO/U \longrightarrow BU$ is induced 
%       from the fiber inclusion of the classical Borel fibration
%       $SO(2n)/U(n) \overset{\chi}{\hooklongrightarrow}   BU(n) \longrightarrow BSO(2n)$. 
%}. 
       Therefore we  have to show that $\chi_{*}: H_{*}(SO/U)  \longrightarrow H_{*}(BU)$ 
       is injective.  In cohomology, it is well known that $\chi^{*} \otimes \Q: 
         H^*(BU) \otimes   \Q \longrightarrow H^*(SO/U) \otimes  \Q$ is surjective. 
         From this and the fact that $H_{*}(SO/U)$ is torsion free, 
         the injectivity of $\chi_{*}$ follows. 
          For a complex oriented generalized homology theory $E_{*}(-)$, the same 
          discussion as in (1) above can be applied, and we obtain the required result. 
\end{proof} 
\noindent
By means of the monomorphism  $(\Omega c)_{*}$, we can regard $E_{*}(\Omega_{0} SO)$ as 
a subalgebra of $E_{*}(\Omega SU)$.  We shall describe this algebra explicitly.  
Since $(\Omega r)_{*}$ is surjective, if we define 
\begin{equation}  \label{eqn:tau_i}  
    \tau^{E}_{i}:= (\Omega r)_{*} (\beta^{E}_{i}) \; (i = 1, 2, \ldots),  
\end{equation} 
  then  $E_{*}  (\Omega_{0} SO)$ is 
    generated by $\tau^{E}_{i} \; (i = 1, 2, \ldots)$ as an algebra.  We shall determine the 
    relations that $\tau^{E}_{i}$'s satisfy.  
Let $\beta^{E}(T) := \sum_{i \geq 0} \beta_{i}^{E} T^{i}  \in  E_{*}(\Omega SU)[[T]]$ be the formal power series 
with coefficients in $E_{*}(\Omega SU)$, and consider the following ring homomorphism 
\begin{equation*} 
    \Omega (c \circ r)_{*}:  E_{*}(\Omega SU)[[T]]  \longrightarrow  E_{*}(\Omega SU)[[T]].  
\end{equation*}   
 Then by the similar manner to 
Proposition \ref{prop:Omega(cq)_*(beta(t))},   one  can show the following:
\begin{prop} 
\begin{equation} 
  \Omega (c \circ r)_{*}(\beta^{E}(T))  =  \dfrac{\beta^{E}(T)} {\beta^{E} ([-1]_{E}(T))}. 
\end{equation}   
\end{prop}  
\noindent
We put $\tau^{E}(T) :=  \sum_{i \geq 0} \tau_{i}^{E} T^{i}  \in E_{*}(\Omega_{0} SO)[[T]] \; (\tau^{E}_{0} := 1)$.  
By definition,  we have 
\begin{equation*} 
   (\Omega r)_{*}(\beta^{E}(T))  =  \tau^{E}(T), 
\end{equation*} 
and hence 
\begin{equation*} 
  \dfrac{\beta^{E}(T)}{\beta^{E}([-1]_{E}(T))} =    \Omega (c \circ r)_{*}(\beta^{E}(T))  =  (\Omega c)_{*} \circ (\Omega r)_{*} (\beta^{E}(T)) 
                                           =   (\Omega c)_{*} (\tau^{E}(T)).  
\end{equation*} 
Therefore we have 
\begin{equation*} 
 (\Omega c)_{*}  (\tau^{E}(T) \tau^{E}([-1]_{E}(T)))  
 =  \dfrac{\beta^{E}(T)} {\beta^{E}([-1]_{E}(T))} \cdot  \dfrac{\beta^{E}([-1]_{E}(T))} {\beta^{E}(T)}   = 1. 
\end{equation*} 
Since $(\Omega c)_{*}$ is injective, we obtain the following relation: 
\begin{equation}   \label{eqn:tau^E(T)tau^E([-1]_E(T))=1}   
   \tau^{E}(T)  \tau^{E}([-1]_{E}(T)) = 1.   
\end{equation} 
For $E = H$, the ordinary homology theory, $H_{*}(\Omega_{0} SO)$ can be easily obtained 
from the result of  Bott \cite[Propositions 9.1 and 10.1]{Bot58}. 
It is generated by $\tau^{H}_{i} \; (i = 1, 2, \ldots)$ and these elements 
satisfy the relation  (\ref{eqn:tau^E(T)tau^E([-1]_E(T))=1}) for $E   = H$, namely
$\tau^{H}(T) \tau^{H}(-T) = 1$.  
Since  $E$-homology Atiyah-Hirzebruch spectral sequence 
\begin{equation*} 
  E^{2} = H_{*}(\Omega_{0} SO; E_{*}) \cong H_{*}(\Omega_{0} SO) \otimes   E_{*} 
  \quad \Longrightarrow \quad E_{*}(\Omega_{0} SO)
\end{equation*} 
collapses by degree reasons, 
no other relations  except  (\ref{eqn:tau^E(T)tau^E([-1]_E(T))=1}) can arise.  
Thus we obtain the following description of $E_{*}(\Omega_{0} SO)$:  %\footnote{
%Baker \cite[Proposition 3.3]{Bak86} described the $E$-homology $E_{*}(SO/U)$  of $SO/U$.  
%By the homotopy equivalence $SO/U  \simeq \Omega_{0} SO$ derived from the Bott periodicity theorem, 
%his description of $E_{*}(SO/U)$ is  the same as  that of  Theorem \ref{thm:E_*(Omega_0SO)}.  
%}: 
%%%%%%%%%%%%%%%%%%%%%%%%%%%%%%%%%%%%%%%%%%%%%%%%%%%%%%%%%%%%%%%%%%%%%%%%%%%%%%%%%%%%%%
\begin{theorem}    \label{thm:E_*(Omega_0SO)}  
The Hopf algebra structure of $E_{*}(\Omega_{0} SO)$ is given as follows$:$ 
\begin{enumerate} 
\item   As an algebra, 
\begin{equation*} 
  E_{*}(\Omega_{0} SO)  =  \dfrac{E_{*}[\tau_{1}^{E}, \tau_{2}^{E}, \ldots, \tau_{i}^{E}, \ldots]} 
                                 {(\tau^{E}(T)  \tau^{E}([-1]_{E}(T)) = 1)}, 
\end{equation*}
where $(\tau^{E}(T)  \tau^{E}([-1]_{E}(T)) = 1)$  means  an ideal in $E_{*}[\tau^{E}_{1}, \tau^{E}_{2}, \ldots]$  generated by  the coefficients of the formal power series 
$\tau^{E}(T)  \tau^{E}([-1]_{E}(T)) - 1$.   %\footnote{
%By abuse of notation, we treat $\tau^{E}_{i} \; (i = 1, 2, \ldots)$ as if 
%they are {\it indeterminates}.  
%}.  
\item The coalgebra structure is given by 
      \begin{equation*} 
     \phi (\tau^{E}_{n})  =   \sum_{i + j = n}    \tau^{E}_{i}  \otimes  \tau^{E}_{j}  \quad  (\tau^{E}_{0} = 1).   
       \end{equation*} 
\end{enumerate}   
\end{theorem}  

\begin{rem} 
Baker \cite[Proposition 3.3]{Bak86} described the $E$-homology $E_{*}(SO/U)$  of $SO/U$.  
By the homotopy equivalence $SO/U  \simeq \Omega_{0} SO$ derived from the Bott periodicity theorem, 
his description of $E_{*}(SO/U)$ is  the same as  that of  Theorem $\ref{thm:E_*(Omega_0SO)}$.  
\end{rem} 
%%%%%%%%%%%%%%%%%%%%%%%%%%%%%%%%%%%%%%%%%%%%%%%%%%%%%%%%%%%%%%%%%%%%%%%%%%%%%%%%%%%%%%%%%%
\begin{ex}   [Bott \cite{Bot58}, Propositions 9.1 and 10.1] \label{thm:H_*(Omega_0(SO))}
  For the ordinary homology theory, we have 
  \begin{equation*} 
   [2]_{H}(T) = 2T  \quad \text{and}  \quad  [-1]_{H}(T) = -T. 
 \end{equation*} 
 From Theorem $\ref{thm:E_*(Omega_0SO)}$, the Hopf algebra structure of $H_{*}(\Omega_{0} SO)$ is given as follows$:$ 
  \begin{enumerate} 
   \item  As an algebra, 
          \begin{equation*} 
             H_{*}(\Omega_{0} SO)  \cong \Z[\tau_{1}, \tau_{2}, \ldots, \tau_{n}, \ldots ]
                                /(\tau_{i}^{2} + 2 \sum_{j=1}^{i} (-1)^{j} \tau_{i+j} \tau_{i-j}  \; (i \geq 1)). 
           \end{equation*} 
   \item   The coalgebra structure is given by 
          \begin{equation*} 
                \phi (\tau_{n})  =  \sum_{i + j = n} \tau_{i} \otimes \tau_{j}.   
          \end{equation*} 
   \end{enumerate} 
\end{ex}

%%%%%%%%%%%%%%%%%%%%%%%%%%%%%%%%%%%%%%%%%%%%%%%%%%%%%%%%%%%%%%%%%%%%%%%%%%%%%%%%%%%%%%%%%%%%%%%
\subsection{$E$-cohomology of the loop space of $Sp$}   \label{subsec:E-cohomology(OmegaSp)}  
In a similar manner  as in \S \ref{subsec:E-homology(OmegaSp)}, we can describe the $E$-cohomology
Hopf algebra  of $\Omega Sp$.   %\footnote{
%The homology ring $E_{*}(\Omega Sp)$ is described in Theorem \ref{thm:E_*(OmegaSp)}. 
%It is a polynomial algebra on $\eta^{E}_{1}, \eta^{E}_{3}, \ldots$. In particular, 
% it is a  a free $E_{*}$-module.  Therefore we have the isomorphism 
%$E^{*}(\Omega Sp) \cong \Hom^{*}_{E_{*}} (E_{*}(\Omega Sp), E_{*})$ by the 
%universal coefficient theorem (see Switzer \cite[p.290]{Swi75}, Adams \cite[Part III,  13. A universal coefficient theorem]{Ada74}).  
%However it seems difficult to determine the structure of $E^{*}(\Omega Sp)$ from 
%this description.  
%}  
We consider the maps of based loop spaces 
\begin{equation*} 
  \Omega SU  \;  \overset{\Omega q}{\longrightarrow}  \;  \Omega Sp  
             \;  \overset{\Omega c}{\longrightarrow}  \;  \Omega SU, 
\end{equation*} 
and the induced homomorphisms in $E$-cohomology: 
\begin{equation}  \label{eqn:Omega(cq)^*}   
     E^{*}(\Omega SU) \;  \overset{(\Omega q)^{*}}{\longleftarrow}   \;  E^{*}(\Omega Sp) 
                      \;  \overset{(\Omega c)^{*}}{\longleftarrow}   \;  E^{*}(\Omega SU). 
\end{equation} 
Then we can show the following fact: 
\begin{lem}    \label{lem:(Omega(c))^*(Omega(q))^*}  
\quad 
\begin{enumerate} 
\item $(\Omega c)^{*}: E^{*}(\Omega SU) \longrightarrow E^{*}(\Omega Sp)$ is surjective. 
\item $(\Omega q)^{*}: E^{*}(\Omega Sp) \longrightarrow E^{*}(\Omega SU)$ is injective. 
\end{enumerate}  
\end{lem} 
\begin{proof} 
(1)  By the Bott periodicity theorem, we have the homotopy equivalences 
       \begin{equation*} 
               Sp/U \overset{\sim}{\longrightarrow} \Omega Sp \quad \text{and}  \quad 
        g_{\infty}:  BU \overset{\sim}{\longrightarrow} \Omega SU, 
       \end{equation*}  
       and the following diagram is commutative 
(see e.g., Cartan \cite[p.10]{Car59I}): 
%Cartan \cite[p.7, 5. Comparaison de $U(X)$ et $Sp(X_{H})$]{Car59II},  
%Cartan \cite[p.1]{Car59III}): 
       \begin{equation}  \label{eqn:CD(Sp/U)}  
        \begin{CD} 
              Sp/U  @>{\chi}>>    BU   \\
               @V{\simeq}VV                   @V{g_{\infty}}V{\simeq}V  \\
              \Omega Sp   @>>{\Omega c}>    \Omega SU,        
        \end{CD} 
       \end{equation}  
       where $\chi$ is induced from the inclusion  
$Sp(n)/U(n) \hooklongrightarrow U(2n)/(U(n) \times U(n))$ 
  (note that $\chi$ is the fiber inclusion of the Borel fibration 
 $Sp/U  \overset{\chi}{\hooklongrightarrow}   BU    \longrightarrow BSp$).  
%   \footnote{
%       Alternatively, the map $\chi:  Sp/U \longrightarrow BU$ is induced from the 
%       fiber inclusion of the classical Borel fibration $Sp(n)/U(n)  \overset{\chi}{\hooklongrightarrow}   BU(n) 
%       \longrightarrow BSp(n)$.  
%}. 
       Therefore  it  suffices to show that 
      $\chi^{*}: E^{*}(BU)  \longrightarrow E^{*}(Sp/U)$ 
       is surjective.  For $E = H$, the ordinary cohomology theory, 
        this follows immediately from the fact that 
       the Leray-Serre  spectral sequence (with integer coefficients)  
      for the  Borel fibration $Sp/U  \overset{\chi}{\longrightarrow}  BU \longrightarrow BSp$
       collapses.  Then the collapsing  of the Atiyah-Hirzebruch spectral sequence 
       \begin{equation*} 
           E_{2} = H^{*}(X; MU^{*})   \cong H^{*}(X)  \, \hat{\otimes}   \,   MU^{*} 
       \quad \Longrightarrow \quad 
           MU^{*}(X) 
       \end{equation*} 
        for $X = BU$ and $Sp/U$  implies that 
        $MU^{*}(X) \cong H^{*}(X)  \, \hat{\otimes}  \,    MU^{*}$ \footnote{
        We need to use the {\it completed tensor product}  $\hat{\otimes}$ 
        because the coefficient ring $MU^{*} = \bigoplus_{i \geq 0} MU^{-2i}$ is 
        negatively graded. 
}
for $X = BU$ and $Sp/U$,   and we obtain the  desired result (by the universality of $MU^{*}(-)$).

(2)    First we show that $(\Omega q)^{*}: H^{*}(\Omega Sp) \longrightarrow H^{*}(\Omega SU)$
is injective.    
  By a result of Kono-Kozima (see Example \ref{ex:H_*(Omega(Sp))}), there exists elements $z_{i}  \in H_{2i}(\Omega Sp)$ 
  such that $(\Omega q)_{*}(\beta_{i}) = \frac{1}{2}z_{i} \; (i = 1, 2, \ldots)$, 
and $H_{*}(\Omega Sp)$ is generated by $z_{i}$'s as an algebra. Therefore 
with rational coefficients, $(\Omega q)_{*} \otimes \Q:  H_{*}(\Omega SU) \otimes   \Q  \longrightarrow H_{*}(\Omega Sp) \otimes  \Q$ is surjective.    This implies that 
$(\Omega q)^{*} \otimes   \Q:  H^{*}(\Omega Sp) \otimes  \Q   \longrightarrow H^{*}(\Omega SU) \otimes  \Q$ is injective.  Since $H^{*}(\Omega Sp)$ and $H^{*}(\Omega SU)$ are torsion free, 
it follows that $(\Omega q)^{*}:  H^{*}(\Omega Sp)  \longrightarrow H^{*}(\Omega SU)$ 
is also  injective.  Again analogous augument as in (1) above shows 
that $(\Omega q)^{*}: E^{*}(\Omega Sp) \longrightarrow E^{*}(\Omega SU)$ is injective 
for any complex oriented generalized cohomology theory. 
\end{proof} 
\noindent
By means of the monomorphism  $(\Omega q)^{*}$, we can regard the algebra $E^{*}(\Omega Sp)$ as a subalgebra of 
$E^{*}(\Omega SU)$. Using the same idea as in  \S \ref{subsec:E-homology(OmegaSO)}, 
 we shall  identify  this algebra explicitly.  
We define the elements of $E^{*}(\Omega Sp)$ to be 
  \begin{equation*} 
    \mu^{E}_{i} := (\Omega c)^{*}(c^{E}_{i}) \; (i = 1, 2, \ldots).   
  \end{equation*} 
  %then   \underline{$E^{*}(\Omega Sp)$ is generated by $\mu^{E}_{i} \; (i = 1, 2, \ldots)$ as 
  %an algebra.}  
 We shall determine the relations that $\mu^{E}_{i}$'s satisfy. 
%First  we need some considerations. 
As stated in \S \ref{subsec:E-(co)homology(OmegaSU)}, 
the Chern classes $c^{E}_{i}$'s can be identified with the elemetary symmetric functions 
in the variables $x_{1}, x_{2}, \ldots$.   This means that  the total Chern class of 
the universal bundle $\gamma$ on $BU$ can be written formaly as 
\begin{equation*} 
   c^{E}(\gamma) = \sum_{i \geq 0} c^{E}_{i} = \prod_{i \geq 1}  (1 + x_{i}). 
\end{equation*} 
Then the total Chern class of the conjugate bundle $\overline{\gamma}$ 
can be written as 
\begin{equation*} 
  c^{E}(\overline{\gamma}) = \sum_{i \geq 0} c^{E}_{i}(\overline{\gamma}) = \prod_{i \geq 1} (1 + \overline{x}_{i}).  
\end{equation*} 
Therefore, by the definition of the map $h: BU \longrightarrow  BU$ (see Remark \ref{rem:h:BU->BU}),
 we have 
\begin{equation}  \label{eqn:c^E(gamma-overline{gamma})} 
  c^{E}(\gamma - \overline{\gamma}) = \dfrac{c^{E}(\gamma)}{c^{E}(\overline{\gamma})}
=  \prod_{i \geq 1} \dfrac{1 + x_{i}}{1 + \overline{x}_{i}}. 
\end{equation} 
With these preliminaries, we argue as follows: 
The algebra homomorphism (\ref{eqn:Omega(cq)^*}) extends to the 
algebra homomorphism of formal power series rings: 
\begin{equation*} 
  \Omega (c \circ q)^{*}: E^{*}(\Omega SU)[[T]]  \longrightarrow E^{*}(\Omega SU)[[T]]. 
\end{equation*} 
We put $c^{E}(T) :=  \sum_{i \geq 0}  c^{E}_{i}T^{i} \in E^{*}(\Omega SU)[[T]]$, 
and we would like to  calculate the image of this formal power series. 
First we can write formally 
\begin{equation*} 
  c^{E}(T)  = \sum_{i \geq 0} c^{E}_{i}T^{i} = \prod_{i \geq 1}  (1 + x_{i}T). 
\end{equation*} 
Define the formal power series $\overline{c^{E}}(T) := \sum_{i \geq 0} \overline{c^{E}_{i}}  T^{i}$
by 
\begin{equation*} 
  \overline{c^{E}}(T) = \sum_{i \geq 0} \overline{c^{E}_{i}} T^{i} 
  = \prod_{i \geq 1} (1 + \overline{x}_{i} T).    
\end{equation*} 
Then  the homotopy-commutative diagram (\ref{eqn:h:BU->BU}) and 
the formula (\ref{eqn:c^E(gamma-overline{gamma})}) imply %the argument dual to Proposition   \ref{prop:Omega(cq)_*(beta(t))} 
 the following: 
\begin{prop}  \label{prop:Omega(cq)^*(c(T))}  
\begin{equation} 
  \Omega (c \circ q)^{*}(c^{E}(T)) =  \dfrac{c^{E}(T)}{\overline{c^{E}}(T)}.  
\end{equation}  
\end{prop} 
\noindent
We also   define 
\begin{equation*} 
   \overline{\mu^{E}_{i}}   :=    (\Omega c)^{*}(\overline{c^{E}_{i}}) \; (i = 1, 2,\ldots), 
\end{equation*} 
and we put    $\mu^{E}(T) := \sum_{i \geq 0} \mu^{E}_{i} T^{i}, \; 
  \overline{\mu^{E}}(T)  = \sum_{i \geq 0}  \overline{\mu^{E}_{i}} T^{i} 
  \in E^{*}(\Omega Sp)[[T]]$
($\mu^{E}_{0} := 1, \;  \overline{\mu^{E}_{0}} := 1$). By definition, we have 
\begin{equation*} 
   (\Omega c)^{*}(c^{E}(T))  = \mu^{E}(T), 
\end{equation*} 
and hence, 
\begin{equation*} 
\begin{array}{llll} 
  \dfrac{c^{E}(T)}{\overline{c^{E}}(T)} 
& = \Omega (c \circ q)^{*} (c^{E}(T)) = (\Omega q)^{*} \circ (\Omega c)^{*} (c^{E}(T)) 
   = (\Omega q)^{*}(\mu^{E}(T)), \medskip \\
  \dfrac{\overline{c^{E}}(T)}{c^{E}(T)} 
&=   \Omega (c \circ q)^{*} (\overline{c^{E}}(T))  =  (\Omega q)^{*} \circ (\Omega c)^{*} 
(\overline{c^{E}}(T))  = (\Omega q)^{*} (\overline{\mu^{E}}(T)).  \medskip 
\end{array}  
\end{equation*} 
Therefore we have 
\begin{equation*}
 (\Omega q)^{*} (\mu^{E}(T)   \overline{\mu^{E}}(T))  
 =  \dfrac{c^{E}(T)}{\overline{c^{E}}(T)} \cdot \dfrac{\overline{c^{E}}(T)}{c^{E}(T)}  = 1.  
\end{equation*} 
Since $(\Omega q)^{*}$ is injective,   the following relation holds:   
\begin{equation*} 
  \mu^{E}(T)  \overline{\mu^{E}} (T) = 1. 
\end{equation*} 
Thus we obtain the following description of $E^{*}(\Omega Sp)$: 
%%%%%%%%%%%%%%%%%%%%%%%%%%%%%%%%%%%%%%%%%%%%%%%%%%%%%%%%%%%%%%%%%%%%%%%%%%%%%%%%%%%%%%%%%%%%%
\begin{theorem}      \label{thm:E^*(OmegaSp)}  
The Hopf algebra structure of $E^{*}(\Omega Sp)$ is given as follows$:$ %\footnote{
%In \cite[Proposition 6.1]{Cla74}, Clarke computed the $K$-theory of a Lagrangian Grassmannian manifold 
%$W_{n} := Sp(n)/U(n)$. In the limit $n \rightarrow \infty$, 
%$Sp/U$ is homotopy equivalent to $\Omega Sp$ by the Bott periodicity theorem,  and therefore his result 
%gives a  description of $K^{*}(\Omega Sp)$ (Here $K^{*}(-) = K^{0}(-) \oplus K^{1}(-)$ means 
%the $\Z/2\Z$-graded $K$-theory).  
%}$:$ 
\begin{enumerate} 
\item   As an algebra, 
\begin{equation*} 
  E^{*}(\Omega  Sp)  =  \dfrac{E^{*}[[\mu_{1}^{E}, \mu_{2}^{E}, \ldots, \mu_{i}^{E}, \ldots]]} 
                                 {(\mu^{E}(T)  \overline{\mu^{E}}(T) = 1)}. 
\end{equation*}
\item The coalgebra structure is given by 
      \begin{equation*} 
     \phi (\mu^{E}_{n})  =   \sum_{i + j = n}    \mu^{E}_{i}  \otimes  \mu^{E}_{j}  \quad  (\mu^{E}_{0} = 1).   
       \end{equation*} 
\end{enumerate}   
\end{theorem}  
%%%%%%%%%%%%%%%%%%%%%%%%%%%%%%%%%%%%%%%%%%%%%%%%%%%%%%%%%%%%%%%%%%%%%%%%%%%%%%%%%%%%%%%%%%%%%%%
\begin{rem} 
In \cite[Proposition 6.1]{Cla74}, Clarke computed the $K$-theory of a Lagrangian Grassmannian manifold 
$W_{n} := Sp(n)/U(n)$. In the limit $n \rightarrow \infty$, 
$Sp/U$ is homotopy equivalent to $\Omega Sp$ by the Bott periodicity theorem,  and therefore his result 
gives a  description of $K^{*}(\Omega Sp)$.  %(Here $K^{*}(-) = K^{0}(-) \oplus K^{1}(-)$ means 
%the $\Z/2\Z$-graded $K$-theory).  
\end{rem}

%\newpage
%%%%%%%%%%%%%%%%%%%%%%%%%%%%%%%%%%%%%%%%%%%%%%%%%%%%%%%%%%%%%%%%%%%%%%%%%%%%%%%%%%%%%%%%%%%%%%%%
\subsection{$E$-cohomology of the loop space of $SO$}    \label{subsec:E-cohomology(Omega_0SO)}  
Analogously we can describe the $E$-cohomology Hopf algebra of $\Omega_{0} SO$.  %\footnote{
%$E_{*}(\Omega_{0} SO)$ is free as an $E_{*}$-module, and hence we have the 
%isomorphism $E^{*}(\Omega_{0} SO) \cong \Hom_{E_{*}}^{*}(E_{*}(\Omega_{0}SO), E_{*})$.  
%}. 
We consider the maps of based loop spaces 
\begin{equation*} 
  \Omega SU  \;  \overset{\Omega r}{\longrightarrow} \;  \Omega_{0} SO \; 
                 \overset{\Omega c}{\longrightarrow} \;  \Omega SU, 
\end{equation*} 
and the induced homomorphism in $E$-cohomology: 
\begin{equation*} 
    E^{*}(\Omega SU)   \;   \overset{(\Omega r)^{*}}{\longleftarrow}  \;
    E^{*}(\Omega_{0} SO) \; \overset{(\Omega c)^{*}}{\longleftarrow} \; 
    E^{*}(\Omega SU). 
\end{equation*} 
Then we can show the following fact: %\footnote{
%The behavior of the  homomorphism $(\Omega c)^{*}: E^{*}(\Omega SU) \longrightarrow E^{*}(\Omega_{0} SO)$ depends on the cohomology %theory one is considering. 
%For instance, for the ordinary cohomology theory $E = H$, we know that 
%$(\Omega c)^{*}: H^{*}(\Omega SU) \longrightarrow H^{*}(\Omega_{0} SO)$ is not 
%surjective.  In fact, we know e.g., from Cartan  \cite[p.17, (88)]{Car59II}
%that the homomorphism $\chi^{*}: H^{*}(BU)  \longrightarrow H^{*}(SO/U)$ is not surjective. 
%To be precise,  the images $\chi^{*}(c_{i}) \; (i = 1, 2, \ldots)$
%are  divisible by $2$ (then  use  the commutative diagram $(\ref{eqn:CD(SO/U)})$).  
%For $K$-theory,  the resolution of the conjecture of 
%Atiyah-Hirzebruch \cite[p.36]{Ati-Hir61} due to Pittie \cite[Theorem 3]{Pit72} 
%implies that the `$\alpha$-construction' $\alpha: R(U(n))  \longrightarrow K^{0}(SO(2n)/U(n))$
%is known to be surjective, where $R(U(n))$ denotes the complex representation ring of 
%$U(n)$.  Since the $K$-theory $K^{0}(BU(n))$ of the classifying space $BU(n)$ is 
%isomorphic to the completion of $R(U(n))$ with respect to the augmentation ideal $I(U(n))$ of 
%$R(U(n))$ (Atiyah-Hirzebruch \cite[p.29, Theorem]{Ati-Hir61}), the homomorphism 
%$\chi^{*}: K^{0}(BU(n)) \longrightarrow K^{0}(SO(2n)/U(n))$ is also surjective.  
%We do not know this implies the surjectivity of the map $\chi^{*}: K^{0}(BU) \longrightarrow K^{0}(SO/U)$. 
%}: 
\begin{lem}    \label{lem:(Omega(r))^*}
  The homomorphism $(\Omega r)^{*}:  E^{*}(\Omega_{0} SO)   \longrightarrow E^{*}(\Omega SU)$
  is a split monomorphism.  
\end{lem}   
\begin{proof} 
In Lemma \ref{lem:(Omega(r))_*(Omega(c))_*}, we proved that 
$(\Omega r)_{*}:  E_{*}(\Omega SU)  \twoheadlongrightarrow E_{*}(\Omega_{0} SO)$ is surjective.  
Since both $E_{*}(\Omega SU)$ and $E_{*}(\Omega_{0} SO)$ are free $E_{*}$-modules, 
the result follows. 
\end{proof}  
\noindent
By means of the monomorphism $(\Omega r)^{*}$, we can regard the algebra 
$E^{*}(\Omega_{0} SO)$ as a subalgebra of $E^{*}(\Omega SU)$.  
However, unlike the case of $Sp$, the homomorphism $(\Omega c)^{*}$ is  not 
a surjection in general, and  we have not been able to find  the
 generators of $E^{*}(\Omega_{0} SO)$ 
for generalized cohomology theory.  Here we only comment the following fact: 
In the case of ordinary cohomology theory, we know  that there exists  a homotopy equivalence
$SO/U \simeq \Omega_{0} SO$ by the Bott periodicity theorem. 
  Therefore we have the isomorphism of algebras: 
$H^{*}(SO/U) \cong H^{*}(\Omega_{0} SO)$.  
The integral cohomology ring $H^{*}(SO/U)$ is well known since Borel 
(see e.g., Cartan \cite[11. Homologie et cohomologie de $SO(X)/U(X)$]{Car59II}),
 and it is a polynomial algebra generated by the elements of degrees $4k + 2 \; (k = 0, 1, 2, \ldots)$. 
Then by the collapse of the Atiyah-Hirzebruch spectral sequence 
\begin{equation*} 
  E_{2} = H^{*}(\Omega_{0}SO; E^{*})  \cong H^{*}(\Omega_{0} SO) \, \hat{\otimes} \, E^{*} 
 \quad   \Longrightarrow   \quad   E^{*}(\Omega_{0} SO), 
\end{equation*} 
we know  that $E^{*}(\Omega_{0} SO)$ is of the form 
\begin{equation*} 
  E^{*}(\Omega_{0} SO)  \cong 
E^{*}[[y_{1}^{E}, y_{3}^{E}, \ldots, y_{2k+1}^{E}, \ldots  ]], 
\end{equation*} 
where $y^{E}_{2k + 1} \in E^{4k + 2} (\Omega_{0} SO) \; (k = 0, 1, 2, \ldots)$.     %\footnote{ 
%Baker \cite[Proposition 3.12]{Bak86} described 
%$E^{*}(SO/U)$.  He asserted that 
%\begin{prop} 
%\begin{equation*} 
%  E^{*}(SO/U)  \cong %E^{*}[[y_{2k+1}^{E} | k \geq 0  ]] = 
%E^{*}[[y_{1}^{E}, y_{3}^{E}, \ldots, y_{2k+1}^{E}, \ldots  ]], 
%\end{equation*} 
%where the elements  $y^{E}_{2k  + 1}  \in E^{4k+2}(SO/U)$ 
%are  dual to a certain $E_{*}$-basis for the module $PE_{*}(SO/U)$ 
%of primitive elements. 
%\end{prop} 
%However his discussion   based on  a wrong formula. 
%Notice that the    formula in  that paper \cite[p.713]{Bak86},  ``$\phi^{*} \chi^{*} c^{E}(T) = c^{E}(T) c^{E}([-1](T))^{-1}$'' 
%should be corrected as in (\ref{eqn:Omega(cr)^*(c^E(T))}).  
%}. 

Finally we state the analogous formula to Proposition \ref{prop:Omega(cq)^*(c(T))}, 
which might be useful for the description of $E^{*}(\Omega_{0} SO)$. 
Let $c^{E}(T) := \sum_{i \geq 0} c^{E}_{i} \, T^{i}  \in E^{*}(\Omega SU)[[T]]$ be 
the formal power series ring with coefficients in $E^{*}(\Omega SU)$, and we write formally 
\begin{equation*} 
  c^{E}(T) = \sum_{i \geq 0} c^{E}_{i} \, T^{i} = \prod_{i \geq 1} (1 + x_{i} T). 
\end{equation*} 
We also put 
\begin{equation*}
   \overline{c^{E}}(T) = \sum_{i \geq 0}  \overline{c^{E}_{i}} \, T^{i} 
   = \prod_{i \geq 1} (1 + \overline{x}_{i} T). 
\end{equation*} 
Consider the ring homomorphism 
\begin{equation*} 
   \Omega (c \circ r)^{*}:  E^{*}(\Omega SU)[[T]]  \longrightarrow E^{*}(\Omega SU)[[T]]. 
\end{equation*} 
Then by the similar manner to Proposition \ref{prop:Omega(cq)^*(c(T))}, 
one can show the following: 
\begin{prop}  
\begin{equation}   \label{eqn:Omega(cr)^*(c^E(T))}  
  \Omega (c \circ r)^{*} (c^{E}(T)) =    \dfrac{c^{E}(T)} {\overline{c^{E}}(T)}. 
\end{equation} 
\end{prop} 

\begin{rem} 
Baker \cite[Proposition 3.12]{Bak86} described 
$E^{*}(SO/U)$.  He asserted that 
%\begin{prop} 
\begin{equation*} 
  E^{*}(SO/U)  \cong %E^{*}[[y_{2k+1}^{E} | k \geq 0  ]] = 
E^{*}[[y_{1}^{E}, y_{3}^{E}, \ldots, y_{2k+1}^{E}, \ldots  ]], 
\end{equation*} 
where the elements  $y^{E}_{2k  + 1}  \in E^{4k+2}(SO/U)$ 
are  dual to a certain $E_{*}$-basis for the module $PE_{*}(SO/U)$ 
of primitive elements. 
%\end{prop} 
%However his discussion   based on  a wrong formula. 
Notice that the    formula in  that paper \cite[p.713]{Bak86},  ``$\phi^{*} \chi^{*} c^{E}(T) = c^{E}(T) c^{E}([-1](T))^{-1}$'' 
should be corrected as in  $(\ref{eqn:Omega(cr)^*(c^E(T))})$.  
\end{rem}

%\newpage
%%%%%%%%%%%%%%%%%%%%%%%%%%%%%%%%%%%%%%%%%%%%%%%%%%%%%%%%%%%%%%%%%%%%%%%%%%%%%%%%%%%%%%%%%%%%%%%%%%%%%%%%%%%%%
\section{Rings of  $E$-(co)homology  Schur $P$- and $Q$-functions}  \label{sec:RingsE-(co)homologySchurPQ-Functions} 
In this section,  we introduce the  $E$-theoretic analogues of the rings of  Schur $P$- and $Q$-functions, 
and   describe the $E$-(co)homology of $\Omega SU$, $\Omega Sp$, and $\Omega_{0} SO$ 
in terms of symmetric functions. 
%%%%%%%%%%%%%%%%%%%%%%%%%%%%%%%%%%%%%%%%%%%%%%%%%%%%%%%%%%%%%%%%%%%%%%%%%%%%%%%%%%%%%%%%%%%%%%%%%%%%%%%%%%%%%%%

\subsection{Ring of symmetric functions}  \label{subsec:RingSymmetricFunctions}  
We use standard notations for symmetric functions as in  Macdonald \cite{Mac95}. 
Let $\Lambda  = \Lambda (\bf{x})$ be the {\it ring of symmetric functions} 
 with integer coefficients 
in infinitely many  variables ${\bf x} = (x_{1}, x_{2}, \ldots)$. 
In the sequel, we provide the variables $x_{i} \; (i = 1, 2, \ldots)$ with 
degree $\deg \, (x_{i}) = 1$, and 
regard $\Lambda$ as a graded algebra over $\Z$. 
When we work with the finite set of variables $\x_{n} = (x_{1}, \ldots, x_{n})$, 
we denote by $\Lambda (\x_{n})$ the ring of symmetric polynomials, namely 
$\Lambda (\x_{n}) = \Z [x_{1}, \ldots, x_{n}]^{S_{n}}$, where $S_{n}$ is the 
symmetric group on $n$ letters acting on $\x_{n} = (x_{1}, \ldots, x_{n})$ by 
permutations. Then $\Lambda (\x)$ is 
the inverse limit of the $\Lambda (\x_{n})$ in the category of {\it graded} rings. %\footnote{
%cf. Hoffman-Humphreys \cite[p.92]{Hof-Hum92}, Macdonald \cite[I, \S 2, Remarks.1]{Mac95}. 
%}. 
Note that it is possible to describe $\Lambda (\x)$ by using the formal power series ring 
instead of  the inverse limits (see Hoffman-Humphreys \cite[p.92]{Hof-Hum92}, Stanley \cite[\S 7.1]{Sta99}). Thus $\Lambda (\x) \subset \Z[[\x]] = \Z[[x_{1}, x_{2}, \ldots]]$.

We denote by $e_{i} = e_{i}({\bf x}) \; (i = 1, 2, \ldots)$ 
(resp. $h_{i}  = h_{i}({\bf x})\; (i = 1, 2, \ldots)$)  
the $i$-th  {\it elementary symmetric function} 
(resp. the $i$-th  {\it complete symmetric function}). 
The generating functions are respectively given by 
\begin{equation*} 
   E(T) = \sum_{i \geq 0}e_{i}T^{i} =  \prod_{i \geq 1}(1  + x_{i}T)  \quad \text{and}  \quad 
   H(T) = \sum_{i \geq 0}h_{i}T^{i} = \prod_{i \geq 1} \dfrac{1}{1 - x_{i}T} 
\end{equation*} 
($e_{0} :=1$ and $h_{0} := 1$).     Note that the relation $H(T)E(-T) = 1$ holds.  
It is well known that $\Lambda$ is a polynomial algebra 
over $\Z$ in both $e_{i} \; (i = 1, 2, \ldots)$ and $h_{i} \; (i = 1, 2, \ldots)$: 
\begin{equation*} 
   \Lambda = \Z[e_{1}, e_{2}, \ldots ]  =  \Z[h_{1}, h_{2}, \ldots]. 
\end{equation*} 
Furthermore, by the coproduct (or comultiplication, diagonal map), 
\begin{equation}   \label{eqn:CoproductLambda}  
       \phi (e_{k})  =  \sum_{i+j = k} e_{i} \otimes e_{j} \quad \text{and} \quad  
       \phi (h_{k})  =  \sum_{i+j = k} h_{i} \otimes h_{j}, 
\end{equation} 
$\Lambda$ is a commutative and co-commutative Hopf algebra over $\Z$
(see Macdonald \cite[I, \S 5, Examples 25]{Mac95}).  
Using the Hall inner product 
$\langle - , -  \rangle : \Lambda \times \Lambda \longrightarrow \Z$
(see Macdonald \cite[I, \S 4, (4.5)]{Mac95}),  
we can identify $\Lambda$ with its {\it graded}    dual $\Lambda^{*} = \Hom \, (\Lambda, \Z)$.

%%%%%%%%%%%%%%%%%%%%%%%%%%%%%%%%%%%%%%%%%%%%%%%%%%%%%%%%%%%%%%%%%%%%%%%%%%%%%%%%%%%%%%%%%
Next we recall the Schur $Q$- and $P$-functions (for details, see Macdonald \cite[III, \S 8]{Mac95}).
Define $Q_{k} \; (k = 1, 2, \ldots)$   %\footnote{
%In \cite[III, \S 8]{Mac95}, $Q_{k}$ is denoted by $q_{k}$. 
%}
 as the coefficients of $T^{k}$ in the generating function 
\begin{equation*} 
     Q(T)  = \displaystyle{\sum_{k \geq 0}} Q_{k} T^{k}  
          :=  \displaystyle{\prod_{i\geq 1}} \dfrac{1 + x_{i}T}{1 - x_{i}T}  =  H(T) E(T) 
          % = \left (\displaystyle{\sum_{i \geq 0}} h_{i}T^{i} \right ) 
           %  \left (\displaystyle{\sum_{j \geq 0}} e_{j}T^{j} \right ) 
           %=  \displaystyle{\sum_{k \geq 0}}
    %\left (\displaystyle{\sum_{i+j = k}} h_{i} e_{j} \right) T^{k}   
\end{equation*} 
($Q_{0} := 1$).   In other words, $Q_{k}$ is defined to be $Q_{k} := \sum_{i+j = k} h_{i} e_{j}$. 
The obvious identity $Q(T) Q(-T) = 1$ yields the following relations: 
\begin{equation}  \label{eqn:relations(Q)} 
    Q_{i}^{2} + 2\sum_{j=1}^{i} (-1)^{j} Q_{i+j} Q_{i-j} = 0  \quad  (i  \geq 1). 
\end{equation} 
Define $\Gamma$ to be  the subalgebra of $\Lambda$ generated by $Q_{i}$'s. 
Then we have 
\begin{equation*} 
  \Gamma  = \Z[Q_{1}, Q_{2}, \ldots, Q_{i}, \ldots]
               /(Q_{i}^{2} + 2\sum_{j=1}^{i} (-1)^{j} Q_{i+j} Q_{i-j} \; (i \geq 1)). 
\end{equation*} 
The function $P_{k} \; (k = 1, 2, \ldots)$ is defined by  
   $P_{k}  := \dfrac{1}{2} Q_{k}  =  \dfrac{1}{2} \sum_{i+j = k} h_{i} e_{j}$. 
By (\ref{eqn:relations(Q)}), $P_{i}$'s satisfy the following relations: 
\begin{equation}  \label{eqn:relations(P)} 
      P_{i}^{2} + 2 \sum_{j=1}^{i-1} (-1)^{j} P_{i+j} P_{i-j}  + (-1)^{i} P_{2i} = 0
     \quad (i \geq 1). 
\end{equation} 
Note that by the above relations (\ref{eqn:relations(P)}), we can eliminate $P_{2i} \; (i \geq 1)$. 
Define $\Gamma'$ to be  the subalgebra of $\Lambda$  generated by $P_{i}$'s. 
By definition, 
$\Gamma \subset \Gamma' \subset \Lambda$.  Then we have 
\begin{equation*} 
\begin{array}{rlll} 
    \Gamma'  & =  & \Z[P_{1}, P_{2}, \ldots, P_{i}, \ldots]
               /(P_{i}^{2} + 2 \sum_{j=1}^{i-1} (-1)^{j} P_{i+j} P_{i-j}  + (-1)^{i} P_{2i} \; (i \geq 1)) \medskip \\
             &=   &  \Z[P_{1}, P_{3}, \ldots, P_{2n-1}, \ldots].     \medskip  
\end{array} 
\end{equation*} 
These two subalgebras $\Gamma$ and $\Gamma'$ also have natural Hopf algebra structures 
(for the Hopf algebra structures on $\Gamma$ and $\Gamma'$, readers are referred to 
e.g., Lam-Schilling-Shimozono \cite[\S 2]{Lam-Sch-Shi2010I}).   
The coproduct is given by 
\begin{equation*} 
   \phi (Q_{k})  = \sum_{i+j = k} Q_{i} \otimes Q_{j} \quad \text{and} \quad 
   \phi (P_{k})  =  P_{k} \otimes 1 + 1 \otimes P_{k} 
                  +  2  \sum_{i+j= k, \, i, j \geq 1}  P_{i} \otimes P_{j}.   
\end{equation*}      
One can  also  define the pairing $[-, -]:  \Gamma' \times \Gamma \longrightarrow \Z$ 
(Lam-Schilling-Shimozono \cite[(2.14)]{Lam-Sch-Shi2010I}, 
Macdonald \cite[III, \S8, (8.12)]{Mac95},  Stembridge \cite[(5.2)]{Ste89}),  
and  then $\Gamma$ and $\Gamma'$ are dual Hopf algebras each other with respect to 
this pairing.

%%%%%%%%%%%%%%%%%%%%%%%%%%%%%%%%%%%%%%%%%%%%%%%%%%%%%%%%%%%%%%%%%%%%%%%%%%%%%%%%%%%%%%%%%%%%%%%%%%%%%%%%%%%
\subsection{Rings of $E$-homology Schur $P$- and $Q$-functions}   
\label{subsec:E-homologySchurP,Q-functions}  
Motivated by the description of $E_{*}(\Omega_{0} SO)$ and $E_{*}(\Omega Sp)$ 
 in the previous two subsections  \S \ref{subsec:E-homology(OmegaSp)} and \ref{subsec:E-homology(OmegaSO)}, 
we define $E$-homology  analogues of the rings of   Schur $P$- and $Q$-functions 
(Definitions \ref{df:Gamma^E_*} and \ref{df:Gamma'^E_*}).  
First we define%\footnote{
%cf. Lenart \cite[\S 2]{Len98}. 
%}
\begin{equation}   \label{eqn:Lambda^E_*Lambda_E^*}   
\begin{array}{rlll}  
  \Lambda^{E}_{*} 
&:=  & E_{*} \otimes_{\Z} \Lambda  \quad (\text{extension of coefficients}), \medskip  \\ 
  \Lambda_{E}^{*}
 &:= & \Hom_{E_{*}} (\Lambda^{E}_{*}, E_{*}) \quad (\text{graded dual of}  \;  \Lambda^{E}_{*}). \medskip 
\end{array} 
\end{equation} 
Then  $\Lambda^{E}_{*}$  inherits  naturally the  Hopf algebra structure from that of $\Lambda$. 
Actually it  is obvious from the definition  that
 $\Lambda^{E}_{*} %= E_{*} \otimes_{\Z}  \Lambda = E_{*}[e_{1}, e_{2}, \ldots, ]
 = E_{*}[h_{1}, h_{2}, \ldots]$, and 
the coalgebra structure is given by (\ref{eqn:CoproductLambda}).  
Dually,  $\Lambda_{E}^{*}  = E^{*}[[e_{1}, e_{2}, \ldots ]]$, 
the  formal power series ring in $e_{i} \; (i = 1, 2, \ldots)$, 
and   the coalgebra structure is also given by (\ref{eqn:CoproductLambda}). 
One can verify this latter assertion by purely algebraic manner (see e.g., 
May-Pont \cite[\S 21.4]{May-Pon2012}).  

%has a  Hopf algebra structure as a graded dual of 
%$\Lambda^{E}_{*}$ \footnote{
%In \S \ref{subsec:IdentificationsSymmetricFunctions}, $\Lambda^{E}_{*}$ and  $\Lambda_{E}^{*}$
%will be identified with $E_{*}(BU)$ and  $E^{*}(BU)$ respectively. 
%Since the $E$-theory universal Chern classes $c^{E}_{i} \; (i = 1, 2, \ldots)$ 
% are naturally identified with 
%the $i$-th elementary symmetric functions $e_{i}$, and $E^{*}(BU)$ is known to be 
%a formal power series ring $E^{*}[[c^{E}_{1}, c^{E}_{2}, \ldots ]]$ in $c^{E}_{i} \; (i = 1, 2, \ldots)$,  one obtains  %$\Lambda_{E}^{*} = E^{*}[[e_{1}, e_{2}, \ldots ]]$, 
%a formal power series ring in $e_{i} \; (i = 1, 2, \ldots)$. 
%We do not know  if  this follows immediately from the definition (\ref{eqn:Lambda^E_*Lambda_E^*})
% of $\Lambda_{E}^{*}$. 
%}.  

\begin{defn}  [{\it $E$-homology Schur $Q$-functions}]     \label{df:wh{q}^E}  
We define   $\wh{q}^{E}_{k}  = \wh{q}^{E}_{k}({\bf y})  
\in  \Lambda^{E}_{*}  \; (k = 1, 2, \ldots)$ %\footnote{
%When we consider ``homology'', we  shall  use the variables $\y = (y_{1}, y_{2}, \ldots)$.
%}  
as the coefficients  of 
the generating function %\footnote{
%Compare this with the equation (\ref{eqn:Omega(cq)_*(beta(t))}). 
%} 
\begin{equation}   \label{eqn:wh{q}^E} 
   \wh{q}^{E}(T)  
  =  \displaystyle{\sum_{k \geq 0}} \wh{q}^{E}_{k} T^{k}  
  %= \dfrac{H(T)}{H([-1]_{E}(T))}
  = \dfrac{H(T)}{H(\overline{T})} \medskip \\
  =  \displaystyle{\prod_{i \geq 1}}   \dfrac{1 - \overline{T}y_{i}}{1 - Ty_{i}} 
  \;   \in \Lambda^{E}_{*}[[T]]      
\end{equation} 
$(\wh{q}^{E}_{0} := 1)$.
\end{defn} 
\noindent
For $E = H$, the ordinary cohomology theory, the formal inverse $\overline{T} = [-1]_{H}(T) = -T$, 
and hence $\wh{q}^{H}_{k} \; (k = 1, 2, \ldots)$ coincide with the usual 
Schur $Q$-functions $Q_{k} = \sum_{i + j = k} h_{i}e_{j} \; (k  = 1, 2, \ldots)$.  
By definition,  
\begin{equation*} 
  \wh{q}^{E}(T) = H(T) E(-\overline{T}) %H(T)  E(-[-1]_{E}(T)) 
  = \left (\sum_{i \geq 0} h_{i} T^{i} \right ) 
  \left (\sum_{j \geq 0} e_{j} (-\overline{T})^{j} \right ), % \left (\sum_{j \geq 0} e_{j} (-[-1]_{E}(T))^{j} \right ),  
\end{equation*} 
and  since $\overline{T} %[-1]_{E}(T) 
= -T +  \text{higher terms in} \;  T$, we have 
\begin{equation*} 
  \wh{q}^{E}_{k} = Q_{k} +  \text{lower terms  in} \;  \y  = (y_{1}, y_{2}, \ldots).  
\end{equation*}

\begin{ex}   \label{ex:wh{q}^E_k(k=1,2,3)}
By the same calculation as in Example $\ref{ex:Omega(cq)_*(beta_k)(k=1,2,3,4,5)}$, 
we have 
\begin{equation*} 
\begin{array}{rllll} 
 \wh{q}^{E}_{1}  &=  & 2h_{1} = Q_{1}, \medskip \\
 \wh{q}^{E}_{2}  &=  & 2h_{1}^{2} - a_{1, 1}  h_{1} = Q_{2} - a_{1, 1}h_{1}, \medskip \\
 \wh{q}^{E}_{3}  &=  & 2(h_{3} -  h_{2}h_{1} + h_{1}^{3})  + 2a_{1, 1}h_{2} 
          - 3a_{1, 1} h_{1}^{2}   + a_{1, 1}^{2} h_{1}   \medskip \\
                    & = & Q_{3} + 2a_{1, 1}h_{2} - 3a_{1, 1} h_{1}^{2}   + a_{1, 1}^{2} h_{1}.  \medskip   
\end{array}
\end{equation*}  
\end{ex}
\noindent
By Definition \ref{df:wh{q}^E},  the relation  
\begin{equation}   \label{eqn:relations(q^E)} 
   \wh{q}^{E}(T)  \wh{q}^{E}(\overline{T})   % \wh{q}^{E}([-1]_{E}(T))  = H(T) E(-[-1]_{E}(T)) H([-1]_{E}(T))E(-T) 
   = 1
\end{equation} 
holds.    
\begin{ex}   For example,  the relations in low degrees are given by 
\begin{equation*} 
\begin{array}{rllll} 
  (\wh{q}^{E}_{1})^{2}  &= & 2\wh{q}^{E}_{2} + a_{1,1} \wh{q}^{E}_{1}, \medskip \\
  (\wh{q}^{E}_{2})^{2}  &= & -2\wh{q}^{E}_{4} + 2\wh{q}^{E}_{3}\wh{q}^{E}_{1}  - 3a_{1,1}\wh{q}^{E}_{3}
  + a_{1,1}\wh{q}^{E}_{2} \wh{q}^{E}_{1} - a_{1, 1}^{2} \wh{q}^{E}_{2}     \medskip \\
  & & + (-a_{1,1}a_{1,2} - 2a_{1,3}+ a_{2,2})\wh{q}^{E}_{1}.   \medskip  
\end{array} 
\end{equation*}  
\end{ex}

\begin{defn} [{\it  Ring of $E$-homology Schur $Q$-functions}]   \label{df:Gamma^E_*}    
Define $\Gamma^{E}_{*}$   
to be the $E_{*}$-subalgebra of $\Lambda^{E}_{*} = E_{*} \otimes_{\Z} \Lambda$ generated by $\wh{q}^{E}_{1},  \wh{q}^{E}_{2}, \ldots$.  
More explicitly  we define 
\begin{equation*} 
   \Gamma^{E}_{*} 
   =  E_{*}[\wh{q}^{E}_{1}, \wh{q}^{E}_{2}, \ldots, \wh{q}^{E}_{i}, \ldots]/
  (\wh{q}^{E}(T) \wh{q}^{E}(\overline{T}) = 1) 
  \;    \hooklongrightarrow   \;  \Lambda^{E}_{*}. 
\end{equation*}
\end{defn} 
\noindent 
 For $E = H$, the ordinary   homology theory, 
$\wh{q}^{H}_{i} = Q_{i} \; (i = 1, 2, \ldots)$, and hence 
$\Gamma^{H}_{*} = \Gamma  = \Z[Q_{1}, Q_{2}, \ldots]/(Q(T) Q(-T)   = 1)$
is the ring of Schur $Q$-functions.

%%%%%%%%%%%%%%%%%%%%%%%%%%%%%%%%%%%%%%%%%%%%%%%%%%%%%%%%%%%%%%%%%%%%%%%%%%%%%%%%%%%%%%%%%%%%%%%%% 
Analogously, we make the following definition. 
\begin{defn} [{\it  $E$-homology Schur $P$-functions}]   \label{df:wh{p}^E}  
We define     $\wh{p}^{E}_{k}  = \wh{p}^{E}_{k}({\bf y})  \in \Lambda_{*}^{E}  \; (k = 1, 2, \ldots)$    by 
the equation%\footnote{
%Compare this with the equation (\ref{eqn:eta^E}).  
%} 
\begin{equation}    \label{eqn:wh{p}^E}
   1 + [2]_{E}(T) \wh{p}^{E}(T)  = 1 + (T +_{\mu} T) \wh{p}^{E} (T)
   =  \wh{q}^{E}(T)   \in \Lambda^{E}_{*}[[T]],     
\end{equation} 
where we put $\wh{p}^{E}(T) := \sum_{j \geq 0} \wh{p}^{E}_{j+1} T^{j}$.   
%and $\wh{p}^{E}_{0} := 1 (0?)$ by convention. 
Equivalently, since 
\begin{equation*} 
  [2]_{E}(T) \wh{p}^{E} (T) = \left ( \sum_{i \geq 1} \alpha^{E}_{i} T^{i} \right ) 
                              \left ( \sum_{j \geq 0} \wh{p}^{E}_{j + 1} T^{j} \right ) 
 % = \sum_{i \geq 1, j \geq 0}  \alpha^{E}_{i} \wh{p}^{E}_{j + 1} T^{i + j}  
  = \sum_{k \geq 1} \left (\sum_{j=1}^{k}  \alpha^{E}_{k + 1 - j} \wh{p}^{E}_{j} \right )T^{k},   
\end{equation*} 
 $\wh{p}^{E}_{k}$'s  are defined   by the recursive formulas$:$ 
\begin{equation}  \label{eqn:wh{q}^E_k=2wh{p}^E_k+...}  
\wh{q}^{E}_{k}  =  \displaystyle{\sum_{j = 1}^{k}}   \alpha_{k + 1 -j}^{E} \wh{p}^{E}_{j} 
  = 2\wh{p}^{E}_{k}  + \alpha^{E}_{2}  \wh{p}^{E}_{k-1}  +  \cdots +  \alpha^{E}_{k} \wh{p}^{E}_{1}   \; (k = 1, 2, \ldots).  
\end{equation} 
\end{defn} 
\noindent
Actually we have to show that each $\wh{p}^{E}_{k} \; (k = 1, 2, \ldots)$ can be determined uniquely 
as an element of $\Lambda^{E}_{*}$  by these formulas. 
By (\ref{eqn:wh{p}^E}),  we have 
\begin{equation}  \label{eqn:[2]_E(T)p^E(T)}   
  [2]_{E}(T) \wh{p}^{E}(T) =   \wh{q}^{E}(T) - 1   = \dfrac{H(T)}{H(\overline{T})} - 1 
  =  \dfrac{H(T) - H(\overline{T})}{H(\overline{T})}.  
\end{equation} 
Since $H(T) - H(\overline{T})$ is divisible by $T - \overline{T}$,  
%\begin{equation*} 
%  H(T) - H(\overline{T}) = \sum_{i \geq 0} h_{i} T^{i}  - \sum_{i \geq 0} h_{i} \overline{T}^{i} 
%  = \sum_{i \geq 1}  h_{i} (T^{i} - \overline{T}^{i}) 
%  =   (T - \overline{T}) \sum_{i \geq 1} h_{i} (T^{i-1} +  T^{i-2} \cdot  \overline{T}  + \cdots + 
%      T \cdot  \overline{T}^{i-2} + \overline{T}^{i-1}),   
%\end{equation*} 
and $T - \overline{T}$ is divisible by $[2]_{E}(T)$, the right hand side of (\ref{eqn:[2]_E(T)p^E(T)}) is divisible by $[2]_{E}(T)$ (see also the argument in \S \ref{subsec:OneRowCase}).  
For the ordinary   cohomology theory,  we have $\wh{q}^{H}_{k} = Q_{k} \; (k = 1, 2, \ldots)$, 
and therefore $\wh{p}^{H}_{k} \; (k = 1, 2, \ldots)$ are the usual 
Schur $P$-functions $P_{k} \; (k = 1, 2, \ldots)$.   
In general, we have 
\begin{equation*} 
  \wh{p}^{E}_{k} = P_{k} + \text{lower terms in} \; \y = (y_{1}, y_{2}, \ldots). 
\end{equation*} 

\begin{ex}    \label{ex:wh{p}^E_k(k=1,2,3)}  
By the same calculation as in Example  $\ref{ex:Omegac(eta_k)(k=1,2,3,4,5)}$,    we have  
\begin{equation*} 
\begin{array}{rlll} 
   \wh{p}^{E}_{1}  &=  & h_{1} = P_{1}, \medskip \\
   \wh{p}^{E}_{2}  &=  & h_{1}^{2} - a_{1, 1} h_{1} = P_{2} - a_{1, 1}h_{1}, \medskip \\
   \wh{p}^{E}_{3}  &=  & (h_{3} - h_{2}h_{1} + h_{1}^{3})  + a_{1, 1}h_{2}  - 2a_{1, 1}h_{1}^{2} 
     + (a_{1, 1}^{2} - a_{1, 2}) h_{1}   \medskip \\
    & =  & P_{3}    + a_{1, 1}h_{2}  - 2a_{1, 1}h_{1}^{2} 
     + (a_{1, 1}^{2} - a_{1, 2}) h_{1}.  \medskip   
\end{array}  
\end{equation*}  
\end{ex}
\noindent
By the definition (\ref{eqn:wh{p}^E}),  the following relation holds: 
\begin{equation}   \label{eqn:relations(p^E)} 
    (1 + [2]_{E}(T)  \wh{p}^{E}(T)) (1 + [2]_{E}(\overline{T})  \wh{p}^{E}(\overline{T}))   
   = \wh{q}^{E}(T) \wh{q}^{E}(\overline{T}) %(1 + [2]_{E}([-1]_{E}(T))  \wh{p}^{E}([-1]_{E}(T)))
 = 1. 
\end{equation} 
Using the same argument as in \S \ref{subsec:E-homology(OmegaSp)},  
 with the above relation (\ref{eqn:relations(p^E)}), we can eliminate 
$\wh{p}^{E}_{2i} \; (i = 1, 2, \ldots)$.
\begin{defn}   [{\it Ring of $E$-homology Schur $P$-functions}]    \label{df:Gamma'^E_*}  
Define ${\Gamma'}^{E}_{*}$   
to be the subalgebra of $\Lambda^{E}_{*}$ generated by $\wh{p}^{E}_{1},  \wh{p}^{E}_{2}, \ldots$. 
More explicitly, we define 
\begin{equation*} 
\begin{array}{llll} 
   {\Gamma'}^{E}_{*}  & = E_{*}[\wh{p}^{E}_{1}, \wh{p}^{E}_{2}, \ldots, \wh{p}^{E}_{i}, \ldots]
                  /((1 + [2]_{E}(T) \wh{p}^{E}(T))(1 + [2]_{E}(\overline{T})  \wh{p}^{E}(\overline{T})) = 1)  \medskip \\
                 & = E_{*}[\wh{p}^{E}_{1}, \wh{p}^{E}_{3}, \ldots, \wh{p}^{E}_{2i-1}, \ldots]  
   \; \hooklongrightarrow   \;  \Lambda^{E}_{*}. 
\end{array} 
\end{equation*}                 
\end{defn}  
\noindent
The coalgebra structures of  $\Gamma^{E}_{*}$ and ${\Gamma'}^{E}_{*}$ are defined  by 
\begin{equation*} 
\begin{array}{rlll} 
  \phi (\wh{q}^{E}_{k})  & =  & \displaystyle{\sum_{i + j = k}} \wh{q}^{E}_{i} \otimes \wh{q}^{E}_{j}  
\quad (k  \geq 1), \medskip \\ 
  \phi (\wh{p}^{E}_{1}) & = & \wh{p}^{E}_{1} \otimes 1 + 1 \otimes \wh{p}^{E}_{1}, \medskip \\
  \phi (\wh{p}^{E}_{l})  & =  &  \wh{p}^{E}_{l} \otimes 1  + 1 \otimes \wh{p}^{E}_{l} + 
           \displaystyle{
                             \hspace{-0.3cm}   \sum_{  
                                  {\tiny 
                                     \begin{array}{cccc} 
                                      &  i + j + k = l, \\
                                      &  i, j \geq 1 
                                     \end{array}  
                                   } 
                                 }    
                         }  \hspace{-0.5cm}              
                                \alpha^{E}_{k + 1} \wh{p}^{E}_{i} \otimes \wh{p}^{E}_{j}   \medskip \\
                      &=  &  \wh{p}^{E}_{l} \otimes 1  + 1 \otimes \wh{p}^{E}_{l} + 
        2 \hspace{-0.3cm} 
            \displaystyle{ 
                               \sum_{
                                 {\tiny 
                                     \begin{array}{cccc} 
                                      &  i + j = l, \\
                                      &  i, j \geq 1 
                                     \end{array}  
                                   } 
                                 }   
                         }   \hspace{-0.3cm}    
                                 \wh{p}^{E}_{i} \otimes \wh{p}^{E}_{j} 
                        +  \; \alpha^{E}_{2} \hspace{-0.3cm}  
                              \displaystyle{ 
                               \sum_{
                                 {\tiny 
                                     \begin{array}{cccc} 
                                      &  i + j = l-1, \\
                                      &  i, j \geq 1 
                                     \end{array}  
                                   } 
                                 }   
                         }   \hspace{-0.3cm} 
                                \wh{p}^{E}_{i} \otimes \wh{p}^{E}_{j}   \medskip \\
             & &      + \cdots 
                        + \alpha^{E}_{l-1} \wh{p}^{E}_{1} \otimes \wh{p}^{E}_{1}  \quad (l  \geq 2). \medskip 
\end{array} 
\end{equation*}  
These colagebra structures make  $\Gamma^{E}_{*}$ and ${\Gamma'}^{E}_{*}$ into 
Hopf algebras  over $E_{*}$.

%\newpage
%%%%%%%%%%%%%%%%%%%%%%%%%%%%%%%%%%%%%%%%%%%%%%%%%%%%%%%%%%%%%%%%%%%%%%%%%%%%%%%%%%%%%%%%%%%%%%%%%%%%%%
\subsection{Rings of $E$-cohomology Schur $P$- and $Q$-functions}
\label{subsec:E-cohomologySchurP,Q-functions}  
Motivated by the description of $E^{*}(\Omega Sp)$ and $E^{*}(\Omega_{0} SO)$ 
in the previous two subsections \S \ref{subsec:E-cohomology(OmegaSp)} and 
\ref{subsec:E-cohomology(Omega_0SO)}, 
 we define the $E$-cohomology analogues of the rings of  Schur $P$- and $Q$-functions. 
 In order to do so, let $\overline{e}_{i} := e_{i}(\overline{x}_{1}, \overline{x}_{2}, 
 \ldots )$ (resp. $\overline{h}_{i} := h_{i}(\overline{x}_{1}, \overline{x}_{2}, \ldots)$)
 be the $i$-th elementary symmetric function (resp. the $i$-th complete symmetric function) 
in the variables 
 $\overline{x}_{1}, \overline{x}_{2}, \ldots$. Namely their generating functions
are  given by   
 \begin{equation*} 
   \overline{E}(T) = \sum_{i \geq 0} \overline{e}_{i} \, T^{i}  = \prod_{i \geq 1} 
 (1 + \overline{x}_{i} \, T) \quad \text{and}  \quad 
  \overline{H}(T) = \sum_{i \geq 0} \overline{h}_{i} \, T^{i}
 = \prod_{i \geq 1} \dfrac{1}{1 - \overline{x}_{i} \, T}.  
 \end{equation*} 
 Note that the relation $\overline{H}(T) \overline{E}(-T) = 1$ holds. 
 Then we make the following definition:
\begin{defn}       \label{df:tilde{q}^E}  
We define $\tilde{q}^{E}_{k} = \tilde{q}^{E}_{k}(\x)  \in \Lambda^{*}_{E} \; (k = 1, 2, \ldots)$ %\footnote{
%When we consider ``cohomology'', we shall use the variables $\x = (x_{1}, x_{2}, \ldots)$. 
%}
 as the coefficients  of the generating function
\begin{equation}     \label{eqn:tilde{q}^E(T)}  
  \tilde{q}^{E}(T)  = \sum_{k \geq 0} \tilde{q}^{E}_{k} T^{k}  
 %=    \overline{H}(-T) E(T) 
= \dfrac{E(T)}{ \overline{E} (T)}  =  \prod_{i \geq 1}  \dfrac{1 + x_{i}T}{1 + \overline{x}_{i}T}  
\;  \in \Lambda^{*}_{E}[[T]]  
\end{equation} 
$(\tilde{q}^{E}_{0} := 1)$.  
\end{defn} 
\noindent
For $E = H$, the ordinary cohomology theory, 
we have $\overline{x}_{i} = -x_{i}$, and hence 
$\tilde{q}^{H}_{k} \; (k = 1, 2, \ldots)$ coincide with the usual Schur $Q$-functions
$Q_{k} = \sum_{i +  j = k}  h_{i}e_{j} \; (k = 1, 2, \ldots)$.  By definition, 
\begin{equation*} 
  \tilde{q}^{E}(T) = \overline{H}(-T) E(T) = \left (\sum_{i \geq 0} \overline{h}_{i} \, (-T)^{i} \right ) \left (\sum_{j \geq 0}  e_{j} T^{j}  \right ) = 
  \sum_{k \geq 0} \left ( \sum_{i + j = k}  (-1)^{i} \overline{h}_{i}  e_{j}  \right ) T^{k}, 
\end{equation*} 
and we have $\tilde{q}^{E}_{k} =  \sum_{i + j = k} (-1)^{i}  \overline{h}_{i} e_{j}$.  
Since $\overline{x}_{i} = -x_{i} + \text{higher terms in} \;   x_{i}$, we have $\overline{h}_{i} 
= (-1)^{i} h_{i} + \text{higher terms}$, and hence we have 
\begin{equation*} 
  \tilde{q}^{E}_{k} = Q_{k} + \text{higher terms in}  \; \x = (x_{1}, x_{2}, \ldots). 
\end{equation*} 
We also define 
\begin{equation} 
   \overline{\tilde{q}^{E}}(T)  =  \sum_{k \geq 0}  \overline{\tilde{q}^{E}_{k}} T^{k} 
   = \dfrac{ \overline{E}(T)}{E(T)}  = \prod_{i \geq 1}  \dfrac{1 + \overline{x}_{i} T}{1 + x_{i}T}  
\end{equation} 
($\overline{\tilde{q}^{E}_{0}} := 1$).  
By Definition \ref{df:tilde{q}^E}, the relation 
\begin{equation}   \label{eqn:tilde{q}^E(T)bar{tilde{q}^E}(T)}  
  \tilde{q}^{E}(T) \overline{\tilde{q}^{E}}(T) 
  =  \dfrac{E(T)}{ \overline{E} (T)} \cdot 
  \dfrac{\overline{E}(T)}{E(T)}  = 1
\end{equation} 
holds.  %\footnote{
%It seems difficult to write down the relations among $\tilde{q}^{E}_{i}$'s explicitly from the equation %(\ref{eqn:tilde{q}^E(T)bar{tilde{q}^E}(T)}). 
%}

\begin{defn}  [Ring of $E$-cohomology Schur $Q$-functions] 
Define    the subalgebra $\Gamma^{*}_{E}$  of $\Lambda_{E}^{*}$  by 
\begin{equation*} 
  \Gamma_{E}^{*} = E^{*}[[\tilde{q}^{E}_{1}, \tilde{q}^{E}_{2}, \ldots, \tilde{q}^{E}_{i}, \ldots]]
                     /(\tilde{q}^{E}(T)  \overline{\tilde{q}^{E}}(T) = 1)  
  \hooklongrightarrow  \Lambda_{E}^{*}. 
\end{equation*}  
\end{defn} 
\noindent
 For $E = H$, the ordinary cohomology theory, $\tilde{q}^{H}_{i} = Q_{i} \; (i = 1, 2, \ldots)$, 
and hence $\Gamma_{H}^{*}  = \Gamma$ is the ring of Schur $Q$-functions.   
The coalgebra structure of $\Gamma^{*}_{E}$ is defined by 
\begin{equation*} 
  \phi (\tilde{q}^{E}_{k}) = \sum_{i + j = k}  \tilde{q}^{E}_{i} \otimes \tilde{q}^{E}_{j}.  
\end{equation*} 
This makes $\Gamma_{E}^{*}$ into a Hopf algebra over $E^{*}$.  

As we saw in subsection  \S \ref{subsec:E-cohomology(Omega_0SO)}, we do not have 
a good grasp of $E^{*}(\Omega_{0} SO) \cong \Hom_{E_{*}}(E_{*}(\Omega_{0} SO), E_{*})$. 
So we merely make the following definition. 
\begin{defn} [Ring of $E$-cohomology Schur $P$-functions] 
  Define ${\Gamma'}_{E}^{*}$ to be the graded dual of $\Gamma^{E}_{*}$ over $E_{*}$, 
namely 
${\Gamma'}_{E}^{*} = \Hom_{E_{*}} (\Gamma^{E}_{*},  E_{*})$.   
\end{defn}  
\noindent
${\Gamma'}_{E}^{*}$ has also a Hopf algebra structure over $E^{*}$ as a graded 
dual of the Hopf algebra $\Gamma^{E}_{*}$ over $E_{*}$.

%\newpage
%%%%%%%%%%%%%%%%%%%%%%%%%%%%%%%%%%%%%%%%%%%%%%%%%%%%%%%%%%%%%%%%%%%%%%%%%%%%%%%%%%%%%%%%%%%%%%%%%%%%%%%%%%%%%%%%%%%%
\subsection{Identifications with symmetric functions}   \label{subsec:IdentificationsSymmetricFunctions}  
In topology, it is well-known that the ordinary homology and cohomology of 
$BU$ can be identified  with the ring of symmetric functions $\Lambda$. 
This identification can be generalized obviously in generalized (co)homology theory
$E^{*}(-)$ and $E_{*}(-)$.   
In the previous subsection,  we defined 
\begin{equation*} 
\begin{array}{rlll}  
  \Lambda^{E}_{*} &=  & E_{*} \otimes_{\Z} \Lambda \quad (\text{extension of coefficients}), \medskip  \\ 
  \Lambda_{E}^{*} &=  & \Hom_{E_{*}} (\Lambda^{E}_{*}, E_{*}) \quad (\text{graded dual of}  \;  \Lambda^{E}_{*}). \medskip 
\end{array} 
\end{equation*} 
Then by Theorem \ref{thm:E^*(OmegaSU)E_*(OmegaSU)},    
we have the following identifications of Hopf algebras:   %\footnote{
%cf. Lenart  \cite[\S 4, p.232]{Len98}. 
%}:  %over $E_{*}$:
\begin{equation}  \label{eqn:identification}  
\begin{array}{rrlll} 
        & E_{*}(\Omega SU)  %&  \cong  & E_{*}(BU) 
         =   E_{*}[\beta^{E}_{1}, \beta^{E}_{2}, \ldots] 
          & \overset{\sim}{\longrightarrow} 
         \Lambda^{E}_{*} 
        = E_{*}[h_{1}, h_{2}, \ldots ],     \medskip \\
       &  \beta^{E}_{i}   & \longmapsto  h_{i}, \medskip \\
       &  E^{*}(\Omega SU) %  & \cong &  E^*(BU) 
          =   E^{*}[[c^{E}_{1}, c^{E}_{2}, \ldots  ]]   
          & \overset{\sim}{\longrightarrow}  
        \Lambda_{E}^{*}   = E^{*}[[e_{1}, e_{2}, \ldots ]],    \medskip \\
       &  c^{E}_{i}   &  \longmapsto  e_{i}.    \medskip
\end{array}  
\end{equation} 
Under  these identifications together with Lemmas \ref{lem:(Omega(c))_*}, 
\ref{lem:(Omega(r))_*(Omega(c))_*} (2),  \ref{lem:(Omega(c))^*(Omega(q))^*} (2), 
and \ref{lem:(Omega(r))^*}, 
  we can realize $E_{*}(\Omega Sp)$,  $E_{*}(\Omega_{0} SO)$, $E^{*}(\Omega Sp)$, 
  and $E^{*}(\Omega_{0} SO)$  as   algebras of certain symmetric functions.  
By Theorems \ref{thm:E_*(Omega_0SO)} and \ref{thm:E_*(OmegaSp)}, we have 
the following:  
\begin{prop}    \label{prop:IdentificationE_*(Omega_0SO)E_*(OmegaSp)}  
There is a natural  identification of Hopf algebras over $E_{*}:$ 
\begin{equation*} 
\begin{array}{rrllll} 
     & E_{*}(\Omega_{0} SO) =  \dfrac{E_{*}[\tau^{E}_{1}, \tau^{E}_{2}, \ldots  % \tau^{E}_{i}, \ldots
                                                        ]}
                               {(\tau^{E}(T) \tau^{E}(\overline{T}) = 1)} 
      &\overset{\sim}{\longrightarrow}
      \quad 
    \Gamma^{E}_{*}  =  \dfrac{E_{*}[\wh{q}^{E}_{1},  \wh{q}^{E}_{2}, \ldots    % \wh{q}^{E}_{i}, \ldots
                                              ]}
                    {(\wh{q}^{E}(T)  \wh{q}^{E}(\overline{T}) = 1)}, \medskip \\
                      %\hooklongrightarrow  \Lambda^{E}_{*}, \medskip \\ 
    &   \tau^{E}_{i}  & \longmapsto    \quad   \wh{q}^{E}_{i}. \medskip  \\
   & E_{*}(\Omega Sp)  = E_{*}[\eta^{E}_{1}, \eta^{E}_{3}, \ldots  % \eta^{E}_{2i-1}, \ldots
                                        ]  
                         &  \overset{\sim}{\longrightarrow} 
   \quad   {\Gamma'}^{E}_{*}  = E_{*}[\wh{p}^{E}_{1}, \wh{p}^{E}_{3}, \ldots   %\wh{p}^{E}_{2i-1}, \ldots
                                                ], \medskip \\  
  %  \hooklongrightarrow \Lambda^{E}_{*},  \medskip  \\
    &  \eta^{E}_{2i-1}   &  \longmapsto  \quad    \wh{p}^{E}_{2i-1}. \medskip 
  \end{array} 
\end{equation*} 
Furthermore,  
\begin{enumerate} 
\item the surjective homomorphism 
     \begin{equation*}
        (\Omega r)_{*}: E_{*}(\Omega SU)  \twoheadlongrightarrow E_{*}(\Omega_{0} SO),  \quad 
        \beta^{E}_{i}  \longmapsto  \tau^{E}_{i} \; (i = 1, 2, \ldots)
     \end{equation*} 
     is identified with  the surjection
     \begin{equation*}
          \Lambda^{E}_{*}  \twoheadlongrightarrow \Gamma^{E}_{*}, \quad  
        h_{i} \longmapsto  \wh{q}^{E}_{i} \; 
(i = 1, 2, \ldots), 
     \end{equation*}  
\item   the injective   homomorphism 
            $(\Omega c)_{*}: E_{*}(\Omega_{0} SO) \hooklongrightarrow E_{*}(\Omega SU)$
       is identified with    the natural inclusion $\Gamma^{E}_{*}   \hooklongrightarrow \Lambda^{E}_{*}$,

\item the homomorphism 
      \begin{equation*} 
            (\Omega q)_{*}: E_{*}(\Omega SU)  \longrightarrow E_{*}(\Omega Sp), \quad 
              \beta^{E}_{i}  \longmapsto   2\eta^{E}_{i} + \alpha^{E}_{2} \eta^{E}_{i-1} + \cdots + \alpha^{E}_{i} \eta^{E}_{1} \; (i = 1, 2, \ldots)
      \end{equation*} 
      is identified with   the homomorphism 
       \begin{equation*} 
     \Lambda^{E}_{*} \longrightarrow {\Gamma'}^{E}_{*}, \; h_{i} \longmapsto  2\wh{p}^{E}_{i} 
  + \alpha^{E}_{2}  \wh{p}^{E}_{i-1}  +  \cdots +  \alpha^{E}_{i} \wh{p}^{E}_{1}   \; (i = 1, 2, \ldots), 
     \end{equation*} 
 
\item  the injective  homomorphism $(\Omega c)_{*}: E_{*}(\Omega Sp) \hooklongrightarrow E_{*}(\Omega SU)$ 
       is identified with   the natural inclusion ${\Gamma'}^{E}_{*} \hooklongrightarrow \Lambda^{E}_{*}$. 
\end{enumerate}  
\end{prop}  
\noindent
Dually from Theorem \ref{thm:E^*(OmegaSp)} and the argument in subsection 
\S  \ref{subsec:E-cohomology(Omega_0SO)}, 
we have the following:  
\begin{prop}      \label{prop:IdentificationE^*(OmegaSp)E^*(Omega_0SO)}  
There is a natural identification of Hopf algebras over $E^{*}:$
\begin{equation*} 
\begin{array}{rrllll} 
   &  E^{*}(\Omega  Sp)  =  \dfrac{E^{*}[[\mu^{E}_{1}, \mu^{E}_{2}, \ldots  %\mu^{E}_{i}, \ldots
                                                     ]]} 
                                 {(\mu^{E}(T)  \overline{\mu^{E}}(T) = 1)}
     & \overset{\sim}{\longrightarrow}  
     \quad  
     \Gamma_{E}^{*} 
     = \dfrac{E^{*}[[\tilde{q}^{E}_{1}, \tilde{q}^{E}_{2}, \ldots  % \tilde{q}^{E}_{i}, \ldots
                        ]]}
       {(\tilde{q}^{E}(T)  \overline{\tilde{q}^{E}}(T) = 1)}, \medskip \\ 
  %\hooklongrightarrow  \Lambda_{E}^{*},    \medskip \\
  &  \mu^{E}_{i}  &  \longmapsto  \quad  \tilde{q}^{E}_{i}.  \medskip \\
  & E^{*}(\Omega_{0} SO)    
    &  \overset{\sim}{\longrightarrow} 
    \quad   {\Gamma'}_{E}^{*}.  \medskip %  \hooklongrightarrow \Lambda_{E}^{*}.  \medskip 
\end{array} 
\end{equation*} 
Furthermore, 
\begin{enumerate} 
\item the surjective  homomorphism 
     \begin{equation*} 
        (\Omega c)^{*}: E^{*}(\Omega SU) \twoheadlongrightarrow E^{*}(\Omega Sp), \quad 
          c^{E}_{i} \longmapsto  \mu^{E}_{i} \; (i = 1, 2, \ldots)
    \end{equation*} 
    is identified with   the surjection 
   \begin{equation*} 
      \Lambda^{*}_{E}  \twoheadlongrightarrow \Gamma^{*}_{E},   \quad  
      e_{i}  \longmapsto  \tilde{q}^{E}_{i} \; (i = 1, 2, \ldots), 
   \end{equation*}   
\item the injective homomorphism  $(\Omega q)^{*}:  E^{*}(\Omega Sp)  \hooklongrightarrow E^{*}(\Omega SU)$
    is identified with   
     the natural inclusion $\Gamma^{*}_{E}  \hooklongrightarrow \Lambda^{*}_{E}$, 
\item the  injective homomorphism $(\Omega r)^{*}: E^{*}(\Omega_{0} SO) \hooklongrightarrow E^{*}(\Omega SU)$
     is identified with  the natural inclusion ${\Gamma'}^{*}_{E} \hooklongrightarrow \Lambda^{*}_{E}$. 
\end{enumerate}
\end{prop} 

\begin{rem} 
For $E = H$, the ordinary  $($co$)$homology theory so that $\Lambda^{H}_{*} = \Lambda^{*}_
{H} = \Lambda$ and 
$\Gamma^{H}_{*} = \Gamma^{*}_{H} =  \Gamma$, the homomorphisms 
\begin{equation*} 
\begin{array}{llll} 
  (\Omega r)_{*}: H_{*}(\Omega SU)  \twoheadlongrightarrow  H_{*}(\Omega_{0} SO),   \quad 
  &  \beta_{i}   \longmapsto   \tau_{i} \; (i = 1, 2, \ldots),  \medskip \\
   (\Omega c)^{*}:  H^{*}(\Omega SU)  \twoheadlongrightarrow H^{*}(\Omega Sp),  \quad 
   &  c_{i}  \longmapsto   \mu_{i} \; (i = 1, 2, \ldots) \medskip 
\end{array}
\end{equation*} 
 give    a  geometric interpretation of the ring homomorphism 
$\varphi: \Lambda  \longrightarrow \Gamma, \; \varphi (e_{n})  =  Q_{n} \; (n \geq 1)$ 
$($Macdonald \cite[III, %p.264, 
Examples 10]{Mac95}$)$. 
\end{rem}

%\newpage
%%%%%%%%%%%%%%%%Chapter4%%%%%%%%%%%%%%%%%%%%%%%%%%%%%%%%%%%%%%%%%%%%%%%%%%%%%%%%%%%%%%%%%%%%%%%%%%%
\section{Universal factorial Schur $P$- and $Q$-functions}  \label{sec:UFSPQF} 
 In the previous section, 
  we  have constructed certain subalgebras $\Gamma^{E}_{*}$, ${\Gamma'}^{E}_{*}$ of 
 $\Lambda^{E}_{*}$  and $\Gamma_{E}^{*}$, ${\Gamma'}_{E}^{*}$ of $\Lambda_{E}^{*}$. 
 Our next task is to construct certain  symmetric functions which may serve as 
nice base for these algebras 
 as free $E_{*} \,  ($or $E^{*})$-modules.  
As mentioned in the introduction (see also Example \ref{ex:FGL}), 
Quillen showed that the complex cobordism theory $MU^{*}(-)$ (with the associated formal group law 
$\mu_{MU}$)  has 
the following  universal property: for any complex oriented cohomology theory 
$E^{*}(-)$ (with the associated formal group law $\mu_{E}$), 
there exists a homomorphism of rings $\theta: MU^{*} \longrightarrow E^{*}$ 
such that  $\mu_{E}(X, Y) = (\theta_{*} \mu_{MU}) (X, Y) 
= X  + Y + \sum_{i, j \geq 1} \theta (a^{MU}_{i, j}) X^{i}Y^{j}$.  
%\begin{thm} [Quillen, \cite{Qui69}, Theorem 2] 
%Since the complex cobordism theory $MU^{*}(-)$ is  ``universal''  among 
%complex oriented generalized cohomology 
%theories,  it will be  sufficient to consider the case $E = MU$. 
Thus it will be sufficient to consider the case when $E = MU$, 
for general case follows immediately from the universal one 
by the specialization $a^{MU}_{i, j} \longmapsto  \theta (a^{MU}_{i, j}) \; (i, j \geq 1)$.  
%In the universal  case, 
Recall that, by Quillen again, the coefficient ring $MU_{*} = MU^{-*}$ is isomorphic to the 
Lazard  ring  $\L$.   
In this section,    we shall construct the ``universal factorial Schur $P$- and $Q$-functions'' $P^{\L}_{\lambda} (\bf{x}|\bf{b})$, 
$Q^{\L}_{\lambda} (\bf{x}|\bf{b})$  
for $\lambda$ strict partitions. 
Then we shall see  at the end of Section  \ref{sec:DUFSPQF} that these functions 
constitute  required base   for ${\Gamma'}^{*}_{MU}$, $\Gamma^{*}_{MU}$ (when $\b = 0$).     
Since the functions $P^{\L}_{\lambda}(\x|\b)$'s and $Q^{\L}_{\lambda}(\x|\b)$'s  
will be of independent interest in terms of, e.g., algebraic combinatorics, 
so apart from geometry, 
 we shall deal with the above problem  purely algebraically in this section.

%%%%%%%%%%%%%%%%%%%%%%%%%%%%%%%%%%%%%%%%%%%%%%%%%%%%%%%%%%%%%%%%%%%%%%%%%%%%%%%%%%%%%%%%%%%%%%%%%%
\subsection{Lazard ring $\L$ and the universal formal group law}  
We begin with collecting  the basic facts about the Lazard ring
 (in Sections \ref{sec:UFSPQF} and \ref{sec:DUFSPQF}, 
we use the convention  as  in Levin-Morel's book \cite{Lev-Mor2007}).  
In \cite{Laz55},  Lazard considered a universal commutative formal group law 
of rank one $(\L, F_{\L})$,  where the ring $\L$, called the {\it Lazard ring}, 
is isomorphic to the polynomial ring 
in countably infinite number of variables with  integer coefficients,    %\footnote{
%Note that there  is a symmetric function realization of this ring by Lenart \cite{Len98}. 
%}.
and $F_{\L} = F_{\L}(u,v)$ is  the {\it universal formal group law} (for a construction 
and basic properties of $\L$, see Levine-Morel \cite[\S 1.1]{Lev-Mor2007}):
\begin{equation*} 
F_{\L} (u,v) =   u + v + \sum_{i,j \geq 1} a_{i,j} u^{i}  v^{j}   \in    \L[[u,v]].
\end{equation*} 
This is a formal power series in $u$,  $v$ with coefficients $a_{i,j}$ of formal variables
which satisfies the axiom of the  formal group law (see  \S \ref{subsec:Generalized(Co)homologyTheory}).
For the universal formal group law, we shall use the notation (see Levine-Morel \cite[\S 2.3.2]{Lev-Mor2007}) 
\begin{equation*} 
\begin{array}{llll} 
  &  a  \pf  b =  F_{\L}(a,b)   \quad  & \text{(formal sum)}, \medskip \\ 
  &  \overline{a} =   \chi_{_{\L}}(a)  & \text{(formal  inverse of} \;     a).    \medskip 
\end{array}
\end{equation*}
Note that  $\overline{a}\in \L[[a]]$ is a formal power series in $a$ with initial term $-a$, 
 and   first few terms appear  in  Levine-Morel  \cite[p.41]{Lev-Mor2007} (see also (\ref{eqn:[-1]_E(X)})).
In what follows, we regard $\L$ as a  {\it graded}   algebra over $\Z$, 
and the grading of $\L$ is given  as follows (see Levin-Morel \cite[p.5]{Lev-Mor2007}): 
The {\it homological}    (resp. {\it cohomological}) degree of $a_{i, j} \; (i, j \geq 1)$
 is defined by
$\deg_{h} \, (a_{i,j})=i+j-1$ (resp. $\deg^{h} (a_{i, j}) =  1 - i - j$).  
We will indicate which   grading of $\L$  we choose  in each time, and 
sometimes we  use the notation  $\L^{*}$ or $\L_{*}$ to avoid confusion. 
Be aware that in topology, it is customary to give   $a_{i, j}$ 
the {\it homological} degree  $2(i + j - 1)$ or the {\it cohomological} degree $2(1 - i - j)$
(see Section \ref{sec:Introduction}).

%%%%%%%%%%%%%%%%%%%%%%%%%%%%%%%%%%%%%%%%%%%%%%%%%%%%%%%%%%%%%%%%%%%%%%%%%%%%%%%%%%%%%%%%%%%%%%%%%%%%%%%%%%%
\subsection{Definition of $P^{\L}_{\lambda} (\x_{n}|\b)$, $Q^{\L}_{\lambda}(\x_{n}|\b)$} 
 \label{subsec:DefinitionP^L(x|b)Q^L(x|b)}
 The {\it factorial Schur $P$- and $Q$-functions} were   first introduced by 
Ivanov \cite[Definitions 2.10 and 2.13]{Iva2004} %\footnote{
%In that paper, he calls these functions the {\it multi-parameter Schur $P$- and $Q$-
%functions}. 
%}
 (for their definition, see also Ikeda-Mihalcea-Naruse \cite[\S 4.2]{IMN2011}), 
and it turns out that these functions represent the 
Schubert classes of the torus equivariant cohomology rings 
of Lagrangian or orthogonal Grassmannians (see 
a series of works due to Ikeda \cite[Theorem 6.2]{Ike2007}, 
Ikeda-Naruse \cite[Theorem 8.7]{Ike-Nar2009}).  
Furthermore, in  \cite[Definition 2.1]{Ike-Nar2013},  they  introduced 
the $K$-theoretic  analogue of the  factorial Schur $P$- and $Q$-functions.
They showed that these functions also represent the Schubert classes of 
the torus equivariant $K$-theory of the above homogeneous spaces.

In this subsection, we shall generalize their definition to the case of {\it complex 
cobordism  cohomology theory}.  
Besides the variables $\x = (x_{1}, x_{2}, \ldots)$, we 
prepare another set of variables $\b = (b_{1}, b_{2}, \ldots)$.  
We provide the variables
$\x=(x_{1},  x_{2},  \ldots)$ and $\b = (b_{1},  b_{2},  \ldots)$ with degree
$\deg  \,  (x_{i}) = \deg \,   (b_{i})= 1 $ for $i = 1, 2, \ldots$, 
and we use $\deg^{h} \, (a_{i, j}) = 1 - i - j
\; (i, j \geq 1)$.   Following \S \ref{subsec:E-homologySchurP,Q-functions}, 
\ref{subsec:IdentificationsSymmetricFunctions}, we introduce the rings 
of symmetric functions 
\begin{equation*} 
\begin{array}{llll} 
  \Lambda_{\L}(\x) & := & \L_{*} \otimes_{\Z}  \Lambda (\x), \medskip \\
  \Lambda^{\L}(\x) & := & \Hom_{\L_{*}} (\Lambda_{\L} (\x),  \L_{*}). \medskip 
\end{array} 
\end{equation*} 
By definition, we have $\Lambda_{\L}(\x) \cong \Lambda_{*}^{MU}$ and 
$\Lambda^{\L}(\x) \cong \Lambda^{*}_{MU}$, and we know that 
\begin{equation*} 
\begin{array}{rllll} 
  \Lambda_{\L}(\x)  & \cong  & \L_{*} [h_{1}(\x), h_{2}(\x), \ldots] \; (\cong MU_{*}(\Omega SU)), \medskip \\
  \Lambda^{\L} (\x) & \cong  & \L^{*} [[e_{1}(\x), e_{2}(\x), \ldots ]] \; (\cong MU^{*}(\Omega SU)).  \medskip 
\end{array}
\end{equation*}       
Moreover we put 
\begin{equation*} 
\begin{array}{rllll} 
   \Lambda_{\L}(\x| \b) &:= & \L[[\b]] \, \hat{\otimes}_{\L}  \, \Lambda_{\L} (\x),  \medskip  \\
   \Lambda^{\L}(\x|\b)  &:= & \L[[\b]] \,  \hat{\otimes}_{\L}  \, \Lambda^{\L} (\x),    \medskip  
\end{array} 
\end{equation*}  
where $\L[[\b]] = \L[[b_{1}, b_{2}, \ldots]]$ is a formal power series ring in 
the variables $b_{1}, b_{2}, \ldots$.  
In this section, we also consider the ring of symmetric formal power series 
 of finite variables 
%When we work with the finite set of variables 
$\x_{n} = (x_{1}, \ldots, x_{n})$ with coefficients in $\L$ (or $\L[[\b]]$),  
and  use the notation  %$\Lambda_{\L}(\x_{n})$, %
% = \L_{*} \otimes_{\Z} \Lambda (\x_{n})
%\cong \L_{*}[h_{1}(\x_{n}), \ldots, h_{n}(\x_{n})]$, 
\begin{equation*} 
    \Lambda^{\L}(\x_{n})  :=  \L^{*} [[e_{1}(\x_{n}), \ldots, e_{n}(\x_{n})]] \quad 
 \text{and}  \quad  \Lambda^{\L}(\x_{n}|\b) := \L[[\b]] \, \hat{\otimes}_{\L} \, \Lambda^{\L}(\x_{n}).
\end{equation*}     
In what follows, when considering polynomials or formal power series 
$f(x_{1}, x_{2}, \ldots)$ with coefficients in $\L$ (or $\L[[\b]]$), 
we shall call the degree with respect to $x_{1}, x_{2}, \ldots, b_{1}, b_{2}, \ldots$, and 
$a_{i, j} \; (i, j \geq 1)$  the {\it total} degree of $f(x_{1}, x_{2}, \ldots)$.  
%we shall distinguish the {\it total} degree and the {\it polynomial degree} 

For an integer $k\geq 1$, we define a generalization of the 
ordinary $k$-th power $t^{k}$  %\footnote{
%cf. Ikeda-Naruse \cite[\S 2.1]{Ike-Nar2013}. 
%}
 by 
\begin{equation*} 
   [t| \b]^{k}  := \displaystyle{\prod_{i=1}^{k}}  (t\pf b_{i})  
               = (t \pf b_{1})(t \pf b_{2}) \cdots (t \pf b_{k}) 
\end{equation*} 
and its variant by 
\begin{equation*} 
   [[t| \b  ]]^{k}  :=(t \pf t)[t| \b ]^{k-1}  
  = (t \pf t) (t \pf b_{1})(t \pf b_{2}) \cdots (t \pf b_{k-1}),
\end{equation*}  
where we  set $[t| \b ]^{0} =   [[t| \b ]]^{0} := 1$.
For a partition, i.e., a non-increasing sequence of non-negative integers 
 $\lambda=(\lambda_1,\ldots,\lambda_r) \;
 (\lambda_{1} \geq \lambda_{2} \geq \cdots \geq \lambda_{r} \geq 0)$, 	
we set
\begin{equation*} 
  [\x| \b ]^{\lambda}  :=  \displaystyle{\prod_{i=1}^{r}}  [x_i|\b]^{\lambda_i}  \quad 
\text{and}  \quad 
[[\x| \b ]]^{\lambda}  :=  \displaystyle{\prod_{i=1}^{r}} [[x_i|\b]]^{\lambda_i}.
\end{equation*} 
Let    $\lambda = (\lambda_{1}, \ldots, \lambda_{r})$
be a strict partition of length $\ell (\lambda) = r$, i.e.,  
a sequence of positive integers 
such that $\lambda_{1}  > \cdots > \lambda_{r} > 0$.   
We denote by $\mathcal{SP}$ the set of all strict partitions, and
$\mathcal{SP}_n$ the subset of $\mathcal{SP}$ consisting of 
strict partitions of length $r  \leq n$.  
The following Hall-Littlewood type  %\footnote{
%cf. Macdonald \cite[III, \S 2]{Mac95}.  
%} 
definition with coefficients in $\L[[\b]]$  
was suggested to us by Anatol Kirillov (cf. Kirillov-Naruse \cite{Kir-Nar}), 
and   we thank him for this.

%%%%%%%%%%%%%%%%%%%%%%%%%%%%%%%%%%%%%%%%%%%%%%%%%%%%%%%%%%%%%%%%%%%%%%%%%%%%%%%%%%%%%%%%%%%%%%%%
\begin{defn}  [Universal factorial Schur $P$- and $Q$-functions]   \label{df:DefinitionP^L(x_n|b)Q^L(x_n|b)} 
For a strict partition $\lambda=(\lambda_{1},  \ldots, \lambda_{r})$ with length $\ell (\lambda ) = r \leq n$, we define 
%\begin{eqnarray}
\begin{equation}    \label{eqn:DefinitionP^L(x_n|b)Q^L(x_n|b)}  
\begin{array}{rlll} 
& P^{\L}_{\lambda}(\x_{n}|\b)  =  P^{\L}_{\lambda} (x_{1}, \ldots, x_{n} |\b) 
:=  
\displaystyle\frac{1}{(n-r)!}
\sum_{w  \in  S_{n}} w 
 \left[
[\x | \b]^{\lambda}
\prod_{i=1}^r
\prod_{j=i+1}^{n}   \frac{x_{i}\pf x_{j}}{x_{i}\pf \overline{x}_{j}}\right],  \medskip \\
%\label{eqn:P^L(x_n|b)}     \\
%%%%%%%%%%%%%%%%%%%%%%%%%%%%%%%%%%%%%%%%%%%%%%%%%%%%%%%%
& Q^{\L}_{\lambda}(\x_{n}|\b) =   Q^{\L}_{\lambda} (x_{1}, \ldots, x_{n} |\b)
:=    
\displaystyle\frac{1}{(n-r)!}
\sum_{w  \in  S_{n}}w\left[
[[\x|\b]]^{\lambda}
\prod_{i=1}^r
\prod_{j=i+1}^n \frac{x_{i}\pf x_{j}}{x_{i}\pf \overline{x}_{j}}
\right],    \medskip 
%\label{eqn:Q^L(x_n|b)}   
\end{array} 
\end{equation} 
%%%%%%%%%%%%%%%%%%%%%%%%%%%%%%%%%%%%%%%%%%%%%%%%%%%%%%%%
%\end{eqnarray}
where the symmetric group $S_{n}$ acts only on the $x$-variables $x_{1}, \ldots, x_{n}$
by permutations.   If $\ell (\lambda)  = r > n$, we set $P^{\L}_{\lambda}(\x_{n}|\b) 
= Q^{\L}_{\lambda}(\x_{n}|\b) = 0$.  
\end{defn} 
\noindent
Since 
\begin{equation*} 
  x_{i} \pf \overline{x}_{j}=  (x_{i} -x_{j})(1+\text{higher degree terms in } x_{i} \text{ and } x_{j}
\text{ with coefficients in }\L)%\footnote{
%More precisely, we compute
%\begin{equation*} 
%\begin{array}{llll} 
%  x \pf \overline{y} & = (x - y) 
%  [   1 - a_{1, 1}y  + (-a_{1, 2}xy + a_{1, 1}^{2}y^{2})  \medskip \\
%       &  + \{ -a_{1, 3}x^{2}y + (a_{2, 2}  + a_{1, 1}a_{1, 2} - a_{1, 3}) xy^{2}  
%             + (a_{2, 2} - a_{1, 1}^{3} - a_{1, 1}a_{1, 2} - 2a_{1, 3}) y^{3} \} + \cdots    
%  ].   \medskip 
%\end{array}  
%\end{equation*} 
%}
, 
\end{equation*}  
$P^{\L}_{\lambda} (x_{1}, \ldots, x_{n}|\b)$ and
$Q^{\L}_{\lambda} (x_{1}, \ldots, x_{n}|\b)$
are well defined  elements of $\Lambda^{\L}(\x_{n}|\b)$,  % $\L[[\b]][[x_1,\ldots,x_n]]^{S_{n}}$, 
namely these are   symmetric formal power series with coefficients in $\L$ 
in the variables $x_{1},  \ldots, x_{n}$ and
formal power series in $b_{1}, b_{2}, \ldots, b_{\lambda_{1}}$ for $P^{\L}_{\lambda}$ 
(resp. $b_{1}, b_{2} \ldots, b_{\lambda_{1} - 1}$ for $Q^{\L}_{\lambda}$).  
Notice that these are  homogeneous formal power series of 
{\it total}     degree  $|\lambda|  = \sum_{i=1}^{r} \lambda_{i}$, 
the size of $\lambda$.  
We call these formal power series the 
{\it universal factorial Schur $P$- and $Q$- functions}.
In  Definition \ref{df:DefinitionP^L(x_n|b)Q^L(x_n|b)}, 
if we put $a_{i, j} = 0$ for all $i, j \geq 1$, the functions 
$P^{\L}_{\lambda} (\x_{n}|\b)$, $Q^{\L}_{\lambda}(\x_{n}|\b)$ reduce to 
the usual  factorial Schur 
$P$- and $Q$-polynomials $P_{\lambda}(\x_{n}|\b)$, $Q_{\lambda}(\x_{n}|\b)$.  
If we put $a_{1, 1} = \beta$, 
$a_{i, j}  = 0$ for all $(i, j) \neq (1, 1)$, then $P^{\L}_{\lambda}(\x_{n}|\b)$, $Q^{\L}_{\lambda}(\x_{n}|\b)$ reduce to the $K$-theoretic factorial Schur 
$P$- and $Q$-polynomials $GP_{\lambda}(\x_{n}|\b)$, $GQ_{\lambda}(\x_{n}|\b)$ 
due to Ikeda-Naruse.  Thus our functions $P^{\L}_{\lambda}(\x_{n}|\b)$, $Q^{\L}_{\lambda}(\x_{n}|\b)$
are  a generalization of these polynomials  and hence {\it universal} in this sense.

\begin{ex}   \label{ex:P^L_1Q^L_1}   
\quad 
\begin{enumerate} 
\item For $n = 1$ and $\lambda  = (1)$, we have 
\begin{equation*} 
\begin{array}{rlll} 
    P^{\L}_{(1)}(x_{1}|\b)  &= & [x_{1}|\b] = x_{1} \pf b_{1} = x_{1} + b_{1} + a_{1, 1}x_{1}b_{1} 
    + \cdots,    \medskip \\
    Q^{\L}_{(1)}(x_{1}|\b)  &= &  [[x_{1}|\b]]
     = x_{1} \pf x_{1}  =  2x_{1} + a_{1, 1}x_{1}^{2} + \cdots.    \medskip 
\end{array}  
\end{equation*} 
\item 
For $n = 2$ and $\lambda = (1) = (1, 0)$, we have 
\begin{equation*} 
\begin{array}{rlll} 
  P^{\L}_{(1)} (x_{1}, x_{2} |  \b) 
  & =  &  (x_{1} \pf x_{2}) 
   \left (\dfrac{x_{1} \pf b_{1}}{x_{1} \pf \overline{x}_{2}}  
 +  \dfrac{x_{2} \pf b_{1}} {x_{2}  \pf \overline{x}_{1}}   \right ), \medskip \\ 
  Q^{\L}_{(1)} (x_{1}, x_{2} |   \b)
  & = &  (x_{1} \pf x_{2}) 
    \left ( \dfrac{x_{1} \pf x_{1}}{x_{1} \pf \overline{x}_{2}}  
          + \dfrac{x_{2} \pf x_{2}}{x_{2} \pf  \overline{x}_{1}}   \right ).  \medskip  
\end{array}   
\end{equation*} 
\end{enumerate}  
\end{ex} 
\noindent 
We also set 
\begin{equation*} 
\begin{array}{rllll} 
P^{\L}_{\lambda}(\x_{n}) & = P^{\L}_{\lambda} (x_{1}, \ldots, x_{n}) 
:= P^{\L}_{\lambda} (x_{1},  \ldots, x_{n}| 0), \medskip \\
Q^{\L}_{\lambda}(\x_{n}) & = Q^{\L}_{\lambda} (x_{1},  \ldots, x_{n}) 
:= Q^{\L}_{\lambda} (x_{1}, \ldots, x_{n}|0).  \medskip 
\end{array} 
\end{equation*} 
These are   homogeneous   symmetric formal power  series in  $\Lambda^{\L}(\x_{n})$
%$\L[[x_{1}, \ldots, x_{n}]]^{S_{n}}$ 
of total degree $|\lambda|$. 
By definition, we have 
\begin{equation*} 
\begin{array}{rlll} 
  P^{\L}_{\lambda}(x_{1}, \ldots, x_{n}) 
 & = &  P_{\lambda}(x_{1}, \ldots, x_{n}) 
   +  g_{\lambda}(x_{1}, \ldots, x_{n}), \medskip \\
  Q^{\L}_{\lambda}(x_{1}, \ldots, x_{n}) 
 & = &  Q_{\lambda}(x_{1}, \ldots, x_{n}) 
  +  h_{\lambda}(x_{1}, \ldots, x_{n}),     \medskip 
\end{array}
\end{equation*} 
where  $P_{\lambda}(x_{1}, \ldots, x_{n})$ and $Q_{\lambda}(x_{1}, \ldots, x_{n})$ 
are the usual  Schur $P$- and $Q$-polynomials, and 
$g_{\lambda}(x_{1}, \ldots, x_{n})$ and $h_{\lambda}(x_{1}, \ldots, x_{n})$
are formal power series with coefficients in $\L$ 
 in $x_{1}, \ldots, x_{n}$ 
whose degrees   with respect to $x_{1}, \ldots, x_{n}$ 
 are strictly greater  than $|\lambda|$. 
  From this, we see that 
$P^{\L}_{\lambda}(x_{1}, \ldots, x_{n})$'s 
(resp.  $Q^{\L}_{\lambda}(x_{1}, \ldots, x_{n})$'s), 
$\lambda \in \mathcal{SP}_{n}$, are linearly independent over $\L$.  
For suppose that there exists a   homogeneous   relation of the form 
\begin{equation}  \label{eqn:LinearIndependenceP^L}  
   \sum_{\lambda  \in \mathcal{A}}  c_{\lambda} P^{\L}_{\lambda}(x_{1}, \ldots, x_{n}) = 0
\end{equation} 
with $c_{\lambda} \in \L$, $\deg  \, (c_{\lambda}) + |\lambda| = N > 0$, 
and the summation ranges over a certain subset $\mathcal{A} \subset \mathcal{SP}_{n}$. 
Since $\deg \, (c_{\lambda}) \leq 0$ by our convention of the grading of $\L = \L^{*}$, 
we have $|\lambda| \geq N$ for all $\lambda$ that appear in the above relation (\ref{eqn:LinearIndependenceP^L}). 
Let $N_{0} \geq N$ be the minimum of the size $|\lambda|$ for all $\lambda \in \mathcal{A}$, 
and $\mathcal{A}_{0} := \{ \lambda \in \mathcal{A} \; | \; |\lambda| = N_{0} \} \subset \mathcal{A}$. 
Then we have from (\ref{eqn:LinearIndependenceP^L}), 
    $\sum_{\lambda \in \mathcal{A}_{0}}  c_{\lambda}  P^{\L}_{\lambda}(x_{1}, \ldots, x_{n}) = 0$.  
 Thus we have 
\begin{equation*} 
\begin{array}{llll} 
 0 & = \displaystyle{\sum_{\lambda  \in \mathcal{A}_{0}}}  
   c_{\lambda}  P^{\L}_{\lambda} (x_{1}, \ldots, x_{n}) 
 = \sum_{\lambda  \in \mathcal{A}_{0}} c_{\lambda} (P_{\lambda}(x_{1}, \ldots, x_{n}) 
  + g_{\lambda}(x_{1}, \ldots, x_{n}))   \medskip \\
  & =  \displaystyle{\sum_{\lambda  \in \mathcal{A}_{0}}}  c_{\lambda} P_{\lambda}(x_{1}, \ldots, x_{n}) 
  +  \sum_{\lambda \in \mathcal{A}_{0}} c_{\lambda} g_{\lambda}(x_{1}, \ldots, x_{n}).   \medskip 
\end{array}  
\end{equation*} 
Considering the degrees with respect to $x_{1}, \ldots, x_{n}$, we have 
$\sum_{\lambda \in \mathcal{A}_{0}} c_{\lambda}  P_{\lambda}(x_{1}, \ldots, x_{n}) = 0$. 
On the other hand, it is known that $P_{\lambda} (x_{1}, \ldots, x_{n})$'s, $\lambda \in \mathcal{SP}_{n}$, form a  $\Z$-basis  for the ring  $\Gamma (\x_{n})$ 
of {\it supersymmetric} polynomials in $n$ variables $x_{1}, \ldots, x_{n}$ 
(see Pragacz \cite[Theorem 2.11]{Pra91}, Ivanov \cite[Proposition 2.12]{Iva2004}). 
Therefore  they form an $\L$-basis   for the ring $\L \otimes_{\Z} \Gamma (\x_{n})$. 
In particular, they are linearly independent over $\L$, and hence we conclude 
that $c_{\lambda} = 0$ for all $\lambda  \in \mathcal{A}_{0}$. 
Next we consider the minimum $N_{1}$ of the size of $|\lambda|$ for all $\lambda \in \mathcal{A} \setminus \mathcal{A}_{0}$, and  repeat the above  argument. 
In this way,   we conclude that $c_{\lambda}  = 0$ for all $\lambda \in \mathcal{A}$.

Similarly, by definition, we have 
\begin{equation*} 
\begin{array}{rlll} 
  P^{\L}_{\lambda}(x_{1}, \ldots, x_{n}|\b) 
 & =  & P_{\lambda}(x_{1}, \ldots, x_{n}|\b) 
  +  g_{\lambda}(x_{1}, \ldots, x_{n}|\b), \medskip \\
  Q^{\L}_{\lambda}(x_{1}, \ldots, x_{n}|\b) 
 & =  & Q_{\lambda}(x_{1}, \ldots, x_{n}|\b) 
  +  h_{\lambda}(x_{1}, \ldots, x_{n}|\b),     \medskip 
\end{array}
\end{equation*} 
where  $P_{\lambda}(x_{1}, \ldots, x_{n}|\b)$ and $Q_{\lambda}(x_{1}, \ldots, x_{n}|\b)$ 
are the usual factorial Schur $P$- and $Q$-polynomials, and 
$g_{\lambda}(x_{1}, \ldots, x_{n}|\b)$ and $h_{\lambda}(x_{1}, \ldots, x_{n}|\b)$
are formal power series  in $x_{1}, \ldots, x_{n}$, $b_{1}, b_{2}, \ldots$ 
whose degrees with respect to $x_{1}, \ldots, x_{n}$, $b_{1}, b_{2}, \ldots$ are 
strictly greater  than $|\lambda|$.  
The  verification of linear independence of $P^{\L}_{\lambda}(x_{1}, \ldots, x_{n}|\b)$'s 
(resp. $Q^{\L}_{\lambda}(x_{1}, \ldots, x_{n}|\b)$'s), $\lambda \in \mathcal{SP}_{n}$, 
over $\L[[\b]]$ will be deferred  until we show the ``vanishing property'' 
of these functions   in \S \ref{subsec:VanishingPropertyP^L(x|b)LQ^L(x|b)}.

%%%%%%%%%%%%%%%%%%%%%%%%%%%%%%%%%%%%%%%%%%%%%%%%%%%%%%%%%%%%%%%%%%%%%%%%%%%%%%%%%%%%%%%%%%%%%%%%%%%%%%%%%%%%%%%%%
\subsection{$\L$-supersymmetric series}    \label{subsec:L-supersymmetricity}
In this subsection, we introduce the notion of  the ``$\L$-supersymmetricity'' 
which is a generalization of the ``$Q$-cancellation property'' due to Pragacz \cite[p.145]{Pra91}, 
``supersymmetricity'' due to Ivanov \cite[Definition 2.1]{Iva2004}, and 
``$K$-supersymmetric property ($K$-theoretic $Q$-cancellation property)''
 due to Ikeda-Naruse \cite[Definition 1.1]{Ike-Nar2013}.  
$\L$-supersymmetric formal series is defined as follows.
\begin{defn}  [$\L$-supersymmetric series]  
A formal power series $f(x_{1}, \ldots, x_{n})$ in the variables $x_{1}, \ldots, x_{n}$ 
 with coefficients in  $\L$  is called    {$\L$-supersymmetric} if
\begin{itemize}
\item[(1)] $f(x_1,\ldots,x_n)$ is symmetric in the  variables $x_{1},  \ldots, x_{n}$, and
\item[(2)] $f(t,  \overline{t},  x_{3}, \ldots, x_{n})$ does not depend on $t$, or in other words, 
           $f(t, \overline{t}, x_{3}, \ldots, x_{n}) = f(0, 0, x_{3}, \ldots, x_{n})$ holds. 
\end{itemize}
\end{defn}
\noindent
For example, the formal power series $x_{1} \pf x_{2} \pf \cdots \pf x_{n}$ 
is $\L$-supersymmetric. 
Formal power series with the $\L$-supersymmetric property form a subring of 
$\Lambda^{\L}(\x_{n})$, 
and we denote it by  $\Gamma^{\L}(\x_{n})$.  
Hereafter we shall call $\Gamma^{\L}(\x_{n})$ the {\it ring of $\L$-supersymmetric functions} 
in the  $n$ variables $x_{1}, \ldots, x_{n}$. 
We also define  $\Gamma_{+}^{\L}(\x_{n})$   to be the subring  of $\Gamma^{\L}(\x_{n})$ consisting of 
$f(x_{1},  \ldots, x_{n}) \in \Gamma^{\L}(\x_{n})$ such that
$f(t, x_{2},  \ldots, x_{n}) - f(0, x_{2},  \ldots, x_{n})$ is divisible by $t \pf t$.
We also define   
\begin{equation*} 
  \Gamma^{\L}(\x_{n}|\b) := \L[[\b]] \,  \hat{\otimes}_{\L}   \,   \Gamma^{\L}(\x_{n}) \quad   
  \text{and} \quad  
   \Gamma^{\L}_{+}(\x_{n}|\b) := \L[[\b]]  \,  \hat{\otimes}_{\L}  \,    \Gamma^{\L}_{+}(\x_{n}).
\end{equation*} 
% Note that
% $a-\overline{b}$ is divisible by $a\pf b$, because 
% it becomes zero when $a=\overline{b}$.
% Therefore $t-\overline{t}$ is divisible by $t\pf t$.

\begin{prop}  \label{prop:L-supersymmetry}  
\quad 
\begin{enumerate} 
\item For $\lambda \in \mathcal{SP}_{n}$, 
$P^{\L}_\lambda(x_{1}, \ldots, x_{n} )$  and 
$Q^{\L}_\lambda(x_{1},  \ldots,x_{n})$ are   in $\Gamma^{\L}(\x_n)$. 
Moreover  
$Q^{\L}_\lambda(x_1,\ldots,x_n)$ is an element of $\Gamma_{+}^{\L}(\x_n)$.
\item For $\lambda \in \mathcal{SP}_{n}$, 
$P^{\L}_\lambda(x_{1}, \ldots, x_{n}|\b)$  is in $\Gamma^{\L}(\x_{n}|\b)$,  
whereas   $Q^{\L}_\lambda(x_{1},  \ldots,x_{n}|\b)$ is    in $\Gamma^{\L}_{+}(\x_{n}|\b)$. 
\end{enumerate} 
\end{prop}
\noindent
We shall prove this proposition along the same lines as in 
Ikeda-Naruse \cite[Propositions 3.1, 3.2]{Ike-Nar2013}. 
However we shall give a proof for convenience of the reader. 
Before proving  this proposition, we prepare two  simple lemmas. 
\begin{lem}%\footnote{
%cf. Ikeda-Naruse \cite[Lemma 3.2]{Ike-Nar2013}.  
%}    
\label{lem:FactorTheorem}  
Let $f(x, y) \in \L[[x, y]]$ be a formal power series. 
Suppose that $f(\overline{y}, y) = 0$. Then $f(x, y)$ is divisible by $x \pf y$. 
\end{lem} 
\begin{proof} 
By the assumption, $f(x, y)$ is divisible by $x - \overline{y}$. 
On the other hand,  we see that $x \pf y = (x -  \overline{y}) \times u(x, y)$
with  $u(x, y) \in \L[[x, y]]$, a unit.   From this the claim follows. 
\end{proof}

\begin{lem}%\footnote{
%cf. Ikeda-Naruse \cite[Lemma 3.1]{Ike-Nar2013}, Ivanov \cite[Proposition 1.1]{Iva97}.  
%} 
 \label{lem:L-supersymmetric}  
Let $r \leq n$, and $f(x_{1}, \ldots, x_{r})$ be a formal power series  in $\L[[x_{1}, \ldots, x_{r}]]$. 
Then the following function  is $\L$-supersymmetric. 
\begin{equation}   \label{eqn:R_n}  
  R_{n}(x_{1}, \ldots, x_{n}) = \sum_{w \in S_{n}}  w 
  \left [  f(x_{1}, \ldots, x_{r})  \prod_{i=1}^{r}  \prod_{j = i + 1}^{n} 
           \dfrac{x_{i} \pf x_{j}} {x_{i} \pf \overline{x}_{j} } \right  ].   
\end{equation}  
Furthermore if $f(x_{1}, \ldots, x_{r})$ is divisible by 
$\prod_{i=1}^{r} x_{i}  = x_{1}\cdots x_{r}$, 
then $R_{n}(x_{1}, \ldots, x_{n})$ has the {\it stability} in the 
following sense$:$ 
\begin{equation*}
  R_{n+1}(x_{1}, \ldots, x_{n}, 0)  = (n + 1 - r) R_{n}(x_{1}, \ldots, x_{n}).
\end{equation*}  
\end{lem}  
\begin{proof} 
By the definition, $R_{n}(x_{1}, \ldots x_{n})$ is a symmetric 
formal power series%\footnote{
%cf. Definition \ref{df:DefinitionP^L(x_n|b)Q^L(x_n|b)}. 
%}
.  We shall show that this formal power series has $\L$-supersymmetric property. 
Let $x_{p} = t$, $x_{q} = \overline{t}$ for arbitrary integers $p$, $q$ such 
that  $1 \leq p < q \leq n$.  We claim that each summand corresponding to 
$w \in S_{n}$ does not depend on $t$.  
For convenience we put $F_{n}(x_{1}, \ldots, x_{n}) = \prod_{i=1}^{r} \prod_{j = i + 1}^{n} 
\dfrac{x_{i} \pf x_{j}}{x_{i} \pf \overline{x}_{j}}$. 
If one of $p$, $q$ is in $\{ w(1), \ldots, w(r) \}$, then we see directly that the term 
\begin{equation*} 
  F_{n}(x_{w(1)}, \ldots, x_{w(n)}) 
= \prod_{i=1}^{r}  
\prod_{j = i  + 1}^{n}  \dfrac{x_{w(i)} \pf x_{w(j)}}{x_{w(i)} \pf \overline{x}_{w(j)}}
\end{equation*} 
vanishes because of $t \pf \overline{t} = 0$. 
If $p, q \in \{ w(r + 1), \ldots, w(n) \}$, then  the term 
\begin{equation*} 
 F_{n}(x_{w(1)}, \ldots, x_{w(n)}) 
=  \prod_{i=1}^{r}  \dfrac{x_{w(i)} \pf x_{w(i + 1)}} {x_{w(i)} \pf \overline{x}_{w(i + 1)}} \times 
      \cdots \times \dfrac{x_{w(i)} \pf x_{w(n)}} {x_{w(i)} \pf \overline{x}_{w(n)}}  
\end{equation*}   
does not depend on $t$ because  of the identity $\dfrac{x_{w(i)} \pf t}{x_{w(i)} \pf \overline{t}} 
\times  \dfrac{x_{w(i)} \pf \overline{t}} {x_{w(i)} \pf t} = 1$. Note that 
in this case $f(x_{w(1)}, \ldots, x_{w(r)})$ does not depend on $t$.  
Thus the first assertion is proved.

For the second assertion, we argue as follows: 
In the defining equation of $R_{n+1}$,  we divide the summation into two parts: 
one is the summation for $w \in S_{n+1}$ such that $n + 1 \in \{ w(1), \ldots, w(r) \}$, 
and the other is the summation for $w \in S_{n+1}$ such that $n + 1  \in \{ w(r + 1), \ldots, w(n + 1) \}$.  Let $R_{n+1}(x_{1}, \ldots, x_{n+1}) = \sum_{1} + \sum_{2}$ be such decomposition. 
If we set $x_{n + 1} = 0$, then each summand in $\sum_{1}$ becomes zero 
because  $f(x_{w(1)},  \ldots, x_{w(r)}) = 0$ by the assumption.
Therefore we have only to consider the entries in $\sum_{2}$.  
If $n + 1 \in \{ w(r + 1), \ldots, w(n + 1) \}$, say $n + 1 = w(k)$ for some $r +  1 \leq k \leq n + 1$.  If $k = n + 1$, namely $w(n + 1) = n + 1$, the element $w \in S_{n + 1}$ naturally 
defines an element $w' \in S_{n}$, i.e., $w'(i) = w(i) \; (1 \leq i \leq n)$. Then 
we have 
$f(x_{w(1)}, \ldots, x_{w(r)}) = f(x_{w'(1)}, \ldots, x_{w'(r)})$. 
If $r + 1 \leq k \leq n$, we consider the permutation $(w(n+1) \; \;  n+1) w$.  %\footnote{
%$(i \; j)$ denotes the transposition of $i$ and $j$.  
%}  
Since this permutation fixes $n + 1$, it defines naturally an element $w' \in S_{n}$, 
i.e., $w'(i) = w(i) \; (1 \leq i \leq n, \; i \neq k)$ and $w'(k) = w(n + 1)$.  
Then we also have $f(x_{w(1)}, \ldots, x_{w(r)}) = f(x_{w'(1)}, \ldots, x_{w'(r)})$. 
In either case, we see directly that 
\begin{equation*} 
   F_{n+1}(x_{w(1)}, \ldots, x_{w(n+1)})  =  F_{n}(x_{w'(1)}, \ldots, x_{w'(n)}), 
\end{equation*} 
when we put $x_{n + 1} = 0$.  Therefore  $\sum_{2}$ becomes 
\begin{equation*}
  (n - r + 1) \sum_{w' \in S_{n}} f(x_{w'(1)}, \ldots, x_{w'(r)}) 
   F_{n}(x_{w'(1)}, \ldots, x_{w'(n)})  
   = (n - r + 1) R_{n}(x_{1},\ldots, x_{n})
\end{equation*} 
when we put $x_{n+1} = 0$, and the second assertion is proved. 
\end{proof}

\begin{proof} [Proof of Proposition $\ref{prop:L-supersymmetry}$]
The assertion (2) follows immediately from (1). Let us prove (1).  
First assertion follows from Lemma \ref{lem:L-supersymmetric} when 
we put $f(x_{1}, \ldots, x_{r}) = [\x|0]^{\lambda}  = \prod_{i=1}^{r} [x_{i}|0]^{\lambda_{i}} =  \prod_{i=1}^{r} x_{i}^{\lambda_{i}}$, 
or $[[\x|0]]^{\lambda} =  \prod_{i=1}^{r} [[x_{i}|0]]^{\lambda_{i}}  
= \prod_{i=1}^{r} (x_{i} \pf x_{i}) x_{i}^{\lambda_{i} - 1}$.  

For the second statement, we shall  show that $Q^{\L}_{\lambda}(t, x_{2}, \ldots, x_{n}) - Q^{\L}_{\lambda}(0, x_{2}, \ldots, x_{n})$ is divisible by $t \pf t$.    
In the defining equation (\ref{eqn:DefinitionP^L(x_n|b)Q^L(x_n|b)}), we divide the summation  into two parts: 
one is the summation for $w \in S_{n}$ 
such that $1 \in \{ w(1), \ldots, w(r) \}$, and 
the other is the summation for $w \in S_{n}$ such that $1 \in \{ w(r + 1), \ldots, w(n) \}$. 
Let $Q^{\L}_{\lambda} (t, x_{2}, \ldots, x_{n}) - Q^{\L}_{\lambda}(0, x_{2}, \ldots, x_{n}) 
= \sum_{1} + \sum_{2}$ be such decomposition. 
 Since $[[x_{1}|\b]]^{\lambda_{1}} = (x_{1} \pf x_{1})[x_{1}|\b]^{\lambda_{1}-1}$, we see easily that 
each summand in $\sum_{1}$ is divisible by $t \pf t$.  
Each summand in $\sum_{2}$ is also divisible by $t \pf t$ because 
the numerator of the following term 
\begin{equation*}
  \prod_{i=1}^{r}  \dfrac{x_{w(i)} \pf t}{x_{w(i)} \pf \overline{t}}  - 1 
  = \prod_{i=1}^{r}  \dfrac{ (x_{w(i)}  \pf t) - (x_{w(i)}  \pf \overline{t} )}{x_{w(i)}  \pf  \overline{t}}
\end{equation*} 
is  divisible by $t \pf t$ from Lemma \ref{lem:FactorTheorem}.  Therefore we have the required result.  
\end{proof}

%%%%%%%%%%%%%%%%%%%%%%%%%%%%%%%%%%%%%%%%%%%%%%%%%%%%%%%%%%%%%%%%%%%%%%%%%%%%%%%%%%%%%%%%%%%
\subsection{Stability Property}    \label{subsec:StabilityProperty}  
In this subsection, we discuss the {\it stability property} of 
our functions $P^{\L}_{\lambda}$ and $Q^{\L}_{\lambda}$. 
First consider   the formal power series $P^{\L}_{\lambda}(x_{1}, \ldots, x_{n})$
and $Q^{\L}_{\lambda}(x_{1}, \ldots, x_{n}) \; (\lambda \in \mathcal{SP}_{n}$). 
They  have the following stability property: 
\begin{prop}%\footnote{
%cf. Ikeda-Naruse \cite[Proposition 3.3]{Ike-Nar2013}. 
%}    
\label{prop:Stability}  
Let $\lambda$ be a strict partition of length $r \leq n$. Then 
\begin{enumerate} 
\item $P^{\L}_{\lambda}(x_{1}, \ldots, x_{n-1}, 0)  = P^{\L}_{\lambda}(x_{1}, \ldots, x_{n-1})$,   
\item $Q^{\L}_{\lambda}(x_{1}, \ldots, x_{n-1}, 0)  = Q^{\L}_{\lambda}(x_{1}, \ldots, x_{n-1})$,  
\end{enumerate} 
where the right-hand sides are zero if $r = n$.  
\end{prop} 
\begin{proof} 
First note that  both $[\x|0]^{\lambda} = \prod_{i=1}^{r} [x_{i}|0]^{\lambda_{i}} = \prod_{i=1}^{r} x_{i}^{\lambda_{i}}$ and $[[\x|0]]^{\lambda} = \prod_{i=1}^{r}   [[x_{i}|0]]^{\lambda_{i}} 
= \prod_{i=1}^{r} (x_{i} \pf x_{i}) x_{i}^{\lambda_{i} - 1}$ 
are divisible by $\prod_{i=1}^{r}x_{i} = x_{1} \cdots x_{r}$. 
Therefore the result follows from Lemma \ref{lem:L-supersymmetric}.  
\end{proof}

For each positive integer $n$, let $\varphi_{n+1}:  \Gamma^{\L} (\x_{n+1})  \longrightarrow \Gamma^{\L} (\x_{n})$ 
be a homomorphism of  graded $\L$-algebras given by the specialization $x_{n+1} = 0$. 
Then  $\{  \Gamma^{\L} (\x_{n}), \varphi_{n} \}_{n = 1, 2, \ldots}$ form an inverse system 
of graded $\L$-algebras. 
Define   the {\it ring of $\L$-supersymmetric fucntions}  %\footnote{
%cf. Ikeda-Naruse \cite[\S 3.4]{Ike-Nar2013}, Ivanov \cite[Definition 2.2]{Iva2004}.  
%} 
    $\Gamma^{\L}(\x)$ 
to be   the inverse limit of this system. Namely, we define 
\begin{equation*} 
  \Gamma^{\L} (\x) := 
\displaystyle{ \lim_{ \substack{\longleftarrow  \\ n } } }  \,  \Gamma^{\L}(\x_{n}). 
\end{equation*}   
For each strict partition   $\lambda \in \mathcal{SP}$,  
by  Proposition \ref{prop:Stability},  the sequence 
$\{ P^{\L}_{\lambda}(x_{1}, \ldots, x_{n}) \}$ %\footnote{
%By definition, $P^{\L}_{\lambda}(x_{1}, \ldots, x_{n}) = 0$ if $n < \ell (\lambda)$.   
%} 
defines an element
\begin{equation*} 
   P^{\L}_{\lambda}(\x)  := \displaystyle{\lim_{\substack{\longleftarrow \\ n}}}   \, 
  P^{\L}_{\lambda}(x_{1}, \ldots, x_{n})    
\in \Gamma^{\L}(\x). 
\end{equation*}  
 Similarly   we  define 
\begin{equation*} 
      \Gamma^{\L}_{+} (\x) := \displaystyle{ \lim_{\substack{\longleftarrow \\ n}}}    
\,  \Gamma^{\L}_{+}(\x_{n}), 
\end{equation*}  
which is a subring of $\Gamma^{\L}(\x)$.  Also  for each strict partition $\lambda \in \mathcal{SP}$, the sequence 
$\{  Q^{\L}_{\lambda}(x_{1}, \ldots, x_{n}) \}$ 
defines an element 
\begin{equation*} 
  Q^{\L}_{\lambda} (\x) :=   \displaystyle{\lim_{ \substack{\longleftarrow \\ n}}}   \,  
Q^{\L}_{\lambda}(x_{1}, \ldots, x_{n})  \in \Gamma^{\L}_{+}(\x).
\end{equation*}   
Note that we have the following inclusion relations: 
\begin{equation*} 
   \Gamma^{\L}_{+}(\x_{n}) \subset \Gamma^{\L}(\x_{n}) \subset \Lambda^{\L}(\x_{n})
= \L[[e_{1}(\x_{n}), \ldots, e_{n}(\x_{n})]].   
\end{equation*}  
This implies that 
\begin{equation*} 
  \Gamma^{\L}_{+}  (\x)  \subset \Gamma^{\L} (\x)  \subset \Lambda^{\L}(\x) 
 =  \L[[e_{1}(\x), e_{2}(\x), \ldots   ]]. 
\end{equation*} 
From this, $\Gamma^{\L}(\x)$ and $\Gamma^{\L}_{+} (\x)$ inherit the Hopf algebra 
structure over $\L$ from $\Lambda^{\L}(\x) = \L[[e_{1}(\x), e_{2}(\x), \ldots]]$.

%%%%%%%%%%%%%%%%%%%%%%%%%%%%%%%%%%%%%%%%%%%%%%%%%%%%%%%%%%%%%%%%%%%%%%%%%%%%%%%%%%%%%%%
Next consider the stability property of $P^{\L}_{\lambda}(x_{1}, \ldots, x_{n}|\b)$ 
and $Q^{\L}_{\lambda}(x_{1}, \ldots, x_{n}|\b)$  $(\lambda \in \mathcal{SP}_{n})$. 
The homomorphisms of  graded $\L[[\b]]$-algebras $\psi_{n+1}:  \Gamma^{\L}(\x_{n+1}|\b) 
\longrightarrow \Gamma^{\L}(\x_{n}|\b)$ given by the specialization $x_{n + 1} = 0$
define an inverse system $\{ \Gamma^{\L}(\x_{n}|\b),  \psi_{n} \}_{n = 1, 2, \ldots}$.  
Then we define 
\begin{equation*}  
  \Gamma^{\L}(\x|\b)  
  :=   \displaystyle{\lim_{\substack{\longleftarrow \\ n}}} \,   \Gamma^{\L}(\x_{n}|\b) 
  =  \displaystyle{\lim_{\substack{\longleftarrow \\ n}}} \, \L[[\b]] \,  \hat{\otimes}_{\L} \,  \Gamma^{\L}(\x_{n}) %\cong 
%     \L[[\b]] \, \hat{\otimes}_{\L}  \, 
%  \displaystyle{\lim_{\substack{\longleftarrow\\n}}}   \,    \Gamma^{\L}(\x_{n})  
 \cong  \L[[\b]] \, \hat{\otimes}_{\L}  \,  \Gamma^{\L} (\x). 
\end{equation*} 
Similarly, we define 
\begin{equation*} 
\Gamma^{\L}_{+} (\x|\b)  
 :=   \displaystyle{\lim_{\substack{\longleftarrow\\n}}}   \,  \Gamma^{\L}_{+}(\x_{n}|\b) 
 =    \displaystyle{\lim_{\substack{\longleftarrow\\n}}}   \,  \L[[\b]] \,   \hat{\otimes}_{\L} \,   \Gamma^{\L}_{+}(\x_{n}) 
 % \cong     \L[[\b]]  \,  \hat{\otimes}_{\L}  \,    
 % \displaystyle{\lim_{\substack{\longleftarrow\\n}}}    \,  
 %\Gamma^{\L}_{+}(\x_{n})  
\cong   \L[[\b]]  \,  \hat{\otimes}_{\L}  \,  \Gamma^{\L}_{+}(\x). 
\end{equation*}     
In contrast to  Proposition \ref{prop:Stability}, the functions   
$P^{\L}_{\lambda} (x_{1}, \ldots, x_{n}|\b)$ and $Q^{\L}_{\lambda}(x_{1}, \ldots, x_{n}|\b)$
 have the   following stability property: 
\begin{prop}    \label{prop:StabilityMod2}  
Let $\lambda$ be a strict partition of length $r \leq n$. Then 
\begin{enumerate} 
\item   
  $P^{\L}_{\lambda}(x_{1}, \ldots, x_{n-2}, 0, 0 |\b) = P^{\L}_{\lambda}(x_{1}, \ldots, x_{n-2}|\b)$,  
\item 
   $Q^{\L}_{\lambda}(x_{1}, \ldots, x_{n-1}, 0 |\b)  = Q^{\L}_{\lambda}(x_{1}, \ldots, x_{n-1}|\b)$, 
\end{enumerate} 
where the right-hand sides are zero if $r = n$.    
\end{prop} 
\begin{proof} 
 The proof of (1) goes along the same lines as in Ikeda-Naruse \cite[Proposition 8.2]{Ike-Nar2009}
(see also Ikeda-Naruse \cite[Remark 3.1]{Ike-Nar2013}). 
 The assertion (2)   follows from Lemma \ref{lem:L-supersymmetric} because 
$[[\x|\b]]^{\lambda} = \prod_{i=1}^{r}  [[x_{i}|\b]]^{\lambda_{i}} 
= \prod_{i=1}^{r} (x_{i} \pf x_{i}) [x_{i}|\b]^{\lambda_{i}-1}$ is divisible 
by $\prod_{i=1}^{r} x_{i} = x_{1} \cdots x_{r}$.   
\end{proof} 
\noindent
Following Ikeda-Naruse \cite[Proposition 8.2]{Ike-Nar2009}, we call the above 
stability property (1)  the {\it stability mod $2$}. 
\begin{rem} 
Note that $($see also Ikeda-Naruse \cite[Remark 3.1]{Ike-Nar2013}$)$ 
the usual stability property   
  $P^{\L}_{\lambda}(x_{1}, \ldots, x_{n-1}, 0|\b)  \\    
  = P^{\L}_{\lambda}(x_{1}, \ldots, x_{n-1}|\b)$ does not hold in general. 
For from  Example  $\ref{ex:P^L_1Q^L_1}$,  we have $P^{\L}_{(1)}(x_{1}|\b) = x_{1} \pf b_{1}$, 
whereas $P^{\L}_{(1)}(x_{1}, 0|\b) = x_{1} \pf b_{1} + b_{1} \dfrac{x_{1}}{\overline{x}_{1}}$.  
\end{rem} 
\noindent
Therefore in the case of $P^{\L}_{\lambda}(x_{1}, \ldots, x_{n}|\b)$, 
there exist well-defined {\it even} and {\it odd} limit functions. 
In what follows, we shall use only the {\it even} limit functions %\footnote{
%cf. Ikeda-Mihalcea-Naruse \cite[\S 4.2]{IMN2011}, Ikeda-Naruse \cite[\S 3.5]{Ike-Nar2013}. 
%}
. 
To be precise, for each strict partition $\lambda \in \mathcal{SP}$, we define 
the functions $P^{\L}_{\lambda}(\x_{n}|\b)^{+}  \in \Gamma^{\L}(\x_{n}|\b)$  %($n \geq \ell (\lambda)$) 
 by 
\begin{equation*}
   P^{\L}_{\lambda}(\x_{n}|\b)^{+} = P^{\L}_{\lambda}(x_{1}, \ldots, x_{n}|\b)^{+} 
    := \left \{ 
                  \begin{array}{llll} 
                     \hspace{-0.3cm} 
                      & P^{\L}_{\lambda}(x_{1}, \ldots, x_{n}|\b)  \quad &  \text{if} \; n \; 
                     \text{is even}, \medskip \\    
                     \hspace{-0.3cm} 
                       & P^{\L}_{\lambda}(x_{1}, \ldots, x_{n}, 0|\b)  \quad & \text{if} \; n \; 
                     \text{is odd}.  \medskip    
                   \end{array} 
       \right. 
\end{equation*}               
By Proposition \ref{prop:StabilityMod2} (1), the sequence  $\{ P^{\L}_{\lambda}(x_{1}, \ldots, x_{n}|\b)^{+} \}$ %\footnote{
%By definition, $P^{\L}_{\lambda}(x_{1}, \ldots, x_{n}|\b)  = 0$ if $n < \ell (\lambda)$. 
%}  
defines an element 
\begin{equation*} 
%\begin{array}{llll} 
  P^{\L}_{\lambda} (\x |\b)^{+}  
     :=    \displaystyle{\lim_{\substack{\longleftarrow \\ n}}}     \,  
    P^{\L}_{\lambda} (x_{1},  \ldots, x_{n}|\b)^{+}   \in \Gamma^{\L}(\x|\b).    
%    P^{\L}_{\lambda} (\x |\b)^{-}  
%   & := &   \displaystyle{\lim_{\substack{\longleftarrow \\ n}}}     \,  
%    P^{\L}_{\lambda} (x_{1},  \ldots, x_{2n+1}|\b) \; (\text{odd limit}).    \medskip \\
%\end{array}  
\end{equation*} 
Similarly, for each strict partition $\lambda \in \mathcal{SP}$, define 
\begin{equation*} 
   Q^{\L}_{\lambda} (\x |\b):=  \lim_{\substack{\longleftarrow \\ n}} \,   Q^{\L}_{\lambda} 
  (x_{1}, \ldots, x_{n}|\b)  \in \Gamma^{\L}_{+}(\x|\b).   
\end{equation*} 
This is a well defined element of $\Gamma^{\L}_{+}(\x|\b)$ because of 
Proposition \ref{prop:StabilityMod2} (2). 
Finally, note that we have the following inclusion relations: 
\begin{equation*} 
\begin{array}{lllll} 
   \Gamma^{\L}_{+}(\x_{n}|\b) \subset \Gamma^{\L}(\x_{n}|\b) 
\subset  \Lambda^{\L}(\x_{n}|\b)      
  & =     \L[[\b]] \, \hat{\otimes}_{\L} \,  \Lambda^{\L}(\x_{n})  \medskip \\
  & =     \L[[\b]]  \,  \hat{\otimes}_{\L}    \,  \L[[e_{1}(\x_{n}), \ldots, e_{n}(\x_{n})]].    \medskip 
\end{array}   
\end{equation*} 
This implies that 
\begin{equation*} 
  \Gamma^{\L}_{+}  (\x|\b)  \subset \Gamma^{\L} (\x|\b)  
\subset   \Lambda^{\L}(\x|\b) \cong  \L[[\b]]  \, \hat{\otimes}_{\L}   \,    \L[[e_{1}(\x), e_{2}(\x), \ldots]]. 
\end{equation*}

%%%%%%%%%%%%%%%%%%%%%%%%%%%%%%%%%%%%%%%%%%%%%%%%%%%%%%%%%%%%%%%%%%%%%%%%%%%%%%%%%%%%%%%%%%%%%%%%%%%%%%%%%%%%%%%%
\subsection{Universal factorial Schur functions}   \label{subsec:UFSF}  
\subsubsection{Definition of $s^{\L}_{\lambda}(\x_{n}|\b)$}  
Let $\mathcal{P}_{n}$ denote  the set of all partitions of length $\leq n$.
For a positive integer $n$,   we set $\rho_{n}=(n, n-1, \ldots, 2, 1)$.
The {\it factorial Schur function} $s_{\lambda}(x|a)$ (for its definition, 
see e.g., Ikeda-Naruse \cite[\S 5.1]{Ike-Nar2009}, Macdonald \cite[I,  \S 3,  Examples 20]{Mac95},  Molev-Sagan \cite[p.4431]{Mol-Sag99}) 
and the {\it factorial   Grothendieck polynomial} $G_{\lambda} (x_{1}, \ldots, x_{n}|b)$  
(for its definition,  see  Ikeda-Naruse \cite[(2.12), (2.13)]{Ike-Nar2013}, 
McNamara \cite[Definition 4.1]{McN2006})
can  also be  generalized in the  universal setting. 
Here we assign the variables $x_{1}, x_{2}, \ldots, b_{1}, b_{2}, \ldots$ 
degree $\deg \, (x_{i}) = \deg \, (b_{i}) = 1 \; (i = 1, 2, \ldots)$,  
and  use $\deg^{h} \, (a_{i, j}) = 1 - i - j \; (i, j \geq 1)$. 
For partitions $\lambda, \mu \in \mathcal{P}_{n}$, $\lambda + \mu$ 
is a partition  of length $\leq n$ defined by $(\lambda + \mu)_{i} := \lambda_{i} + \mu_{i} \; (1 \leq i \leq n)$. 

\begin{defn}   [Universal factorial Schur functions]  \label{df:Definitions^L(x|b)}  
For a partition  
$\lambda=(\lambda_1,\ldots,\lambda_n)\in \mathcal{P}_n$, 
we define the {\it  universal factorial Schur functions} 
$s^{\L}_{\lambda}(x_{1}, \ldots, x_{n}| \b)$  to be 
\begin{equation}   \label{eqn:Definitions^L}  
s^{\L}_{\lambda} (\x_{n}|\b) = s^{\L}_{\lambda} (x_{1},  \ldots, x_{n}|\b):=
\displaystyle
\sum_{w  \in S_{n}} w  \left[\frac{[\x |\b]^{\lambda + \rho_{n-1}}}
{\prod_{1\leq i<j\leq n}(x_{i}\pf \overline{x}_{j})}
\right]. 
\end{equation} 
\end{defn}
\begin{rem}
The non-equivariant version, i.e., $\b  = 0$,  of our functions  are  already defined by Fel'dman 
\cite[Definition 4.2]{Fel2003}.  These are called the {\it generalized Schur polynomials} there. 
We thank the referee who pointed out this fact to us. 
\end{rem}  
\noindent
By the same reason for the universal factorial Schur $P$- and $Q$-functions, 
$s^{\L}_{\lambda}(x_{1}, \ldots, x_{n} | \b)$ is a  well-defined element of 
$\Lambda^{\L}(\x_{n}|\b)$, %$\L[[\b]] [[x_{1}, \ldots, x_{n}]]^{S_{n}}$, 
namely it is a symmetric formal power series with coefficients in $\L$ 
in the variables $x_{1}, \ldots, x_{n}$ 
%and a formal power series in
and  $b_{1}, b_{2}, \ldots, b_{\lambda_{1} + n-1}$. 
It is also a  homogeneous  formal power series  
%in $x_{1}, x_{2}, \ldots, b_{1}, b_{2}, \ldots$
of  total     degree  $|\lambda|$.  In Definition \ref{df:Definitions^L(x|b)}, 
if we put $a_{i, j} = 0$ for all $i, j \geq 1$, the functions $s^{\L}_{\lambda}(\x_{n}|\b)$ 
reduce to the usual factorial Schur polynomials $s_{\lambda}(\x_{n}|\b)$.  
If we put $a_{1, 1} = \beta$, $a_{i, j} = 0$ for all $(i, j) \neq (1, 1)$, 
then $s^{\L}_{\lambda}(\x_{n}|\b)$ reduce to the factorial Grothendieck 
polynomials $G_{\lambda}(\x_{n}|\b)$.  
 Note that unlike the usual Schur polynomials, 
$s_{\emptyset}^{\L}(x_{1}, \ldots, x_{n}|\b)  \neq 1$. For instance, 
\begin{equation*} 
  s^{\L}_{\emptyset} (x_{1}, x_{2}|\b) 
=  \dfrac{x_{1} \pf b_{1}}{x_{1} \pf \overline{x}_{2}} 
 + \dfrac{x_{2} \pf b_{1}}{x_{2} \pf \overline{x}_{1}} 
= 1 + a_{1, 2}x_{1}x_{2} + a_{1, 1}a_{1, 2} b_{1} x_{1}x_{2} + \cdots. 
\end{equation*} 
We also define 
\begin{equation*} 
s^{\L}_{\lambda} (\x_{n}) 
= s^{\L}_{\lambda} (x_{1}, \ldots, x_{n}) :=s^{\L}_{\lambda} (x_{1}, \ldots, x_{n}|0).
\end{equation*} 
It is a   homogeneous
 symmetric formal power series in $\Lambda^{\L}(\x_{n})$ %$\L[[x_{1}, \ldots, x_{n}]]^{S_{n}}$ 
of total degree $|\lambda|$. 
By definition, we have 
\begin{equation}  \label{eqn:s^L_lambda(x_1,...,x_n)} 
   s^{\L}_{\lambda} (x_{1}, \ldots, x_{n}) = s_{\lambda} (x_{1}, \ldots, x_{n})+ 
           g_{\lambda}(x_{1}, \ldots, x_{n}), 
\end{equation}  
where $s_{\lambda}(x_{1}, \ldots, x_{n})$ is the usual Schur polynomial and 
 $g_{\lambda}(x_{1},\ldots, x_{n})$ is a symmetric formal power series with 
 coefficients in $\L$ in the variables $x_{1}, \ldots, x_{n}$ 
 whose degree with respect to $x_{1}, \ldots, x_{n}$ is strictly higher than 
 $|\lambda|$.  From this,  one can show that $s^{\L}_{\lambda}(x_{1}, \ldots, x_{n})$, 
 $\lambda \in \mathcal{P}_{n}$, are linearly independent over $\L$ 
 by the same argument as  with the case of $P^{\L}_{\lambda}$'s (see \S \ref{subsec:DefinitionP^L(x|b)Q^L(x|b)}).

%%%%%%%%%%%%%%%%%%%%%%%%%%%%%%%%%%%%%%%%%%%%%%%%%%%%%%%%%%%%%%%%%%%%%%%%%%%%%%%%%%%%%%%%
\subsubsection{Basis Theorem for $s^{\L}_{\lambda}(\x_{n})$}   \label{subsubsec:BasisTheorem(s^L(x))} 
It is well-known that the usual Schur functions (polynomials) 
$s_{\lambda} (x_{1}, \ldots, x_{n})$, where $\ell (\lambda) \leq n$, form
a $\Z$-basis of $\Lambda_{n} = \Lambda (\x_{n}) = \Z[x_{1}, \ldots, x_{n}]^{S_{n}}$
 (see Macdonald  \cite[I, \S 3,  (3.2)]{Mac95}). Also  the  Grothendieck polynomials 
$G_{\lambda} (x_{1}, \ldots, x_{n}), \; (\lambda \in \mathcal{P}_{n})$ 
form a $\Z[\beta]$-basis of $\Z[\beta][x_{1}, \ldots, x_{n}]^{S_{n}}$ 
(see e.g., Ikeda-Naruse \cite[Corollary 2.1]{Ike-Nar2013}, McNamara \cite[Theorem 3.4]{McN2006}). 
Our functions $s^{\L}_{\lambda}(x_{1}, \ldots, x_{n})$, $\lambda \in \mathcal{P}_{n}$, 
also have the similar property.   
Indeed one can prove the following {Basis Theorem}: 
\begin{prop} [Basis Theorem]  \label{prop:BasisTheorem(s^L(x))}  
$s^{\L}_{\lambda}(x_{1}, \ldots, x_{n}) \; (\lambda \in \mathcal{P}_{n})$
form a formal     $\L$-basis for $\Lambda^{\L}(\x_{n})$. 
 \end{prop} 
 \begin{proof} 
 We  already remarked  
 the linear independence of $s^{\L}_{\lambda}(x_{1}, \ldots, x_{n})$'s, $\lambda \in \mathcal{P}_{n}$,
 over $\L$.   Therefore it only remains to prove that  an arbitrary symmetric formal 
power series $f(x_{1}, \ldots, x_{n}) \in  \Lambda^{\L}(\x_{n})$  
can be written as  a  formal $\L$-linear combination 
 of $s^{\L}_{\lambda}(x_{1}, \ldots, x_{n})$'s.  
%We argue as follows: 
This folllows immediately from (\ref{eqn:s^L_lambda(x_1,...,x_n)}) and the 
fact that 
% First note that 
the usual Schur polynomials $s_{\lambda}(x_{1}, \ldots, x_{n})$, 
 $\lambda \in \mathcal{P}_{n}$, form a  formal   $\L$-basis for the ring 
 $\Lambda^{\L}(\x_{n}) =  \L[[e_{1}(\x_{n}), \ldots, e_{n}(\x_{n})]]$.   
 \end{proof}

%%%%%%%%%%%%%%%%%%%%%%%%%%%%%%%%%%%%%%%%%%%%%%%%%%%%%%%%%%%%%%%%%%%%%%%%%%%%%%%%%%%%%%%%%%%%%%%%%
\subsubsection{Vanishing Property for $s^{\L}_{\lambda}(\x_{n}|\b)$} 
In order to prove the linear independence of $s^{\L}_{\lambda}(x_{1}, \ldots, x_{n}|\b)$, 
$\lambda \in \mathcal{P}_{n}$, over $\L[[\b]]$, 
we shall make use of the following 
vanishing property for $s^{\L}_{\lambda} (\x_{n}|\b)$ 
(for the vanishing property of the factorial Schur polynomials 
$s_{\lambda} (x|a)$, see Molev-Sagan \cite[Theorem 2.1]{Mol-Sag99}. 
For the  Grothendieck polynomials $G_{\lambda}(x_{1}, \ldots, x_{n}|b)$, 
see Ikeda-Naruse \cite[Proposition 2.2]{Ike-Nar2013}, McNamara \cite[Theorem 4.4]{McN2006}).  
Given a partition $\mu \in \mathcal{P}_{n}$, define the 
sequence 
\begin{equation*} 
  \overline{\b}_{\mu}  
   := (\overline{b}_{\mu_{1}}, \overline{b}_{\mu_{2}}, \ldots,  
   \overline{b}_{\mu_{i}}, \ldots, \overline{b}_{\mu_{n}}). 
\end{equation*} 
Then we have the following.  
Here we identify a partition $\lambda = (\lambda_{1}, \lambda_{2}, \ldots)$ with 
its Young diagram 
\begin{equation*} 
   D(\lambda) =  \{ (i, j) \in \Z^{2} \; | \; i \geq 1, \;  1 \leq j \leq \lambda_{i} \} 
\end{equation*} 
(see Macdonald \cite[I,  p.2]{Mac95}): 
\begin{prop} [Vanishing Property]  \label{prop:VanishingPropertys^L(x|b)}  
Let $\lambda, \mu \in \mathcal{P}_{n}$.  Then we have 
\begin{equation*} 
  s^{\L}_{\lambda} (\overline{\b}_{\mu + \rho_{n}} |\b) 
  = \left \{ 
    \begin{array}{llllll} 
          &  0           \quad   &   \text{if} \quad   \mu  \not\supset \lambda, \medskip \\
          &  \displaystyle{\prod_{(i, j) \in \lambda}}  
               (\overline{b}_{\lambda_{i} + n - i + 1}  \pf b_{n + j - {}^{t} \!  \lambda_{j}})    \quad 
                          &   \text{if} \quad  \mu = \lambda,  \medskip 
    \end{array}  
    \right. 
\end{equation*}
where $\rho_{n} = (n, n-1, \ldots, 2, 1)$ and ${}^{t} \!  \lambda$ is the diagram conjugate to $\lambda$.  
\end{prop} 
\begin{proof} 
The proof is essentially identical to that of Molev-Sagan \cite[Theorem 2.1 (Vanishing Theorem)]{Mol-Sag99}. 
See also Ivanov \cite[Theorem 1.5 (the zero property)]{Iva97}. 
We shall exhibit a proof for the sake of completeness.   
Let us prove the first assertion. 
We use the expression (\ref{eqn:Definitions^L}): 
\begin{equation}   \label{eqn:Definitions^LII}   
    s^{\L}_{\lambda} (\x_{n}|\b) 
    = \displaystyle
\sum_{w  \in S_{n}} w  \left[\frac{[\x |\b]^{\lambda + \rho_{n-1}}}
{\prod_{1\leq i<j\leq n}(x_{i}\pf \overline{x}_{j})}
\right]
=     
\displaystyle
\sum_{w  \in S_{n}} 
  \left[\frac{  \prod_{i=1}^{n} [x_{w(i)} |\b]^{\lambda_{i}  + n-i}}
{\prod_{1\leq i<j\leq n}(x_{w(i)}\pf \overline{x}_{w(j)})}
\right]. 
\end{equation} 
We wish to show that this becomes zero when we make a substitution 
\begin{equation*} 
  \x_{n} = \overline{\b}_{\mu + \rho_{n}} 
= (\overline{b}_{\mu_{1} + n},  \overline{b}_{\mu_{2} + n-1}, \ldots, \overline{b}_{\mu_{i} + n - i + 1}, \ldots, \overline{b}_{\mu_{n} + 1}). 
\end{equation*} 
The condition $\lambda \not\subset \mu$ implies that 
there exists an index $k$ such that $\mu_{k} < \lambda_{k}$ (and hence $\mu_{k} + 1 \leq \lambda_{k}$).  
For an arbitrary permutation $w \in S_{n}$, 
there exists a positive integer $1 \leq j \leq k$ such that $w(j) \geq k$. 
Thus  we have $\mu_{w(j)} \leq \mu_{k} < \lambda_{k} \leq \lambda_{j}$, and hence 
the following inequalities hold: 
\begin{equation*} 
   \mu_{w(j)} + n - w(j) + 1  \leq \mu_{k} + n - k + 1 \leq \lambda_{k} + n - k 
   \leq \lambda_{j} + n - j. 
\end{equation*}  
Therefore the term 
%\begin{equation*} 
  $[x_{w(j)}|\b]^{\lambda_{j} + n - j}$ 
%  = (x_{w(j)} \pf b_{1})(x_{w(j)} \pf b_{2}) \cdots (x_{w(j)} \pf b_{\mu_{w(j)} + n - w(j) + 1})  \cdots (x_{w(j)} \pf %b_{\lambda_{j} + n - j}) 
%\end{equation*} 
becomes zero when we specialize $x_{w(j)}$ to $\overline{b}_{\mu_{w(j)} + n - w(j) + 1}$, 
and the  first assertion follows. 

Next consider the case $\mu = \lambda$.  
We shall show that in the expression (\ref{eqn:Definitions^LII}), 
the terms other than the one corresponding to $w = e$  vanish when 
we set $\x_{n} = \overline{\b}_{\lambda + \rho_{n}}$. 
For $w \neq e$, there exists a positive integer $1 \leq k \leq n$ such that 
$w(k) > k$ (and hence $w(k) \geq k + 1$). 
 Thus we have an inequality 
\begin{equation*} 
  \lambda_{w(k)} + n - w(k) + 1  \leq \lambda_{k} + n - k. 
\end{equation*} 
Therefore the term 
%\begin{equation*} 
 $[x_{w(k)}|\b]^{\lambda_{k} + n-k}$  
% = (x_{w(k)} \pf b_{1})(x_{w(k)} \pf b_{2}) \cdots 
%    (x_{w(k)} \pf b_{\lambda_{w(k)} + n - w(k) + 1}) \cdots (x_{w(k)} \pf b_{\lambda_{k} + n - k}). 
%\end{equation*} 
becomes zero when we specialize $x_{w(k)}$ to $\overline{b}_{\lambda_{w(k)} + n - w(k) + 1}$. 
The term corresponding to $w = e$ is 
\begin{equation*} 
\begin{array}{llll} 
&   \dfrac{[\overline{\b}_{\lambda + \rho_{n}}|\b]^{\lambda + \rho_{n-1} } }  
{\prod_{1 \leq i < j \leq n} (\overline{b}_{\lambda_{i} + n - i + 1} \pf  b_{ \lambda_{j} + n - j + 1})}  
=  \dfrac{\prod_{i=1}^{n} [\overline{b}_{\lambda_{i} + n-i + 1}|\b]^{\lambda_{i} + n - i}}
   {\prod_{1 \leq i < j \leq n} (\overline{b}_{\lambda_{i} + n - i + 1}  \pf  b_{ \lambda_{j} + n - j + 1}) }  \medskip \\
%&=  \dfrac{ \prod_{i=1}^{n} (\overline{b}_{\lambda_{i} + n -i + 1}   \pf b_{1})
%                          (\overline{b}_{\lambda_{i} + n -i + 1}   \pf b_{2}) 
%                 \cdots   (\overline{b}_{\lambda_{i} + n -i + 1}   \pf b_{\lambda_{i} + n-i})  }
%{\prod_{i=1}^{n} \prod_{j =i + 1}^{n}  (\overline{b}_{\lambda_{i} + n - i + 1}  \pf  b_{ %\lambda_{j} + n - j + 1} )  }  \medskip 
=  & \dfrac{ \prod_{i=1}^{n} \prod_{j=1}^{\lambda_{i} + n-i} 
        (\overline{b}_{\lambda_{i} + n -i + 1}   \pf b_{j})} 
    {\prod_{i=1}^{n} \prod_{j =i + 1}^{n}  
      (\overline{b}_{\lambda_{i} + n - i + 1}  \pf  b_{\lambda_{j} + n - j + 1} )  }.   \medskip 
\end{array} 
\end{equation*} 
By cancellation, we obtain the required formula. 
\end{proof}

%%%%%%%%%%%%%%%%%%%%%%%%%%%%%%%%%%%%%%%%%%%%%%%%%%%%%%%%%%%%%%%%%%%%%%%%%%%%%%%%%%%%%%%%
\subsubsection{Algebraic localization map of type $A_{\infty}$}   \label{subsubsec:AlgebraicLocalizationMap(A)}  
In \S \ref{subsubsec:BasisTheorem(s^L(x))}, we have proven     
the Basis Theorem for $s^{\L}_{\lambda}(\x_{n})$'s. 
We shall  prove the Basis Theorem for $s^{\L}_{\lambda}(\x_{n}|\b)$'s
in \S \ref{subsubsec:BasisTheorem(s^L(x|b))}. 
In order to prove this,  we shall exploit the  localization technique. 
This subsubsection will be  devoted to a  brief introduction 
of such technique. 
Here we shall  introduce the following two devices: the {\it GKM ring} %\footnote{
%cf. Ikeda-Naruse \cite[\S 5.1]{Ike-Nar2013}, 
%Lam-Schilling-Shimozomo \cite[\S 2.4, 2.5, 3.2]{Lam-Sch-Shi2010II}. 
%} 
and the {\it algebraic localization map}.  %\footnote{
%cf. Ikeda-Mihalcea-Naruse \cite[\S 6]{IMN2011}, Ikeda-Naruse \cite[\S 7.2]{Ike-Nar2013}.  
%}
%Before introducing these notions, we need to prepare some notations 
%about Weyl group, root system of  type $A_{\infty}$.
In order to define these two notions, we freely use the standard notations and 
conventions about  the Weyl group,  the root system, the Grassmannian elements of 
type $A_{\infty}$.
We collect them in the Appendix \ref{subsubsec:RootDeta(A)}  for reader's convenience.

%%%%%%%%%%%%%%%%%%%%%%%%%%%%%%%%%%%%%%%%%%%%%%%%%%%%%%%%%%%%%%%%%%%%%%%%%%%%%%%%%%%%%%
%We are now   ready  to introduce the promised two tools. 
%Let us
First we   introduce the GKM ring.  
Let $L$ denote a free $\Z$-module with a basis $\{ t_{i} \}_{i \geq 1}$. 
The   positive roots  $\Delta^{+}_{A} \subset L$ of type $A_{\infty}$ 
are given by 
\begin{equation*} 
     \Delta^{+}_{A} = \{ \alpha_{j, i} = t_{j} - t_{i} \; | \; j > i \geq 1 \}. 
\end{equation*} 
The simple roots are given by $\alpha_{i} = \alpha_{i+1, i}  \; (i \geq 1)$. 
We  define a map  %\footnote{
%As the notation suggests, $e (\alpha)$ for $\alpha  \in L$ signifies 
%the  `Euler class' or the `first Chern class'(in the  
%complex cobordism cohomology theory $MU^{*}(-)$) of the
%} 
$e: L \longrightarrow \L[[\b]]$ 
by setting $e (t_{i}) := b_{i} \; (i \geq 1)$ and by the rule 
$e (\alpha + \alpha') := e(\alpha) \pf e (\alpha')$ for $\alpha, \alpha' \in L$. 
Note that by definition, we have $e(-\alpha) = \overline{e(\alpha)}$ for $\alpha \in L$. 
For the simple root  $\alpha_{i} = t_{i + 1} - t_{i} \; (i \geq 1)$, we have 
\begin{equation*} 
  e(\alpha_{i}) = e(t_{i + 1} - t_{i})  = b_{i + 1} \pf \overline{b}_{i} \quad 
  (i \geq 1). 
\end{equation*} 
Let $\mathrm{Map} \, (\mathcal{P}_{n}, \L[[\b]])$ denote the set of 
all maps $\psi: \mathcal{P}_{n} \longrightarrow \L[[\b]], \; 
\lambda \longmapsto \psi_{\lambda}$. % from $\mathcal{P}_{n}$ to $\L[[\b]]$. 
%$\mathrm{Map} \, (\mathcal{P}_{n}, \L[[\b]]$ 
It has  a natural  $\L[[\b]]$-algebra structure under pointwise multiplication  
$(\psi \cdot \varphi)_{\lambda} := \psi_{\lambda} \cdot \varphi_{\lambda}$ 
for $\psi, \varphi \in \mathrm{Map} \, (\mathcal{P}_{n},  \L[[\b]])$  
and 
scalar multiplication $ (c \cdot \psi)_{\lambda} := c \cdot \psi_{\lambda}$
for  $c \in \L[[\b]]$, 
$\psi \in \mathrm{Map} \, (\mathcal{P}_{n}, \L[[\b]])$.   Hereafter we use the identification 
$\mathrm{Map} \, (\mathcal{P}_{n}, \L[[\b]]) \cong \prod_{\mu \in \mathcal{P}_{n}} 
(\L[[\b]])_{\mu}$. Thus an element $\psi \in \mathrm{Map} \, (\mathcal{P}_{n}, \L[[\b]])$ 
can be regarded as a collection $\psi = (\psi_{\mu})_{\mu \in \mathcal{P}_{n}}$
with $\psi_{\mu}  \in \L[[\b]]$.  
With these notations, we introduce the GKM ring. 
\begin{defn} [GKM ring of type $A_{\infty}$]  %\footnote{
%cf. Ganter-Ram \cite[Theorem 3.1]{Gan-Ram2012}, 
%    Harada-Henriques-Holm \cite[Theorem 3.1]{HHH2005},  
%    Ikeda-Naruse \cite[Definition 5.1]{Ike-Nar2013}, 
%    Kiritchenko-Krishna \cite[Theorem 3.9]{Kir-Kri2013},  
%    Lam-Schilling-Shimozono \cite[Proposition 2.6]{Lam-Sch-Shi2010II}.   
%} 
Let $\Psi^{(n)}_{A}$ be a subring 
 of $\mathrm{Map} \, (\mathcal{P}_{n}, \L[[\b]]) \cong 
\prod_{\mu \in \mathcal{P}_{n}}  (\L[[\b]])_{\mu}$  that consists  of elements 
$\psi = (\psi_{\mu})_{\mu \in \mathcal{P}_{n}}$ 
satisfying  the following 
condition $(${\it GKM condition}$):$ 
\begin{equation}   \label{eqn:GKMCondition}  
  \psi_{s_{\alpha} \mu}  - \psi_{\mu}   \in e (-\alpha) \, \L[[\b]] 
  \quad \text{for all} \; \mu  \in \mathcal{P}_{n} \; \text{and all} \; 
  \alpha \in \Delta^{+} = \Delta^{+}_{A}. 
\end{equation} 
\end{defn}   
\noindent
It is easy to see that $\Psi^{(n)}_{A}$ is indeed  a subring  (more precisely,  an $\L[[\b]]$-subalgebra)
of $\prod_{\mu \in \mathcal{P}_{n}} 
(\L[[\b]])_{\mu}$.  We call the ring $\Psi^{(n)}_{A}$ the {\it GKM ring of type $A_{\infty}$}.

Next we shall introduce the algebraic localization map of type $A_{\infty}$.  
For each partition $\mu \in \mathcal{P}_{n}$, define an $\L[[\b]]$-algebra 
homomorphism 
\begin{equation*} 
  \phi^{(n)}_{\mu, A}:  \Lambda^{\L}(\x_{n}|\b) = \L[[\b]] \, \hat{\otimes}_{\L} \Lambda^{\L}(\x_{n}) %= \L[[\b]][[\x_{n}]]^{S_{n}} 
 \longrightarrow \L[[\b]], \quad 
  F = F(\x_{n}|\b)  \longmapsto  \phi^{(n)}_{\mu, A}(F)
\end{equation*} 
by $\phi^{(n)}_{\mu, A}(F) :=  F(\overline{\b}_{\mu + \rho_{n}}|\b)$.  

\begin{defn} [Algebraic localization map of type $A_{\infty}$]  %\footnote{
%cf. Ikeda-Mihalcea-Naruse \cite[Definition 6.1]{IMN2011}, 
%    Ikeda-Naruse \cite[Definition 7.1]{Ike-Nar2013}.  
%} 
 \label{df:AlgebraicLocalizationMap}  
  Define the homomorphism of $\L[[\b]]$-algebras to be 
  \begin{equation}  \label{eqn:AlgebraicLocalizationMap(A)}  
    \Phi^{(n)}_{A}:  \Lambda^{\L}(\x_{n}|\b) %\L[[\b]][[\x_{n}]]^{S_{n}} 
                  \longrightarrow 
                  \prod_{\mu \in \mathcal{P}_{n}}  (\L[[\b]])_{\mu}, \; 
                 F \longmapsto  \Phi^{(n)}_{A}(F)  := (\phi^{(n)}_{\mu, A}(F))_{\mu \in \mathcal{P}_{n}}. 
  \end{equation} 
\end{defn} 
\noindent
  We call the homomorphism $\Phi^{(n)}_{A}$ the {\it algebraic localization map of type $A_{\infty}$}.

\begin{lem}  %\footnote{
%cf. Ikeda-Naruse \cite[Proposition 7.2]{Ike-Nar2013}, \cite[Corollary 7.1]{Ike-Nar2013}. 
%}  
 \label{lem:PropertyPhi^{(n)}_A} 
\quad 
\begin{enumerate} 
\item  The image of $\Phi^{(n)}_{A}$ is contained in the GKM ring $\Psi^{(n)}_{A}$, 
       that is, 
       $\Im \, (\Phi^{(n)}_{A})   \subset \Psi^{(n)}_{A}$.  %\footnote{
%       Actually,  one can prove that 
%       $\Im \, (\Phi^{(n)}_{A})$ coincides with $\Psi^{(n)}_{A}$ 
%       (see Corollary \ref{cor:ImPhi^(n)_A=Psi^(n)_A}).  
%}.
\item  The homomorphism $\Phi^{(n)}_{A}$ is injective.  
\end{enumerate}   
\end{lem} 
\begin{proof}
(1) For an arbitrary $F = F(\x_{n}|\b)  \in   \Lambda^{\L}(\x_{n}|\b)$, %\L[[\b]][[\x_{n}]]^{S_{n}}$, 
    we have to show that $F(\overline{\b}_{s_{\alpha} \mu + \rho_{n}} |\b)
     - F(\overline{\b}_{\mu + \rho_{n}} |\b)  \in  e(-\alpha) \L[[\b]]$ 
    for all $\mu \in \mathcal{P}_{n}$ and for all $\alpha \in \Delta^{+}$.  
    More  	accurately, we have to show that 
    \begin{equation*} 
       F(\overline{\b}_{s_{t_{j} - t_{i}} \mu + \rho_{n}}|\b) 
      -  F(\overline{\b}_{\mu + \rho_{n}}|\b)  \in  \langle \overline{b}_{j} \pf b_{i} \rangle  
       \quad \text{for} \; j > i \geq 1, 
    \end{equation*} 
    where $\langle \overline{b}_{j} \pf b_{i} \rangle$ denotes  an ideal of $\L[[\b]]$ 
    generated by $\overline{b}_{j} \pf b_{i}$. 
     This is a direct consequence of the action of the reflection $s_{t_{j} - t_{i}}$ 
     on $\mu \in \mathcal{P}_{n}$  (see Appendix \ref{subsubsec:RootDeta(A)})  
     and Lemma \ref{lem:FactorTheorem}.

(2)  %It is enough to show the following statement: 
     We shall prove the following statement:  

    \begin{center} 
     For $F  \in   \Lambda^{\L}(\x_{n}|\b)$,   %\L[[\b]][[\x_{n}]]^{S_{n}}$, 
    suppose that 
     there exists a positive integer $N$ (which may depend on $n$) 
     such that $\phi^{(n)}_{\mu}(F) = 0$ 
     for all  partitions $\mu \in \mathcal{P}_{n}$ containing   
      $(N^{n}) = (N, N, \ldots, N)$.  %\footnote{     
%}. 
Then we have $F =  0$. 
   \end{center}  
     
     From this, the conclusion of (2) immediately follows.  
     Let us prove the above statement  by induction on the number $n$ of $x$-variables. 
     For the case $n = 1$, the proof is easy. 
     For if $F  = F(x_{1}|\b)  \in \L[[\b]][[x_{1}]]$ satisfies  the 
     assumption,  we have $F(\overline{b}_{N + k}|\b) = 0$ for all $k  \geq 1$. 
     This implies that $F(x_{1}|\b)$ is divisible by $x_{1} - \overline{b}_{N + k}$ 
     for all $k  \geq 1$.  From this, we conclude that $F = 0$ as a formal power series 
     in $x_{1}$ with coefficients in $\L[[\b]]$.  
     
     Next we consider the case $n > 1$, and assume that the above assertion 
     holds for the case of $n-1$ variables $\x_{n-1}$.  
      Suppose  that $F = F(\x_{n}|\b)  \in   \Lambda^{\L}(\x_{n}|\b)$  %\L[[\b]][[\x_{n}]]^{S_{n}}$ 
     satisfies  the 
     assumption $\phi^{(n)}_{\mu} (F) = 0$ for all $\mu \supset (N^{n})$. 
     For  each $k \geq 1$, we put $F_{k}  = F_{k}(\x_{n-1}|\b)  := F(\x_{n-1}, \overline{b}_{N+k}|\b)$. 
   %= F_{k}(x_{1}, \ldots, x_{n-1}|\b) 
    % := F(x_{1}, \ldots, x_{n-1}, \overline{b}_{N + k}|\b)$.  
     %We set $G(y|\b) := F(\x_{n-1}, y|\b)$.  
     Then  we see  easily that $F_{k}  \in   \Lambda^{\L}(\x_{n-1} |\b)$  %\L[[\b]][[\x_{n-1}]]^{S_{n-1}}$
     and 
     $\phi^{(n-1)}_{\nu} (F_{k}) = 0$ for all $\nu \in \mathcal{P}_{n-1}$ 
     such that $\nu  \supset  ((N + k)^{n-1}))$.   
     Therefore by the induction hypothesis, we have $F_{k} = 0$ 
     as a formal power series in $x_{1}, \ldots, x_{n-1}$ with coefficients in $\L[[\b]]$. 
     Thus $F$ depends only on  the variable $x_{n}$.  But  $F$ vanishes 
     when we set $x_{n} = \overline{b}_{N + k}$ for all $k \geq 1$, and hence 
     we deduce that $F = 0$ as required. 
\end{proof}

%%%%%%%%%%%%%%%%%%%%%%%%%%%%%%%%%%%%%%%%%%%%%%%%%%%%%%%%%%%%%%%%%%%%%%%%%%%%%%%%%%%%%%%%%%
\subsubsection{Basis Theorem for $s^{\L}_{\lambda}(\x_{n}|\b)$}   \label{subsubsec:BasisTheorem(s^L(x|b))} 
The factorial Schur polynomials 
$s_{\lambda}(x_{1}, \ldots, x_{n}|a)$, where $\ell (\lambda) \leq n$, form 
a $\Z[a]$-basis of $\Z[a] \otimes_{\Z} \Lambda_{n}$ (see Macdonald \cite[I, \S 3,  Examples 20]{Mac95}).   For the factorial Grothendieck polynomials $G_{\lambda}(x_{1}, \ldots, x_{n}|\b)$, 
readers are referred to  Ikeda-Naruse \cite[Lemma 2.5]{Ike-Nar2013}, 
McNamara \cite[Theorems 4.6, 4.9]{McN2006}.  
By the technique introduced in the previous subsubsection, we shall prove the 
Basis Theorem for  the universal factorial Schur functions $s^{\L}_{\lambda}(x_{1}, \ldots, x_{n}|\b)$'s.     
\begin{theorem}  [Basis Theorem]  \label{thm:BasisTheorem(s^L(x|b))}  
  $s^{\L}_{\lambda} (\x_{n} | \b) \; (\lambda \in \mathcal{P}_{n})$ 
  form a formal $\L[[\b]]$-basis for  $\Lambda^{\L}(\x_{n}|\b)$.  %$\L[[\b]][[\x_{n}]]^{S_{n}}$. 
\end{theorem} 
\noindent
The proof of this theorem   will be divided into two steps: 
 the linear independence property and the generation (spanning) property. 
\begin{proof} [Proof of Theorem $\ref{thm:BasisTheorem(s^L(x|b))}$ $($Linear independence property$)$] 
%  \footnote{
%cf. Knutson-Tao \cite[Proposition 1]{Knu-Tao2003}.  
%}
We show the linear independence of 
$s^{\L}_{\lambda}(\x_{n}|\b)$'s over $\L[[\b]]$ by using the Vanishing Property (Proposition \ref{prop:VanishingPropertys^L(x|b)}). 
Suppose that there exists a linear relation of the form 
\begin{equation}   \label{eqn:LinearRelation(s^L)}  
    \sum_{\lambda}  c_{\lambda} (\b)  \, s^{\L}_{\lambda}(\x_{n}|\b) = 0  
    \quad (c_{\lambda}(\b)  \in \L[[\b]]). 
\end{equation} 
Let $\mu$ be {\it minimal} (with respect to the containement) 
among all partitions in (\ref{eqn:LinearRelation(s^L)}) such that $c_{\lambda}(\b) \neq 0$. 
We set  $\x_{n} = \overline{\b}_{\mu + \rho_{n}}$ in (\ref{eqn:LinearRelation(s^L)}). 
Then   using the first part of Proposition \ref{prop:VanishingPropertys^L(x|b)} and 
the choice of $\mu$, we obtain 
\begin{equation*} 
    c_{\mu} (\b)  s^{\L}_{\mu} (\overline{\b}_{\mu + \rho_{n}} |\b) = 0.  
\end{equation*}  
By the second part of Proposition \ref{prop:VanishingPropertys^L(x|b)},  we have 
$s^{\L}_{\mu} (\overline{\b}_{\mu + \rho_{n}} |\b) \neq 0$, and hence we have 
$c_{\mu} (\b) = 0$. We repeat this process, and we finally conclude 
that all the coefficients $c_{\lambda}(\b)$ turn  out to be zero.  
\end{proof} 
\noindent

\begin{proof}  [Proof of Theorem  $\ref{thm:BasisTheorem(s^L(x|b))}$ $($generation $($spanning$)$ 
 property$)$] %\footnote{
% cf. Ikeda-Naruse \cite[Proposition 5.3]{Ike-Nar2013}, 
%      Knutson-Tao \cite[Proposition 1]{Knu-Tao2003}, 
%      Lam-Schilling-Shimozono \cite[Proposition 2.6]{Lam-Sch-Shi2010II}. 
%} 
Let $F = F(\x_{n}|\b)  \in  \Lambda^{\L}(\x_{n}|\b)$. %\L[[\b]][[\x_{n}]]^{S_{n}}$. 
We wish to express $F$ as a formal  $\L[[\b]]$-linear combination of $s^{\L}_{\lambda}$, 
$\lambda \in \mathcal{P}_{n}$. 
 Define the {\it support} of $F$  by 
\begin{equation*} 
  \mathrm{Supp} \, (F) :=  \{\mu \in \mathcal{P}_{n} \; | \;  \phi^{(n)}_{\mu}(F) \neq 0   \}  
   \subset \mathcal{P}_{n}.  
\end{equation*}  
Let $\nu \in \mathrm{Supp} \, (F)$ be a {\it minimal} element (with respect to 
the containment).  We know from Lemma \ref{lem:PropertyPhi^{(n)}_A} (1)  that 
$\Phi^{(n)} (F)  = (\phi^{(n)}_{\mu}(F))_{\mu \in \mathcal{P}_{n}} \in \Psi^{(n)}$, 
therefore by  the GKM condition (\ref{eqn:GKMCondition}) and the minimality of $\nu$, 
we see that $\phi^{(n)}_{\nu} (F)$  is divisible by $e (-\alpha)$
for all  $\alpha \in  \mathrm{Inv} \, (\nu)$. Since the elements  $\{ e (-\alpha) \; | \; \alpha \in \mathrm{Inv} \, (\nu) \}$ are   relatively prime, 
$\phi^{(n)}_{\nu} (F)$ is divisible by their product $\prod_{\alpha \in \mathrm{Inv} \, (\nu)} e(-\alpha)$. 
Here we know from the Vanishing Property (Proposition \ref{prop:VanishingPropertys^L(x|b)}) and 
(\ref{eqn:Product(e(-alpha))(A)}) (see Appendix \ref{subsubsec:RootDeta(A)})  
that 
\begin{equation*} 
\phi^{(n)}_{\nu} (s^{\L}_{\nu})  = s^{\L}_{\nu} (\overline{\b}_{\nu + \rho_{n}} |\b) 
= \prod_{(i, j) \in \nu}  (\overline{b}_{\nu_{i} + n - i + 1}  \pf b_{n + j - {}^{t} \! \nu_{j}})
= \prod_{\alpha \in \mathrm{Inv} \, (\nu)}  e(-\alpha).   
\end{equation*} 
Thus we have $\phi^{(n)}_{\nu} (F) = c_{\nu}  \cdot \phi^{(n)}_{\nu} (s^{\L}_{\nu})$
for some $c_{\nu}  \in \L[[\b]]$.  Let 
\begin{equation*} 
   F' := F - c_{\nu}   \cdot s^{\L}_{\nu}.  
\end{equation*} 
Then we have $\phi^{(n)}_{\nu} (F')  = \phi^{(n)}_{\nu}(F)  - c_{\nu}  \cdot \phi^{(n)}_{\nu}(s^{\L}_{\nu})
= 0$, and hence $\nu \not\in \mathrm{Supp} \, (F')$. 
Moreover, for every $\mu \in \mathrm{Supp} \, (F') \setminus \mathrm{Supp} \, (F)$ 
(obviously, $\mu \neq \nu$), we have
$0 \neq \phi^{(n)}_{\mu}(F')  = \phi^{(n)}_{\mu} (F) - c_{\nu}  \cdot \phi^{(n)}_{\mu}(s^{\L}_{\mu}) 
= - c_{\nu}  \cdot \phi^{(n)}_{\mu} (s^{\L}_{\nu}) = -c_{\nu} \cdot s^{\L}_{\nu} (\overline{\b}_{\mu + \rho_{n}} |\b)$, and hence $\nu <  \mu$ . 
Therefore $\mathrm{Supp} \, (F') \setminus \mathrm{Supp} \, (F)$ consists of 
elements  strictly greater   than $\nu$.  Then we apply to $F'$ the above argument, 
and repeat this. Eventually, we will obtain the function $\tilde{F}$ of the 
form $\tilde{F} = F - \sum_{\lambda   \in \mathcal{P}_{n}}  c_{\lambda} \cdot  s^{\L}_{\lambda}$
with $c_{\lambda}  \in \L[[\b]]$   whose restriction  $\phi^{(n)}_{\mu} (\tilde{F})$
to  all $\mu \in \mathcal{P}_{n}$ vanish, i.e., $\Phi^{(n)}(\tilde{F}) = 0$.    
Since the homomorphism $\Phi^{(n)}$ is injective (Lemma \ref{lem:PropertyPhi^{(n)}_A} (2)), 
we have $\tilde{F} = 0$, and hence we obtain the required expression.  
\end{proof}

\begin{cor} [Corollary to the proof of Theorem \ref{thm:BasisTheorem(s^L(x|b))}]  
\label{cor:ImPhi^(n)_A=Psi^(n)_A} 
The algebraic localization map $\Phi^{(n)}_{A}$ is onto the GKM ring $\Psi^{(n)}_{A}$. 
\end{cor} 
\begin{proof} 
 Let $\psi  = (\psi_{\mu})_{\mu \in \mathcal{P}_{n}}  \in \Psi^{(n)}_{A}$, and put $\mathrm{Supp} \, (\psi) = \{ \mu \in \mathcal{P}_{n} \; | \; \psi_{\mu} \neq 0 \} \subset \mathcal{P}_{n}$.  Applying the same argument as in 
 the proof of the above theorem,  one sees that $\psi$ is the image of a certain 
 function of the form $\sum_{\lambda \in \mathcal{P}_{n}}  c_{\lambda} \cdot s^{\L}_{\lambda}$, 
 whence the result follows.  
\end{proof}

%\newpage
%%%%%%%%%%%%%%%%%%%%%%%%%%%%%%%%%%%%%%%%%%%%%%%%%%%%%%%%%%%%%%%%%%%%%%%%%%%%%%%%%%%%%%%%%%%%%%%%
\subsection{Factorization Formula}    \label{subsec:FactorizationFormula} 
The following {\it factorization} property (cf. Pragacz \cite[Proposition 2.2]{Pra91}, 
Ikeda-Naruse \cite[Proposition 2.3]{Ike-Nar2013}) will be useful 
in the proof of the Basis Theorem  below (Theorem \ref{thm:BasisTheoremP^L(x_n)Q^L(x_n)}).  
\begin{prop} [Factorization Formula]    \label{prop:Factorization}     %(cf.\cite{Ike-Nar} Prop. 2.3.)
For a positive integer $k \geq 1$, let $\rho_{k}$ denote the partition $(k, k-1, \ldots, 2, 1)$
$($and $\rho_{0}  = \emptyset$ by convention$)$. 
\begin{enumerate}
\item  For a positive integer $n$,  
 we have
\begin{equation*} 
\begin{array}{llll} 
   P^{\L}_{\rho_{n-1}}(x_1,\ldots,x_{n} |\b)  
& =
& \left(\displaystyle
\prod_{1\leq i<j\leq n}(x_i\pf x_j)\right)
s_{\emptyset}^{\L}(x_1,\ldots,x_{n} |\b),  \medskip \\
Q^{\L}_{\rho_{n}}(x_1,\ldots,x_{n} |\b)  
& =
& \left(\displaystyle
\prod_{1\leq i\leq j\leq n}(x_i\pf x_j)\right)
s_{\emptyset}^{\L}(x_1,\ldots,x_{n} |\b).  \medskip 
\end{array}
\end{equation*} 

\item    For a a positive integer $n$ and 
  a partition  $\lambda= (\lambda_{1}, \lambda_{2}, \ldots, \lambda_{n}) \in \mathcal{P}_{n}$, 
we have 
\begin{equation*} 
\begin{array}{llll} 
P^{\L}_{\rho_{n-1}+\lambda}(x_1,\ldots,x_n|\b) 
& =
%P^{\L}_{\rho_{n-1}}(x_1,\ldots,x_n) 
& \left(\displaystyle
\prod_{1\leq i<j\leq n}(x_i\pf x_j)\right)
s^{\L}_\lambda(x_1,\ldots,x_n|\b),  \medskip \\
Q^{\L}_{\rho_{n}+\lambda}(x_1,\ldots,x_n|\b)
& =
%Q^{\L}_{\rho_{n}}(x_1,\ldots,x_n)
& \left(\displaystyle
\prod_{1\leq i\leq j\leq n}(x_i\pf x_j)\right)
 s^{\L}_\lambda(x_1,\ldots,x_n|\b).  \medskip 
\end{array}  
\end{equation*} 
\end{enumerate}
\end{prop}

\begin{proof}
The first formulas (1) follow immediately from (2) once we put $\lambda = \emptyset$.  
We shall show the formulas (2). We first  prove the case of $P^{\L}_{\lambda}$. 
Note that the length of $\rho_{n-1}  + \lambda$ is $n-1$ or $n$.  In any case, 
the product  in  the expression (\ref{eqn:DefinitionP^L(x_n|b)Q^L(x_n|b)})  becomes 
$[\x |\b]^{\lambda + \rho_{n-1} }  \prod_{1 \leq i < j \leq n}  \dfrac{x_{i} \pf x_{j}}{x_{i} \pf \overline{x}_{j}}$.  
Therefore we modify it as follows: 
\begin{equation*} 
  [\x | \b]^{\lambda + \rho_{n-1}}  \prod_{1 \leq i < j \leq n} 
  \dfrac{x_{i} \pf x_{j}}{x_{i} \pf  \overline{x}_{j}}  
  =  \prod_{1 \leq i < j \leq n}  (x_{i} \pf x_{j}) \times 
     \dfrac{[x | \b]^{\lambda + \rho_{n-1}}}
     {\prod_{1 \leq i < j \leq n} (x_{i} \pf \overline{x}_{j})}. 
\end{equation*} 
Since $\prod_{1 \leq i < j \leq n} (x_{i} \pf x_{j})$ is symmetric, 
we have from (\ref{eqn:DefinitionP^L(x_n|b)Q^L(x_n|b)}), 
\begin{equation*} 
\begin{array}{llll} 
 P^{\L}_{\rho_{n-1} + \lambda} (x_{1}, \ldots, x_{n} | \b) 
 & =  \displaystyle{\prod_{1 \leq i < j \leq n}}   (x_{i} \pf x_{j}) \times 
\sum_{w  \in  S_{n}} w  \left[
   \dfrac{[x | \b]^{\lambda + \rho_{n-1}}}
     {\prod_{1 \leq i < j \leq n} (x_{i} \pf \overline{x}_{j})}  
\right]  \medskip \\
&=   \displaystyle{\prod_{1 \leq i < j \leq n}}  (x_{i} \pf x_{j})  \times 
     s^{\L}_\lambda(x_1,\ldots,x_n|\b).    \medskip  
\end{array}
\end{equation*} 
For the case of $Q^{\L}_{\lambda}$, the product in the expression (\ref{eqn:DefinitionP^L(x_n|b)Q^L(x_n|b)}) 
becomes
\begin{equation*} 
\begin{array}{llll} 
   [[\x|\b]]^{\lambda + \rho_{n}}   \displaystyle{\prod_{1 \leq i < j \leq n}} 
   \dfrac{x_{i} \pf x_{j}} {x_{i} \pf \overline{x}_{j}}
   & =  \displaystyle{\prod_{i=1}^{n}}   (x_{i} \pf x_{i})[x_{i}|\b]^{\lambda_{i} + n - i}  \times 
   \prod_{1 \leq i < j \leq n}   \dfrac{x_{i} \pf x_{j}} {x_{i} \pf \overline{x}_{j}}  \medskip \\
   &= \displaystyle{\prod_{1 \leq i \leq j \leq n}}    (x_{i} \pf x_{j}) \times 
      \dfrac{ [ \x| \b]^{\lambda + \rho_{n-1}}} 
            { \prod_{1 \leq i < j \leq n}    (x_{i} \pf \overline{x}_{j}) },    \medskip 
\end{array} 
\end{equation*} 
and the result follows.  
\end{proof}

%%%%%%%%%%%%%%%%%%%%%%%%%%%%%%%%%%%%%%%%%%%%%%%%%%%%%%%%%%%%%%%%%%%%%%%%%%%%%%%%%%%%%%%%%%%%%%%%
\subsection{Basis Theorem for $P^{\L}_{\lambda} (\x)$, $Q^{\L}_{\lambda}  (\x)$} 
\label{subsec:BasisTheoremP^L(x)Q^L(x)}  
In \cite[Theorem 2.11]{Pra91}, Pragacz showed that 
the usual Schur $P$-polynomials $P_{\lambda} (x_{1}, \ldots, x_{n})$   ($\lambda \in \mathcal{SP}_{n}$)
form  a $\Z$-basis for  the ring $\Gamma_{n} = \Gamma (\x_{n})$ of ``supersymmetric polynomials'' (cf. Macdonald \cite[III, (8.9)]{Mac95}).
Their  $K$-theoretic analogues $GP_{\lambda} (x_{1}, \ldots, x_{n}) \; (\lambda \in \mathcal{SP}_{n})$
form a $\Z[\beta]$-basis for the ring $G\Gamma_{n}$ of ``$K$-supersymmetric polynomials'' 
(Ikeda-Naruse \cite[Theorem 3.1]{Ike-Nar2013}). 
Our functions $P^{\L}_{\lambda} (x_{1}, \ldots, x_{n})  \; (\lambda \in \mathcal{SP}_{n})$ also 
have the similar property. Namely we have the following theorem: 
\begin{theorem}  [Basis Theorem]  \label{thm:BasisTheoremP^L(x_n)Q^L(x_n)}  
\quad 
\begin{enumerate}
\item  The formal power series   
 $P^{\L}_\lambda(x_{1}, \ldots, x_{n})$ 
$(\lambda  \in \mathcal{SP}_{n})$   form a formal $\L$-basis of $\Gamma^{\L}(\x_{n})$.

\item   The formal power series    
   $Q^{\L}_\lambda(x_{1}, \ldots,  x_{n})$ 
$(\lambda  \in \mathcal{SP}_{n})$ 
 form a formal $\L$-basis of $\Gamma^{\L}_{+}(\x_{n})$.
\end{enumerate}
\end{theorem}

\begin{proof}   
 For the proof of this theorem, we  make use of  the same strategy as in Ikeda-Naruse \cite[Theorem 3.1]{Ike-Nar2013}, 
Pragacz \cite[Theorem 2.11 (Q)]{Pra91}, and thus 
we shall  only prove the case (1) when $n$ is even.
First note that 
$P^{\L}_{\lambda}(x_{1}, \ldots, x_{n})$ %(resp. $Q^{\L}_{\lambda}(x_{1}, \ldots, x_{n})$) 
 $(\lambda \in \mathcal{SP}_{n})$ are linearly independent over $\L$
 as we showed earlier (see \S \ref{subsec:DefinitionP^L(x|b)Q^L(x|b)}).  
   We use the induction on the number of the variables $n$. 
Let $n  = 2$ and  $f(x_{1}, x_{2})  $ be an $\L$-supersymmetric function. 
We may assume that the constant term of $f$ is zero, namely $f(0, 0) = 0$. 
Then $\L$-supersymmetricity implies that $f(t, \overline{t}) = f(0, 0) = 0$. 
Therefore by Lemma \ref{lem:FactorTheorem}, $f(x_{1}, x_{2})$ is divisible by $x_{1} \pf x_{2}$. 
Thus $f$ can be written as $f(x_{1}, x_{2}) = (x_{1} \pf x_{2}) g(x_{1}, x_{2})$
for some symmetric function $g (x_{1}, x_{2})$.  
By the  Basis Theorem for $s^{\L}_{\lambda}(\x_{n})$'s 
(Proposition \ref{prop:BasisTheorem(s^L(x))}), 
 we can write $g(x_{1}, x_{2}) = \sum_{\lambda \in \mathcal{P}_{2}} 
c_{\lambda}  s^{\L}_{\lambda}(x_{1}, x_{2}), \;  c_{\lambda} \in \L$.  
Then by the factorization property (Proposition \ref{prop:Factorization}), we have 
\begin{equation*} 
    f(x_{1}, x_{2})  
=  \sum_{\lambda  \in \mathcal{P}_{2}}  c_{\lambda} (x_{1} \pf x_{2}) s^{\L}_{\lambda}(x_{1}, x_{2}) 
= \sum_{\lambda \in \mathcal{P}_{2}}  c_{\lambda}  P^{\L}_{\rho_{1} + \lambda}  (x_{1}, x_{2}). 
\end{equation*} 
Thus $f(x_{1}, x_{2})$ is an $\L$-linear combination of $P^{\L}_{\mu}(x_{1}, x_{2})$'s, 
$\mu \in \mathcal{SP}_{2}$
(note that for $\lambda \in \mathcal{P}_{2}$,  we have $\rho_{1} + \lambda  \in \mathcal{SP}_{2}$). 

For $n \geq 4$,  we proceed as follows. Let $f(x_{1}, \ldots, x_{n})$ be an $\L$-supersymmetric function. 
Notice that  $f(x_{1}, \ldots, x_{n-2}, t, \overline{t}) = f(x_{1}, \ldots, x_{n-2}, 0, 0)$ holds. 
Put $f_{1}(x_{1}, \ldots, x_{n-2}) := f(x_{1}, \ldots, x_{n-2}, 0, 0)$. Since $f_{1}(x_{1}, \ldots, x_{n-2})$ 
is also $\L$-supersymmetric, we can write $f_{1}$ as an $\L$-linear combination of $P^{\L}_{\lambda} (x_{1}, \ldots, x_{n-2})$'s with $\lambda \in \mathcal{SP}_{n-2}$ by the induction hypothesis. Thus we  have the following expression: 
\begin{equation*}   
  f_{1}(x_{1}, \ldots, x_{n-2}) = \sum_{\lambda \in \mathcal{SP}_{n-2}}  c_{\lambda}  P^{\L}_{\lambda}(x_{1}, \ldots, x_{n-2}), \quad c_{\lambda}  \in \L.   
\end{equation*}  
Consider the function $g(x_{1}, \ldots, x_{n}) :=  \sum_{\lambda \in \mathcal{SP}_{n-2}} c_{\lambda}  P^{\L}_{\lambda}(x_{1}, \ldots, x_{n})$. Note that $g(x_{1}, \ldots, x_{n-2}, 0, 0) = f_{1}(x_{1}, \ldots, x_{n-2})  = f(x_{1}, \ldots, x_{n-2}, 0, 0)$ holds 
because of the  stability of $P^{\L}_{\lambda}$.      
Put $h(x_{1}, \ldots, x_{n}) := f(x_{1}, \ldots, x_{n}) - g(x_{1}, \ldots, x_{n})$. 
Then we have 
\begin{equation*}
\begin{array}{rllll}  
h(x_{1}, \ldots, x_{n-2}, t, \overline{t})  & = & h(x_{1}, \ldots, x_{n-2}, 0, 0)   \medskip \\
                                                   & = & f(x_{1}, \ldots, x_{n-2}, 0, 0) - g(x_{1}, \ldots, x_{n-2}, 0, 0)  = 0.  \medskip 
\end{array} 
\end{equation*} 
This implies that $x_{n-1} \pf x_{n}$ divides $h$.  Since $h$ is symmetric, 
we see that $h$ is a multiple of $V :=  \prod_{1 \leq i < j \leq n} (x_{i} \pf x_{j})$.  
Thus we have 
\begin{equation*} 
   f(x_{1}, \ldots, x_{n}) = g(x_{1}, \ldots, x_{n}) +  V 
   \cdot s(x_{1}, \ldots, x_{n}) 
\end{equation*} 
for some symmetric function $s$.  By the Basis Theorem for $s^{\L}_{\lambda}(\x_{n})$'s  
  (Proposition \ref{prop:BasisTheorem(s^L(x))})
again, we can write 
\begin{equation*}
    s(x_{1}, \ldots, x_{n})  = \sum_{\lambda \in \mathcal{P}_{n}} d_{\lambda}  s^{\L}_{\lambda}(x_{1}, \ldots, x_{n}), \quad d_{\lambda}   \in \L. 
\end{equation*} 
Using the factorization theorem (Proposition \ref{prop:Factorization}) again, we have 
\begin{equation*} 
\begin{array}{rllll} 
   f & = & g + V  \cdot s  \medskip \\
     &  = & \displaystyle{\sum_{\lambda  \in \mathcal{SP}_{n-2}}}   c_{\lambda}   P^{\L}_{\lambda} (x_{1}, \ldots, x_{n}) 
     + \sum_{\lambda  \in \mathcal{P}_{n}}   d_{\lambda} 
    \left ( \prod_{1 \leq i < j \leq n}  (x_{i} \pf x_{j}) \right ) 
       s^{\L}_{\lambda}(x_{1}, \ldots, x_{n})   \medskip \\
     &=  & \displaystyle{\sum_{\lambda  \in \mathcal{SP}_{n-2}}}  c_{\lambda}  P^{\L}_{\lambda}(x_{1}, \ldots, x_{n})  
     +  \sum_{\lambda  \in \mathcal{P}_{n}}  d_{\lambda}  
        P^{\L}_{\rho_{n-1} + \lambda} (x_{1}, \ldots, x_{n}).  
 \end{array} 
 \end{equation*}  
 Note that if $\lambda \in \mathcal{P}_{n}$,  then $\rho_{n-1} + \lambda  \in \mathcal{SP}_{n}$.  Thus 
 $f$ can be written as an $\L$-linear combination of $P^{\L}_{\mu}(x_{1}, \ldots, x_{n})$'s,  
 $\mu \in \mathcal{SP}_{n}$.  
\end{proof}
\noindent
Taking limit  $n \rightarrow \infty$, 
we obtain the following:  
\begin{cor} [Basis Theorem]    \label{cor:BasisTheoremP^L(x)Q^L(x)}  
\quad 
\begin{enumerate}
\item  The formal power series %polynomials 
 $P^{\L}_\lambda(\x)$ 
$(\lambda  \in \mathcal{SP})$   form a formal $\L$-basis of $\Gamma^{\L}(\x)$.

\item   The formal power series %polynomials
   $Q^{\L}_\lambda(\x)$ 
$(\lambda  \in \mathcal{SP})$ 
 form a formal $\L$-basis of $\Gamma^{\L}_{+}(\x)$.
\end{enumerate}
\end{cor}

%%%%%%%%%%%%%%%%%%%%%%%%%%%%%%%%%%%%%%%%%%%%%%%%%%%%%%%%%%%%%%%%%%%%%%%%%%%%%%%%%%%%%%%%%%%%%%%%%%%
\subsection{Vanishing Property}   \label{subsec:VanishingPropertyP^L(x|b)LQ^L(x|b)} 
Various factorial analogues of Schur $P$- and $Q$- polynomials (functions)
satisfy the  vanishing property (see Ivanov  \cite[Theorem 1.5 (the zero property)]{Iva97}, 
\cite[Theorem 5.3 (Vanishing property)]{Iva2004}, 
Ikeda-Naruse \cite[Proposition 8.3]{Ike-Nar2009},  
Ikeda-Mihalcea-Naruse \cite[Proposition 4.2]{IMN2011}, Ikeda-Naruse \cite[Proposition 7.1]{Ike-Nar2013}).  
In this subsection, we shall prove the vanishing property 
for the functions $P^{\L}_{\lambda}(\x_{n}|\b)^{+}$'s (for its definition, see \S  \ref{subsec:StabilityProperty})
and $Q^{\L}_{\lambda}(\x_{n}|\b)$'s. 
For a strict partition $\mu=(\mu_1,\ldots,\mu_r)$ of length $r$, we set
\begin{equation*} 
  \overline{\b}_{\mu}  :=(\overline{b}_{\mu_1},\ldots,\overline{b}_{\mu_r},0,0,\ldots). 
\end{equation*} 
We also set 
\begin{equation*} 
\begin{array}{llll} 
{\rm sh}(\mu)  &:=  (\mu_{1}+1,  \ldots,  \mu_{r}+1)   &  \quad  \text{if}  \; r \;  \text{is even},   \medskip \\
{\rm sh}(\mu)  &:=  (\mu_{1}+1,  \ldots,  \mu_{r} +1,1) &  \quad  \text{if}  \;  r  \; \text{is odd}.  \medskip 
\end{array}
\end{equation*} 
Here we consider only  {\it even}    variable case  
$P^\L_\lambda(x_1,\ldots,x_{2n}|\b)$ because of the mod $2$ stability
(see Proposition \ref{prop:StabilityMod2}).  
 %\cite{Ike-Nar2013} Remark 3.1)

\begin{prop}  [Vanishing Property]  \label{prop:VanishingProperty(P^LQ^L)}
Let $\lambda, \mu \in \mathcal{SP}$.  
Then we have 
\begin{enumerate}
\item  
\begin{equation*} 
\begin{array}{llll} 
& P^{\L}_{\lambda} (\overline{\b}_{{\rm sh}(\mu)} |\b)^{+}=0   
    \quad  \;  \text{if}   \quad      \mu  \not \supset \lambda, \medskip  \\
& P^{\L}_{\lambda} (\overline{\b}_{{\rm sh}(\lambda)}|\b)^{+} 
=\displaystyle
\prod_{i=1}^{r}
\left(
\prod_{  \substack{1 \leq  j \leq  \lambda_{i},   \\  j \neq \lambda_{p} +1 
   \text{ for }  \, i +1 \leq  p  \leq  r}}(\overline{b}_{\lambda_{i} +1} \pf b_{j})  \cdot 
\prod_{j=i+1}^{r}\left( 
\overline{b}_{\lambda_i+1}\pf \overline{b}_{\lambda_{j}+1}
\right)
\right).  \medskip 
\end{array}
\end{equation*}

\item  
\begin{equation*} 
\begin{array}{llll} 
 & Q^{\L}_{\lambda} (\overline{\b}_\mu|\b)=0 
  \quad \text{if} \quad    \mu\not \supset \lambda, \medskip \\
 & Q^{\L}_\lambda (\overline{\b}_\lambda|\b)=\displaystyle
\prod_{i=1}^r\left(
\prod_{ \substack{1 \leq j \leq  \lambda_{i}-1,     \\
        j \neq  \lambda_{p}  \text{ for } \,  i + 1 \leq  p  \leq  r}}  
\left(\overline{b}_{\lambda_i}\pf b_j\right)  \cdot 
\prod_{j=i}^{r} 
\left(\overline{b}_{\lambda_i}\pf \overline{b}_{\lambda_j}\right)
\right).  \medskip 
\end{array}  
\end{equation*} 

\end{enumerate}
\end{prop}

\begin{proof}

We will only prove (1) when the length $\ell (\mu) = r$ is {\it even}. 
 The proofs of the remaining cases   are similar. 
%We may assume that the length of $\lambda$ is less than or equal to $n$.
Using the  stability property (Proposition \ref{prop:StabilityMod2}),    
$P^{\L}_{\lambda} (\overline{\b}_{\mathrm{sh} (\mu)}|\b)^{+}$ 
can be evaluated as $P^{\L}_{\lambda}(\overline{b}_{\mu_{1} + 1}, \overline{b}_{\mu_{2} + 1}, 
\ldots, \overline{b}_{\mu_{r} + 1} |\b)$.

If $\mu \not \supset \lambda$,  there exists an index $ 1 \leq k  \leq r$ such that
$\mu_{k} <  \lambda_{k}$ (and hence $\mu_{k} + 1 \leq \lambda_{k}$). 
For an arbitrary permutation $w \in S_{r}$, there exists a positive integer 
$1 \leq j \leq k$ such that $w(j) \geq k$.  
Then we have  inequalities $\mu_{w(j)} \leq  \mu_{k} < \lambda_{k} \leq \lambda_{j}$.  
Therefore    the term  
\begin{equation*}  
 [x_{w(j)} |\b]^{\lambda_{j}}
   = (x_{w(j)} \pf b_{1}) (x_{w(j)} \pf b_{2}) \cdots (x_{w(j)} \pf b_{\mu_{w(j)} + 1}) 
 \cdots  (x_{w(j)} \pf b_{\lambda_{j}}) 
\end{equation*}  
vanishes  when we specialize $x_{w(j)}$ to $\overline{b}_{\mu_{w(j)} + 1}$.  
This means that $P^{\L}_{\lambda}(\overline{\b}_{{\rm sh}(\mu)}|\b)=0$.

For the case  $\mu=\lambda$,
we shall show  that in the defining equation (\ref{eqn:DefinitionP^L(x_n|b)Q^L(x_n|b)}) of $P^{\L}_{\lambda}(x_{1}, \ldots, x_{r}|\b)$, 
each term corresponding to $w \in S_{r}$ other than the identity $e$  becomes   zero 
when we evaluate them at 
$(x_{1},  \ldots, x_{r}) = \overline{\b}_{{\rm sh}(\lambda)} = (\overline{b}_{\lambda_{1} + 1}, 
\ldots, \overline{b}_{\lambda_{r} + 1})$. For $w \neq e$, there exists an index 
$1 \leq k \leq r$ such that $w(k) > k$. Thus we have $\lambda_{w(k)}  < \lambda_{k}$ 
(and hence $\lambda_{w(k)} + 1 \leq \lambda_{k}$). 
Then it is obvious that 
\begin{equation*} 
  [x_{w(k)}|\b]^{\lambda_{k}} = (x_{w(k)} \pf b_{1}) (x_{w(k)} \pf b_{2}) \cdots 
  (x_{w(k)} \pf b_{\lambda_{w(k)} + 1}) \cdots (x_{w(k)} \pf b_{\lambda_{k}}) 
\end{equation*} 
becomes zero  when we specialize $x_{w(k)}$  to $\overline{b}_{\lambda_{w(k)} + 1}$. 
The term corresponding to $w = e$  is 
\begin{equation*} 
\begin{array}{llll} 
&  \displaystyle{\prod_{i=1}^{r}} [\overline{b}_{\lambda_{i} + 1}|\b]^{\lambda_{i}} 
 \prod_{i=1}^{r} \prod_{j = i + 1}^{r} 
\dfrac{\overline{b}_{\lambda_{i} + 1}  \pf  \overline{b}_{\lambda_{j} + 1}}
      {\overline{b}_{\lambda_{i} + 1}  \pf  b_{\lambda_{j} + 1} }    \medskip \\
= & \displaystyle{\prod_{i=1}^{r}}  \left ( (\overline{b}_{\lambda_{i} + 1} \pf b_{1}) \cdots (\overline{b}_{\lambda_{i} + 1} \pf b_{\lambda_{i}})  \times 
\prod_{j= i + 1}^{r}  \dfrac{\overline{b}_{\lambda_{i} + 1}  \pf  \overline{b}_{\lambda_{j} + 1}}
      {\overline{b}_{\lambda_{i} + 1}  \pf  b_{\lambda_{j} + 1} }  \right ).     \medskip 
\end{array} 
\end{equation*} 
By cancellation, we obtain the desired value.  
%(cf. \cite{Ike-Nar2013} Prop.7.1).
\end{proof}

%%%%%%%%%%%%%%%%%%%%%%%%%%%%%%%%%%%%%%%%%%%%%%%%%%%%%%%%%%%%%%%%%%%%%%%%%%%%%%%%%%%%%%%%%%%%%%%
\subsection{Algebraic localization map of types $B_{\infty}$, $C_{\infty}$, or $D_{\infty}$}   
In \S \ref{subsec:BasisTheoremP^L(x)Q^L(x)}, we have proven the Basis Theorem 
for $P^{\L}_{\lambda}(\x_{n})$'s and $Q^{\L}_{\lambda}(\x_{n})$'s. 
We shall prove the Basis Theorem for $P^{\L}_{\lambda}(\x_{n}|\b)^{+}$'s and 
$Q^{\L}_{\lambda}(\x_{n}|\b)$'s  in \S  \ref{subsec:BasisTheoremP^L(x|b)Q^L(x|b)}. 
For the proof, we apply the same technique as  the case of $s^{\L}_{\lambda} (\x_{n}|\b)$
(see \S \ref{subsubsec:BasisTheorem(s^L(x|b))}), 
namely the Vanishing Property (Proposition \ref{prop:VanishingProperty(P^LQ^L)}) 
and the localization technique.  
Because  the idea of the proof is the same as that of Theorem \ref{thm:BasisTheorem(s^L(x|b))}, 
we only exhibit the definitions and the results. %necessary data  and the results.   
We  mainly follow the notation and convention as in 
Ikeda-Naruse \cite[\S 4]{Ike-Nar2013} and Ikeda-Mihalcea-Naruse \cite[\S 3]{IMN2011}
(we collect the necessary data in the Appendix \ref{subsec:RootDeta(BCD)}).  

Let us introduce the GKM ring  and the algebraic localization maps
of type $X_{\infty}$ for $X = B, C$, or $D$.  %$B_{\infty}$, $C_{\infty}$ or $D_{\infty}$. 
 As with the type $A_{\infty}$ case 
 (\S \ref{subsubsec:AlgebraicLocalizationMap(A)}), denote by 
$\mathrm{Map} \, (\mathcal{SP}_{n}, \L[[\b]])$ the set of all maps from 
$\mathcal{SP}_{n}$ to $\L[[\b]]$. It is an $\L[[\b]]$-algebra under  pointwise 
multiplication and scalar multiplication.  
We identify $\mathrm{Map} \, (\mathcal{SP}_{n}, \L[[\b]])$ with the 
product ring $\prod_{\mu \in \mathcal{SP}_{n}} (\L[[\b]])_{\mu}$.  
\begin{defn} [GKM ring of type  $X_{\infty}$]  %\footnote{
%cf. Ikeda-Naruse \cite[Definition 5.1]{Ike-Nar2013}.  
%} %$B_{\infty}$, $C_{\infty}$, or $D_{\infty}$] 
Let $X$ be $B$, $C$, or $D$. 
Let $\Psi^{(n)}_{X}$ be a subring of $\mathrm{Map} \, (\mathcal{SP}_{n}, \L[[\b]])$ 
that consists of elements $\psi = (\psi_{\mu})_{\mu \in \mathcal{SP}_{n}}$ satisfying the following 
GKM condition$:$
\begin{equation*} 
  \psi_{s_{\alpha} \mu} - \psi_{\mu}  \in e(-\alpha) \,  \L[[\b]] \quad 
  \text{for all} \; \mu \in \mathcal{SP}_{n} \; \text{and all} \; 
  \alpha \in \Delta^{+}_{X}. 
\end{equation*}  
\end{defn} 
\noindent 
We call $\Psi^{(n)}_{X}$ the {\it GKM ring of type $X_{\infty}$}.  

Next we shall introduce the algebraic localization map of type $X_{\infty}$ for $X = C, D$. 
For each strict partition $\mu \in \mathcal{SP}_{n}$, 
define  the $\L[[\b]]$-algebra homomorphisms to be 
\begin{equation*} 
\begin{array}{llll} 
  & \phi^{(n)}_{\mu, C}: \Gamma^{\L}_{+}(\x_{n}|\b) 
          = \L[[\b]] \, \hat{\otimes}_{\L}  \, \Gamma^{\L}_{+}(\x_{n}) 
                     \longrightarrow  \L[[\b]], \medskip \\ 
                     & F = F(\x_{n}|\b) \longmapsto  \phi^{(n)}_{\mu, C}(F)
                     := F(\overline{\b}_{\mu}|\b), \medskip \\
   & \phi^{(n)}_{\mu, D}: \Gamma^{\L} (\x_{n}|\b) 
          = \L[[\b]]  \, \hat{\otimes}_{\L}  \, \Gamma^{\L} (\x_{n})
                      \longrightarrow \L[[\b]],  \medskip \\
                     & F = F(\x_{n}|\b) \longmapsto  \phi^{(n)}_{\mu, D}(F) 
                     := F(\overline{\b}_{\mathrm{sh}(\mu)}|\b). \medskip 
\end{array} 
\end{equation*} 
%\vspace{2cm} 

\begin{defn}  [Algebraic localization map of type $X_{\infty}$] %\footnote{
%cf. Ikeda-Naruse \cite[Definition 7.1 (Localization map)]{Ike-Nar2013}.  
%} 
Define the homomorphisms of $\L[[\b]]$-algebras to be 
\begin{equation*} 
\begin{array}{llll} 
 &  \Phi^{(n)}_{C}:  \Gamma^{\L}_{+}(\x_{n}|\b) %= \L[[\b]] \hat{\otimes}_{\L} \Gamma^{\L}_{+}(\x_{n})  
  \longrightarrow  \displaystyle{\prod_{\mu \in \mathcal{SP}_{n}}} (\L[[\b]])_{\mu},  \quad 
  &  F \longmapsto \Phi^{(n)}_{C} (F) := (\phi^{(n)}_{\mu, C}(F))_{\mu  \in \mathcal{SP}_{n}}, \medskip \\
  &  \Phi^{(n)}_{D}: \Gamma^{\L} (\x_{n}|\b) %= \L[[\b]] \hat{\otimes}_{\L} \Gamma^{\L}(\x_{n})  
  \longrightarrow  \displaystyle{\prod_{\mu \in \mathcal{SP}_{n}}} (\L[[\b]])_{\mu},  \quad 
  &  F \longmapsto \Phi^{(n)}_{D} (F) := (\phi^{(n)}_{\mu, D}(F))_{\mu  \in \mathcal{SP}_{n}}.   \medskip 
 \end{array} 
\end{equation*} 
\end{defn} 
\noindent
We call the homomorphism $\Phi^{(n)}_{X}$ the {{\it algebraic localization map of type $X_{\infty}$}. 
Concerning the above algebraic localization maps, 
the following lemma can be proved analogously 
(see the proof of Lemma \ref{lem:PropertyPhi^{(n)}_A} %\footnote{
%See also Ikeda-Naruse \cite[Proposition 7.2, Lemma 7.1, Corollary 7.1]{Ike-Nar2013}
%}
): 
\begin{lem} 
Let $X$ be $C$ or $D$. Then we have 
\begin{enumerate} 
\item  The image of $\Phi^{(n)}_{X}$ agrees with the GKM ring $\Psi^{(n)}_{X}: 
       \Im \, (\Phi^{(n)}_{X})  =   \Psi^{(n)}_{X}$.
\item  The homomorphism $\Phi^{(n)}_{X}$ is injective.  
\end{enumerate}   
\end{lem}

%\newpage
%%%%%%%%%%%%%%%%%%%%%%%%%%%%%%%%%%%%%%%%%%%%%%%%%%%%%%%%%%%%%%%%%%%%%%%%%%%%%%%%%%%%%%%%%%%%%%%%%%
\subsection{Basis Theorem for $P^{\L}_{\lambda}(\x|\b)^{+}$, $Q^{\L}_{\lambda}(\x |\b)$}  
\label{subsec:BasisTheoremP^L(x|b)Q^L(x|b)}  
The GKM ring  and the algebraic localization map  
of type $X_{\infty}$ for $X = B, C$, or $D$, and an analogous argument 
as  in the Basis Theorem for $s^{\L}_{\lambda}(\x_{n}|\b)$'s  
(Theorem \ref{thm:BasisTheorem(s^L(x|b))}) 
enable  us to  prove  the Basis Theorem for $P^{\L}_{\lambda}(\x|\b)^{+}$'s 
and $Q^{\L}_{\lambda} (\x|\b)$'s.  
\begin{theorem} [Basis Theorem]    \label{thm:BasisTheoremP^L(x_n|b)Q^L(x_n|b)}  
\quad 
\begin{enumerate}
\item  The formal power series %polynomials 
 $P^{\L}_\lambda(\x_{n}|\b)^{+}$ 
$(\lambda  \in \mathcal{SP}_{n})$   form a formal $\L[[\b]]$-basis of
 $\Gamma^{\L}(\x_{n}|\b)$.

\item   The formal power series %polynomials
   $Q^{\L}_\lambda(\x_{n}|\b)$ 
$(\lambda  \in \mathcal{SP}_{n})$ 
 form a formal $\L[[\b]]$-basis of $\Gamma^{\L}_{+}(\x_{n}|\b)$.
\end{enumerate}
\end{theorem} 

%\begin{proof} 
%We only prove the statement (1). The assertion (2) will be proved analogously. 
%By the similar augument as in the case of $s^{\L}_{\lambda}(x_{n}|\b)$ (Theorem %\ref{thm:BasisTheorem(s^L(x|b))}), %
%the linear independence of $P^{\L}_{\lambda}(x_{1}, \ldots, x_{n}|\b)$'s 
%over $\L[[\b]]$ will follow  from the Vanishing Property (Proposition %\ref{prop:VanishingProperty(P^LQ^L)}). %
%
%\vspace{3cm}  
%\end{proof} 
\noindent
%\newpage
Taking limit $n \rightarrow \infty$, we obtain the following: 
\begin{cor}   [Basis Theorem]   \label{cor:BasisTheoremP^L(x|b)Q^L(x|b)}   
\quad 
\begin{enumerate}
\item  The formal power series  
 $P^{\L}_\lambda(\x|\b)^{+}$ 
$(\lambda  \in \mathcal{SP})$   form a formal $\L[[\b]]$-basis of $\Gamma^{\L}(\x|\b)$.

\item   The formal power series   
   $Q^{\L}_\lambda(\x|\b)$ 
$(\lambda  \in \mathcal{SP})$ 
 form a formal $\L[[\b]]$-basis of $\Gamma^{\L}_{+}(\x|\b)$.
\end{enumerate}
\end{cor}

%\newpage
%%%%%%%%%%%%%%%%%%%%%%%%%%%%%%%%%%%%%%%%%%%%%%%%%%%%%%%%%%%%%%%%%%%%%%%%%%%%%%%%%%%%%%%%%%%%%%%%%%%
\section{Dual universal (factorial) Schur $P$- and $Q$-functions}   \label{sec:DUFSPQF}  
%%%%%%%%%%%%%%%%%%%%%%%%%%%%%%%%%%%%%%%%%%%%%%%%%%%%%%%%%%%%%%%%%%%%%%%%%
In the previous section,  we have  constructed  ``cohomology  base'' 
$\{ P^{\L}_{\lambda}(\x|\b)^{+} \}$  %_{\lambda \in \mathcal{SP} }$ 
and $\{ Q^{\L}_{\lambda}(\x|\b) \}$. %_{\lambda \in \mathcal{SP}}$.  
Our next task is to construct the corresponding 
``homology base'' $\{ \wh{p}^{\L}_{\lambda} (\y|\b) \}$  %_{\lambda \in \mathcal{SP}}$
 and $\{ \wh{q}^{\L}_{\lambda} (\y|\b) \}$. %_{\lambda \in \mathcal{SP}}$.  
In this section,  we consider this problem  only in the non-equivariant case, i.e., $\b = 0$. 
Namely, we shall construct  certain functions $ \wh{p}^{\L}_{\lambda} (\y)$
(resp.  $\wh{q}^{\L}_{\lambda} (\y)$)  for  strict partitions $\lambda \in \mathcal{SP}$
dual to $Q^{\L}_{\lambda}(\x)$  (resp. $P^{\L}_{\lambda}(\x)$).  
Here we use  the countably infinite set of variables 
$\y =(y_{1}, y_{2}, \ldots)$.  %and $\b = (b_{1}, b_{2},\ldots)$. 
Their   degrees are given by   $\deg \, (y_{i}) =  1$  % and   \underline{$\deg \, (b_{i}) = -1$}  for $i= 1, 2, \ldots$, 
and $\deg_{h} (a_{i, j}) = i + j - 1 \; (i, j \geq 1)$.  
Let $\Lambda_{\L} (\y) = \L_{*} \otimes_{\Z} \Lambda (\y)$ 
be the ring of symmetric functions in $\y$ 
with coefficients in $\L_{*}$   %and put $\Lambda_{\L}(\y|\b)= \L[[\b]] \, \hat{\otimes}_{\L} \, \Lambda_{\L}(\y)$
 (see the beginning of \S \ref{subsec:DefinitionP^L(x|b)Q^L(x|b)}).

%%%%%%%%%%%%%%%%%%%%%%%%%%%%%%%%%%%%%%%%%%%%%%%%%%%%%%%%%%%%%%%%%%%%%%%%%%%%%%%%%%%%
\subsection{One row case}  \label{subsec:OneRowCase}  
First we shall construct the required functions corresponding to the  ``one rows'',
that is, strict partitions $(k) \; (k = 1, 2, \ldots)$.  
We put 
\begin{equation*} 
  \Delta(t; \y):=  \prod_{j=1}^{\infty}\frac{1-\overline{t} y_j}{1-t y_j} 
  \in   \Lambda_{\L} (\y) [[t]]. 
\end{equation*} 
Then we define  $\widehat{q}^\L_{k}(\y)  \in    \Lambda_{\L}(\y) \; (k = 0, 1, 2, \ldots)$ 
as the coefficients of the following expansion: 
 \begin{equation}   \label{eqn:q^L_k}  
\Delta(t;\y)
=\sum_{k=0}^{\infty}  \widehat{q}^\L_{k}(\y) t^{k}.  
\end{equation} 
%%%%%%%%%%%%%%%%%%%%%%%%%%%%%%%%%%%%%%%%%%%%%%%%%%%%%%%%%%%%%%%%%%%%%%%%%%%%%%%%%%%%%%%%
Next we shall define $\wh{p}^{\L}_{k}(\y)  \in \Lambda_{\L}(\y) \; (k =  1, 2, \ldots)$  
as the coefficients of the following expansion:  
\begin{equation}   \label{eqn:p^L_k}
    \Delta (t; \y)  =  1 + (t \pf t) \sum_{k=1}^{\infty}  \wh{p}^{\L}_{k}(\y) t^{k-1}. 
\end{equation} 
We set $\wh{p}^{\L}_{0}(\y) := 1$ by convention.  
In order to make the above  definition valid, we have to verify that 
$\Delta (t; \y) - 1$ is divisible by $t \pf t$.  This follows from 
Lemma \ref{lem:FactorTheorem}.  More concretely, if we write 
 $t \pf t = 2t + \sum_{k=2}^{\infty} \alpha_{k}^{\L} \, t^{k}$ with $\alpha^{\L}_{k} \in \L$, 
we have 
\begin{equation}  \label{eqn:wh{q}^L_k=2wh{p}^L_k+...}   
  \wh{q}^{\L}_{1}(\y) = 2\wh{p}^{\L}_{1}(\y), \quad  
 \wh{q}^{\L}_{k}(\y) = 2 \wh{p}^{\L}_{k}(\y) + \sum_{j=1}^{k-1} \alpha_{k + 1 -j}^{\L} \, \wh{p}^{\L}_{j} (\y)
\; (k \geq 2). 
\end{equation}

\begin{rem} 
\quad 
\begin{enumerate} 
\item Notice that comparing $(\ref{eqn:q^L_k})$  $($resp.  $(\ref{eqn:p^L_k}))$ 
 with the equation $(\ref{eqn:wh{q}^E})$  $($resp. $(\ref{eqn:wh{p}^E}))$,
 we see immediately that 
$\widehat{q}^{\L}_{k}(\y) \; (k =  1, 2, \ldots)$  $($resp.  $\wh{p}^{\L}_{k}(\y) \; (k = 1, 2, \ldots))$  coincide with $\wh{q}^{MU}_{k}(\y)$  $($resp.  $\wh{p}^{MU}_{k}(\y))$ 
in Definition $\ref{df:wh{q}^E}$  $($resp. Definition  $\ref{df:wh{p}^E})$.  
\item It follows from Definition $\ref{df:DefinitionP^L(x_n|b)Q^L(x_n|b)}$ that 
$P^{\L}_{(k)} (x_{1}) = x_{1}^{k}$ and $Q^{\L}_{(k)} (x_{1}) = (x_{1} \pf x_{1})x_{1}^{k-1}$
for $k = 1, 2, \ldots$.  Therefore we can write $(\ref{eqn:q^L_k})$ and $(\ref{eqn:p^L_k})$ 
as 
\begin{equation}   \label{eqn:Delta(x_1;y)}   
  \Delta (x_{1}; \y) = \prod_{j=1}^{\infty}  \dfrac{ 1 - \overline{x}_{1} y_{j}}{1 - x_{1} y_{j}} 
  = \sum_{k=0}^{\infty}  P^{\L}_{(k)} (x_{1}) \,  \wh{q}^{\L}_{k}(\y)  
  = \sum_{k=0}^{\infty}  Q^{\L}_{(k)} (x_{1}) \, \wh{p}^{\L}_{k}(\y).  
\end{equation} 
\end{enumerate} 
\end{rem}

%%%%%%%%%%%%%%%%%%%%%%%%%%%%%%%%%%%%%%%%%%%%%%%%%%%%%%%%%%%%%%%%%%%%%%%%%%%%%%%%%%%%%%%%
%\subsection{Kernel}
\subsection{Definition of $\wh{p}^{\L}_{\lambda}(\y)$ and $\wh{q}^{\L}_{\lambda}(\y)$} 
\label{subsec:Definitionwh{p}^L(y)wh{q}^L(y)}  
In \S \ref{subsec:OneRowCase}, we  have  constructed $\wh{p}^{\L}_{k} (\y)$ and 
$\wh{q}^{\L}_{k}(\y)$ corresponding to  the one rows $(k)  \; (k = 1, 2, \ldots)$. 
In this subsection, we extend these functions to arbitrary strict partitions 
$\lambda \in \mathcal{SP}$. 
For this end, we argue as follows. Define 
\begin{equation*} 
\begin{array}{rlll} 
\Gamma_{\L}(\y)        & :=  & \text{ the } \L \text{-subalgebra of } 
        \Lambda_{\L} (\y)
\text{ generated by } \widehat{p}^{\L}_{k}(\y),\; k=1,2,\ldots,  \medskip \\
 \Gamma_{\L}^{+}(\y)    & :=  & \text{ the }\L  \text{-subalgebra of }
         \Lambda_{\L} (\y)
\text{ generated by } \widehat{q}^{\L}_{k}(\y),\; k= 1,2,\ldots. \medskip 
\end{array}
\end{equation*} 
%\begin{rem}   \label{rem:Gamma_L(y)=MU_*(OmegaSp)}  
As we remarked in the previous subsection, $\wh{p}^{\L}_{k}(\y)
= \wh{p}^{MU}_{k}(\y)$ and $\wh{q}^{\L}_{k}(\y) = \wh{q}^{MU}_{k}(\y)$ for 
$k = 1, 2, \ldots$,  and hence  the above algebras $\Gamma_{\L}(\y)$ and $\Gamma_{\L}^{+}(\y)$ 
coincide with the algebras ${\Gamma'}^{MU}_{*}$ and $\Gamma^{MU}_{*}$ respectively  
defined in  Definitions $\ref{df:Gamma'^E_*}$ and $\ref{df:Gamma^E_*}$. 
%\S  \ref{subsec:E-homologySchurP,Q-functions}. 
In particular, we have 
\begin{equation*} 
\begin{array}{rlll} 
  \Gamma_{\L}(\y)  &= &  \L [\wh{p}^{\L}_{1}(\y),  \wh{p}^{\L}_{3}(\y), \ldots, 
                             \wh{p}^{\L}_{2i-1} (\y), \ldots]     \; ( \cong MU_{*}(\Omega Sp)),     \medskip \\
  \Gamma_{\L}^{+}(\y)  
    &= &  \L[\wh{q}^{\L}_{1}(\y), \wh{q}^{\L}_{2}(\y), \ldots, \wh{q}^{\L}_{i}(\y), \ldots]
      /(\wh{q}^{\L} (T) \wh{q}^{\L}(\overline{T}) = 1)   \; (\cong MU_{*}(\Omega_{0} SO)),  \medskip 
\end{array} 
\end{equation*} 
where $\wh{q}^{\L}(T) := \sum_{k \geq 0} \wh{q}^{\L}_{k}(\y) T^{k}$.  
Also $\Gamma_{\L}(\y)$ and $\Gamma_{\L}^{+}(\y)$ have the Hopf algebra structure over $\L$ 
as explained in \S $\ref{subsec:E-homologySchurP,Q-functions}$.    
%\end{rem} 

For later discussion, we prepare the following: 
For a partition $\lambda = (\lambda_{1}, \ldots, \lambda_{r})$, %of length $r \leq n$,
define the monomials 
\begin{equation*} 
\begin{array}{rlll} 
   \widehat{p}^{\L}_{[\lambda]}(\y) 
& :=  &  \displaystyle{\prod_{i=1}^{r}}  
   \widehat{p}^{\L}_{\lambda_i}(\y) 
= \wh{p}^{\L}_{\lambda_{1}}  (\y) \wh{p}^{\L}_{\lambda_{2}} (\y) \cdots 
  \wh{p}^{\L}_{\lambda_{r}} (\y), \medskip \\
   \widehat{q}^{\L}_{[\lambda]} (\y) 
& :=  & \displaystyle{\prod_{i=1}^{r}}   \widehat{q}^{\L}_{\lambda_{i}}(\y)  
= \wh{q}^{\L}_{\lambda_{1}} (\y) \wh{q}^{\L}_{\lambda_{2}} (\y)  \cdots
   \wh{q}^{\L}_{\lambda_{r}}(\y). \medskip 
\end{array} 
\end{equation*} 
Note that  there exist some relations among $\wh{p}^{\L}_{k} (\y)$'s 
and $\wh{q}^{\L}_{k} (\y)$'s  (see  the relations (\ref{eqn:relations(q^E)}) and (\ref{eqn:relations(p^E)})), and 
these  monomials are not linearly independent over $\L$. 
Then we define the $\L$-submodule  of $\Gamma_{\L}(\y)$ (resp.  $\Gamma_{\L}^{+}(\y)$) 
spanned  by  these  monomials $\wh{p}^{\L}_{[\lambda]}(\y)$'s (resp. $\wh{q}^{\L}_{[\lambda]} (\y)$'s)  with  $\lambda \in \mathcal{P}_{n}$:    
\begin{equation*} 
\begin{array}{rlll}
\Gamma_{\L}^{(n)}(\y) 
&:=   &  \displaystyle{\sum_{\lambda \in \mathcal{P}_{n}}} 
    \L \: \wh{p}^{\L}_{[\lambda]}(\y)   \;  \subset  \;  \Gamma_{\L}(\y),  \medskip \\
\Gamma_{\L}^{(n),+}(\y) 
&:=   &\displaystyle\sum_{\lambda\in \mathcal{P}_n} 
\L\:  \wh{q}^{\L}_{[\lambda]}(\y)   \;  \subset  \;  \Gamma_{\L}^{+}(\y).  \medskip 
\end{array}
\end{equation*}

%%%%%%%%%%%%%%%%%%%%%%%%%%%%%%%%%%%%%%%%%%%%%%%%%%%%%%%%%%%%%%%%%%%%%%%%%%%%%%%%%%%%%%%%%%%%%%%%%%%
Now   consider the iterated product of $\Delta(x_{i};\y)$'s, and their limit.
\begin{equation*} 
\begin{array}{rlll} 
  \Delta (\x_{n}; \y) & =  &  \Delta (x_{1},  \ldots, x_{n}; \y) 
   :=  \;  \displaystyle{\prod_{i=1}^{n}}   \Delta(x_{i};\y)
       = \prod_{i=1}^{n} \prod_{j=1}^{\infty} \dfrac{1 - \overline{x}_{i}y_{j}}{1 - x_{i}y_{j}},
   \medskip \\ 
   \Delta(\x;\y) 
  & :=   &   \displaystyle{\lim_{\substack{\longleftarrow \\ n}}} \,   \Delta( \x_{n};\y) 
        = \prod_{i=1}^{\infty} \prod_{j=1}^{\infty} 
          \dfrac{1 - \overline{x}_{i} y_{j}} {1 - x_{i}y_{j} }.  \medskip 
\end{array} 
\end{equation*} 
We can think of $\Delta (\x_{n}; \y)$ (resp. $\Delta (\x; \y)$) 
as an element of  $\Lambda^{\L} (\x_{n}) \, \hat{\otimes}_{\L} \, \Lambda_{\L}(\y)$
(resp.  $\Lambda^{\L}(\x)   \,  \hat{\otimes}_{\L}  \Lambda_{\L}(\y)$).

\begin{prop}   
\label{prop:Delta(x_1,...,x_n;y)}  
\quad 
\begin{enumerate}
\item  
$\Delta(x_1,\ldots,x_n;\y)$ is $\L$-supersymmetric in the variables $x_{1},\ldots, x_{n}$. 
Therefore  $\Delta (x_{1}, \ldots, x_{n}; \y)$ is in
 $\Gamma^{\L}(\x_{n})   \, \hat{\otimes}_{\L}  \, \Lambda_{\L}(\y)$.  
Moreover, $\Delta(x_{1}, \ldots, x_{n};\y)$ is in 
$\Gamma^{\L}_{+}(\x_{n})  \,  \hat{\otimes}_{\L}   \, \Lambda_{\L}(\y)$. 

\item  Furthermore, we have 
$\Delta(x_{1},  \ldots,x_{n};\y)$ is in  
 $\Gamma^\L_{+}(\x_{n}) \,  \hat{\otimes}_{\L} \,  \Gamma_{\L}^{(n)}(\y)$
and
$\Delta(x_{1},\ldots,x_{n}; \y)$   is in  \\
$\Gamma^{\L}(\x_{n})  \,  \hat{\otimes}_{\L}  \,   \Gamma_{\L}^{(n),+}(\y)$.
\end{enumerate}
\end{prop}
\begin{proof} 
(1) By the definition, it is obvious that $\Delta (t, \overline{t}, x_{3}, \ldots, x_{n})
= \Delta (0, 0, x_{3}, \ldots, x_{n})$ holds. 
Moreover, using Lemma \ref{lem:FactorTheorem}, we see easily that 
%$\Delta(x_1,\ldots,x_n;\y)$ has the property that
$\Delta(t,x_{2}, \ldots, x_{n};\y) - \Delta(0,x_{2},  \ldots,x_{n};\y)$ is divisible by 
$t\pf t$, and hence the first assertion is proved. 

(2)  Using (\ref{eqn:q^L_k}), we have 
\begin{equation*} 
   \displaystyle{\prod_{i=1}^{n}} 
   \prod_{j=1}^{\infty}  \dfrac{1 - \overline{x}_{i}y_{j}} {1 - x_{i} y_{j}} 
   =     \displaystyle{\prod_{i=1}^{n}} 
    \left ( \sum_{\lambda_{i} = 0}^{\infty}   x_{i}^{\lambda_{i}} 
    \wh{q}^{\L}_{\lambda_{i}} (\y)   \right ) 
    = \sum_{\lambda \in \mathcal{P}_{n}}  m_{\lambda} (\x_{n})  \,  \wh{q}^{\L}_{[\lambda]} (\y), \end{equation*} 
where $m_{\lambda}(\x_{n})$ denotes the {\it monomial symmetric polynomial} 
corresponding to a partition $\lambda \in \mathcal{P}_{n}$. 
Using (\ref{eqn:wh{q}^L_k=2wh{p}^L_k+...}),  we see that $\wh{q}^{\L}_{[\lambda]}(\y)$ 
can be written as a certain  $\L$-linear combination of the  monomials 
of the form $\wh{p}^{\L}_{[\mu]}(\y)$ with $\mu \in \mathcal{P}_{n}$. 
From these,  the result follows immediately. 
\end{proof} 

\noindent
Taking limit $n \rightarrow \infty$ and using the Basis Theorem (Corollary \ref{cor:BasisTheoremP^L(x)Q^L(x)}), 
we can expand $\Delta (\x; \y)$ in terms of a basis $\{  Q^{\L}_{\lambda}(\x) \}_{\lambda \in \mathcal{SP}}$ for $\Gamma^{\L}_{+}(\x)$ 
(resp. a basis $\{  P^{\L}_{\lambda}(\x) \}_{\lambda \in \mathcal{SP}}$ for $\Gamma^{\L}(\x)$). 
Then we will obtain the required functions as the coefficients of these expansions. 
Thus we make the  following definition: 
\begin{defn}  [Dual universal  Schur $P$ and $Q$-functions]    \label{df:Definitionwh{p}^Lwh{q}^L}  
We define $\widehat{p}^{\L}_\lambda(\y)$ and $\widehat{q}^{\L}_\lambda(\y)$
for strict partitions $\lambda \in \mathcal{SP}$ by the following identities
$($Cauchy identities$):$ 
\begin{equation}  \label{eqn:CauchyIdentities}
\begin{array}{rllll}    
\Delta(\x;\y)  & = &  \displaystyle{\prod_{i,j  \geq 1}} 
   \frac{1-\overline{x}_{i} y_{j}}{1- x_{i} y_{j}}
=  \sum_{\lambda  \in \mathcal{SP}}
Q^{\L}_{\lambda} (\x) \, \widehat{p}^\L_{\lambda} (\y),  \medskip \\ 
%\end{equation} 
%\begin{equation}    \label{eqn:CauchyIdentitiesII}   
\Delta(\x; \y) & =  &  \displaystyle{\prod_{i,j \geq 1}} 
 \frac{1-\overline{x}_{i} y_{j}}{1-x_{i} y_{j}}
=\sum_{\lambda  \in \mathcal{SP}}
P^{\L}_{\lambda} (\x) \,  \widehat{q}^\L_{\lambda} (\y).   \medskip 
\end{array}  
\end{equation} 
\end{defn}
\noindent
By the Definition \ref{df:Definitionwh{p}^Lwh{q}^L} and Proposition \ref{prop:Delta(x_1,...,x_n;y)} (2), 
we see that for a strict partition $\lambda$ of length $r$,
\begin{equation*} 
   \widehat{p}^{\L}_{\lambda} (\y)  \in \Gamma_{\L}^{(r)}(\y)  \subset \Gamma_{\L}(\y) 
    \quad  \text{and}  \quad 
   \widehat{q}^{\L}_{\lambda} (\y)  \in \Gamma_{\L}^{(r),+}(\y)  \subset \Gamma_{\L}^{+}(\y).
\end{equation*}

\begin{rem} 
\quad 
\begin{enumerate} 
\item By using the Basis Theorem $($Theorem $\ref{thm:BasisTheoremP^L(x_n)Q^L(x_n)})$, 
we can also  define $\wh{p}^{\L}_{\lambda} (\y)$ and $\widehat{q}^{\L}_\lambda(\y)$
for strict partitions $\lambda \in \mathcal{SP}_{n}$ by the following identities$:$ 
\begin{equation*} 
\begin{array}{rllll} 
\Delta(\x_{n};\y)  & =  &   \displaystyle{\prod_{i=1}^{n}} 
 \prod_{j= 1}^{\infty} 
      \frac{1-\overline{x}_{i} y_{j}}{1- x_{i} y_{j}}
=  \sum_{\lambda  \in \mathcal{SP}_{n}}
Q^{\L}_{\lambda} (\x_{n}) \, \widehat{p}^\L_{\lambda} (\y),   \medskip \\
\Delta(\x_{n}; \y)  & =  &  \displaystyle{\prod_{i = 1}^{n}} 
 \prod_{j = 1}^{\infty} 
 \frac{1-\overline{x}_{i} y_{j}}{1-x_{i} y_{j}}
=\sum_{\lambda  \in \mathcal{SP}_{n}}
P^{\L}_{\lambda} (\x_{n}) \,  \widehat{q}^\L_{\lambda} (\y).  \medskip 
\end{array}  
\end{equation*} 
Then by the stability property of $P^{\L}_{\lambda}(\x_{n})$ and $Q^{\L}_{\lambda}(\x_{n})$ 
$($see Proposition $\ref{prop:Stability})$,  we see that  the definition of $\wh{p}^{\L}_{\lambda}(\y)$ 
and $\wh{q}^{\L}_{\lambda}(\y)$ does not depend on $n$. 
In particular, in view of $(\ref{eqn:Delta(x_1;y)})$, we see that $\wh{p}^{\L}_{(k)}(\y)$ 
and $\wh{q}^{\L}_{(k)}(\y)$ corresponding to the  one rows $(k) \; (k = 1, 2, \ldots)$ 
agree with $\wh{p}^{\L}_{k}(\y)$ and $\wh{q}^{\L}_{k}(\y)$ defined in \S $\ref{subsec:OneRowCase}$ 
respectively. 

\item  By using $\Delta (\x; \y)$ and 
   the Basis Theorem $($Corollary $\ref{cor:BasisTheoremP^L(x|b)Q^L(x|b)})$, 
       one can formally  define the {\it dual universal factorial Schur $P$- and $Q$-functions} 
       $\wh{p}^{\L}_{\lambda}(\y|\b)$, $\wh{q}^{\L}_{\lambda}(\y|\b)$. 
       However, in order to make this definition valid, we have to remove some 
       technical difficulties.   We hope to return to this problem elsewhere.  
\end{enumerate} 
\end{rem} 
\noindent
The functions  $\widehat{p}^\L_\lambda(\y)$  
(resp.  $\widehat{q}^\L_\lambda(\y)$)  are  elements of $\Gamma_{\L}(\y)  \subset \Lambda_{\L} (\y)$ (resp. $\Gamma_{\L}^{+}(\y) \subset  \Lambda_{\L} (\y)$),  
and hence 
symmetric functions   of total  degree $|\lambda|$. 
If we put $a_{i, j} = 0$ for all $i, j \geq 1$, $\wh{p}^{\L}_{\lambda}(\y)$ (resp. $\wh{q}^{\L}_{\lambda}(\y)$) reduce to the usual Schur $P$-functions $P_{\lambda} (\y)$
(resp.  $Q$-functions $Q_{\lambda}(\y)$).  Therefore we have 
\begin{equation}  \label{eqn:hat{p}=P+lower}   
\begin{array}{llll} 
  \wh{p}^{\L}_{\lambda}(\y)  & = P_{\lambda}(\y)  +  \text{lower order terms in } \; \y, \medskip \\
  \wh{q}^{\L}_{\lambda}(\y)  & = Q_{\lambda}(\y)  +  \text{lower order terms in } \; \y. \medskip 
\end{array}  
\end{equation} 
Here is some examples of these functions (see also Examples \ref{ex:wh{q}^E_k(k=1,2,3)} and 
\ref{ex:wh{p}^E_k(k=1,2,3)}): 
\begin{ex}   \label{ex:p^L_1(y|b)}  
\begin{equation*} 
\begin{array}{rlll} 
%   \wh{p}^{\L}_{1}(\y|\b) & =   P_{1}(\y)  + \{ a_{1, 1}h_{1}(\y) - P_{2}(\y)\} b_{1} 
%+ \{ -a_{1, 2}h_{1}(\y)  - a_{1, 1} (h_{1}(\y)^2 - h_{2}(\y))  + P_{3}(\y) \} b_{1}^{2}  \medskip \\
% & +  \{ (2a_{1, 3} + a_{1, 2} a_{1, 1} )h_{1}(\y)   + a_{1, 2} h_{1}(\y)^{2} - a_{1, 1} %(2h_{1}(\y) h_{2}(\y) - h_{1}(\y)^{3})  - P_{4} (\y) \}  b_{1}^{3} +  \cdots.    \medskip  \\
%\end{array} 
%\end{equation*} 
%\item 
%\begin{equation*} 
%\begin{array}{llll} 
  \widehat{p}^{\L}_{1} (\y) & =   P_{1}(\y),  \medskip \\
  \widehat{p}^{\L}_{2} (\y) & =   P_{2}(\y) - a_{1,1} h_{1}(\y), \medskip \\
  \widehat{p}^{\L}_{3} (\y) & =   P_{3}(\y) + a_{1,1}h_{2} (\y) - 2 a_{1,1} {h_{1}(\y)}^{2}  +(a_{1,1}^2-a_{1,2})h_1(\y). \medskip \\%
  \widehat{q}^{\L}_{1}(\y)  & =   Q_{1}(\y),  \medskip \\
  \widehat{q}^{\L}_{2} (\y) & =   Q_{2}(\y)-a_{1,1}h_{1}(\y),  \medskip \\
  \widehat{q}^{\L}_{3} (\y) & =   Q_{3}(\y) + 2 a_{1,1} h_{2}(\y) -3 a_{1,1} h_{1}(\y)^2  
 +a_{1,1}^{2} h_{1}(\y). \medskip 
\end{array}
\end{equation*}  
\end{ex}  
\noindent

%%%%%%%%%%%%%%%%%%%%%%%%%%%%%%%%%%%%%%%%%%%%%%%%%%%%%%%%%%%%%%%%%%%%%%%%%%%%%%%%%%%%%%%%%%%%%%%%%%%
\subsection{Basis Theorem for $\wh{p}^{\L}_{\lambda}(\y)$, $\wh{q}^{\L}_{\lambda}(\y)$} 
As expected, the functions  $\{ \wh{p}^{\L}_{\lambda}  (\y) \}_{\lambda \in \mathcal{SP}}$
(resp. $\{  \wh{q}^{\L}_{\lambda} (\y) \}_{\lambda  \in \mathcal{SP}}$) 
constitute an   $\L$-basis for $\Gamma_{\L}(\y)$ (resp. $\Gamma_{\L}^{+}(\y)$). 
\begin{theorem} [Basis Theorem]  \label{thm:BasisTheoremwh{p}^L(y)wh{q}^L(y)}  
\quad   
\begin{enumerate}
\item  
$\left\{\widehat{p}^\L_\lambda(\y)\right\}_{\lambda\in \mathcal{SP}}$
are linearly independent  over $\L$ and form an $\L$-basis
of $\Gamma_{\L}(\y)$.

\item  
$\left\{\widehat{q}^\L_\lambda(\y)\right\}_{\lambda\in \mathcal{SP}}$ 
are linearly independent  over $\L$ and form an $\L$-basis
of $\Gamma_\L^{+} (\y)$.

\end{enumerate}
\end{theorem}
\begin{proof} 
We  only prove  the assertion (1).  (2) can be proved similarly.  
From (\ref{eqn:hat{p}=P+lower}) and  the similar argument as in \S \ref{subsec:DefinitionP^L(x|b)Q^L(x|b)}, 
we see that 
$\wh{p}^{\L}_{\lambda}(\y)$'s,  
$\lambda \in \mathcal{SP}$, are linearly independent over $\L$. 
In particular, an $\L$-submodule $\tilde{\Gamma}_{\L}(\y)$
 of $\Gamma_{\L}(\y)$ generated by $\wh{p}^{\L}_{\lambda}(\y)$'s, 
$\lambda \in \mathcal{SP}$,  is a direct sum: $\tilde{\Gamma}_{\L}(\y) = \bigoplus_{\lambda \in \mathcal{SP}} \L \, \wh{p}^{\L}_{\lambda}(\y)   \subset \Gamma_{\L}(\y)$.  
We wish to show that $\tilde{\Gamma}_{\L}(\y)$ agrees with $\Gamma_{\L}(\y)$, or equivalently, 
the set $\{ \wh{p}^{\L}_{\lambda}(\y)  \}_{\lambda \in \mathcal{SP}}$ 
%(resp.  $\{ \wh{q}^{\L}_{\lambda}(\y) \}_{\lambda \in \mathcal{SP}}$) 
spans $\Gamma_{\L} (\y)$ %(resp. $\Gamma_{\L}^{+} (\y)$) 
over $\L$.  To this end,  it is sufficient to show that monomials 
$\wh{p}^{\L}_{k_{1}}(\y)  \wh{p}^{\L}_{k_{2}} (\y)  \cdots \wh{p}^{\L}_{k_{r}} (\y) \; 
(k_{1} \geq 1, \ldots, k_{r} \geq 1, \; r \geq 1)$ are expressed 
as $\L$-linear combinations of $\wh{p}^{\L}_{\lambda} (\y)$'s, $\lambda \in \mathcal{SP}$, 
since $\Gamma_{\L}(\y)$ %(resp. $\Gamma_{\L}^{+}(\y)$) 
is generated by 
$\wh{p}^{\L}_{k}(\y) \; (k = 1, 2, \ldots)$ %(resp.  $\wh{q}^{\L}_{k}(\y) \; (k = 1, 2, \ldots)$) 
as an $\L$-algebra.  
We prove this by making use of the {\it  Hopf algebra property} %\footnote{
%See Hoffmann-Humphreys \cite[p.95]{Hof-Hum92}.  
%}
 of $\Gamma^{\L}_{+}(\x)$. Let $\x = (x_{1}, x_{2}, \ldots)$ and $\x' = (x'_{1}, x'_{2}, \ldots)$ 
be two countable sets of variables.  Then the function $Q^{\L}_{\lambda}(\x, \x')$ can be 
written in terms of the two sets separately, that is, 
\begin{equation}  \label{eqn:CoproductQ^L(x)}  
   Q^{\L}_{\lambda} (\x, \x') = \sum_{\mu, \,  \nu \in \mathcal{SP}}  
  \wh{c}_{\mu,  \, \nu}^{\lambda}  \, Q^{\L}_{\mu}(\x)  \,  Q^{\L}_{\nu}  (\x')   \quad 
  (\wh{c}_{\mu, \,  \nu}^{\lambda}  \in \L). 
\end{equation} 
This gives the  coproduct (comultiplication, diagonal map)
$\phi:  \Gamma^{\L}_{+}(\x) \longrightarrow \Gamma^{\L}_{+}(\x) \otimes_{\L}  \Gamma^{\L}_{+}(\x)$. 
%\footnote{
%cf. Macdonald \cite[I,  \S 5, Examples 25]{Mac95}, Hoffman-Humphreys \cite[p.94]{Hof-Hum92}
%}. 
Note that by the Basis Theorem for $Q^{\L}_{\lambda}(\x)$'s (Corollary \ref{cor:BasisTheoremP^L(x)Q^L(x)}), the coefficients $\wh{c}_{\mu, \, \nu}^{\lambda} \in \L$ 
are uniquely determined in the above expression. 
On the other hand, we have by Definition \ref{df:Definitionwh{p}^Lwh{q}^L}, 
\begin{equation*} 
\begin{array}{llll} 
   \Delta (\x; \y) & = \displaystyle{\prod_{i, j \geq 1}}    \dfrac{1 - \overline{x}_{i}y_{j}} {1 - x_{i} y_{j}} 
                     = \sum_{\mu \in \mathcal{SP}}  Q^{\L}_{\mu}(\x) \, \wh{p}^{\L}_{\mu}(\y), \medskip \\
    \Delta (\x'; \y) & = \displaystyle{\prod_{i, j \geq 1}}    \dfrac{1 - \overline{x'_{i}} y_{j}} {1 - x'_{i} y_{j}} 
                     = \sum_{\nu \in \mathcal{SP}}  Q^{\L}_{\nu}(\x') \, \wh{p}^{\L}_{\nu}(\y).    \medskip             
\end{array}
\end{equation*}  
Multiplying these two expressions, we have 
\begin{equation*} 
  \sum_{\lambda  \in \mathcal{SP}}  Q^{\L}_{\lambda}(\x, \x') \,  \wh{p}^{\L}_{\lambda}(\y) 
=  \Delta (\x, \x'; \y)  =  \sum_{\mu,  \, \nu \in \mathcal{SP}}  
                            \wh{p}^{\L}_{\mu} (\y) \wh{p}^{\L}_{\nu}(\y)  Q^{\L}_{\mu} (\x) Q^{\L}_{\nu} (\x'). 
\end{equation*} 
By (\ref{eqn:CoproductQ^L(x)}), the left-hand side turns into 
\begin{equation*} 
  \sum_{\mu, \,   \nu \in \mathcal{SP}} 
\left ( \sum_{\lambda \in \mathcal{SP}} \wh{c}_{\mu, \, \nu}^{\lambda}   \wh{p}^{\L}_{\lambda}(\y)  \right ) Q^{\L}_{\mu} (\x)  Q^{\L}_{\nu} (\x'), 
\end{equation*} 
and therefore we have the following product formula (we used the Basis Theorem for $Q^{\L}_{\lambda}(\x)$'s again): 
\begin{equation}  \label{eqn:Productwh{p}^L(y)}   
  \wh{p}^{\L}_{\mu} (\y)  \wh{p}^{\L}_{\nu} (\y) 
= \sum_{\lambda \in \mathcal{SP}}   \wh{c}_{\mu, \,   \nu}^{\lambda}  \,  \wh{p}^{\L}_{\lambda}(\y). 
\end{equation} 
Notice that by our convention of the grading of $\L = \L_{*}$, the right-hand side 
is necessarily a finite sum, and hence is contained in an $\L$-submodule 
$\tilde{\Gamma}_{\L}(\y) = \bigoplus_{\lambda \in \mathcal{SP}}  \L \, \wh{p}^{\L}_{\lambda}(\y)$. 
Thus $\tilde{\Gamma}_{\L}(\y)$ is closed under  multiplication.  
Especially, in the case $\mu = (k)$ and $\nu = (l)$, i.e., one rows,  the product $\wh{p}^{\L}_{k}(\y) \wh{p}^{\L}_{l}(\y)$ is contained in $\tilde{\Gamma}_{\L}(\y)$.  
Iterating use of the product formula (\ref{eqn:Productwh{p}^L(y)}) yields  the 
required result.
\end{proof}

%%%%%%%%%%%%%%%%%%%%%%%%%%%%%%%%%%%%%%%%%%%%%%%%%%%%%%%%%%%%%%%%%%%%%%%%%%%%%%%%%%%%%%%%%%%%%
\subsection{Hopf algebra structure}
As mentioned in  \S \ref{subsec:RingSymmetricFunctions}, 
the ring of symmetric functions $\Lambda$ has a structure of a self-dual 
Hopf algebra over $\Z$.  Also  its subalgebras $\Gamma$ and $\Gamma'$ 
have  Hopf algebra structures over $\Z$  
which are dual to each other.  
In this subsection, we shall mention the Hopf algebra structure 
of our rings $\Gamma^{\L}(\x)$, $\Gamma^{\L}_{+}(\x)$ (see \S \ref{subsec:StabilityProperty}) and 
$\Gamma_{\L}(\y)$, $\Gamma_{\L}^{+}(\y)$ (see \S \ref{subsec:Definitionwh{p}^L(y)wh{q}^L(y)}).  
As we saw in the proof of the Basis Theorem (Theorem \ref{thm:BasisTheoremwh{p}^L(y)wh{q}^L(y)}), 
through the Cauchy identity (\ref{eqn:CauchyIdentities}), the coproduct formula (\ref{eqn:CoproductQ^L(x)}) 
of $Q^{\L}_{\lambda}(\x)$ determines the product formula (\ref{eqn:Productwh{p}^L(y)}) of 
$\wh{p}^{\L}_{\lambda}(\y)$'s.  
By an  analogous argument,  one sees easily that the product formula 
\begin{equation*}  
   Q^{\L}_{\mu}(\x)  \,  Q^{\L}_{\nu} (\x)  
  = \sum_{\lambda \in \mathcal{SP}}  c_{\mu, \,  \nu}^{\lambda} \,   Q^{\L}_{\lambda}(\x)  \quad  (c_{\mu, \,  \nu}^{\lambda}  \in \L) 
\end{equation*} 
determines the coproduct formula 
\begin{equation*} 
   \wh{p}^{\L}_{\lambda} (\y, \y')  = \sum_{\mu, \,   \nu \in \mathcal{SP}} c_{\mu, \,  \nu}^{\lambda}  \, \wh{p}^{\L}_{\mu}(\y)  \,  \wh{p}^{\L}_{\nu}(\y'), 
\end{equation*} 
where $\y = (y_{1}, y_{2}, \ldots)$ and $\y'  = (y_{1}', y_{2}', \ldots)$ are two 
countable sets of variables.   
Thus we obtain the following result as a 
 formal consequence of the Cauchy identities (\ref{eqn:CauchyIdentities}). 
Here we   write $\phi$ for the coproduct map.

\begin{prop}  [Duality]  \label{prop:Duality}  
\quad 
\begin{enumerate}
\item  $($Duality between $\Gamma^{\L}_{+}(\x)$ and $\Gamma_{\L}(\y) )$  \\
If  \, $Q^{\L}_{\lambda} (\x) \,  Q^{\L}_{\mu} (\x)
= \displaystyle{\sum_{\nu  \in \mathcal{SP}}}  c_{\lambda, \,  \mu}^\nu  \,  Q^{\L}_{\nu}(\x)$, then
\begin{equation*} 
   \phi(\widehat{p}^{\L}_{\nu} (\y))=
\sum_{\lambda, \,   \mu \in \mathcal{SP}}c_{\lambda,  \, \mu}^{\nu}  \;
\widehat{p}^{\L}_{\lambda} (\y)  \otimes  \widehat{p}^{\L}_{\mu} (\y).
\end{equation*}  
If %\begin{equation*} 
   $\phi(Q^{\L}_{\nu} (\x))= \displaystyle{\sum_{\lambda, \,  \mu  \in \mathcal{SP}}}  
\widehat{c}_{\lambda, \, \mu}^{\nu}  \, 
Q^{\L}_{\lambda} (\x)  \otimes Q^{\L}_{\mu} (\x)$, then 
%\end{equation*} 
\begin{equation*} 
  \widehat{p}^{\L}_{\lambda} (\y)  \, \widehat{p}^{\L}_{\mu}(\y)=\displaystyle
\sum_{\nu \in \mathcal{SP}} 
\widehat{c}_{\lambda,  \,  \mu}^\nu \,  \widehat{p}^{\L}_{\nu} (\y). 
\end{equation*}

\item $($Duality between $\Gamma^{\L}(\x)$ and $\Gamma_{\L}^{+}(\y) )$  \\
If \,   $P^{\L}_{\lambda} (\x)  \, P^{\L}_{\mu} (\x)=\displaystyle
\sum_{\nu  \in \mathcal{SP}} d_{\lambda, \,   \mu}^\nu \,  P^{\L}_{\nu} (\x)$, then
\begin{equation*} 
  \phi(\widehat{q}^{\L}_{\nu} (\y))=
\sum_{\lambda, \,  \mu \in \mathcal{SP}} d_{\lambda, \, \mu}^{\nu}  \, 
\widehat{q}^{\L}_{\lambda} (\y)  \otimes \widehat{q}^{\L}_{\mu} (\y).
\end{equation*}  
If %\begin{equation*} 
$\phi(P^{\L}_{\nu} (\x))=
\displaystyle{\sum_{\lambda, \,  \mu  \in \mathcal{SP}}} 
\widehat{d}_{\lambda, \,  \mu}^{\nu}   \,  
P^{\L}_{\lambda} (\x) \otimes P^{\L}_{\mu} (\x)$,  then 
%\end{equation*} 
\begin{equation*} 
   \widehat{q}^{\L}_{\lambda} (\y) \widehat{q}^\L_{\mu} (\y)=\displaystyle
\sum_{\nu\in \mathcal{SP}} 
\widehat{d}_{\lambda,  \, \mu}^\nu \,  \widehat{q}^{\L}_{\nu} (\y). 
\end{equation*} 
\end{enumerate}
\end{prop}

%\begin{rem} 
By the Basis Theorems $($Corollary $\ref{cor:BasisTheoremP^L(x)Q^L(x)}$ and  Theorem $\ref{thm:BasisTheoremwh{p}^L(y)wh{q}^L(y)})$,  
we can define the $\L$-bilinear pairing between the rings $\Gamma^{\L}_{+} (\x)$ and $\Gamma_{\L}(\y)$, 
\begin{equation*} 
  [ -, -]:  \Gamma^{\L}_{+}(\x)  \times \Gamma_{\L}(\y) \longrightarrow \L, 
\end{equation*} 
by setting 
%\begin{equation*} 
     $[ Q^{\L}_{\lambda} (\x), \wh{p}^{\L}_{\mu}(\y) ]  = \delta_{\lambda, \,  \mu}$. 
%\end{equation*} 
This pairing induces an $\L$-module homomorphism  
\begin{equation*} 
  \kappa:  \Gamma^{\L}_{+}(\x)  \longrightarrow \Hom_{\L} (\Gamma_{\L}(\y), \L), \quad 
  F \longmapsto  \kappa (F) = [F, -].  
\end{equation*} 
Using the Basis Theorems (Corollary \ref{cor:BasisTheoremP^L(x)Q^L(x)} and 
Theorem \ref{thm:BasisTheoremwh{p}^L(y)wh{q}^L(y)}) again, one can show that 
$\kappa$ is an isomorphism of  $\L$-modules.  Furthermore, both $\Gamma^{\L}_{+}(\x)$ and 
$\Hom_{\L} (\Gamma_{\L}(\y), \L)$ have a Hopf algebra structure over $\L$. 
Using the above duality (Proposition \ref{prop:Duality}), one can also show that 
$\kappa$ is actually an isomorphism of Hopf algebras over $\L$.  
As we remarked earlier,  %(Remark \ref{rem:Gamma_L(y)=MU_*(OmegaSp)}), 
we know that $\Gamma_{\L}(\y) \cong \Gamma_{*}^{MU}  \cong MU_{*}(\Omega Sp)$, 
and therefore $\Hom_{\L} (\Gamma_{\L}(\y), \L) \cong MU^{*}(\Omega Sp)$ as 
Hopf algebras over $\L$. Thus we have the following isomorphism of Hopf algebras over $\L$:  
\begin{equation*} 
  \Gamma^{\L}_{+}(\x) \cong \Gamma^{*}_{MU} \cong MU^{*}(\Omega Sp). 
\end{equation*} 
Similarly we define the pairing  
\begin{equation*}  
  [-, -]:  \Gamma^{\L}(\x)  \times \Gamma_{\L}^{+}  (\y) \longrightarrow \L, 
\end{equation*} 
by setting $[P^{\L}_{\lambda}(\x),  \wh{q}^{\L}_{\mu}(\y)] = \delta_{\lambda, \, \mu}$.  
By means of this pairing and the same argument as above, we have the following isomorphism 
of Hopf algebras over $\L$: 
\begin{equation*} 
   \Gamma^{\L}(\x)  \cong  {\Gamma'}^{*}_{MU}  \cong MU^{*}(\Omega_{0} SO). 
\end{equation*} 
Summing up the results so far, we have the following: 
\begin{theorem}   \label{thm:MU_*(OmegaSp)MU_*(Omega_0SO)MU^*(OmegaSp)MU^*(Omega_0SO)} 
There is a symmetric function realization  as Hopf algebras over $\L_{*}:$ 
\begin{equation*} 
\begin{array}{rllll} 
  MU_{*}(\Omega Sp)     
& \cong   & {\Gamma'}_{*}^{MU} \cong  \Gamma_{\L}(\y) 
= \L[\wh{p}^{\L}_{1} (\y), \wh{p}^{\L}_{3}(\y), \ldots, \wh{p}^{\L}_{2i-1}(\y), \ldots]   \medskip \\
& = &  \displaystyle{\bigoplus_{\lambda \in \mathcal{SP}}}  \L \, \wh{p}^{\L}_{\lambda}(\y) \;  \subset \;  \Lambda_{\L}(\y),   \medskip \\
  MU_{*}(\Omega_{0} SO) 
& \cong  & \Gamma_{*}^{MU} \cong   \Gamma_{\L}^{+}(\y)  = \L[\wh{q}^{\L}_{1}(\y), \wh{q}^{\L}_{2}(\y), \ldots, \wh{q}^{\L}_{i}(\y), \ldots]/(\wh{q}^{\L}(T) \wh{q}^{\L}(\overline{T}) = 1)  \medskip \\
& = &  \displaystyle{\bigoplus_{\lambda \in \mathcal{SP}}}   \L \, \wh{q}^{\L}_{\lambda}(\y)  \; 
\subset \;  \Lambda_{\L}(\y).  \medskip 
\end{array}  
\end{equation*} 
Dually, there is a symmetric function realization as Hopf algebras over $\L^{*}$ %\footnote{
%In \S \ref{subsec:E-cohomologySchurP,Q-functions}, we introduced the 
%functions $\tilde{q}^{MU}_{k}(\x)  \in \Gamma^{*}_{MU} \; (k = 1, 2, \ldots)$, 
%and described the ring $\Gamma^{*}_{MU}$ in terms of these functions.    
%At present, we have not been able to obtain 
%a general formula relating the functions $\tilde{q}^{MU}_{k}(\x) \; (k = 1, 2, \ldots)$ 
%and $Q^{\L}_{\lambda}(\x) \; (\lambda \in \mathcal{SP})$ under the 
%isomorphism $\Gamma^{*}_{MU} \cong \Gamma^{\L}_{+}(\x)$.  
%}
$:$ 
\begin{equation*} 
\begin{array}{rllll} 
  MU^{*}(\Omega Sp)   &  \cong & \Gamma^{*}_{MU} \cong  \Gamma^{\L}_{+}(\x)  = \displaystyle{\prod_{\lambda \in \mathcal{SP}}}  \L \, Q^{\L}_{\lambda} (\x)  \; \subset \; \Lambda^{\L}(\x), \medskip \\
  MU^{*}(\Omega_{0} SO) & \cong &  {\Gamma'}^{*}_{MU} \cong \Gamma^{\L} (\x)  = \displaystyle{\prod_{\lambda \in \mathcal{SP}}}  \L \, P^{\L}_{\lambda} (\x) \; \subset \; \Lambda^{\L}(\x).   \medskip 
\end{array}
\end{equation*} 
\end{theorem}

%\begin{rem} 
\subsection{Concluding remarks}
\begin{enumerate} 
\item In \S $\ref{subsec:E-cohomologySchurP,Q-functions}$, we introduced the 
functions $\tilde{q}^{MU}_{k}(\x)  \in \Gamma^{*}_{MU} \; (k = 1, 2, \ldots)$, 
and described the ring $\Gamma^{*}_{MU}$ in terms of these functions.    
At present, we have not been able to obtain 
a general formula relating the functions $\tilde{q}^{MU}_{k}(\x) \; (k = 1, 2, \ldots)$ 
and $Q^{\L}_{\lambda}(\x) \; (\lambda \in \mathcal{SP})$ under the 
isomorphism $\Gamma^{*}_{MU} \cong \Gamma^{\L}_{+}(\x)$.

\item  By the result of Buch \cite{Buc2002}, we  observe  that 
%For the (non-equivariant) $K$-theory of the usual complex Grassmannians, 
% Buch \cite{Buc2002} showed that 
the {\it stable Grothendieck  polynomials} $G_{\lambda}(\x)$, $\lambda \in \mathcal{P}$, 
represent the Schubert classes of the $K$-theory (more precisely, $K$-cohomology) 
of the infinite Grassmannain $BU \simeq \Omega SU$.   
Dual functions  $g_{\lambda}(\y)$, $\lambda \in \mathcal{P}$, called the {\it dual stable Grothendieck 
polynomials},  are defined  and studied by Lam-Pylyavskyy \cite[\S 9.1]{Lam-Pyl2007}. 
By the construction, these dual functions $\{ g_{\lambda} (\y) \}$ 
represent the Schubert classes of the $K$-homology 
$K_{*}(\Omega SU)$.  
There is a ``type $C$'' analogue  of the above story: 
 If we specialize $a_{1, 1}  = \beta$, $a_{i, j} = 0$ for all $(i, j) \neq (1, 1)$
in Definition $\ref{df:Definitionwh{p}^Lwh{q}^L}$, the resulting functions 
will be denoted by  $gp_{\lambda}(\y)$ and $gq_{\lambda}(\y)$, $\lambda \in \mathcal{SP}$. 
By definition, these functions are dual to the $K$-theoretic factorial Schur $P$- and $Q$-functions 
$GP_{\lambda} (\x)$ and $GQ_{\lambda}(\x)$ due to Ikeda-Naruse \cite[Definition 2.1 and \S 6.1]{Ike-Nar2013}. As shown in that paper, the functions $\{ GP_{\lambda}(\x) \}$
and $\{ GQ_{\lambda} (\x) \}$  represent the Schubert classes for 
the %torus equivariant 
$K$-cohomology of infinite 
maximal isotropic Grassmannians $Sp/U \simeq \Omega Sp$ and $SO/U \simeq \Omega_{0} SO$
$($see  Ikeda-Naruse  \cite[Theorem 8.3, Corollary 8.1]{Ike-Nar2013}$)$. 
Therefore  
the functions $\{ gp_{\lambda} (\y) \}$ and $\{ gq_{\lambda} (\y) \}$ 
are expected to represent the Schubert classes for the $($non-equivariant$)$ 
$K$-homology  $K_{*}(\Omega Sp)$ and $K_{*}(\Omega_{0} SO)$ of the loop 
spaces $\Omega Sp$ and $\Omega_{0} SO$. 
We return to this problem  elsewhere.    
\end{enumerate} 
%\end{rem} 

%\newpage
%%%%%%%%%%%%%%%%%%%%%%%%%%%%%%%%%%%%%%%%%%%%%%%%%%%%%%%%%%%%%%%%%%%%%%%%%%%%%%%%%%%%%%%%%%%%%%%%%%%%%%%%%%%%%%%%
\section{Appendix}
\subsection{Universal factorial Schur functions}   \label{subsec:UFSF2}  
In \S \ref{subsec:UFSF}, we defined the {\it universal factorial Schur functions} 
in the $n$-variables 
$s^{\L}_{\lambda}(\x_{n}|\b)$, $\lambda  \in \mathcal{P}_{n}$. 
In order to consider the limit function as $n \rightarrow \infty$, 
we need to modify its definition.  In fact,  we are able to generalize the 
{\it double Schur function}   $s_{\lambda} (x || a)$  and the 
{\it dual Schur function} $\wh{s}_{\lambda}(y||a)$ due to  Molev  \cite[\S 2.1, \S 3.1]{Mol2009} 
in the universal setting. 
In this appendix, we only exhibit the definition and basic property of our functions. Details will be  discussed elsewhere.  
Here we use a doubly infinite sequence  
$\b_{\Z}  = %\b_{\pm}= 
(\ldots,b_{-2},b_{-1},b_0,b_{1},b_{2},\ldots)$  in place of 
$\b = (b_{1}, b_{2}, \ldots)$.
%Using Lazard ring $\L$ it is possible to extend Molev's dual Schur functions
%$\widehat{s}_\lambda(\x||{\bf a})$
%\cite{Mol}
%to universal setting  as follows.
%(A geometric meaning of Molev's dual Schur function is explained in \cite{Lam-Shi}.) 
%(e.g. root system of type $A_\infty$).
First we introduce another variant of the ordinary $k$-th power $t^{k}$ (see \S \ref{subsec:DefinitionP^L(x|b)Q^L(x|b)} and Molev \cite[p.7]{Mol2009}). For a fixed positive integer $n$, we define  
\begin{equation*} 
  [t \,|| \, \b_{\Z}]_{n}^{k} :=  \prod_{i=1}^{k} (t \pf b_{n + 1 - i}) 
    = (t \pf b_{n}) (t \pf b_{n-1}) \cdots (t \pf b_{n + 1 - k}), 
\end{equation*} 
where we set $[t \, || \, \b_{\Z}]_{n}^{0} := 1$.  For a partition $\lambda = (\lambda_{1}, \lambda_{2}, \ldots, \lambda_{r})$, we set 
\begin{equation*} 
  [\x \, || \, \b_{\Z}]_{n}^{\lambda} 
 :=  \prod_{i=1}^{r}  [x_{i} \, || \, \b_{\Z}]_{n}^{\lambda_{i}}. 
\end{equation*} 
 
\begin{defn}  [Universal factorial Schur functions]    \label{df:Definitions^L(x||b)}  
For a partition 
$\lambda  = (\lambda_{1}, \lambda_{2}, \ldots, \lambda_{n})  
%=(\lambda_{1} \geq  \lambda_{2}  \geq  \cdots \geq \lambda_{n} \geq 0)  
\in \mathcal{P}_{n}$,
we define
\begin{equation*} 
   s^{\L}_{\lambda}(\x_{n} \, || \, \b_{\Z}) = 
   s^{\L}_{\lambda}(x_{1}, \ldots, x_{n} \, || \, \b_{\Z}) 
:=\displaystyle{
\sum_{w  \in S_{n}}}  
 w   \left( 
      \dfrac{[\x \, || \, \b_{\Z}]_{n}^{\lambda + \rho_{n-1}} }
     %\dfrac{(x_{1} \: || \: \b_{\pm})^{\lambda_{1}+{n-1}}_{n}
     %             (x_{2}  \: || \: \b_{\pm})^{\lambda_{2}+{n-2}}_{n}
     %              \cdots
     %             (x_{n}  \: || \: \b_{\pm})^{\lambda_{n}}_{n}}
            {  \prod_{1\leq i < j \leq n}  (x_i  \pf   \overline{x}_j)}
    \right),
\end{equation*} 
where $\rho_{n-1} = (n-1, n-2, \ldots, 1, 0)$. 
\end{defn} 
\noindent
This is a symmetric formal power series with coefficients in $\L$ in the variables 
$x_{1}, \ldots, x_{n}$ and $b_{n}, b_{n-1}, \ldots, b_{2 - \lambda_{1}}$.  
%$(t \:|| \:   \b_{\pm})^{k}_{n}:=  \displaystyle\prod_{i=1}^{k} (t   \pf b_{n+1-i})$.
One sees immediately from the Definitions   \ref{df:Definitions^L(x|b)} and \ref{df:Definitions^L(x||b)}
that $s^{\L}_{\lambda}(\x_{n} |\b)$  turns into  $s^{\L}_{\lambda}( \x_{n} \, || \, \b_{\Z})$ 
by changing the parameters $b_{i} \rightarrow b_{n - i + 1} \; (i = 1, 2, \ldots)$, and 
vice versa.

As with the case of $s^{\L}_{\lambda}(\x_{n}|\b)$ (see  Proposition \ref{prop:VanishingPropertys^L(x|b)}),  the functions $s^{\L}_{\lambda}(\x_{n} \, || \, \b_{\Z})$ also  satisfy the vanishing property. 
Given a partition $\mu  \in \mathcal{P}_{n}$, we introduce the sequence 
$\overline{\b}_{I - \mu}$ by  
\begin{equation*} 
 \overline{\b}_{I - \mu} :=(\overline{b}_{1-\mu_{1}},  \overline{b}_{2- \mu_{2}},  \ldots, 
                            \overline{b}_{i - \mu_{i}}, \ldots).
\end{equation*} 

\begin{prop}  [Vanishing  Property] %\footnote{
%cf. Molev \cite[p.14]{Mol2009}.  
%}  
Let $\lambda, \mu \in \mathcal{P}_{n}$. Then we have 
\begin{equation*} 
  s^{\L}_{\lambda} (\overline{\b}_{I  - \mu} \, || \, \b_{\Z}) 
  = \left \{ 
    \begin{array}{llllll} 
          &  0           \quad   &   \text{if} \quad   \mu  \not\supset \lambda, \medskip \\
          &  \displaystyle{\prod_{(i, j) \in \lambda}}  
             (\overline{b}_{i-\lambda_{i}}  \pf b_{ {}^{t} \! \lambda_{j}-j+1} )  
    \quad 
                          &   \text{if} \quad  \mu = \lambda.    \medskip 
    \end{array}  
    \right. 
\end{equation*}
%where  ${}^{t} \!  \lambda$ is the diagram conjugate to $\lambda$.  

%\begin{itemize}
%\item[(1)]
%If $\mu\not\supset \lambda$ then
%\hspace{2.5cm}
%$
%s^\L_\lambda(\overline{b}_{I-\mu}||\:   \b_{\pm})=0.
%$
%\item[(2)]
%$$
%s^\L_\lambda(\overline{b}_{I-\lambda}||\:\b_{\pm})=
%\prod_{(i,j)\in \lambda}(\overline{b}_{i-\lambda_{i}}  \pf b_{\lambda'_{j}-j+1} ).
%$$
%
%\end{itemize}
\end{prop}
Since the functions $\{ s^{\L}_{\lambda} (\x_{n} \, || \, \b_{\Z})\}_{n \geq 1}$
have the stability property under the evaluation map $x_{n} = \overline{b}_{n}$, 
we can  take limit $n  \to \infty$ to  obtain the limit function 
$s^{\L}_{\lambda} (\x||\,  \b_{\Z})$.
%(specialization map is  the evaluation $x_n=\overline{b}_n$).

%%%%%%%%%%%%%%%dual%%%%%%%%%%%%%%%%%%%%%%%%%%%%%%%%%%%%%%%%%%%%%%%%%%%%%%%%%%%%%%%
Using the Cauchy identity and an analogous argument that we did in \S  \ref{sec:DUFSPQF}, 
we can define the dual functions $\wh{s}^{\L}_{\lambda}(\y || \b_{\Z})$ for  $\lambda \in \mathcal{P}_{n}$. 
These functions are a generalization of 
%it is possible to extend 
Molev's   dual Schur functions  
$\widehat{s}_\lambda(y||a)$
%(\cite[p.15]{Mol2009})
in  the universal setting.    
%We can also define (type $A$) universal factorial Schur function 
%$s^\L_\lambda(\x||\b)$ and its dual
%$\widehat{s}^\L_\lambda(\y||\b)$.  
Here we use  one more set of variables $\y = (y_{1}, y_{2}, \ldots)$.  
\begin{defn}  [Dual universal factorial Schur functions]  
For a partition $\lambda \in \mathcal{P}_{n}$, 
we define $\widehat{s}^{\L}_{\lambda} (\y \, || \, \b_{\Z})$ 
 by the following identity $($Cauchy identity$):$ 
\begin{equation*} 
\prod_{i=1}^{n}\prod_{j=1}^{\infty}\frac{1-\overline{b}_{i} y_{j}}{1- x_{i} y_{j}}
=
\sum_{\lambda\in \mathcal{P}_{n}}
s^\L_\lambda(\x_{n}  \, ||\, \b_{\Z}) \, \widehat{s}^{\L}_{\lambda} (\y \, || \, \b_{\Z}).
\end{equation*}  
%We can also take limit $n\to \infty$ to get $s^\L_\lambda(\x||\:\b_{\pm})$.
%(specialization map is  the evaluation $x_n=\overline{b}_n$).
\end{defn} 
\noindent
For a geometric meaning of Molev's double and dual  Schur functions, 
readers are referred to  Lam-Shimozono  \cite{Lam-Shi2013}.

%%%%%%%%%%%%%%%%%%%%%%%%%%%%%%%%%%%%%%%%%%%%%%%%%%%%%%%%%%%%%%%%%%%%%%%%%%%%%%%%%%
\subsection{Root data}     
\subsubsection{Type $A_{\infty}$}  \label{subsubsec:RootDeta(A)} 
%\vspace{0.3cm} 
\quad  \\
%\vspace{0.3cm} 
\underline{Weyl group of type $A_{\infty}$} \quad 
Let $S_{\infty}  := \displaystyle{\lim_{ \substack{\rightarrow \\ n} } } \,   S_{n}$ 
be the {\it infinite} 
symmetric group, that is,  the Weyl group of type $A_{\infty}$: 
$W(A_{\infty}) \cong S_{\infty}$. 
 An   element $w$ of $S_{\infty}$ is 
identified  with a permutation of the set of positive integers $\N := \{ 1, 2, \ldots \}$
such that $w(i) = i$ for all but a finite number of $i \in \N$.  
We shall use the {\it one-line notation}  $w =   w(1) \, w(2) \, \cdots$ 
to denote an element $w \in S_{\infty}$.   
The group $S_{\infty}$ is also a Coxeter group generated by the {\it simple 
reflections} $\{ s_{i} \}_{i \in I}$, where  the index set $I = \{ 1, 2, \ldots \}$. 
More explicitly, the simple reflections are given by 
the {\it simple transpositions}, i.e.,  $s_{i} = (i \; i+1) \; (i \geq 1)$.  
The defining relations of $\{ s_{i} \}_{i \in I}$ are given by 
\begin{equation*} 
\begin{array}{rlll} 
    s_{i}^{2}              & = &  1  \quad  & (i \geq 1), \medskip \\
    s_{i} s_{i + 1} s_{i}  & =  & s_{i + 1} s_{i} s_{i + 1}  \quad & (i \geq 1), \medskip \\
             s_{i}s_{j}    & =  & s_{j} s_{i}                \quad & (i, j \geq 1, \; 
             |i - j| \geq 2). \medskip 
\end{array} 
\end{equation*}

\vspace{0.3cm}  
%%%%%%%%%%%%%%%%%%%%%%%%%%%%%%%%%%%%%%%%%%%%%%%%%%%%%%%%%%%%%%%%%%%%%%%%%%%%%%%%
\underline{Root system of type $A_{\infty}$}   \quad 
We fix notation about the root system of type $A_{\infty}$. 
Let $L$ denote a free $\Z$-module with a basis $\{ t_{i} \}_{i \geq 1}$. 
For $j > i \geq 1$, we define the 
 {\it positive root} $\alpha_{j, i}$ to be $\alpha_{j, i} := t_{j} - t_{i}$, and 
 denote the set of  all positive roots  by $\Delta^{+}_{A}$ \footnote{
 In what follows, we  often drop the subscript $A$ for brevity when there is no fear of confusion. 
}, namely 
\begin{equation*}
  \Delta^{+}_{A} :=  \{ \alpha_{j, i} =  t_{j} -  t_{i} \; | \; j > i \geq 1 \}  \subset L. 
\end{equation*} 
The set of {\it negative roots} is defined by $\Delta^{-}_{A} := - \Delta^{+}_{A}
= \{ - \alpha \;  | \; \alpha \in \Delta^{+}_{A}  \}$. We also set 
$\Delta_{A} := \Delta^{+}_{A} \coprod \Delta^{-}_{A}$ and call it the {\it root system of type $A_{\infty}$}.  
The following elements of $\Delta^{+}_{A}$ are called the {\it simple roots}: 
\begin{equation*} 
  \alpha_{i} := \alpha_{i + 1, i} =  t_{i + 1} - t_{i} \quad (i \geq 1). 
\end{equation*} 
The Weyl group $S_{\infty}$ acts on the  lattice $L$ by the usual permutation on $t_{i}$'s, 
and hence on the root system $\Delta$.  The action of an element $w \in S_{\infty}$ 
on a root $\alpha \in \Delta$ will be denoted by $w \cdot \alpha$ or $w (\alpha)$ 
in the sequel. 
For a positive root $\alpha \in \Delta^{+}$, we have $\alpha = w (\alpha_{i})$ 
for some $i \in I$ and some $w \in S_{\infty}$. Then we set 
$s_{\alpha} = s_{w(\alpha_{i})} = w s_{i} w^{-1}$ 
(thus for a simple root $\alpha_{i}$ itself, we have $s_{\alpha_{i}} = s_{i} \; (i \in I)$).   
%For a positive root $\alpha_{j, i} = t_{j} - t_{i}$, we also denote 
%$s_{\alpha_{j, i}}$ by $r_{j, i}$.  

\vspace{0.3cm}  
%%%%%%%%%%%%%%%%%%%%%%%%%%%%%%%%%%%%%%%%%%%%%%%%%%%%%%%%%%%%%%%%%%%%%%%%%%%%%%%%%%%%%%%%%%%%%
\underline{Grassmannian elements} \quad 
\begin{defn} 
 An  element    $w \in S_{\infty}$ is called a {\it Grassmannian element}   
 of  {\it descent $n$} if the condition 
 \begin{equation*} 
   w(1) < w(2) < \cdots < w(n), \; w(n + 1) < w(n + 2) < \cdots
 \end{equation*}  
 is satisfied  for some fixed positive integer $n$. 
\end{defn} 
\noindent
The set of all Grassmannian elements of  descent $n$   will be denoted by 
$S_{\infty}^{(n)}$ in the sequel.  Note that $S_{\infty}^{(n)}$ is equal to 
the set of elements $w  \in S_{\infty}$ sucht that $\ell (w s_{i}) > \ell (w)$ 
for $i \neq n$, where $\ell (w)$ is the length of $w$. 
In other words, $S_{\infty}^{(n)}$ is the set of {\it minimal length  coset representatives} 
of the quotient group $S_{\infty}/(S_{n} \times S_{\infty})$, 
where  $S_{n} \times S_{\infty}$ % := \langle s_{i} \; (i \neq n) \rangle$
 is a subgroup of $S_{\infty}$ generated by the elements $\{ s_{i} \}_{I \setminus \{ n \}}$.   
It is well known that there is a bijection between the set
$\mathcal{P}_{n}$ of partitions of length $\leq n$ and 
$S^{(n)}_{\infty}$%\footnote{
%cf. Ikeda-Naruse \cite[\S 3.3]{Ike-Nar2009}. 
%}
. To be precise, the bijection is given as follows: 
For $w \in S^{(n)}_{\infty}$, we define a partition 
$\lambda_{w}   = ((\lambda_{w})_{1},  (\lambda_{w})_{2}, \ldots, (\lambda_{w})_{n}) \in \mathcal{P}_{n}$ 
by 
\begin{equation*} 
  (\lambda_{w})_{i} := w(n + 1 - i) - (n + 1  - i) \; (1 \leq i \leq n). 
\end{equation*} 
Conversely, given a partition $\lambda \in \mathcal{P}_{n}$, we can construct 
a Grassmannian  permutation $ w_{\lambda}$ of descent $n$ by the following manner.  %\footnote{
%cf.  %Fulton \cite[Lecture seven]{Fulton2007},   
%     Ikeda-Naruse \cite[\S 3.3]{Ike-Nar2009}.  
%}
First note that when considered as a Young diagram, $\lambda$ is contained in the 
rectangle $n \times (N - n)$ for sufficiently large $N$, and we identify $\lambda$
 with   a path  starting from the 
south-west corner to the north-east corner of the rectangle. 
We assign numbers $1, 2, \ldots, N$ to each step. For example, for $
\lambda = (4, 2, 1, 0)$ with $n = 4$, $N = 10$,  we have the following picture:

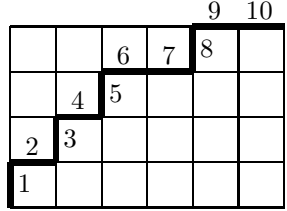
\setlength{\unitlength}{1mm}
     \begin{picture}(150,35)
        %\put(0,0){\circle{1}}
        %\put(0,50){\circle{1}}
        %\put(150,0){\circle{1}}
        %\put(150,50){\circle{1}}
        
        %horizontal lines
        \put(60, 29){\line(1,0){36}}
        \put(60, 23){\line(1,0){36}}
        \put(60, 17){\line(1,0){36}}
        \put(60, 11){\line(1,0){36}}
        \put(60, 5){\line(1,0){36}}

        %vertical lines
        \multiput(60,29)(6,0){7}{\line(0,-1){6}}
        \multiput(60,23)(6,0){7}{\line(0,-1){6}}
        \multiput(60,17)(6,0){7}{\line(0,-1){6}}
        \multiput(60,11)(6,0){7}{\line(0,-1){6}}

        \linethickness{2pt} 
        \put(84, 29){\line(1,0){12}}
        \put(84,29){\line(0,-1){6}} 
        \put(72, 23){\line(1,0){12}}
        \put(72,23){\line(0,-1){6}} 
        \put(66, 17){\line(1,0){6}} 
        \put(66, 17){\line(0, -1){6}}  
        \put(60, 11){\line(1,0){6}} 
        \put(60, 11){\line(0, -1){6}}

        %numbers
        \put(91, 30){$10$}
        \put(86, 30){$9$}
        \put(85, 25){$8$}
        \put(80, 24){$7$}
        \put(74, 24){$6$} 
        \put(73, 19){$5$} 
        \put(68, 18){$4$} 
        \put(67, 13){$3$}
        \put(62, 12){$2$} 
        \put(61, 7){$1$} 
   \end{picture}

\noindent  
If the assigned numbers of the vertical steps are $i_{1} < i_{2} < \cdots < i_{n}$, 
and those of the horizontal steps are $j_{1} < j_{2} <  \cdots < j_{N-n}$, 
then the corresponding Grassmannian permutation is 
\begin{equation*}
   w_{\lambda}   =  i_{1} \, i_{2} \,  \ldots \,  i_{n} \;  j_{1} \,  j_{2} \,  \ldots \,  j_{N-n}. 
\end{equation*}    
More explicity, we have 
\begin{equation}   \label{eqn:Number1-N}  
  i_{k} = \lambda_{n - k + 1} + k \;   (1 \leq k \leq n),  \quad   
  j_{k} = n + k  - {}^{t} \! \lambda_{k}  \;    (1 \leq k \leq N-n), 
\end{equation} 
where ${}^{t} \! \lambda$ is the conjugate of $\lambda$.  
In the above example, we have $w_{\lambda}  = 1 \, 3 \,  5 \,  8 \; 2 \,  4 \,  6 \, 7 \,  9 \,  10$. 
Note that the set $\mathcal{P}_{n}$  (resp. $S^{(n)}_{\infty}$ ) 
is a partially ordered set given by the containment 
$\lambda \subset \mu$ of partitions (resp. the Bruhat-Chevalley ordering), and 
 the above bijection preserves these partial orderings.

 Furthermore a   reduced expression of $w_{\lambda}$ can be obtained 
 by the following manner.  %\footnote{
% This expression is called the {\it row-reading expression}. 
% cf. Buch \cite[\S 2, p.41]{Buc2002}, Ikeda-Naruse \cite[\S 4.2]{Ike-Nar2009}. 
%}. 
For  each box (cell) $\alpha = (i, j) \in \lambda$,  the {\it content} of $\alpha$ 
is defined to be $c(\alpha) := j - i$ %\footnote{
%cf. Macdonald \cite[I,  \S 1,  Examples 3]{Mac95}, Molev \cite[p.8]{Mol2009}, 
%Molev-Sagan \cite[p.4431]{Mol-Sag99}.  
%}
. We fill in  each box $\alpha = (i, j) \in \lambda$ with the number 
$n + c(\alpha) = n   - i + j$. For the above example ($\lambda = (4, 2, 1)$, $n = 4$),  the 
numbering is given by the following picture:

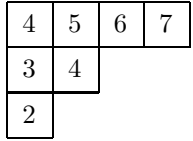
\setlength{\unitlength}{1mm}
     \begin{picture}(150,30)

        \put(60, 24){\line(1,0){24}}
        \put(60, 18){\line(1,0){24}}
        \put(60, 12){\line(1,0){12}}
        \put(60, 6){\line(1,0){6}}

        \multiput(60,24)(6,0){5}{\line(0,-1){6}}
        \multiput(60,18)(6,0){3}{\line(0,-1){6}}
        \multiput(60,12)(6,0){2}{\line(0,-1){6}}

        \put(62, 20){4} 
        \put(68, 20){5} 
        \put(74, 20){6} 
        \put(80, 20){7} 
        \put(62, 14){3} 
        \put(68, 14){4} 
        \put(62, 8){2}

   \end{picture}

\noindent
We read the entries of the boxes of the Young diagram of $\lambda$ from 
right to left starting from  the bottom row 
 to the top  row   and obtain the sequence $i_{1} \; i_{2} \; \cdots  \; i_{|\lambda|}$. 
We then let $s_{i_{1}} s_{i_{2}} \cdots s_{i_{|\lambda|}}$ be the product 
of the corresponding simple reflections, which gives a reduced expression 
of $w_{\lambda}$.  For the above example, we have 
\begin{equation*} 
  w_{\lambda} = 1 \, 3 \, 5 \, 8 \; 2 \, 4 \, 6 \, 7 \,  9 \,  10 
  = s_{2} \cdot s_{4} s_{3} \cdot s_{7} s_{6} s_{5} s_{4}. 
\end{equation*}

\vspace{0.3cm}   
%%%%%%%%%%%%%%%%%%%%%%%%%%%%%%%%%%%%%%%%%%%%%%%%%%%%%%%%%%%%%%%%%%%%%%%%%%%%%%%%%%%%%%%%%
\underline{Action of $S_{\infty}$ on $\mathcal{P}_{n}$}  \quad  
 By means of the above bijection   $\mathcal{P}_{n}  \overset{\sim}{\longrightarrow}   S^{(n)}_{\infty}, \; 
\lambda \longmapsto w_{\lambda}$, the group $S_{\infty}$ acts naturally on 
 the set $\mathcal{P}_{n}$.   %\footnote{
% Notice that the action of $S_{\infty}$ does  {\it not}   send 
% a Grassmannian element to a Grassmannian element. 
% For example, for $w = s_{4} s_{3} =  1 \, 2 \, 5 \; 3 \, 4 \, 6 \cdots 
% \in S_{\infty}^{(3)}$ and 
%$s_{\alpha_{4} + \alpha_{3}} = s_{t_{5} - t_{3}} = (3 \; 5)$, we have 
%$s_{\alpha_{4} + \alpha_{3}} w = s_{4} \not\in S_{\infty}^{(3)}$. 
%We should think of $S_{\infty}^{(3)}$ as $S_{\infty}/(S_{3} \times S_{\infty})$. 
%Then we have  the `correct' action 
%$s_{\alpha_{4} + \alpha_{3}}   w  = 1$.   
%Since $w = s_{4}s_{3}$ corresponds to a partition $(2) = (2, 0, 0) \in \mathcal{P}_{3}$, 
%the above action induces an action $w \cdot (2) = \emptyset$. 
%}. 
We only describe  the action of the simple reflections 
 $s_{i} \; (i \in I)$ on partitions $\lambda \in \mathcal{P}_{n}$. 
 Given a partition  $\lambda \in \mathcal{P}_{n}$, 
 %(not necessarily $\ell (\lambda) \leq n)$,  %be a partition. 
 a box $\alpha  = (i, j)  \in \lambda$ (resp. $\alpha = (i, j) \not\in \lambda$ and $i \leq n$) 
 is {\it removable}    (resp. {\it addable}) 
if  $\lambda \setminus \{ \alpha \}$ (resp.  $\lambda \cup \{ \alpha \}$)   
 is again a Young diagram of a partition in $\mathcal{P}_{n}$, 
 i.e., $(i, j + 1) \not\in \lambda$ and $(i + 1, j) \not\in \lambda$
 (resp. $(i, j-1) \in \lambda$ and $(i-1, j)  \in \lambda$). 
  Furthermore, we fill in  each 
  box $\alpha  = (i, j) \in \lambda$ with   the number $n + c(\alpha) = n + j - i$. 
Then 
  $\lambda$ is called {\it $k$-removable}  (resp. {\it $k$-addable}) 
  if there is a removable  box  $\alpha \in \lambda$  (resp. an addable box $\alpha \not\in \lambda$) 
  such that $c(\alpha) = k  - n$.  
  Under the above convention, the action of the simple reflections $s_{i} \; (i \in I)$ 
  on partitions $\lambda \in \mathcal{P}_{n}$ is given by the following manner: 
  \begin{enumerate} 
  \item $s_{i} \lambda < \lambda$ if and only if $\lambda$ is $i$-removable. 
  \item $s_{i} \lambda > \lambda$ if and only if $\lambda$ is $i$-addable. 
  \end{enumerate}  
  Otherwise $s_{i}$ acts trivially on $\lambda$. 
  For a general $w \in S_{\infty}$, write $w = s_{i_{1}} s_{i_{2}} \cdots s_{i_{r}}$ 
  as a product of simple reflections, and apply the above process successively 
  to a partition $\lambda$ to obtain the action $w \cdot \lambda$.

Let  $\lambda \in \mathcal{P}_{n}$ be a partition, and 
$w_{\lambda} \in S_{\infty}^{(n)}$ the corresponding Grassmannian 
element.   Define its {\it inversion set} by 
\begin{equation*} 
  \mathrm{Inv} \, (\lambda) := \{ \alpha \in \Delta^{+}_{A}  \; | \;  s_{\alpha} \lambda < \lambda \}
 = \{ \alpha \in  \Delta^{+}_{A}  \; | \; s_{\alpha} w_{\lambda} < w_{\lambda} \}. 
\end{equation*}  
In view of the above action of $S_{\infty}$ on $\lambda$ 
(or more directly by considering the condition $s_{\alpha} w_{\lambda} < w_{\lambda}$), 
 we can describe the inversion set 
$\mathrm{Inv} \, (\lambda)$ explicitly for a given $\lambda \in \mathcal{P}_{n}$.

%\setlength{\unitlength}{1mm}
%     \begin{picture}(150,35)

%        \put(60, 29){\line(1,0){36}}
%        \put(60, 23){\line(1,0){36}}
%        \put(60, 17){\line(1,0){36}}
%        \put(60, 11){\line(1,0){36}}
%        \put(60, 5){\line(1,0){36}}

%        \multiput(60,29)(6,0){7}{\line(0,-1){6}}
%        \multiput(60,23)(6,0){7}{\line(0,-1){6}}
%        \multiput(60,17)(6,0){7}{\line(0,-1){6}}
%        \multiput(60,11)(6,0){7}{\line(0,-1){6}}

%        \linethickness{2pt} 
%        \put(90, 29){\line(1,0){6}}
%        \put(90,29){\line(0,-1){6}} 
%        \put(84, 23){\line(1,0){6}}
%        \put(84, 23){\line(0,-1){6}} 
%        \put(72, 17){\line(1,0){12}} 
%        \put(72, 17){\line(0, -1){6}}  
%        \put(60, 11){\line(1,0){12}} 
%        \put(60, 11){\line(0, -1){6}} 

%        \put(57, 19){$i$} 
%        \put(85, 19){{\tiny $\lambda_{i} + n - i + 1 = w_{\lambda}(n + 1 - i)$}}
        
%        \put(74, 31){$j$} 
%        \put(70, 14){{\tiny $n + j - {}^{t} \! \lambda_{j}$}}  
%        \put(74, 11){{\tiny $||$}}  
%        \put(71, 8){{\tiny $w_{\lambda}(n + j)$}}  
%   \end{picture}
 
%\noindent
%Then we see  that the inversion set $\mathrm{Inv} \, (\lambda)$ of $\lambda$ is given 
%by\footnote{
%The above picture might be helpful to see this. 
%} 
\begin{equation}  \label{eqn:Inv(lambda)(A)}  
\begin{array}{rllll} 
  \mathrm{Inv} \, (\lambda)   
 & = &  \{ t_{\lambda_{i}  + n- i + 1} -  t_{n + j - {}^{t} \! \lambda_{j}}   \; | \; 
    (i, j)  \in \lambda   \}   \medskip \\
 & = & \{  t_{w_{\lambda} (n + 1 - i)}   - t_{w_{\lambda}(n + j)}  \; | \;  (i, j) \in \lambda \}.   \medskip 
\end{array}  
\end{equation}

\vspace{0.3cm} 
%%%%%%%%%%%%%%%%%%%%%%%%%%%%%%%%%%%%%%%%%%%%%%%%%%%%%%%%%%%%%%%%%%%%%%%%%%%%%%%%%%%%%%%%%%%%%%
\underline{Euler classes}   \quad  
Lastly we  define a map  %\footnote{
%As the notation suggests, $e (\alpha)$ for $\alpha  \in L$ signifies 
%the  `Euler class' or the `first Chern class'(in the  
%complex cobordism cohomology theory $MU^{*}(-)$) of the
%} 
$e: L \longrightarrow \L[[\b]]$ 
by setting $e (t_{i}) := b_{i} \; (i \geq 1)$ and by the rule 
$e (\alpha + \alpha') := e(\alpha) \pf e (\alpha')$ for $\alpha, \alpha' \in L$. 
Note that by definition, we have $e(-\alpha) = \overline{e(\alpha)}$ for $\alpha \in L$. 
For the simple root  $\alpha_{i} = t_{i + 1} - t_{i} \; (i \geq 1)$, we have 
\begin{equation*} 
  e(\alpha_{i}) = e(t_{i + 1} - t_{i})  = b_{i + 1} \pf \overline{b}_{i} \quad 
  (i \geq 1). 
\end{equation*} 
In particular, for a partition $\lambda \in \mathcal{P}_{n}$, it follows from 
(\ref{eqn:Inv(lambda)(A)}) that 
\begin{equation}    \label{eqn:Product(e(-alpha))(A)}  
\begin{array}{rllll} 
  \displaystyle{\prod_{\alpha \in \mathrm{Inv} \, (\lambda)}}  e (-\alpha) 
  &  = &  \displaystyle{\prod_{(i, j) \in \lambda}}   (\overline{b}_{\lambda_{i} + n  - i + 1}  \pf 
                                 b_{n + j - {}^{t} \! \lambda_{j}})   \medskip \\
  & = &   \displaystyle{\prod_{(i, j) \in \lambda}}   (\overline{b}_{w_{\lambda}(n + 1 - i)}  \pf 
                                 b_{w_{\lambda}(n + j)}).  \medskip 
\end{array} 
\end{equation}

%%%%%%%%%%%%%%%%%%%%%%%%%%%%%%%%%%%%%%%%%%%%%%%%%%%%%%%%%%%%%%%%%%%%%%%%%%%%%%%%%%%%%%%%%%
\subsubsection{Type $B_{\infty}$, $C_{\infty}$, $D_{\infty}$}   \label{subsec:RootDeta(BCD)}   
\quad \\
\underline{Weyl groups of type $X_{\infty}$}  \quad 
Let $X = B, C$, or $D$,  and $W = W(X_{\infty})$ be the Weyl group of type $X_{\infty}$. 
This is a Coxeter group  generated by the simple reflections $\{ s_{i} \}_{i \in J}$, 
where the index set $J = I  \sqcup \{ 0 \} 
  = \{ 0, 1, 2, \ldots \}$ for  $X  = B, C$, 
and $J =  I \sqcup  \{ \hat{1} \} = \{ \hat{1}, 1, 2, \ldots \}$ for $X =  D$.  
If $W$ is of type $C_{\infty}$ (or $B_{\infty})$, the defining relations of 
$\{ s_{i} \}_{i \in J}$ are give by 
\begin{equation*} 
\begin{array}{rlllll} 
   s_{i}^{2}            &=  &  1  \quad   & (i = 0, 1, 2, \ldots), \medskip \\
   s_{i} s_{i + 1}s_{i} &=  &  s_{i+1} s_{i} s_{i+1}   \quad & (i \geq 1), \medskip \\
   s_{i} s_{j}          &=  &  s_{j} s_{i}   \quad & (i, j \geq 1, \; |i  - j| \geq 2), \medskip \\
   s_{0}s_{1}s_{0}s_{1} &=  &  s_{1}s_{0}s_{1}s_{0},   \quad &  s_{0}s_{i} = s_{i}s_{0}  
   \quad (i \geq 1).   \medskip  
\end{array} 
\end{equation*} 
If $W$ is of type $D_{\infty}$, the defining relations  are given by 
\begin{equation*} 
\begin{array}{rllll} 
   s_{i}^{2}     & = & 1  \quad  & (i = \hat{1}, 1, 2, \ldots), \medskip \\
    s_{i} s_{i + 1}s_{i} &=  &  s_{i+1} s_{i} s_{i+1}   \quad & (i \geq 1), \medskip \\
   s_{i} s_{j}          &=  &  s_{j} s_{i}   \quad & (i, j \geq 1, \; |i  - j| \geq 2), \medskip \\
   s_{\hat{1}} s_{2} s_{\hat{1}} &= & s_{2}s_{\hat{1}} s_{2},  \quad 
   & s_{\hat{1}} s_{j} = s_{j} s_{\hat{1}} \quad (j \neq 2).  \medskip 
\end{array} 
\end{equation*} 
It is also well  known that the Weyl group $W(X_{\infty})$ of type $X_{\infty}$ 
can be realized as  a (sub) group of {\it signed} or {\it barred}  permutations. 
As before, let $\N = \{ 1, 2, \ldots \}$ be the set of 
positive integers.  Denote by $\overline{\N} = \{ \overline{1}, \overline{2}, \ldots \}$
 a ``negative'' copy of $\N$ (thus $\overline{\overline{i}} = i$ for $i \in \N$). 
 A {\it signed permutation} $w$ of $\N$ 
 is defined as a bijection on the set $\N \cup \overline{\N}$ such that 
 $\overline{w(i)} = w(\overline{i})$ for all $i \in \N$ and 
$w (i) = i$ for all but a finite number of $i$.  
Denote by $\overline{S}_{\infty}$ the group of all signed permutations. 
We shall use the one-line notation $w = w(1) \, w(2) \, \cdots$ to denote 
an element $w \in \overline{S}_{\infty}$
(we have only to specify $w(i)$ for $i \in \N$ because of the 
condition $w (\overline{i}) = \overline{w(i)}$).  
Then the above-mentioned realization is given by the following identifications:
\begin{equation*} 
  s_{0}  = (1 \; \overline{1}),  \quad 
  s_{i}  = (i \; i + 1) (\overline{i + 1} \; \overline{i}) \;\quad  \text{for} \; i \geq 1. 
\end{equation*} 
Then $W(B_{\infty})$ and $W(C_{\infty})$ can be identified with $\overline{S}_{\infty}$. 
Further the simple reflection $s_{\hat{1}}$ is identified with $s_{0}s_{1}s_{0}$, 
in other words, $s_{\hat{1}} = (\overline{2} \; 1)(\overline{1} \; 2)$. 
Then $W(D_{\infty})$ can be identified with the subgroup $\overline{S}_{\infty, +}$ of 
$\overline{S}_{\infty}$ generated by $s_{\hat{1}}$ and $s_{i} \; (i \geq 1)$.

\vspace{0.3cm}  
%%%%%%%%%%%%%%%%%%%%%%%%%%%%%%%%%%%%%%%%%%%%%%%%%%%%%%%%%%%%%%%%%%%%%%%%%%%%%%%%%%%%%%%%%%%
\underline{Root systems of type $X_{\infty}$}  \quad 
Let $L$ denote a   free $\Z$-module with a basis $\{ t_{i} \}_{i \geq 1}$. 
The set of  positive roots is  defined respectively  by 
\begin{equation*} 
\begin{array}{llll} 
  & \text{Type} \; B_{\infty}:   \quad \Delta^{+}_{B} &=  & \{ t_{i} \; | \; i \geq 1 \}  \cup 
                                 \{  t_{j} \pm t_{i} \; | \; j >  i \geq 1 \},  \medskip \\
  & \text{Type} \; C_{\infty}:  \quad \Delta^{+}_{C}  &=  & \{ 2t_{i} \; | \; i \geq 1 \} \cup 
                                 \{  t_{j} \pm t_{i} \; | \; j > i \geq 1 \}, \medskip \\
  & \text{Type} \; D_{\infty}:  \quad \Delta^{+}_{D}  &=  & \{ t_{j} \pm t_{i} \; | \; j > i \geq 1 \}.  \medskip 
\end{array} 
\end{equation*} 
The set of negative roots is defined by $\Delta^{-}_{X} 
:= -\Delta^{+}_{X} = \{ -\alpha \; | \; \alpha \in 
\Delta^{+}_{X} \}$, and set $\Delta_{X} := \Delta^{+}_{X} \coprod \Delta^{-}_{X}$. 
The following elements of $\Delta^{+}_{X}$ are the simple roots: 
\begin{equation}  \label{eqn:SimpleRoots(BCD)}   
\begin{array}{llll} 
 & \text{Type} \; B_{\infty}:  \quad \alpha_{0} = t_{1}, \quad \alpha_{i} = t_{i + 1} - t_{i} \; (i \geq 1), \medskip \\
 & \text{Type} \; C_{\infty}:  \quad \alpha_{0} = 2t_{1}, \quad \alpha_{i} = t_{i + 1} - t_{i} \; (i \geq 1), \medskip \\
 & \text{Type} \; D_{\infty}:  \quad \alpha_{\hat{1}} = t_{1} + t_{2}, \quad \alpha_{i} = t_{i + 1} - t_{i} \; (i\geq 1). \medskip 
\end{array} 
\end{equation}

\vspace{0.3cm}  
%%%%%%%%%%%%%%%%%%%%%%%%%%%%%%%%%%%%%%%%%%%%%%%%%%%%%%%%%%%%%%%%%%%%%%%%%%%%%%%%%%%%%%%%%%%%%
\underline{Grassmannian elements}    \quad  
We introduce the Grassmannian elements in the case $X  = B, C$, or $D$. 
For simplicity, we  only  deal with the case $X = B, C$. For $X = D$, 
we need  an appropriate modification (see e.g., Ikeda-Mihalcea-Naruse \cite[\S 3.4]{IMN2011}, 
Ikeda-Naruse \cite[\S 4.3]{Ike-Nar2013}).  
\begin{defn}  
  An element $w  \in \overline{S}_{\infty}$ is called a Grassmannian element 
  if the condition 
  \begin{equation*} 
       w(1) < w(2) < \cdots < w(i) < \cdots   
  \end{equation*} 
  is satisfied,  where the ordering is given by 
   \begin{equation*} 
     \cdots < \overline{m} < \cdots < \overline{2}   < \overline{1} < 1 < 2 < \cdots < m < \cdots. 
   \end{equation*} 
\end{defn} 
\noindent 
Let $n $ be a fixed positive integer.  
We are concerned with a  Grassmannian element $w \in \overline{S}_{\infty}$ satisfying the condition 
  \begin{equation*}
        w(1) < w(2) < \cdots < w(r) < 1, \quad 
        \overline{1} < w(r + 1) < w(r + 2) < \cdots
   \end{equation*} 
for some $1 \leq r \leq n$.  In other words, it is a Grassmannian element in $\overline{S}_{\infty}$ 
whose first $r$ values $w(1), w(2), \ldots, w(r) \; (r \leq n)$ are negative.  
The set of all such  Grassmannian elements is  denoted by $\overline{S}_{\infty}^{(n)}$
in the sequel.  %Note that 
Then there is a bijection between the set $\mathcal{SP}_{n}$ of strict partitions of length $\leq n$ 
and $\overline{S}^{(n)}_{\infty}$. 
The bijection is given explicitly  as follows: 
For $w \in \overline{S}^{(n)}_{\infty}$, let $1 \leq r \leq n$ be the number such that 
$w(1) < \cdots < w(r) < 1$ and $\overline{1} < w(r + 1) < w(r + 2) < \cdots$. 
Then we define an $r$-tuple of positive integers $\lambda_{w} = ((\lambda_{w})_{1}, (\lambda_{w})_{2}, \ldots, (\lambda_{w})_{r})$ by $(\lambda_{w})_{i} :=  \overline{w(i)}$ 
for $1 \leq i \leq r$.  By the above ordering, we see immediately that 
$\lambda_{w}$ is a strict partition of length $\leq n$, i.e., $\lambda_{w} \in \mathcal{SP}_{n}$.  
Conversely, for a strict partition $\lambda = (\lambda_{1}, \ldots, \lambda_{r})$ of 
length $r \leq n$,  we define the element $w_{\lambda}  \in \overline{S}_{\infty}^{(n)}$ 
by $w_{\lambda}(i) = \overline{\lambda}_{i} \; (1 \leq i \leq r)$, 
and the remaining values $w_{\lambda} (j) \; (j \geq r + 1)$
 are given by the increasing sequence of 
positive integers from $\N \setminus \{ \lambda_{r},  \lambda_{r-1}, \ldots, \lambda_{1} \}$.  
 For example, the  Grassmannian element $w = \overline{6}  \, \overline{4}  \,  \overline{3} \, \overline{1} \, 2 \,  5 \, 7  \cdots$  corresponds to  the strict partition $\lambda = (6, 4, 3, 1)$. 
Note that the set $\mathcal{SP}_{n}$ (resp.  $\overline{S}_{\infty}^{(n)}$) is a 
partially ordered set given by the containment $\lambda \subset \mu$ of strict partitions 
(resp. Bruhat-Chevalley ordering), and the above bijection preserves these 
partial orderings. %\footnote{
%cf. Ikeda-Naruse \cite[Proposition 4.1]{Ike-Nar2013}. 
%}. 

Furthermore   a   reduced expression of $w_{\lambda}$ can be obtained by the following 
manner.  %\footnote{
%cf. Ikeda-Naruse \cite[\S 7]{Ike-Nar2009}, Ikeda-Mihalcea-Naruse \cite[\S 3.4, Example]{IMN2011}. 
%}.  
First we associate to each strict partition $\lambda = (\lambda_{1}, \ldots, \lambda_{r})$, 
$\lambda_{1} > \lambda_{2} > \cdots > \lambda_{r} > 0$ the {\it shifted Young diagram} 
\begin{equation*} 
  D'(\lambda) := \{(i, j) \in \Z^{2} \; | \; 1 \leq i \leq r, \; 
     i \leq j  \leq \lambda_{i} + i - 1 \}. 
\end{equation*} 
For example,

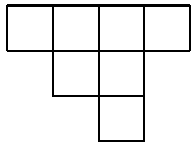
\setlength{\unitlength}{1mm}
     \begin{picture}(150,30)

        \put(60, 24){\line(1,0){24}}
        \put(60, 18){\line(1,0){24}}
        \put(66, 12){\line(1,0){12}}
        \put(72, 6){\line(1,0){6}}

        \multiput(60,24)(6,0){5}{\line(0,-1){6}}
        \multiput(66,18)(6,0){3}{\line(0,-1){6}}
        \multiput(72,12)(6,0){2}{\line(0,-1){6}}

   \end{picture}

\noindent  
is the shifted Young diagram of a strict partition $\lambda = (4, 2, 1)$.  
We shall identify a strict partition $\lambda$ with its shifted Young diagram 
if there is no fear of confusion.  
For each box $\alpha = (i, j) \in D'(\lambda)$, 
we define its content $c(\alpha)  \in I \sqcup \{ 0 \}$ 
%(resp.  $c' (\alpha) \in I \sqcup \{ \hat{1} \}$  for $X = D$) 
 by 
$c(\alpha) := j - i$ (for $X = B, C$).    
%(resp.   $c'(\alpha) = j - i + 1$ of $i \neq j$, 
%$c' (\alpha) = \hat{1}$ if $j = i$ is  \underline{odd}, 
%$c' (\alpha) = 1$ if $j = i$ is \underline{even}).  
We fill in the number $c(\alpha)$ %for $X = B, C$ %and $c'(\alpha)$ for $X = D$ 
to each box $\alpha \in  D'(\lambda)$.  For example, for $\lambda = (4, 2, 1)$, 
the numbering is given by the following picture:

\setlength{\unitlength}{1mm}
     \begin{picture}(150,34)

        \put(60, 28){\line(1,0){24}}
        \put(60, 22){\line(1,0){24}}
        \put(66, 16){\line(1,0){12}}
        \put(72, 10){\line(1,0){6}}

        \multiput(60,28)(6,0){5}{\line(0,-1){6}}
        \multiput(66,22)(6,0){3}{\line(0,-1){6}}
        \multiput(72,16)(6,0){2}{\line(0,-1){6}}
       
        \put(66, 4){Type $B$, $C$}

        \put(62, 24){$0$}
        \put(68, 24){$1$} 
        \put(74, 24){$2$} 
        \put(80, 24){$3$}  
        
        \put(68, 18){$0$} 
        \put(74, 18){$1$}
        \put(74, 12){$0$}

   \end{picture}

\noindent  
We read the entries  of the boxes of the shifted Young diagram of $\lambda$ 
from  right to left starting from the bottom row  to the top row
 and obtain the sequence $i_{1} \, i_{2}, \ldots, i_{|\lambda|}$. 
We let $s_{i_{1}} s_{i_{2}} \cdots s_{i_{|\lambda|}}$ be the product of 
the corresponding simple reflections, which gives  a reduced expression 
of $w_{\lambda}$.  For the above example, we have 
\begin{equation*} 
  w_{\lambda} =  \overline{4} \, \overline{2} \, \overline{1} 3 \cdots = 
  s_{0} \cdot s_{1} s_{0} \cdot s_{3} s_{2} s_{1} s_{0}. 
\end{equation*}

\vspace{0.3cm}  
%%%%%%%%%%%%%%%%%%%%%%%%%%%%%%%%%%%%%%%%%%%%%%%%%%%%%%%%%%%%%%%%%%%%%%%%%%%%%%%%%%%%%%%%%%%
\underline{Action of the Weyl group $W(X_{\infty})$ on $\mathcal{SP}_{n}$}    \quad  
By means of the above bijection $\mathcal{SP}_{n} \overset{\sim}{\longrightarrow} 
\overline{S}_{\infty}^{(n)}$, $\lambda \longmapsto w_{\lambda}$, the 
group $\overline{S}_{\infty}$ acts on the set $\mathcal{SP}_{n}$. 
We shall  describe the action of the simple reflections $s_{i} \; (i \in I \sqcup \{ 0 \})$
on  strict partitions $\lambda \in \mathcal{SP}_{n}$.  
Given a strict partition $\lambda \in \mathcal{SP}_{n}$, 
a box $\alpha = (i, j) \in \lambda$  (resp. $\alpha  = (i, j) \not\in \lambda$ and $i \leq n$) 
is removable  (resp. addable) if 
$\lambda \setminus \{ \alpha \}$  (resp.  $\lambda \cup \{ \alpha \}$) 
is again a shifted Young diagram 
of a strict partition in $\mathcal{SP}_{n}$, i.e., $j = \lambda_{i} + i - 1$ and 
$\lambda_{i + 1} \leq \lambda_{i} - 2$ (resp. $j = \lambda_{i} + i$ and 
$\lambda_{i} \leq \lambda_{i-1} - 2$).  
Furthemore, we fill in  each 
box $\alpha = (i, j) \in \lambda$ with   the number $c(\alpha) = j - i$. 
  Then $\lambda$ is called 
$k$-removable  (resp.  $k$-addable)  
if there is a removable box $\alpha \in \lambda$  (resp. an addable box $\alpha \not\in \lambda$)
such that $c(\alpha) = k$.  
Under the above covention, the action of the simple reflections $s_{i} \; (i \in I \sqcup \{ 0 \})$ 
on strict partitions $\lambda \in \mathcal{SP}_{n}$ is given by the following manner: %\footnote{
%cf. Ikeda-Naruse \cite[Proposition 4.3]{Ike-Nar2013}, Ikeda-Mihalcea-Naruse \cite[Lemma 9.2]{IMN2011}. 
%}: 
\begin{enumerate} 
\item $s_{i} \lambda < \lambda$ if and only if $\lambda$ is $i$-removable. 
\item $s_{i} \lambda > \lambda$ if and only if $\lambda$ is $i$-addable. 
\end{enumerate} 
Otherwise $s_{i}$ acts trivially on $\lambda$. 
For a general $w \in \overline{S}_{\infty}$, write $w = s_{i_{1}} s_{i_{2}} \cdots s_{i_{r}}$ 
as a product of simple reflections, and apply the above process successively to 
a strict partition $\lambda$ to obtain the action $w \cdot \lambda$.

Let $\lambda  \in \mathcal{SP}_{n}$ be a strict partition, and $w_{\lambda} \in \overline{S}_{\infty}^{(n)}$ the corresponding Grassmannian element. 
Define its inversion set by 
\begin{equation*} 
   \mathrm{Inv}_{X}   (\lambda) = \{ \alpha \in \Delta^{+}_{X} \; | \; s_{\alpha} \lambda < \lambda \} 
     = \{ \alpha \in \Delta^{+}_{X} \; | \; s_{\alpha} w_{\lambda} < w_{\lambda} \}. 
\end{equation*} 
In view of the above action of $\overline{S}_{\infty}$ on $\lambda$, 
we can describe the inversion set $\mathrm{Inv}_{X}   (\lambda)$ explicitly for 
a given $\lambda \in \mathcal{SP}_{n}$.   
For instance, the set $\mathrm{Inv}_{C} (\lambda)$ for a strict partition 
$\lambda  = (\lambda_{1}, \lambda_{2}, \ldots, \lambda_{r}) 
\in \mathcal{SP}_{n}$, $\ell (\lambda) = r$, 
 is described as follows: 
Let $w_{\lambda}  \in \overline{S}_{\infty}^{(n)}$ be the corresponding 
Grassmannian element.   Recall that 
 $w_{\lambda} (i)  = \overline{\lambda}_{i} \; (1 \leq i \leq r)$
and  $w_{\lambda} (j) \; (j \geq r + 1)$ are obtained by the increasing sequence 
of positive integers from $\N \setminus \{ \lambda_{r}, \lambda_{r-1}, \ldots, \lambda_{1} \}$. 
From this, we have 
\begin{equation}    \label{eqn:Inv(lambda)(C)}  
\begin{array}{llll} 
  \mathrm{Inv}_{C} (\lambda) 
      & =  \{  \;  
             t_{\overline{w_{\lambda} (i)}} + t_{\overline{w_{\lambda} (j)}} \; | \; 
  (i, j) \in \lambda \}  \medskip \\
       &= \{ t_{\lambda_{i}} + t_{\lambda_{j}} \; (1 \leq i \leq r, \;  i \leq j \leq r)  \}  \medskip \\
       & \hspace{0.1cm}     
       \cup  \;   \{  t_{\lambda_{i}} + t_{\overline{j}}  \;   (1 \leq i \leq r, \; 
                   1 \leq j \leq \lambda_{i} - 1, \; j \neq \lambda_{i + 1}, \ldots, \lambda_{r})  
        \} . \medskip 
\end{array}  
\end{equation} 
Here  $\lambda$ is identified with  its associated   shifted    diagram, and 
$t_{\overline{i}}$ for $i \in \N$  is understood to be  $-t_{i}$.

\vspace{0.3cm}  
%%%%%%%%%%%%%%%%%%%%%%%%%%%%%%%%%%%%%%%%%%%%%%%%%%%%%%%%%%%%%%%%%%%%%%%%%%%%%%%%%%%%%%%%
\underline{Euler classes}    \quad  
We define a map $e: L \longrightarrow \L[[\b]]$ by setting 
$e(t_{i}) := b_{i} \; (i \geq 1)$ and by the rule $e(\alpha + \alpha')  := e(\alpha) \pf e (\alpha')$ 
for $\alpha, \alpha' \in L$.  
For the simple roots $\alpha_{i} \; (i \geq 0)$ and $\alpha_{\hat{1}}$ given in 
(\ref{eqn:SimpleRoots(BCD)}), we have 
\begin{equation*} 
\begin{array}{llll} 
 & \text{Type} \; B_{\infty}:  \quad  e (\alpha_{0})  =  b_{1}, 
    \quad   e(\alpha_{i})  = b_{i + 1}  \pf  \overline{b}_{i} \; (i \geq 1), \medskip \\
 & \text{Type} \; C_{\infty}:  \quad  e( \alpha_{0}) = b_{1} \pf b_{1}, 
    \quad   e(\alpha_{i})  =  b_{i + 1}  \pf  \overline{b}_{i} \; (i \geq 1), \medskip \\
 & \text{Type} \; D_{\infty}:  \quad   e(\alpha_{\hat{1}}) = b_{1}  \pf  b_{2}, 
   \quad   e(\alpha_{i}) = b_{i + 1} \pf   \overline{b}_{i} \; (i\geq 1). \medskip 
\end{array} 
\end{equation*} 
In particular, for a strict partition $\lambda \in \mathcal{SP}_{n}$, it follows from 
(\ref{eqn:Inv(lambda)(C)}) that 
\begin{equation}   \label{eqn:Prodct(e(-alpha))(C)}  
\begin{array}{llll}  
  \displaystyle{\prod_{\alpha \in \mathrm{Inv}_{C}    \, (\lambda)}}   e(-\alpha)
  & =  \displaystyle{\prod_{(i, j) \in \lambda }}  (b_{w_{\lambda}(i)}   \pf b_{w_{\lambda}(j)})    \medskip \\
  &=   \displaystyle{\prod_{i=1}^{r}}   
    \left ( \prod_{j = i}^{r}  (\overline{b}_{\lambda_{i}}  \pf  \overline{b}_{\lambda_{j}})  
  \cdot \prod_{ \substack{  1 \leq j \leq \lambda_{i} - 1 \\  
                 j \neq \lambda_{p} \; \text{for} \; i + 1 \leq p \leq r  } }  (\overline{b}_{\lambda_{i}} \pf  b_{j} )  \right ).  \medskip 
\end{array}  
\end{equation} 
Notice that this value is equal to the specialization 
$Q^{\L}_{\lambda}(\overline{\b}_{\lambda}|\b)$ in Proposition \ref{prop:VanishingProperty(P^LQ^L)}. 
Similarly, an analogous result holds for $P^{\L}_{\lambda}(\x_{n}|\b)^{+}$.

%%%%%%%%%%%%%%%%%%%%%%%%%%%%%%%%%
% References
%%%%%%%%%%%%%%%%%%%%%%%%%%%%%%%%%


\begin{thebibliography}{99}

%%%%%%%%%%%%%%%%%%%%%%%%%%%%%%%%%%%%%%%
%\bibitem{B1}
%T. Bridgeland,
%Equivalences of triangulated categories and Fourier-Mukai transforms,
%Bull. London Math. Soc., {\bf 31} (1999), 25--34.
%%%%%%%%%%%%%%%%%%%%%%%%%%%%%%%%%%%%%%%%
%\bibitem{Sa}
%I. Satake, 
%Algebraic Structures of Symmetric Domains,
%Publ. Math. Soc. Japan, {\bf 14}, 
%Iwanami Shoten, Tokyo; 
%Princeton University Press, Princeton, N.J., 1980.
%%%%%%%%%%%%%%%%%%%%%%%%%%%%%%%%%%%%%%%%%
%\bibitem{Vo} D. A. Vogan,  Jr.,
%A Langlands classification for unitary representations,
%In: Analysis on Homogeneous Spaces and Representation Theory of Lie Groups,
%Okayama-Kyoto,  1997,
%(eds. T. Kobayashi, M. Kashiwara, T. Matsuki, K. Nishiyama and T. Oshima),
%Adv. Stud. Pure Math., {\bf 26}, Math. Soc. Japan, 2000, pp. 299--324.
%%%%%%%%%%%%%%%%%%%%%%%%%%%%%%%%%%%%%%%%%%






%%%%%%%%%%%%%%%%%%%%%%%%%%%%%%%%%%%%%%%%%%%%%%%%%%%%%%%%%%%%%%%%%%%%%%%%%%%%%%%%%%%%%%%%%%%%%%
\bibitem{Ada74}  
J. F. Adams, 
Stable Homotopy and Generalised Homology, 
Chicago Lectures in Mathematics (1974), The University of Chicago Press. 


%%%%%%%%%%%%%%%%%%%%%%%%%%%%%%%%%%%%%%%%%%%%%%%%%%%%%%%%%%%%%%%%%%%%%%%%%%%%%%%%%%%%%%%%%%%%%%%%%
%\bibitem{ABP66}  
%D. W. Anderson, E. H. Brown, Jr., and F. P. Peterson, 
%$SU$-cobordism, $KO$-characteristic number, and the Kervaire invariant, 
%Ann. Math.  {\bf  83} (1966), 54--67.

%%%%%%%%%%%%%%%%%%%%%%%%%%%%%%%%%%%%%%%%%%%%%%%%%%%%%%%%%%%%%%%%%%%%%%%%%%%%%%%%%%%%%%%%%%%%%%%%
%\bibitem{Ati-Hir61}  
%M. F, Atiyah and F. Hirzebruch, 
%Vector bundles and homogeneous spaces, 
%Proc. Symposia in Pure Math. {\bf 3} (1961), 7--38.  


%%%%%%%%%%%%%%%%%%%%%%%%%%%%%%%%%%%%%%%%%%%%%%%%%%%%%%%%%%%%%%%%%%%%%%%%%%%%%%%%%%%%%%%%%%%%%%%%%
\bibitem{Bak86} 
A. Baker, 
On the spaces classifying complex vector bundles with given real dimension, 
Rocky Mountain J.  Math. {\bf 16} (1986), 703--716.   



%%%%%%%%%%%%%%%%%%%%%%%%%%%%%%%%%%%%%%%%%%%%%%%%%%%%%%%%%%%%%%%%%%%%%%%%%%%%%%%%%%%%%%%%%%%%%%%%
\bibitem{Bak-Ric2008}  
A. Baker and B. Richter, 
Quasisymmetric functions from a topological point of view, 
Math. Scand. {\bf 103} (2008),  208--242. 



%%%%%%%%%%%%%%%%%%%%%%%%%%%%%%%%%%%%%%%%%%%%%%%%%%%%%%%%%%%%%%%%%%%%%%%%%%%%%%%%%%%%%%%%%%%%%%%%
%\bibitem{Bor53}  A. Borel, 
%\newblock{\em Sur la cohomologie des espaces fibr\'{e}s principaux et des espaces homog\`{e}nes
%de groupes de Lie compacts}, 
%\newblock{Ann. of Math. \textbf{57} (1953), 115--207.}  


%%%%%%%%%%%%%%%%%%%%%%%%%%%%%%%%%%%%%%%%%%%%%%%%%%%%%%%%%%%%%%%%%%%%%%%%%%%%%%%%%%%%%%%%%%%%%%%%%
\bibitem{Bot54} 
R. Bott, 
 On torsion in Lie groups, 
Proc. Nat. Acad. Sci. {\bf 40} (1954), 586--588. 


%%%%%%%%%%%%%%%%%%%%%%%%%%%%%%%%%%%%%%%%%%%%%%%%%%%%%%%%%%%%%%%%%%%%%%%%%%%%%%%%%%%%%%%%%%%%%%%%%%%
\bibitem{Bot56} 
R. Bott, 
An application of the Morse theory to the topology of Lie-groups, 
Bull. Soc. Math. France  {\bf  84} (1956), 251--281. 


%%%%%%%%%%%%%%%%%%%%%%%%%%%%%%%%%%%%%%%%%%%%%%%%%%%%%%%%%%%%%%%%%%%%%%%%%%%%%%%%%%%%%%%%%%%%%%%%
\bibitem{Bot58} 
R. Bott, 
The   space of loops on a Lie group, 
Michigan  Math. J.  {\bf 5} (1958), 35--61. 



%%%%%%%%%%%%%%%%%%%%%%%%%%%%%%%%%%%%%%%%%%%%%%%%%%%%%%%%%%%%%%%%%%%%%%%%%%%%%%%%%%%%%%%%%%%%%%%%%
\bibitem{Bot59}  
R. Bott, 
The stable homotopy of the classical groups, 
Ann. of Math. (2)  {\bf 70}  (1959), 313--337.   


%%%%%%%%%%%%%%%%%%%%%%%%%%%%%%%%%%%%%%%%%%%%%%%%%%%%%%%%%%%%%%%%%%%%%%%%%%%%%%%%%%%%%%%%%%%%%%%%
\bibitem{Bot69}
R. Bott, 
Lectures on $K(X)$, 
W. A. Benjamin, Inc., New York-Amsterdam 1969. 
%%%%%%%%%%%%%%%%%%%%%%%%%%%%%%%%%%%%%%%%%%%%%%%%%%%%%%%%%%%%%%%%%%%%%%%%%%%%%%%%%%%%%%%%%%%%%%%%%%
%\bibitem[Bre-Eve90]{Bre-Eve}
\bibitem{Bre-Eve90} 
P.  Bressler and  S. Evens,
The Schubert calculus, braid relations, and generalized cohomology,
Trans. Amer. Math. Soc.  {\bf 317}  (1990), no. 2, 799--811. 



%%%%%%%%%%%%%%%%%%%%%%%%%%%%%%%%%%%%%%%%%%%%%%%%%%%%%%%%%%%%%%%%%%%%%%%%%%%%%%%%%%%%%%%%%%%%%%%
%\bibitem[Bre-Eve92]{Bre-Eve2}
%P.\: Bressler and  S.\: Evens,
%Schubert calculus in complex cobordism, Trans. Amer. Math. Soc. \textbf{331} (1992), no. 2, 799--813.
\bibitem{Bre-Eve92}  
P. Bressler and  S. Evens,
Schubert calculus in complex cobordism, 
Trans. Amer. Math. Soc. {\bf 331} (1992), no. 2, 799--813.  


%%%%%%%%%%%%%%%%%%%%%%%%%%%%%%%%%%%%%%%%%%%%%%%%%%%%%%%%%%%%%%%%%%%%%%%%%%%%%%%%%%%%%%%%%%%%%
\bibitem{Buc2002} 
A. S. Buch, 
A Littlewood-Richardson rule for the $K$-theory of Grassmannians, 
Acta Math., {\bf 189} (2002),  37--78. 



%%%%%%%%%%%%%%%%%%%%%%%%%%%%%%%%%%%%%%%%%%%%%%%%%%%%%%%%%%%%%%%%%%%%%%%%%%%%%%%%%%%%%%%%%%%%%%
\bibitem{Car59I} 
H. Cartan, 
D\'{e}monstration homologique des th\'{e}or\`{e}mes de 
                p\'{e}riodicit\'{e} de Bott, I, 
S\'{e}minaire Henri Cartan, {\bf 12} no.2 (1959-1960), Exp. No.16, p.1--16. 

%%%%%%%%%%%%%%%%%%%%%%%%%%%%%%%%%%%%%%%%%%%%%%%%%%%%%%%%%%%%%%%%%%%%%%%%%%%%%%%%%%%%%%%%%%%%%
\bibitem{Car59II} 
H. Cartan, 
D\'{e}monstration homologique des th\'{e}or\`{e}mes de 
                p\'{e}riodicit\'{e} de Bott, II.    
               Homologie et cohomologie des groupes classiques et de leurs espaces homog\`{e}ne,
S\'{e}minaire Henri Cartan, {\bf 12} no.2 (1959-1960), Exp. No.17, p.1--32. 

%%%%%%%%%%%%%%%%%%%%%%%%%%%%%%%%%%%%%%%%%%%%%%%%%%%%%%%%%%%%%%%%%%%%%%%%%%%%%%%%%%%%%%%%%%%%%%
%\bibitem{Car59III} 
%H. Cartan, 
%D\'{e}monstration homologique des th\'{e}or\`{e}mes de 
%                p\'{e}riodicit\'{e} de Bott, III, 
%S\'{e}minaire Henri Cartan, {\bf 12} no.2 (1959-1960), Exp. No.18, p.1--9. 




%%%%%%%%%%%%%%%%%%%%%%%%%%%%%%%%%%%%%%%%%%%%%%%%%%%%%%%%%%%%%%%%%%%%%%%%%%%%%%%%%%%%%%%%%%%%%%%
\bibitem{Cla74}  
F. Clarke, 
On the $K$-theory of the loop space on a Lie group, 
Proc. Camb. Phil. Soc.  {\bf 76} (1974),  1--20. 


%%%%%%%%%%%%%%%%%%%%%%%%%%%%%%%%%%%%%%%%%%%%%%%%%%%%%%%%%%%%%%%%%%%%%%%%%%%%%%%%%%%%%%%%%%%%%%%%
\bibitem{Cla81}  
F. Clarke, 
The $K$-theory of  $\Omega Sp(n)$, 
Quart. J. Math. Oxford Ser. (2)  {\bf 32} (1981),  11--22. 


%%%%%%%%%%%%%%%%%%%%%%%%%%%%%%%%%%%%%%%%%%%%%%%%%%%%%%%%%%%%%%%%%%%%%%%%%%%%%%%%%%%%%%%%%%%%%%%
\bibitem{Cla81(2)}  
F. Clarke, 
On the  homology of $\Omega Sp$, 
J. London Math. Soc. (2)  {\bf 24} (1981),  346--350. 


%%%%%%%%%%%%%%%%%%%%%%%%%%%%%%%%%%%%%%%%%%%%%%%%%%%%%%%%%%%%%%%%%%%%%%%%%%%%%%%%%%%%%%%%%%%%%%%%
%\bibitem{Con-Flo66}  
%P. E. Conner and E. E. Floyd, 
%The relation of cobordism to $K$-theories, 
%Lecture Notes in Mathemetics  {\bf  28} (1966), Springer-Verlag. 



%\bibitem{Dou59} A. Douady,  
%\newblock{\em P\'{e}riodicit\'{e} du groupe unitaire}, 
%\newblock{ S\'{e}minaire Henri Cartan, \textbf{12} no. 2 (1959-1960), Exp. No.11, p.1--16.
%}  


%%%%%%%%%%%%%%%%%%%%%%%%%%%%%%%%%%%%%%%%%%%%%%%%%%%%%%%%%%%%%%%%%%%%%%%%%%%%%%%%%%%%%%%%%%%%%%
\bibitem{Fel2003}  
K. E. Fel'dman, 
An equivariant analog of the Poincar\'{e}-Hopf theorem, 
J. Math. Sci., {\bf 113} (2003), 906--914; 
Translated from Zap. Nauchn. Sem. POMI, {\bf 267} (2001), 303--318. 


%\bibitem{Fus94} D. Fusemoller, 
%\newblock{Fibre Bundles}, 
%\newblock{3rd edition, Graduate Texts in Math. \textbf{20}, Springer-Verlag.} 



%%%%%%%%%%%%%%%%%%%%%%%%%%%%%%%%%%%%%%%%%%%%%%%%%%%%%%%%%%%%%%%%%%%%%%%%%%%%%%%%%%%%%%%%
\bibitem{Gan-Ram2012}  
N. Ganter and A. Ram, 
Generalized Schubert calculus,  
J. Ramanujan Math. Soc., {\bf 28} A (2013), 149--190.  
%arXiv:math.RT/12125742.  




%%%%%%%%%%%%%%%%%%%%%%%%%%%%%%%%%%%%%%%%%%%%%%%%%%%%%%%%%%%%%%%%%%%%%%%%%%%%%%%%%%%%%%%%
\bibitem{Gar-Rag75}  
H. Garland and R. S. Raghunathan, 
A Bruhat decomposition for the loop space of a compact group: 
               A new approach to results of Bott, 
Proc. Nat. Acad. Sci. USA  {\bf 72} (1975), 4716--4717. 





%\bibitem{Git-Lam7?}  S. Gitler and K. Y. Lam, 
%\newblock{\em  The $K$-theory of Stiefel manifolds}, 
%\newblock{}.  



%%%%%%%%%%%%%%%%%%%%%%%%%%%%%%%%%%%%%%%%%%%%%%%%%%%%%%%%%%%%%%%%%%%%%%%%%%%%%%%%%%%%%%%%%%%
%\bibitem[HHH05]{H-H-H}
\bibitem{HHH2005}  
M.  Harada, A.  Henriques, and T.  S. Holm,
Computation of generalized equivariant
cohomologies of Kac-Moody flag varieties,
Adv. Math. {\bf 197}  (2005), 198--221.


%%%%%%%%%%%%%%%%%%%%%%%%%%%%%%%%%%%%%%%%%%%%%%%%%%%%%%%%%%%%%%%%%%%%%%%%%%%%%%%%%%%%%%%%%%%%%
\bibitem{Hof-Hum92}  
P. N. Hoffman and J. F. Humphreys, 
Projective representations of the symmetric groups. 
$Q$-functions and shifted tableaux, 
Oxford Mathematical Monographs.    Oxford Science Publications. 
The Clarendon Press, Oxford University Press, New York, 1992. 


%%%%%%%%%%%%%%%%%%%%%%%%%%%%%%%%%%%%%%%%%%%%%%%%%%%%%%%%%%%%%%%%%%%%%%%%%%%%%%%%%%%%%%%%%%%%
\bibitem{Ike2007}  
T. Ikeda, 
Schubert classes in the equivariant cohomology of the Lagrangian Grassmannian, 
Adv. in Math. \textbf{215} (2007),   1--23. 



%%%%%%%%%%%%%%%%%%%%%%%%%%%%%%%%%%%%%%%%%%%%%%%%%%%%%%%%%%%%%%%%%%%%%%%%%%%%%%%%%%%%%%%%%%%
\bibitem{Ike-Nar2009}  
T. Ikeda and H. Naruse, 
Excited Young diagrams and equivariant Schubert calculus, 
Trans. Amer. Math. Soc., {\bf 361} (2009),  5193--5221.  




%%%%%%%%%%%%%%%%%%%%%%%%%%%%%%%%%%%%%%%%%%%%%%%%%%%%%%%%%%%%%%%%%%%%%%%%%%%%%%%%%%%%%%%%%%%
\bibitem{IMN2011} 
T. Ikeda, L. C. Mihalcea, and H. Naruse, 
Double Schubert polynomials for the classical groups, 
Adv. in Math.  \textbf{226} (2011),   840--866.




%%%%%%%%%%%%%%%%%%%%%%%%%%%%%%%%%%%%%%%%%%%%%%%%%%%%%%%%%%%%%%%%%%%%%%%%%%%%%%%%%%%%%%%%%%%
\bibitem{Ike-Nar2013}
T. Ikeda and H.  Naruse,
$K$-theoretic analogue of factorial Schur $P$- and $Q$-functions,
%arXiv:1112.5223; 
Adv. in Math. {\bf 243} (2013),  22--66. 


%%%%%%%%%%%%%%%%%%%%%%%%%%%%%%%%%%%%%%%%%%%%%%%%%%%%%%%%%%%%%%%%%%%%%%%%%%%%%%%%%%%%%%%%%%%%%
\bibitem{Iva97}  
V. N. Ivanov, 
Dimensions of skew-shifted Young diagrams and projective 
characters of the infinite symmetric group, 
Zapiski Nauchnykh Seminarov POMI,  {\bf 240} (1997), 115--135 
(J. of Mat. Sci., \textbf{96}, No.5 (1999), 3517--3530). 


%%%%%%%%%%%%%%%%%%%%%%%%%%%%%%%%%%%%%%%%%%%%%%%%%%%%%%%%%%%%%%%%%%%%%%%%%%%%%%%%%%%%%%%%%%%%%%
\bibitem{Iva2004}  
V. N. Ivanov, 
Interpolation analogs   of Schur $Q$-functions, 
Zapiski  Nauchnykh Seminarov POMI,  {\bf 307} (2004), 99--119
 (J. of Math. Sci. {\bf 131}, No.2 (2005),  5495--5507). 


%%%%%%%%%%%%%%%%%%%%%%%%%%%%%%%%%%%%%%%%%%%%%%%%%%%%%%%%%%%%%%%%%%%%%%%%%%%%%%%%%%%%%%%%%%%%%
\bibitem{Joz91}  
T. J\'{o}zefiak, 
Schur $Q$-functions and cohomology of isotropic Grasmannians, 
Math. Proc. Camb. Phil. Soc. {\bf 109} (1991),  471--478. 


%%%%%%%%%%%%%%%%%%%%%%%%%%%%%%%%%%%%%%%%%%%%%%%%%%%%%%%%%%%%%%%%%%%%%%%%%%%%%%%%%%%%%%%%%%%%%%%%
%\bibitem{Kir-Kri2013}  
%V. Kiritchenko, and A. Krishna, 
%Equivariant cobordism of flag varieties and of symmetric 
%varieties, 
%Transform. Groups {\bf 18} (2013),  no.2, 391--413.  


%%%%%%%%%%%%%%%%%%%%%%%%%%%%%%%%%%%%%%%%%%%%%%%%%%%%%%%%%%%%%%%%%%%%%%%%%%%%%%%%%%%%%%%%%%%%    
%\bibitem[Kir-Nar]{Kir-Nar}
\bibitem{Kir-Nar} 
    A.  N.  Kirillov and H. Naruse,
Construction of double Grothendieck polynomials  of classical type
using Id-Coxeter algebras, 
preprint.  


%%%%%%%%%%%%%%%%%%%%%%%%%%%%%%%%%%%%%%%%%%%%%%%%%%%%%%%%%%%%%%%%%%%%%%%%%%%%%%%%%%%%%%%%%%%%%%%%
\bibitem{Knu-Tao2003}   
A. Knutson and T. Tao, 
Puzzles and $($equivariant$)$ cohomology of Grassmannians, 
Duke Math. J. {\bf 119} (2003),  221--260. 



%%%%%%%%%%%%%%%%%%%%%%%%%%%%%%%%%%%%%%%%%%%%%%%%%%%%%%%%%%%%%%%%%%%%%%%%%%%%%%%%%%%%%%%%%%%%%%%%%
\bibitem{Kon-Koz78} 
A.  Kono and K. Kozima, 
The   space of loops on a symplectic  group, 
Japanese J. Math.  {\bf 4} (1978), 461--486. 


%%%%%%%%%%%%%%%%%%%%%%%%%%%%%%%%%%%%%%%%%%%%%%%%%%%%%%%%%%%%%%%%%%%%%%%%%%%%%%%%%%%%%%%%%%%%%%%%
\bibitem{Kos-Kum90} 
B. Kostant and S. Kumar, 
$T$-equivariant $K$-theory of generalized flag varieties, 
J. Differential Geom. {\bf 32} (1990), no.2, 549--603. 


%%%%%%%%%%%%%%%%%%%%%%%%%%%%%%%%%%%%%%%%%%%%%%%%%%%%%%%%%%%%%%%%%%%%%%%%%%%%%%%%%%%%%%%%%%%%%%%
\bibitem{Koz79}  
K. Kozima, 
The  Hopf algebra structure of $K_{*}(\Omega Sp(n))$, 
J. Math. Kyoto Univ. {\bf 19} (1979), 315--326. 


%%%%%%%%%%%%%%%%%%%%%%%%%%%%%%%%%%%%%%%%%%%%%%%%%%%%%%%%%%%%%%%%%%%%%%%%%%%%%%%%%%%%%%%%%%%%%
\bibitem{Koz80}  
K. Kozima, 
The comodule structure of $K_{*}(\Omega Sp(n))$, 
J. Math. Kyoto Univ.  {\bf 20} (1980),  315--325. 


%%%%%%%%%%%%%%%%%%%%%%%%%%%%%%%%%%%%%%%%%%%%%%%%%%%%%%%%%%%%%%%%%%%%%%%%%%%%%%%%%%%%%%%%%%%
\bibitem{Koz83}  
K. Kozima, 
The Hopf algebra structure of $MU_{*}(\Omega Sp(n))$, 
J. Math. Kyoto Univ. {\bf  23} (1983), 225--232. 



%%%%%%%%%%%%%%%%%%%%%%%%%%%%%%%%%%%%%%%%%%%%%%%%%%%%%%%%%%%%%%%%%%%%%%%%%%%%%%%%%%%%%%%%%%%%%
\bibitem{Lam2008} 
T. Lam, 
Schubert polynomials for the affine Grassmannian, 
J. Amer. Math. Soc. {\bf 21}, No.1 (2008), 259--281. 


%%%%%%%%%%%%%%%%%%%%%%%%%%%%%%%%%%%%%%%%%%%%%%%%%%%%%%%%%%%%%%%%%%%%%%%%%%%%%%%%%%%%%%%%%%%%
\bibitem{Lam2011}  
T. Lam, 
Affine Schubert classes, Schur positivity, and combinatorial 
Hopf algebras, 
Bull. London Math. Soc. {\bf 43} (2011),  328--334.  



%%%%%%%%%%%%%%%%%%%%%%%%%%%%%%%%%%%%%%%%%%%%%%%%%%%%%%%%%%%%%%%%%%%%%%%%%%%%%%%%%%%%%%%%%%%%
%\bibitem[Lam-Pyl07]{Lam-Pyl}
\bibitem{Lam-Pyl2007} 
    T.  Lam and P. Pylyavskyy,
Combinatorial Hopf algebras and $K$-homology of Grassmannians,
Int. Math. Res. Not. IMRN 2007, no.24, rnm 125. % arXiv:0705.2189. 
 
    
%%%%%%%%%%%%%%%%%%%%%%%%%%%%%%%%%%%%%%%%%%%%%%%%%%%%%%%%%%%%%%%%%%%%%%%%%%%%%%%%%%%%%%%%%%%
\bibitem{Lam-Sch-Shi2010I} 
T. Lam, A. Schilling, and M. Shimozono, 
Schubert polynomials  for the affine Grassmannian of the symplectic 
group, 
Math. Z.  {\bf 264} (2010), 765--811. 


%%%%%%%%%%%%%%%%%%%%%%%%%%%%%%%%%%%%%%%%%%%%%%%%%%%%%%%%%%%%%%%%%%%%%%%%%%%%%%%%%%%%%%%%%%%
\bibitem{Lam-Sch-Shi2010II} 
T. Lam, A. Schilling, and M. Shimozono, 
$K$-theory Schubert calculus of the affine Grassmannian, 
Compos. Math. {\bf 146} (2010),  811--852. 



%%%%%%%%%%%%%%%%%%%%%%%%%%%%%%%%%%%%%%%%%%%%%%%%%%%%%%%%%%%%%%%%%%%%%%%%%%%%%%%%%%%%%%%%%%% 
%\bibitem[Lam-Shi]{Lam-Shi}
\bibitem{Lam-Shi2013}    
T. Lam and M. Shimozono,
$k$-double Schur functions and equivariant $($co$)$homology of
the affine Grassmannian, 
%arXiv:1105.2170, 
Math. Ann. {\bf 356} (2013), no.4, 1379--1404. 


%%%%%%%%%%%%%%%%%%%%%%%%%%%%%%%%%%%%%%%%%%%%%%%%%%%%%%%%%%%%%%%%%%%%%%%%%%%%%%%%%%%%%%%%%
%\bibitem[Laz55]{Laz}
\bibitem{Laz55}  
M.  Lazard,
Sur les groupes de Lie formels \`a un param\`etre,
Bull. Soc. Math. France  {\bf  83} (1955), 251--274.


%%%%%%%%%%%%%%%%%%%%%%%%%%%%%%%%%%%%%%%%%%%%%%%%%%%%%%%%%%%%%%%%%%%%%%%%%%%%%%%%%%%%%%%%%
%\bibitem[L-M07]{Lev-Mor}
\bibitem{Lev-Mor2007}  
M.  Levine and F.  Morel,
Algebraic Cobordism, 
Springer Monographs in Math. 2007. 



%\bibitem{Len96} C. Lenart, 
%\newblock{Combinatorial models for certain structures in algebraic topology 
%and formal group theory}, 
%\newblock{Thesis submitted to the University of Manchester (May 1996).} 


%%%%%%%%%%%%%%%%%%%%%%%%%%%%%%%%%%%%%%%%%%%%%%%%%%%%%%%%%%%%%%%%%%%%%%%%%%%%%%%%%%%%%%%%%%%%%%
\bibitem{Len98}  
C. Lenart, 
Symmetric functions, formal group laws, and Lazard's theorem, 
Adv. Math. {\bf 134} (1998), 219--239. 


%%%%%%%%%%%%%%%%%%%%%%%%%%%%%%%%%%%%%%%%%%%%%%%%%%%%%%%%%%%%%%%%%%%%%%%%%%%%%%%%%%%%%%%%%%%%%%%%
\bibitem{Liu80}  
A. Liulevicius, 
Arrows, symmetries, and representation rings, 
J. of Pure and App. Alg.  {\bf 19} (1980),   259--273. 


%%%%%%%%%%%%%%%%%%%%%%%%%%%%%%%%%%%%%%%%%%%%%%%%%%%%%%%%%%%%%%%%%%%%%%%%%%%%%%%%%%%%%%%%%%%%%%%%%
\bibitem{Mac95}  
I. G. Macdonald, 
Symmetric functions and Hall polynomials, 
2nd edition, Oxford  Univ. Press, Oxford, 1995.



%%%%%%%%%%%%%%%%%%%%%%%%%%%%%%%%%%%%%%%%%%%%%%%%%%%%%%%%%%%%%%%%%%%%%%%%%%%%%%%%%%%%%%%%%%%%%%%
%\bibitem[May96]{May}
%\bibitem{May96}
%J.  P.  May, 
%Equivariant homotopy and cohomology theory, 
%vol. {\bf  91}, CBMS Regional Conference
%Series, 1996. 


%%%%%%%%%%%%%%%%%%%%%%%%%%%%%%%%%%%%%%%%%%%%%%%%%%%%%%%%%%%%%%%%%%%%%%%%%%%%%%%%%%%%%%%%%%%%%%%%
\bibitem{May99} 
J. P. May,    
A Concise Course in Algebraic Topology, 
Chicago Lectures in Mathematics, The Univ. Chicago Press, Chicago and London, 
1999. 



\bibitem{May-Pon2012} 
J. P. May and K. Ponto, 
More Concise Algebraic Topology: Localization, completion, and model categories, 
Chicago Lectures in Mathematics, The Univ. Chicago Press, Chicago and London, 
2012. 



%%%%%%%%%%%%%%%%%%%%%%%%%%%%%%%%%%%%%%%%%%%%%%%%%%%%%%%%%%%%%%%%%%%%%%%%%%%%%%%%%%%%%%%%%%%%%%%%%
\bibitem{McN2006}  
P. J. McNamara, 
Factorial Grothendieck polynomials, 
Electron. J. Combin., {\bf 13} (2006),  no.1, Research Paper 71. 

 
   
%\bibitem{Mil94}  H. Miller, 
%\newblock{\em Notes on Cobordism}, 
%\newblock{Notes based on lectures of Haynes Miller, Spring, 1994.} 


%%%%%%%%%%%%%%%%%%%%%%%%%%%%%%%%%%%%%%%%%%%%%%%%%%%%%%%%%%%%%%%%%%%%%%%%%%%%%%%%%%%%%%%%%%%%%%%%%%%
\bibitem{Mit86}  
S. A. Mitchell, 
A filtration of the loops on $SU(n)$ by Schubert varieties, 
Math. Z. {\bf 193}, no.3 (1986), 347--362. 


%%%%%%%%%%%%%%%%%%%%%%%%%%%%%%%%%%%%%%%%%%%%%%%%%%%%%%%%%%%%%%%%%%%%%%%%%%%%%%%%%%%%%%%%%%%%%%%%%
\bibitem{Mit87}  
S. A. Mitchell, 
The Bott filtration of a loop group, 
Algebraic topology, Barcelona, 1986, 215--226, 
Lecture Notes in Math., {\bf 1298}, Springer, Berlin, 1987.  

%%%%%%%%%%%%%%%%%%%%%%%%%%%%%%%%%%%%%%%%%%%%%%%%%%%%%%%%%%%%%%%%%%%%%%%%%%%%%%%%%%%%%%%%%%%%%%%%
\bibitem{Mit88} 
S. A. Mitchell, 
Quillen's theorem on buildings and the loops on a symmetric 
space, 
Enseign. Math. (2)  {\bf 34} (1988), no.1-2, 123--166. 



%%%%%%%%%%%%%%%%%%%%%%%%%%%%%%%%%%%%%%%%%%%%%%%%%%%%%%%%%%%%%%%%%%%%%%%%%%%%%%%%%%%%%%%%%%%%%%%%   
%\bibitem[Mol09]{Mol}
\bibitem{Mol2009}  
    A. I.  Molev,
Comultiplication rules for the double Schur functions and
    Cauchy identities,
Electron. J.   Combin. 16 (2009),  Research paper 13, 44 pp.   %arXiv:0808.2127. 



%%%%%%%%%%%%%%%%%%%%%%%%%%%%%%%%%%%%%%%%%%%%%%%%%%%%%%%%%%%%%%%%%%%%%%%%%%%%%%%%%%%%%%%%%%%%%%%%%
\bibitem{Mol-Sag99}  
A. I.  Molev and B. E. Sagan, 
A Littlewood-Richardson rule for factorial Schur functions, 
Trans. Amer. Math. Soc., {\bf 351} (1999),   4429--4443.  




%\bibitem[Nar]{Nar}
%\bibitem{Nar} 
%H.  Naruse,
%\newblock{\em  Homology factorial Schur $P,Q$-functions}, 
%\newblock{In preparation.}  




%%%%%%%%%%%%%%%%%%%%%%%%%%%%%%%%%%%%%%%%%%%%%%%%%%%%%%%%%%%%%%%%%%%%%%%%%%%%%%%%%%%%%%%%%%%%%%%%%%%
%\bibitem[Pra91]{Pra}
\bibitem{Pra91}    
P.  Pragacz, 
Algebro-geometric applications of  Schur $S$- and 
$Q$-polynomials, 
Topics in invariant theory (Paris, 1989/1990), 130--191, 
Lecture Notes in Math.,  {\bf 1478}, Springer, Berlin, 1991. 


%%%%%%%%%%%%%%%%%%%%%%%%%%%%%%%%%%%%%%%%%%%%%%%%%%%%%%%%%%%%%%%%%%%%%%%%%%%%%%%%%%%%%%%%%%%%%%%%%%%%%%%%
\bibitem{Pet97} D. Peterson, 
Quantum Cohomology of $G/P$,  
Lecture Notes at MIT, Spring, 1997. 



%%%%%%%%%%%%%%%%%%%%%%%%%%%%%%%%%%%%%%%%%%%%%%%%%%%%%%%%%%%%%%%%%%%%%%%%%%%%%%%%%%%%%%%%%%%%%%%%%%%%%%
%\bibitem{Pit72}  
%H. Pittie, 
%Homegeneous vector bundles on homogeneous spaces, 
%Topology \textbf{11} (1972),   199--203. 


%%%%%%%%%%%%%%%%%%%%%%%%%%%%%%%%%%%%%%%%%%%%%%%%%%%%%%%%%%%%%%%%%%%%%%%%%%%%%%%%%%%%%%%%%%%%%%%%%%%%%
\bibitem{Qui69}  
D. Quillen, 
On the formal group laws of unoriented and complex cobordism theory, 
Bull. Amer. Math. Soc., {\bf 75}, No.6 (1969), 1293--1298. 



%\bibitem{Rao89}  V. Rao, 
%\newblock{\em  The Hopf algebra structure of the complex bordism of the loop spaces of 
%               the special orthogonal groups}, 
%\newblock{Indiana Univ. Math. J. \textbf{38} (1989), 277--291.}


%%%%%%%%%%%%%%%%%%%%%%%%%%%%%%%%%%%%%%%%%%%%%%%%%%%%%%%%%%%%%%%%%%%%%%%%%%%%%%%%%%%%%%%%%%%%%%%%%%%%%%
\bibitem{Rav2004} 
D. C. Ravenel, 
Complex cobordism and stable homotopy groups of spheres, 
2nd ed.  AMS Chelsea Publishing, Amer. Math. Soc., Providence, Rhode Island, 2004. 


%\bibitem{Rav86}  D. C. Ravenel, 
%\newblock{\em  Complex cobordism theory for number theorists}, 
%\newblock{Elliptic curves and modular forms in algebraic topology, Princeten, NJ, USA 1986, 
%Lecture Notes in Math. \textbf{1326}, 123--133.}  



\bibitem{Ray74}  N. Ray,  
Bordism  $J$-homomorphisms, 
Illinois J. Math.  \textbf{18} (1974), 290--309. 

%%%%%%%%%%%%%%%%%%%%%%%%%%%%%%%%%%%%%%%%%%%%%%%%%%%%%%%%%%%%%%%%%%%%%%%%%%%%%%%%%%%%%%%%%%%%%%%%%%%%
\bibitem{Sta99} 
R. P. Stanley, 
Enumerative Combinatorics. Volume $2$, 
Cambridge Studies in Advanced Mathemtics  {\bf 62}. 
Cambridge University Press, Cambridge, 1999. 


%%%%%%%%%%%%%%%%%%%%%%%%%%%%%%%%%%%%%%%%%%%%%%%%%%%%%%%%%%%%%%%%%%%%%%%%%%%%%%%%%%%%%%%%%%%%%%%%%%
\bibitem{Ste89} 
J. R. Stembridge, 
Shifted tableaux and the projective representations of symmetric groups, 
Adv. in Math.  {\bf 74} (1989),  87--134.  


%%%%%%%%%%%%%%%%%%%%%%%%%%%%%%%%%%%%%%%%%%%%%%%%%%%%%%%%%%%%%%%%%%%%%%%%%%%%%%%%%%%%%%%%%%%%%%%%%
\bibitem{Swi75}  
R. Switzer, 
Algebraic Topology  -Homology and Homotopy, 
Classics in Mathematics, Reprint of the 1975 Edition, Springer-Verlag, 
Berlin, 2002. 




%\bibitem{Wat75}T. Watanabe, 
%\newblock{\em The integral cohomology ring of the symmetric space $EVII$}, 
%\newblock{J. Math. Kyoto Univ. \textbf{15} (1975), 363--385.} 


%\bibitem{Wil84} S. Wilson, 
%\newblock{\em  The complex cobordism of $BO_{n}$}, 
%\newblock{J. London Math. Soc. (2), \textbf{29} (1984),  352--366.} 




%\bibitem{Wil??} S. Wilson, 
%\newblock{\em  Brown-Peterson homology: An introduction and sampler}, 
%\newblock{Regional conference series in mathemetics \textbf{48}, AMS.} 


%\bibitem{Yag05}N. Yagita, 
%\newblock{\em   Algebraic cobordism of simply connected Lie groups}, 
%\newblock{Math. Proc. Camb. Phil. Soc.  \textbf{139} (2005), 243--260.} 






\end{thebibliography}
\end{document}